\documentclass[11pt]{article}

\usepackage[a4paper,margin=1in]{geometry}
\usepackage[T1]{fontenc}
\usepackage[utf8]{inputenc}
\usepackage{lmodern}
\usepackage{microtype}
\usepackage{amsmath,amssymb,mathtools}
\usepackage{amsthm}
\usepackage{longtable,array,booktabs}
\usepackage{tikz}
\allowdisplaybreaks
\usepackage{hyperref}

\hypersetup{hidelinks}

\newcommand{\PP}{\mathbb{P}}
\newcommand{\C}{\mathbb{C}}
\newcommand{\Gm}{\mathbb{G}_m}
\newcommand{\SL}{\operatorname{SL}}
\newcommand{\Sym}{\operatorname{Sym}}
\newcommand{\diag}{\operatorname{diag}}
\newcommand{\Span}{\operatorname{Span}}
\newcommand{\Supp}{\operatorname{Supp}}
\newcommand{\Conv}{\operatorname{Conv}}
\newcommand{\Int}{\operatorname{Int}}
\newcommand{\Sing}{\operatorname{Sing}}
\newcommand{\rank}{\operatorname{rank}}
\newcommand{\wt}{\operatorname{wt}}
\newcommand{\mexp}{\widetilde{\alpha}}
\newcommand{\nf}{\phi^{\mathrm{nf}}}

\newcommand{\PGL}{\mathrm{PGL}}
\newcommand{\GL}{\mathrm{GL}}
\theoremstyle{plain}
\newtheorem{theorem}{Theorem}[section]
\newtheorem{proposition}[theorem]{Proposition}
\newtheorem{lemma}[theorem]{Lemma}
\newtheorem{corollary}[theorem]{Corollary}

\newtheorem*{theoremA}{Theorem A}

\theoremstyle{definition}
\newtheorem{definition}[theorem]{Definition}

\theoremstyle{remark}
\newtheorem{remark}[theorem]{Remark}

\theoremstyle{plain}

\newcounter{casecounter}
\renewcommand{\theequation}{\arabic{equation}}
\providecommand*{\theHequation}{\arabic{equation}}
\providecommand{\QH}[2]{\mathrm{QH}(#1)_{#2}}
\providecommand{\Cstar}{\mathbb C^{\times}}

\newcommand{\StartCase}[1]{%
  \setcounter{casecounter}{#1}%
  \setcounter{equation}{0}%
  \renewcommand{\theequation}{\arabic{equation}}%
  \renewcommand{\theHequation}{case.#1.eq.\arabic{equation}}%
 
  \phantomsection
  \hypertarget{case-k-#1}{}%
  \label{case:#1}%
  \pdfbookmark[2]{Case k=#1}{bookmark-case-k-#1}%
  \subsection*{Case $k=#1$}%
}

\newcommand{\Sfam}[2]{\mathbf{S}_{#1}^{\mathrm{Y}#2}}

\title{The GIT Boundary of Quintic Threefolds}
\author{Yasutaka Shibata}
\date{}

\begin{document}
\maketitle

\begin{abstract}
We determine the quotient-side strictly semistable boundary for the natural
\(G=\SL(5)\)-action on \(\PP(\Sym^5\C^5)\), proving that it has exactly
twenty-one irreducible components, of dimensions between \(1\) and \(24\).
Using the Hilbert--Mumford criterion and exact integer linear algebra, we
enumerate, up to coordinate permutation, thirty-eight maximal strictly
\(T\)-semistable monomial supports for the diagonal maximal torus \(T\) and
prove that their \(G\)-saturations cover the strictly semistable locus.
Zero-weight reduction shows that reversing a supporting half-space does not
change its quotient image; the thirty-eight supports therefore form seventeen
pairs and four fixed cases.

For each support we construct a polystable weight-zero normal form and compute
the corresponding quotient dimension.  Exact rank certificates identify the
generic connected projective stabilizers of the resulting twenty-one families
as pairwise nonconjugate one-dimensional tori.  Luna's \'{e}tale slice theorem
then excludes containments between distinct families.  Lakhani's four
first-level families give the four largest components; the other seventeen are
additional components.

For a general closed-orbit representative in every component, we determine the
saturated Jacobian scheme and local analytic stratification.  The
positive-dimensional singular loci are explicit configurations of lines,
smooth conics, cuspidal plane curves, planes, and smooth quartic surfaces.  The
isolated singularities form exactly eleven weighted-homogeneous local families
whose normalized weights occur in Yonemura's list.  Every general representative
has global minimal exponent \(\mexp(X)=1=(4+1)/5\), the critical value
\((n+1)/d\) for quintic hypersurfaces in \(\PP^4\).
\end{abstract}

\section{Introduction}
\label{sec:introduction}

\subsection*{The problem and the complete boundary catalog}

Let
\[
  V:=H^0\bigl(\PP^4,\mathcal O_{\PP^4}(1)\bigr)\simeq\C^5,
  \qquad
  W:=\Sym^5V,
  \qquad
  G:=\SL(V)\simeq\SL(5).
\]
A nonzero form \(f\in W\) defines a quintic threefold
\[
  X_f:=V(f)\subset\PP(V^\vee)\simeq\PP^4,
\]
and a smooth such hypersurface is a Calabi--Yau threefold.  The natural
projective GIT quotient is
\[
  \mathcal M^{\mathrm{GIT}}
  :=
  \PP(W)^{ss}//G.
\]
We denote the good quotient map by
\[
  \pi:\PP(W)^{ss}\longrightarrow\mathcal M^{\mathrm{GIT}}
\]
and write
\[
  \Sigma^{ss}:=\PP(W)^{ss}\setminus\PP(W)^s.
\]
The object classified in this paper is the quotient of the strictly
semistable locus,
\[
  \partial\mathcal M^{\mathrm{GIT}}
  :=
  \pi(\Sigma^{ss}).
\]
Thus the term \emph{GIT boundary} always means
\(\pi(\PP(W)^{ss}\setminus\PP(W)^s)\); it does not mean the entire
complement of the smooth-quintic locus in \(\mathcal M^{\mathrm{GIT}}\).
For later reference, if
\(r=(r_0,\ldots,r_4)\in\mathbb Z^5\) satisfies
\(\sum_{i=0}^4r_i=0\), we write
\(\lambda_r(t):=\diag(t^{r_0},\ldots,t^{r_4})\), \(t\in\Gm\), for the
corresponding one-parameter subgroup of the diagonal maximal torus, and put
\(H_r:=\lambda_r(\Gm)\).

The purpose of the present introduction is to record the complete output of
the paper before explaining the proofs.  The twenty-one explicit
normal-form families are given in Table~\ref{tab:intro-normal-forms}; the
component classification is listed in
Table~\ref{tab:intro-twenty-one-components}; the singularities of a general
closed orbit in every component are listed in
Table~\ref{tab:intro-generic-singularities}; and the defining equations, weights, and occurrences of the isolated
weighted-homogeneous local families are displayed in
Table~\ref{tab:intro-isolated-families}.  These four catalogs use the same
component index \(a=1,\ldots,21\), so that the equation, quotient-theoretic
invariants, and singularity data of each component can be read across the
introduction.  Throughout the catalogs, \emph{general} means after restriction
to a fixed nonempty Zariski-open subset of the parameter space in the
corresponding row; all orbit-theoretic and singularity-theoretic assertions
are made on that open subset.

\begin{theoremA}[Complete classification of the strictly semistable boundary]
For the natural \(\SL(5)\)-action on \(\PP(\Sym^5\C^5)\), the following
statements hold.
\begin{enumerate}
\item Up to permutation of the coordinates, there are exactly thirty-eight
maximal strictly semistable monomial supports for the diagonal maximal torus.
Their \(G\)-saturations cover \(\Sigma^{ss}\).

\item The thirty-eight support-indexed quotient families are identified by
seventeen two-cycles and four fixed points under reversal of the supporting
half-space.  The resulting twenty-one closed irreducible families are
precisely the irreducible components
\[
  \partial\mathcal M^{\mathrm{GIT}}
  =\Theta_1\cup\cdots\cup\Theta_{21}.
\]
There is no equality or containment between two distinct members of this
list.

\item For every \(a=1,\ldots,21\), a general point of \(\Theta_a\) is
represented by a member of the explicit polystable weight-zero normal-form
family \(F_a\) displayed in Table~\ref{tab:intro-normal-forms}.  The
parameters are restricted to the indicated nonempty Zariski-open generic
loci; no uniqueness of the displayed parameterizations is asserted.  Its
quotient dimension and generic connected projective stabilizer are
\[
  \dim\Theta_a=d_a,
  \qquad
  \operatorname{Stab}_G([F_a])^\circ
  =H_a:=H_{r_a}=\lambda_{r_a}(\Gm),
\]
where the complete list of \(\Theta_a\), \(r_a\), and \(d_a\) is given in
Table~\ref{tab:intro-twenty-one-components}.  The twenty-one tori \(H_a\)
are pairwise nonconjugate in \(G\).

\item For \(X_a:=V(F_a)\subset\PP^4\), the support and local stratification of
the saturated Jacobian scheme, including its generic transverse structure and
all isolated singular points, are summarized in
Tables~\ref{tab:intro-generic-singularities} and
\ref{tab:intro-isolated-families}.  Exactly eleven isolated
weighted-homogeneous local normal-form families occur.  Their defining
equations, integral weight systems, Milnor numbers, and component occurrences
are displayed in Table~\ref{tab:intro-isolated-families}; the corresponding
ordered normalized weights all occur in Yonemura's Table~2.2~\cite{Yon90}.
No other isolated local family occurs for a general boundary representative.

\item Every general closed-orbit representative has global minimal exponent
\[
  \mexp(X_a)=1=\frac{4+1}{5},
  \qquad
  a=1,\ldots,21.
\]
\end{enumerate}
Tables~\ref{tab:intro-normal-forms}--\ref{tab:intro-isolated-families}
form part of the statement of the theorem.
\end{theoremA}

\medskip
\noindent\textbf{Logical architecture.}
The quotient-side component count is established in three distinct stages.
Proposition~\ref{prop:strictly-semistable-cover} first gives a cover by
thirty-eight support-indexed quotient images.  Zero-weight reduction and
reversal of the supporting half-space, formalized in
Corollary~\ref{cor:supporting-hyperplane-invariance} and
Proposition~\ref{prop:opposite-half-space-identifications}, reduce this to the
twenty-one candidate images of
Corollary~\ref{cor:twenty-one-quotient-side-candidate-cover}.
Corollary~\ref{cor:twenty-one-quotient-side-cover} proves that those candidates
are nonempty, closed, and irreducible.
Theorem~\ref{thm:pairwise-noninclusion} then uses the generic stabilizers and
Luna's slice theorem to prove pairwise non-inclusion, so the candidates are
exactly the irreducible components.
Appendices~\ref{app:normal-forms} and~\ref{app:singular-analysis} provide,
respectively, the support-by-support normal-form constructions and the
component-by-component singularity and minimal-exponent calculations.

\subsection*{The twenty-one normal-form families}

For the component numbering in Theorem~A, we choose the support-case
representatives
\[
  k(a)=1,2,3,4,5,6,7,8,9,10,11,12,13,14,15,19,20,23,29,32,36
\]
in that order and set
\[
  F_a:=\nf_{k(a)},
  \qquad
  X_a:=V(F_a)\subset\PP^4.
\]
Thus \(k(a)=a\) for \(1\leq a\leq15\), whereas the component index and the
chosen oriented-support label differ for \(a\geq16\).  The column \(k(a)\)
should therefore be retained when passing between the twenty-one-row catalog
and the thirty-eight support-by-support calculations.
The following table reproduces the complete equations rather than merely
referring to the case-by-case construction later in the paper.  It is meant
to be directly usable: a reader can recover a general closed-orbit equation
for any boundary component without consulting the thirty-eight-case table.
The displayed families are normal forms in the sense of explicit
polystable representatives after a change of projective coordinates; the
parameterizations are not claimed to be unique.

We use
\[
  B_{\lambda}(u,v):=uv(u-v)(u-\lambda v),
  \qquad
  B_{\lambda,\mu}(u,v)
  :=v(u-v)(u-\lambda v)(u-\mu v).
\]
Symbols such as \(A_d,B_d,C_d,D_d,Q_d,L_d,M_d\) denote general forms of the
indicated degree in the displayed variables, and \(Q_{2,2}\) denotes a
bihomogeneous form of bidegree \((2,2)\).  A bold coefficient vector
parametrizes the binary form whose degree is forced by homogeneity.  All
parameters are understood to lie in the relevant nonempty Zariski-open
locus, in addition to the explicit conditions in the last column.

\begingroup
\scriptsize
\renewcommand{\arraystretch}{1.18}
\setlength{\tabcolsep}{1.2pt}
\setlength{\extrarowheight}{1pt}
\providecommand{\IntroNFcell}[1]{\(\begin{aligned}[t]#1\end{aligned}\)}
\providecommand{\IntroParamcell}[1]{\begin{tabular}[t]{@{}l@{}}#1\end{tabular}}
\begin{longtable}{@{}
  >{\centering\arraybackslash}p{0.058\textwidth}
  >{\centering\arraybackslash}p{0.045\textwidth}
  >{\raggedright\arraybackslash}p{0.625\textwidth}
  >{\raggedright\arraybackslash}p{0.245\textwidth}
@{}}
\caption{The twenty-one polystable weight-zero normal-form families.  A general point of \(\Theta_a\) is represented by a member of the family \(F_a\) in row \(a\).  The label \(k(a)\) refers to the corresponding case among the thirty-eight oriented support cases.  All displayed parameters are restricted to the relevant nonempty Zariski-open generic locus.  This table forms part of Theorem~A.}
\label{tab:intro-normal-forms}\\
\toprule
\(\Theta_a\) & \(k\) & Normal-form family \(F_a\) & Parameters and genericity\\
\midrule
\endfirsthead
\toprule
\(\Theta_a\) & \(k\) & Normal-form family \(F_a\) & Parameters and genericity\\
\midrule
\endhead
\midrule
\multicolumn{4}{r}{\emph{continued on the next page}}\\
\endfoot
\bottomrule
\endlastfoot
\(\Theta_{1}\) & 1 & \IntroNFcell{
F_{1}={}&x_1^4x_2+x_0x_3^4+x_0x_2^3x_4\\
&+x_0x_1x_2x_3x_4+\alpha x_0^2x_4^3
} & \IntroParamcell{$\alpha\in\Cstar$} \\
\midrule

\(\Theta_{2}\) & 2 & \IntroNFcell{
F_{2}={}&x_1^3x_2^2+x_1^4x_3+x_0x_2x_3^3+x_0x_2^3x_4\\
&+\alpha x_0x_1x_2x_3x_4+\beta x_0^2x_4^3
} & \IntroParamcell{$\alpha,\beta\in\Cstar$} \\
\midrule

\(\Theta_{3}\) & 3 & \IntroNFcell{
F_{3}={}&B_5(x_1,x_2)+x_0x_2x_3^3\\
&+x_0(x_1^2+x_2^2)x_3x_4+x_0^2x_4^3
} & \IntroParamcell{$B_5$: general\\ binary quintic} \\
\midrule

\(\Theta_{4}\) & 4 & \IntroNFcell{
F_{4}={}&x_1^2x_2^3+x_1^3x_3x_4+x_0x_2^4+x_0x_1x_3^3\\
&+\alpha x_0x_1x_2x_3x_4+\beta x_0^2x_4^3
} & \IntroParamcell{$\alpha,\beta\in\Cstar$} \\
\midrule

\(\Theta_{5}\) & 5 & \IntroNFcell{
F_{5}={}&x_1^2x_2^3+x_1^3x_3^2+x_1^3x_2x_4+x_0x_2^3x_3\\
&+\alpha x_0x_1x_3^3+\beta x_0x_1x_2x_3x_4+\gamma x_0^2x_4^3
} & \IntroParamcell{$\alpha,\beta,\gamma\in\Cstar$} \\
\midrule

\(\Theta_{6}\) & 6 & \IntroNFcell{
F_{6}={}&x_1^2x_2^3+x_1^3x_2x_3+x_1^4x_4+x_0x_2^2x_3^2\\
&+\alpha x_0x_2^3x_4+\beta x_0x_1x_3^3+\gamma x_0x_1x_2x_3x_4\\
&+\delta x_0x_1^2x_4^2+\epsilon x_0^2x_4^3
} & \IntroParamcell{$\alpha,\beta,\gamma,\delta,$\\ $\epsilon\in\Cstar$} \\
\midrule

\(\Theta_{7}\) & 7 & \IntroNFcell{
F_{7}={}&x_2^5+x_1^2x_2^2x_3+x_1^4x_4+x_0x_1x_3^3\\
&+\alpha x_0x_1x_2x_3x_4+\beta x_0^2x_3x_4^2
} & \IntroParamcell{$\alpha,\beta\in\Cstar$} \\
\midrule

\(\Theta_{8}\) & 8 & \IntroNFcell{
F_{8}={}&x_1^2x_2^2x_3+x_1^3x_4^2+x_0x_2^4+x_0x_1x_3^3\\
&+\alpha x_0x_1x_2x_3x_4+\beta x_0^2x_3x_4^2
} & \IntroParamcell{$\alpha,\beta\in\Cstar$} \\
\midrule

\(\Theta_{9}\) & 9 & \IntroNFcell{
F_{9}={}&x_1x_2^4+x_1^2x_2^2x_3+\alpha x_1^3x_3^2+x_1^4x_4\\
&+x_0x_2x_3^3+\beta x_0x_2^3x_4+\gamma x_0x_1x_2x_3x_4\\
&+\delta x_0^2x_3x_4^2
} & \IntroParamcell{$\alpha,\beta,\gamma,\delta\in\Cstar$} \\
\midrule

\(\Theta_{10}\) & 10 & \IntroNFcell{
F_{10}={}&x_1^3Q_{\mathbf q}(x_2,x_3)+x_1^4x_4+x_0B_\lambda(x_2,x_3)\\
&+x_0x_1x_4R_{\mathbf r}(x_2,x_3)+\beta x_0x_1^2x_4^2+x_0^2x_4^3
} & \IntroParamcell{$\lambda\in\C\setminus\{0,1\}$;\\ $\mathbf q,\mathbf r\in\C^3$, $\beta\in\C$} \\
\midrule

\(\Theta_{11}\) & 11 & \IntroNFcell{
F_{11}={}&x_1x_2^4+\alpha x_1^2x_2^2x_3+\beta x_1^3x_3^2+x_1^3x_2x_4\\
&+\gamma x_0x_2^2x_3^2+x_0x_2^3x_4+\delta x_0x_1x_3^3\\
&+\epsilon x_0x_1x_2x_3x_4+\zeta x_0x_1^2x_4^2+x_0^2x_3x_4^2
} & \IntroParamcell{$\alpha,\beta,\gamma,\delta,$\\ $\epsilon,\zeta\in\Cstar$} \\
\midrule

\(\Theta_{12}\) & 12 & \IntroNFcell{
F_{12}={}&x_1x_2^4+\alpha x_1^2x_2x_3^2+\beta x_1^2x_2^2x_4+x_1^3x_4^2\\
&+x_0x_2^3x_3+\gamma x_0x_1x_3^3+\delta x_0x_1x_2x_3x_4\\
&+\epsilon x_0^2x_3^2x_4+x_0^2x_2x_4^2
} & \IntroParamcell{$\alpha,\beta,\gamma,\delta,$\\ $\epsilon\in\Cstar$} \\
\midrule

\(\Theta_{13}\) & 13 & \IntroNFcell{
F_{13}={}&x_2^5+\alpha x_1x_2^3x_3+\beta x_1^2x_2x_3^2+x_1^3x_4^2\\
&+x_0x_2^3x_4+\gamma x_0x_1x_2x_3x_4+x_0^2x_3^3\\
&+\delta x_0^2x_2x_4^2
} & \IntroParamcell{$\alpha,\beta,\gamma,\delta\in\Cstar$} \\
\midrule

\(\Theta_{14}\) & 14 & \IntroNFcell{
F_{14}={}&x_2^5+\alpha x_1x_2^3x_3+\beta x_1^2x_2x_3^2+x_1^4x_4+x_0x_3^4\\
&+\gamma x_0x_2^3x_4+\delta x_0x_1x_2x_3x_4+x_0^2x_2x_4^2
} & \IntroParamcell{$\alpha,\beta,\gamma,\delta\in\Cstar$} \\
\midrule

\(\Theta_{15}\) & 15 & \IntroNFcell{
F_{15}={}&x_1x_2^4+x_1^2x_2^2x_3\\
&+x_1^3(\gamma x_3^2+\beta x_3x_4+x_4^2)+x_0x_2^3x_4\\
&+x_0x_1x_2R(x_3,x_4)+x_0^2B(x_3,x_4)
} & \IntroParamcell{$\gamma,\beta\in\C$;\\
$R\in\Sym^2\langle x_3,x_4\rangle$;\\
$B\in\Sym^3\langle x_3,x_4\rangle$} \\
\midrule

\(\Theta_{16}\) & 19 & \IntroNFcell{
F_{16}={}&x_2^5+\alpha x_1x_2^3x_3+\beta x_1^2x_2x_3^2+\gamma x_1^2x_2^2x_4\\
&+x_1^3x_3x_4+\delta x_0x_2^2x_3^2+\epsilon x_0x_2^3x_4+x_0x_1x_3^3\\
&+\zeta x_0x_1x_2x_3x_4+\xi x_0x_1^2x_4^2+x_0^2x_3^2x_4\\
&+\theta x_0^2x_2x_4^2
} & \IntroParamcell{$\alpha,\beta,\gamma,\delta,$\\ $\epsilon,\zeta,\xi,\theta\in\Cstar$} \\
\midrule

\(\Theta_{17}\) & 20 & \IntroNFcell{
F_{17}={}&x_0B_{\lambda,\mu}(x_2,x_3)+x_1^2C_{\mathbf c}(x_2,x_3)+x_1^3x_2x_4\\
&+x_0x_1x_4Q_{\mathbf q}(x_2,x_3)+x_0x_1^2x_4^2\\
&+x_0^2x_4^2M_{\mathbf m}(x_2,x_3)
} & \IntroParamcell{$\lambda,\mu$ distinct in\\ $\C\setminus\{0,1\}$;\\ $\mathbf c\in\C^4$, $\mathbf q\in\C^3$,\\ $\mathbf m\in\C^2$} \\
\midrule

\(\Theta_{18}\) & 23 & \IntroNFcell{
F_{18}={}&x_0B_4(x_1,x_2,x_3,x_4)
} & \IntroParamcell{$B_4$: smooth quartic\\ surface in $\PP^3$} \\
\midrule

\(\Theta_{19}\) & 29 & \IntroNFcell{
F_{19}={}&x_3A_4(x_1,x_2)+x_4C_4(x_1,x_2)\\
&+x_0Q_{2,2}(x_1,x_2;x_3,x_4)+x_0^2x_3x_4(x_3-x_4)
} & \IntroParamcell{$A_4,C_4$: binary quartics;\\ $Q_{2,2}$ of bidegree $(2,2)$} \\
\midrule

\(\Theta_{20}\) & 32 & \IntroNFcell{
F_{20}={}&x_0^2C_0(x_2,x_3,x_4)+x_0x_1C_1(x_2,x_3,x_4)\\
&+x_1^2C_2(x_2,x_3,x_4)
} & \IntroParamcell{$C_0,C_1,C_2$: ternary cubics} \\
\midrule

\(\Theta_{21}\) & 36 & \IntroNFcell{
F_{21}={}&B_5(x_1,x_2,x_3)+x_0x_4C_3(x_1,x_2,x_3)\\
&+x_0^2x_3x_4^2
} & \IntroParamcell{$B_5,C_3$: ternary forms\\ of degrees $5,3$} \\
\end{longtable}
\endgroup

\subsection*{The twenty-one irreducible components at a glance}

For each oriented support case \(k\), let \(\Phi_k\) denote the image in
\(\mathcal M^{\mathrm{GIT}}\) of the \(G\)-saturated family determined by the
\(k\)-th maximal support; the formal definition is given in
Section~\ref{sec:boundary-families-sl5}.
The normal-form table chooses one of the thirty-eight support cases for each
component and denotes its label by \(k(a)\).  The vector \(r_a\) in the
following table is primitive and defines
both the supporting hyperplane and the one-dimensional torus
\[
  H_a=H_{r_a}=\lambda_{r_a}(\Gm).
\]
The equality in the second column records the complete orbit of the original
support label under reversal of the supporting half-space.  In particular,
the table simultaneously records the reduction from thirty-eight oriented
supports to twenty-one quotient components, the chosen closed-orbit case,
the generic connected stabilizer, and the quotient dimension.

\begingroup
\scriptsize
\renewcommand{\arraystretch}{1.13}
\setlength{\tabcolsep}{2.5pt}
\begin{longtable}{@{}
  >{\centering\arraybackslash}p{0.035\textwidth}
  >{\raggedright\arraybackslash}p{0.245\textwidth}
  >{\centering\arraybackslash}p{0.060\textwidth}
  >{\centering\arraybackslash}p{0.470\textwidth}
  >{\centering\arraybackslash}p{0.055\textwidth}
@{}}
\caption{The complete quotient-side component catalog.  For a general closed-orbit representative in row \(a\), one has \(\operatorname{Stab}_G([F_a])^\circ=\lambda_{r_a}(\Gm)\).}
\label{tab:intro-twenty-one-components}\\
\toprule
\(a\) & Quotient component \(\Theta_a\) & \(k(a)\) & Supporting normal \(r_a\) & \(d_a\)\\
\midrule
\endfirsthead
\toprule
\(a\) & Quotient component \(\Theta_a\) & \(k(a)\) & Supporting normal \(r_a\) & \(d_a\)\\
\midrule
\endhead
\midrule
\multicolumn{5}{r}{\emph{continued on the next page}}\\
\endfoot
\bottomrule
\endlastfoot
1  & \(\Theta_1=\Phi_1=\Phi_{26}\)    & 1  & \((36,1,-4,-9,-24)\)  & 1\\
2  & \(\Theta_2=\Phi_2=\Phi_{27}\)    & 2  & \((27,2,-3,-8,-18)\)  & 2\\
3  & \(\Theta_3=\Phi_3=\Phi_{33}\)    & 3  & \((3,0,0,-1,-2)\)     & 6\\
4  & \(\Theta_4=\Phi_4=\Phi_{21}\)    & 4  & \((24,9,-6,-11,-16)\) & 2\\
5  & \(\Theta_5=\Phi_5=\Phi_{24}\)    & 5  & \((21,6,-4,-9,-14)\)  & 3\\
6  & \(\Theta_6=\Phi_6=\Phi_{28}\)    & 6  & \((18,3,-2,-7,-12)\)  & 5\\
7  & \(\Theta_7=\Phi_7=\Phi_{18}\)    & 7  & \((5,1,0,-2,-4)\)     & 2\\
8  & \(\Theta_8=\Phi_8=\Phi_{16}\)    & 8  & \((4,2,-1,-2,-3)\)    & 2\\
9  & \(\Theta_9=\Phi_9=\Phi_{22}\)    & 9  & \((19,4,-1,-6,-16)\)  & 4\\
10 & \(\Theta_{10}=\Phi_{10}=\Phi_{34}\) & 10 & \((12,2,-3,-3,-8)\)   & 8\\
11 & \(\Theta_{11}=\Phi_{11}=\Phi_{25}\) & 11 & \((14,4,-1,-6,-11)\)  & 6\\
12 & \(\Theta_{12}=\Phi_{12}=\Phi_{17}\) & 12 & \((13,8,-2,-7,-12)\)  & 5\\
13 & \(\Theta_{13}=\Phi_{13}\)            & 13 & \((3,2,0,-2,-3)\)     & 4\\
14 & \(\Theta_{14}=\Phi_{14}\)            & 14 & \((4,1,0,-1,-4)\)     & 4\\
15 & \(\Theta_{15}=\Phi_{15}=\Phi_{30}\) & 15 & \((9,4,-1,-6,-6)\)    & 9\\
16 & \(\Theta_{16}=\Phi_{19}\)            & 19 & \((2,1,0,-1,-2)\)     & 8\\
17 & \(\Theta_{17}=\Phi_{20}=\Phi_{31}\) & 20 & \((8,3,-2,-2,-7)\)    & 11\\
18 & \(\Theta_{18}=\Phi_{23}=\Phi_{38}\) & 23 & \((4,-1,-1,-1,-1)\)   & 19\\
19 & \(\Theta_{19}=\Phi_{29}=\Phi_{35}\) & 29 & \((6,1,1,-4,-4)\)     & 15\\
20 & \(\Theta_{20}=\Phi_{32}=\Phi_{37}\) & 32 & \((3,3,-2,-2,-2)\)    & 18\\
21 & \(\Theta_{21}=\Phi_{36}\)            & 36 & \((1,0,0,0,-1)\)      & 24\\
\end{longtable}
\endgroup

The unsigned permutation signatures of the twenty-one displayed vectors are
pairwise distinct.  Hence the tori \(H_a\) are pairwise nonconjugate.  The
largest component is \(\Theta_{21}\), of dimension \(24\); the next largest
are \(\Theta_{18}\), \(\Theta_{20}\), and \(\Theta_{19}\), of dimensions
\(19\), \(18\), and \(15\), respectively.  The dimensions are part of the
classification, but the proof that the rows are distinct irreducible
components uses stabilizers rather than dimension comparisons.

\subsection*{The singularity catalog}

We now summarize the geometry of a general \(X_a\).  To align this
catalog with the component-indexed singular-locus table
(Table~\ref{tab:component-singular-loci}), set
\[
  P_0:=(1:0:0:0:0),
  \qquad
  P_\infty:=(0:0:0:0:1).
\]
Throughout the introduction, we use the fixed coordinate
labels \(P_0\) and \(P_\infty\) in place of case-local
labels such as \(P\) and \(Q\); the latter are retained
when convenient in the case-by-case analysis of
Appendix~\ref{app:singular-analysis}.

In Table~\ref{tab:intro-generic-singularities}, the symbols \(A_n\) and
\(D_n\) refer to the generic transverse hypersurface singularity along the
indicated smooth stratum.  Their indices also give the generic transverse
Jacobian length.  More generally, by the \emph{generic transverse Jacobian
length} we mean the length obtained from the full Jacobian ideal after
localizing at the generic point of the stratum and passing to a transverse
local ring.  It need not equal the Milnor number of a transverse polynomial
obtained by freezing the stratum parameter; this distinction is relevant in
row~\(19\).  The phrase ``special'' marks a nonisolated point at which the
generic transverse model changes; every such point is treated explicitly in
Appendix~\ref{app:singular-analysis} and has local minimal exponent one.  For
the few non-ADE strata, the table records an analytic model or the generic
Jacobian length.  A dash in the last column means that there is no isolated
singular point.

\providecommand{\IntroCell}[1]{\parbox[t]{\linewidth}{\raggedright #1}}
\providecommand{\IntroNone}{\textemdash}

\begingroup
\scriptsize
\renewcommand{\arraystretch}{1.16}
\setlength{\tabcolsep}{1.5pt}
\begin{longtable}{@{}
  >{\centering\arraybackslash}p{0.032\textwidth}
  >{\raggedright\arraybackslash}p{0.315\textwidth}
  >{\raggedright\arraybackslash}p{0.405\textwidth}
  >{\raggedright\arraybackslash}p{0.190\textwidth}
@{}}
\caption{Generic singularity catalog for the twenty-one boundary components.  In every row, \(\mexp(X_a)=1\).}
\label{tab:intro-generic-singularities}\\
\toprule
\(a\) & Reduced singular support of \(X_a\) & Generic transverse structure and special strata & Isolated local family\\
\midrule
\endfirsthead
\toprule
\(a\) & Reduced singular support of \(X_a\) & Generic transverse structure and special strata & Isolated local family\\
\midrule
\endhead
\midrule
\multicolumn{4}{r}{\emph{continued on the next page}}\\
\endfoot
\bottomrule
\endlastfoot

1 & \IntroCell{rational plane quartic \(C\) with a\ \((3,4)\)-cusp at \(P_\infty\); isolated point \(P_0\)}
  & \IntroCell{\(A_3\) on \(C\setminus\{P_\infty\}\); \(P_\infty\) special}
  & \(P_0:\ \Sfam{87}{52}\)\\
\midrule

2 & \IntroCell{line \(L\) and cuspidal plane cubic \(C\),\ \(L\cap C=\{P_\infty\}\); isolated point \(P_0\)}
  & \IntroCell{\(A_3\) on \(L\setminus\{P_\infty\}\); \(A_2\) on\ \(C\setminus\{P_\infty\}\); \(P_\infty\) special}
  & \(P_0:\ \Sfam{88}{84}\)\\
\midrule

3 & \IntroCell{line \(L\) with distinguished point \(P_\infty\);\ isolated point \(P_0\)}
  & \IntroCell{\(A_4\) on \(L\setminus\{P_\infty\}\); \(P_\infty\) special}
  & \(P_0:\ \Sfam{88}{15}\)\\
\midrule

4 & two disjoint lines \(M\sqcup L\)
  & \IntroCell{\(A_2\) on the general point of \(M\);\ \(A_{10}\) on the general point of \(L\);\ special points \(P_0\in M\) and \(P_\infty\in L\)}
  & \IntroNone\\
\midrule

5 & \IntroCell{three lines \(N\sqcup(L\cup M)\),\ \(L\cap M=\{P_\infty\}\)}
  & \IntroCell{\(A_1\) on general points of \(N\) and \(M\);\ \(A_8\) on \(L\setminus\{P_\infty\}\);\ distinguished \(P_0\in N\) and intersection \(P_\infty\) special}
  & \IntroNone\\
\midrule

6 & \IntroCell{line \(L\) and smooth conic \(C\),\ \(L\cap C=\{P_\infty\}\); isolated point \(P_0\)}
  & \IntroCell{\(A_6\) on \(L\setminus\{P_\infty\}\); \(A_1\) on\ \(C\setminus\{P_\infty\}\); \(P_\infty\) special}
  & \(P_0:\ \Sfam{91}{53}\)\\
\midrule

7 & \IntroCell{line \(L\) with distinguished point \(P_\infty\);\ isolated point \(P_0\)}
  & \IntroCell{\(A_4\) on \(L\setminus\{P_\infty\}\); \(P_\infty\) special}
  & \(P_0:\ \Sfam{96}{86}\)\\
\midrule

8 & two disjoint lines \(N\sqcup L\)
  & \IntroCell{on general \(N\): suspension of an ordinary\ triple plane-curve point (Jacobian length \(4\));\ on general \(L\): \(A_9\); special points \(P_0\in N\), \(P_\infty\in L\)}
  & \IntroNone\\
\midrule

9 & \IntroCell{line \(L\) with distinguished point \(P_\infty\);\ isolated point \(P_0\)}
  & \IntroCell{\(A_2\) on \(L\setminus\{P_\infty\}\); \(P_\infty\) special}
  & \(P_0:\ \Sfam{96}{94}\)\\
\midrule

10 & \IntroCell{four lines \(C_\lambda\) through the common vertex \(P_\infty\);\ isolated point \(P_0\)}
   & \IntroCell{\(A_2\) away from the common vertex;\ \(P_\infty\) special}
   & \(P_0:\ \Sfam{90}{2}\)\\
\midrule

11 & two disjoint lines \(N\sqcup L\)
   & \IntroCell{\(A_1\) on general \(N\); \(A_5\) on general \(L\);\ special points \(P_0\in N\), \(P_\infty\in L\)}
   & \IntroNone\\
\midrule

12 & two disjoint lines \(N\sqcup L\)
   & \IntroCell{\(D_5\) on general \(N\); \(A_6\) on general \(L\);\ special points \(P_0\in N\), \(P_\infty\in L\)}
   & \IntroNone\\
\midrule

13 & two disjoint lines \(N\sqcup L\)
   & \IntroCell{\(D_6\) on the general point of each line;\ special points \(P_0\in N\), \(P_\infty\in L\)}
   & \IntroNone\\
\midrule

14 & two isolated points \(P_0\) and \(P_\infty\)
   & \IntroNone
   & \IntroCell{\(P_0,P_\infty:\ \Sfam{102}{62}\)}\\
\midrule

15 & two disjoint lines \(N\sqcup L\)
   & \IntroCell{\(A_3\) on \(N\setminus\{P_0\}\); \(P_0\) special;\ along \(L\): relative weighted cone\ of normal type \((3,2,1;6)\), Jacobian length \(9\)}
   & \IntroNone\\
\midrule

16 & two disjoint lines \(N\sqcup L\)
   & \IntroCell{\(A_4\) on the general point of each line;\ special points \(P_0\in N\), \(P_\infty\in L\)}
   & \IntroNone\\
\midrule

17 & \IntroCell{line \(N\) disjoint from four lines\ \(C_{\lambda,\mu}\) through the common vertex \(P_\infty\)}
   & \IntroCell{\(A_2\) on general \(N\); \(A_1\) on each of the\ four other lines away from \(P_\infty\);\ distinguished \(P_0\in N\) and vertex \(P_\infty\) special}
   & \IntroNone\\
\midrule

18 & \IntroCell{smooth quartic surface \(S\);\ isolated point \(P_0\)}
   & simple normal crossing along \(S\)
   & \(P_0:\ \Sfam{81}{1}\)\\
\midrule

19 & \IntroCell{line \(L\); isolated point \(P_0\)}
   & \IntroCell{generic transverse model \(U^2+G_4(V,W)\),\ where \(G_4\) is square-free binary quartic;\ generic Jacobian length \(8\)}
   & \(P_0:\ \Sfam{100}{3}\)\\
\midrule

20 & disjoint union of a line \(N\) and a plane \(\Pi\)
   & \IntroCell{on \(\Pi\setminus\Delta_\Pi\): \(U^2+V^2\);\ on the smooth discriminant curve \(\Delta_\Pi\):\ \(U^2+zV^2\); along \(N\): a cubic-cone family\ \(C_t(y_2,y_3,y_4)=0\), Jacobian length \(7\)}
   & \IntroNone\\
\midrule

21 & two isolated points \(P_0\) and \(P_\infty\)
   & \IntroNone
   & \IntroCell{\(P_0,P_\infty:\ \Sfam{96}{21}\)}\\

\end{longtable}
\endgroup

The notation in the last column is decoded completely in the following
table, which also displays a defining equation for each local normal-form
family.  If a family is written as \(\Sfam{\mu}{N}\), then \(\mu\) is its
Milnor number and \(N\) is the row of Yonemura's Table~2.2
\cite{Yon90} containing its ordered normalized weight.  The precise
nonempty Zariski-open genericity conditions on the displayed parameters are
given in Definitions~\ref{def:isolated-family-81-Y1}--
\ref{def:isolated-family-102-Y62}.  Every displayed integral weight system
\((w_1,w_2,w_3,w_4;D)\) satisfies
\[
  w_1+w_2+w_3+w_4=D,
\]
so every isolated germ in the table has local minimal exponent one.

\begingroup
\scriptsize
\renewcommand{\arraystretch}{1.18}
\setlength{\tabcolsep}{0.8pt}
\setlength{\extrarowheight}{1pt}
\providecommand{\IntroIsoFamily}[1]{\begin{tabular}[t]{@{}l@{}}#1\end{tabular}}
\providecommand{\IntroIsoModel}[1]{\(\begin{aligned}[t]#1\end{aligned}\)}
\providecommand{\IntroIsoHead}[1]{\begin{tabular}[c]{@{}c@{}}#1\end{tabular}}
\begin{longtable}{@{}
  >{\raggedright\arraybackslash}p{0.128\textwidth}
  >{\raggedright\arraybackslash}p{0.475\textwidth}
  >{\centering\arraybackslash}p{0.115\textwidth}
  >{\centering\arraybackslash}p{0.035\textwidth}
  >{\centering\arraybackslash}p{0.045\textwidth}
  >{\raggedright\arraybackslash}p{0.125\textwidth}
@{}}
\caption{The eleven isolated weighted-homogeneous local normal-form families.  The second column displays a defining equation for each family.  All parameters are restricted to the nonempty Zariski-open generic loci specified in Definitions~\ref{def:isolated-family-81-Y1}--\ref{def:isolated-family-102-Y62}.  The integral weights determine the ordered normalized weights used to identify the corresponding rows of Yonemura's Table~2.2~\cite{Yon90}.  This table forms part of Theorem~A.}
\label{tab:intro-isolated-families}\\
\toprule
Local family & Local normal form & \IntroIsoHead{Integral\\weights\\\((w;D)\)} & \(\mu\) & \IntroIsoHead{Yon.\\No.} & \IntroIsoHead{Component /\\point}\\
\midrule
\endfirsthead
\toprule
Local family & Local normal form & \IntroIsoHead{Integral\\weights\\\((w;D)\)} & \(\mu\) & \IntroIsoHead{Yon.\\No.} & \IntroIsoHead{Component /\\point}\\
\midrule
\endhead
\midrule
\multicolumn{6}{r}{\emph{continued on the next page}}\\
\endfoot
\bottomrule
\endlastfoot

\(\Sfam{81}{1}(B_4)\)
& \IntroIsoModel{B_4(X_1,X_2,X_3,X_4)=0}
& \((1,1,1,1;4)\) & 81 & 1 & \(\Theta_{18}:P_0\)
\\
\midrule

\(\Sfam{87}{52}(\tau)\)
& \IntroIsoModel{X_1^4X_2+X_3^4+X_2^3X_4\\
{}+\tau X_1X_2X_3X_4+X_4^3=0}
& \((7,8,9,12;36)\) & 87 & 52 & \(\Theta_1:P_0\)
\\
\midrule

\(\Sfam{88}{84}(\alpha,\beta)\)
& \IntroIsoModel{X_1^3X_2^2+X_1^4X_3+X_2X_3^3+X_2^3X_4\\
{}+\alpha X_1X_2X_3X_4+\beta X_4^3=0}
& \((5,6,7,9;27)\) & 88 & 84 & \(\Theta_2:P_0\)
\\
\midrule

\(\Sfam{88}{15}(B_5)\)
& \IntroIsoModel{B_5(X_1,X_2)+X_2X_3^3\\
{}+(X_1^2+X_2^2)X_3X_4+X_4^3=0}
& \((3,3,4,5;15)\) & 88 & 15 & \(\Theta_3:P_0\)
\\
\midrule

\IntroIsoFamily{\(\Sfam{90}{2}\)\\\((B_4,Q_2,R_2,\beta)\)}
& \IntroIsoModel{B_4(X_2,X_3)+X_1^3Q_2(X_2,X_3)+X_1^4X_4\\
{}+X_1X_4R_2(X_2,X_3)+\beta X_1^2X_4^2+X_4^3=0}
& \((2,3,3,4;12)\) & 90 & 2 & \(\Theta_{10}:P_0\)
\\
\midrule

\IntroIsoFamily{\(\Sfam{91}{53}\)\\\((\alpha,\beta,\gamma,\delta,\epsilon)\)}
& \IntroIsoModel{X_1^2X_2^3+X_1^3X_2X_3+X_1^4X_4+X_2^2X_3^2\\
{}+\alpha X_2^3X_4+\beta X_1X_3^3+\gamma X_1X_2X_3X_4\\
{}+\delta X_1^2X_4^2+\epsilon X_4^3=0}
& \((3,4,5,6;18)\) & 91 & 53 & \(\Theta_6:P_0\)
\\
\midrule

\(\Sfam{96}{86}(\alpha,\beta)\)
& \IntroIsoModel{X_2^5+X_1^2X_2^2X_3+X_1^4X_4+X_1X_3^3\\
{}+\alpha X_1X_2X_3X_4+\beta X_3X_4^2=0}
& \((4,5,7,9;25)\) & 96 & 86 & \(\Theta_7:P_0\)
\\
\midrule

\IntroIsoFamily{\(\Sfam{96}{94}\)\\\((\alpha,\beta,\gamma,\delta)\)}
& \IntroIsoModel{X_1X_2^4+X_1^2X_2^2X_3+\alpha X_1^3X_3^2+X_1^4X_4\\
{}+X_2X_3^3+\beta X_2^3X_4+\gamma X_1X_2X_3X_4\\
{}+\delta X_3X_4^2=0}
& \((3,4,5,7;19)\) & 96 & 94 & \(\Theta_9:P_0\)
\\
\midrule

\IntroIsoFamily{\(\Sfam{96}{21}\)\\\((B_5,C_3,L_1)\)}
& \IntroIsoModel{B_5(X_1,X_2,X_3)+X_4C_3(X_1,X_2,X_3)\\
{}+X_4^2L_1(X_1,X_2,X_3)=0}
& \((1,1,1,2;5)\) & 96 & 21 & \IntroIsoFamily{\(\Theta_{21}:P_0,\)\\\(P_\infty\)}
\\
\midrule

\IntroIsoFamily{\(\Sfam{100}{3}\)\\\((A_4,C_4,\)\\\(Q_{2,2},B_3)\)}
& \IntroIsoModel{X_3A_4(X_1,X_2)+X_4C_4(X_1,X_2)\\
{}+Q_{2,2}(X_1,X_2;X_3,X_4)+B_3(X_3,X_4)=0}
& \((1,1,2,2;6)\) & 100 & 3 & \(\Theta_{19}:P_0\)
\\
\midrule

\IntroIsoFamily{\(\Sfam{102}{62}\)\\\((\alpha,\beta,\gamma,\delta)\)}
& \IntroIsoModel{X_2^5+\alpha X_1X_2^3X_3+\beta X_1^2X_2X_3^2\\
{}+X_1^4X_4+X_3^4+\gamma X_2^3X_4\\
{}+\delta X_1X_2X_3X_4+X_2X_4^2=0}
& \((3,4,5,8;20)\) & 102 & 62 & \IntroIsoFamily{\(\Theta_{14}:P_0,\)\\\(P_\infty\)}
\\

\end{longtable}

\par\smallskip
{\footnotesize
\noindent\textit{Parameter notation.}
The arguments displayed in each equation specify the variables of the
auxiliary forms.  Thus \(B_4(X_1,X_2,X_3,X_4)\) is a homogeneous quartic
whose projectivization in \(\PP^3\) is smooth, whereas
\(B_4(X_2,X_3)\) is a square-free binary quartic.  The symbols
\(B_5(X_1,X_2)\) and \(B_5(X_1,X_2,X_3)\) denote, respectively, a binary
and a ternary quintic.  The forms \(Q_2,R_2\) are binary quadratics in
\((X_2,X_3)\); \(C_3\) and \(L_1\) are, respectively, a ternary cubic and
a nonzero ternary linear form in \((X_1,X_2,X_3)\); \(A_4,C_4\) are
binary quartics in \((X_1,X_2)\); \(Q_{2,2}\) is bihomogeneous of bidegree
\((2,2)\) in \((X_1,X_2)\) and \((X_3,X_4)\); and \(B_3\) is a binary
cubic in \((X_3,X_4)\).  Greek letters denote complex scalar parameters.
The precise smoothness, square-freeness, nonvanishing, isolatedness, and
discriminant conditions defining the relevant nonempty Zariski-open loci
are given in Definitions~\ref{def:isolated-family-81-Y1}--
\ref{def:isolated-family-102-Y62}.\par
}
\endgroup

The Yonemura assertion is deliberately an assertion about normalized weight
systems.  A row of Yonemura's list~\cite{Yon90} gives a weight system and one
particular nondegenerate realization.  We do not claim that every parameter value in
one of the families above is analytically equivalent to the polynomial
displayed in that row.  The term \emph{Yonemura-type} means exactly that the
normalized weight system occurs in Yonemura's list.  What is proved here is
an occurrence-and-exhaustion statement on the quintic GIT boundary: the
eleven families in Table~\ref{tab:intro-isolated-families} occur, and no
other isolated weighted-homogeneous local family occurs for a general closed
orbit in any boundary component.

\subsection*{Proof map for Theorem~A}

Theorem~A is a synthesis of the structural results in the body and the
case-by-case verifications in the appendices.  Item~(1) is
Proposition~\ref{prop:quintic-maximal-t-supports} together with
Proposition~\ref{prop:strictly-semistable-cover}.  Item~(2) is proved in three
stages.  Corollary~\ref{cor:supporting-hyperplane-invariance} and
Proposition~\ref{prop:opposite-half-space-identifications} identify the
seventeen opposite-support pairs and the four fixed cases, while
Corollary~\ref{cor:twenty-one-quotient-side-candidate-cover} reduces the
thirty-eight-family cover to twenty-one candidate quotient images.
Corollary~\ref{cor:twenty-one-quotient-side-cover} proves that the candidates
are nonempty, closed, and irreducible, and
Theorem~\ref{thm:pairwise-noninclusion} proves pairwise non-inclusion and hence
the irreducible-component decomposition.  The normal forms and quotient
dimensions in item~(3) are organized in
Section~\ref{sec:normal-form} and proved support by support in
Appendix~\ref{app:normal-forms}; its stabilizer assertions are
Proposition~\ref{prop:generic-connected-stabilizers} and
Corollary~\ref{cor:generic-tori-pairwise-nonconjugate}.  Item~(4) is stated
in Section~\ref{sec:singular-loci} and proved by the component-by-component
local calculations in Appendix~\ref{app:singular-analysis}.  Item~(5) is
proved in the same appendix.  Thus
Tables~\ref{tab:intro-normal-forms}--\ref{tab:intro-isolated-families} are
both a statement of the classification and a row-by-row index to its proofs.

\subsection*{Why thirty-eight supports give twenty-one components}

Fix the diagonal maximal torus \(T\subset G\), and let
\[
  I
  :=
  \left\{
    u=(u_0,\ldots,u_4)\in\mathbb Z_{\geq0}^5
    \ \middle|\
    u_0+\cdots+u_4=5
  \right\}
\]
be the exponent set of quintic monomials.  Its barycenter is
\[
  \eta=(1,1,1,1,1).
\]
For a nonzero vector
\[
  r=(r_0,\ldots,r_4)\in\mathbb Z^5,
  \qquad
  \sum r_i=0,
\]
put
\[
  I(r)_{\geq0}:=\{u\in I\mid r\cdot u\geq0\}.
\]
The convex-geometric Hilbert--Mumford criterion identifies \(T\)-semistability
with the condition that the convex hull of the support contain \(\eta\), and
strict \(T\)-semistability with the case in which \(\eta\) lies on a
supporting hyperplane.  Exact integer linear algebra, inclusion tests, and
reduction by the Weyl group \(S_5\) yield exactly thirty-eight maximal
supports
\[
  S_k=I(r_k)_{\geq0},
  \qquad
  k=1,\ldots,38.
\]
For
\[
  V_k:=\Span_\C\{x^u\mid u\in S_k\},
  \qquad
  Y_k:=\overline{G\cdot\PP(V_k)},
\]
one has
\[
  \Sigma^{ss}
  =
  \bigcup_{k=1}^{38}\bigl(Y_k\cap\Sigma^{ss}\bigr),
  \qquad
  \partial\mathcal M^{\mathrm{GIT}}
  =
  \bigcup_{k=1}^{38}\Phi_k,
\]
where
\[
  \Phi_k:=\pi(Y_k\cap\Sigma^{ss}).
\]
This is a support-indexed cover, not yet a component decomposition.

For an arbitrary \(r\), write
\[
  W=\bigoplus_{m\in\mathbb Z}W_m(r),
  \qquad
  V(r):=\bigoplus_{m\geq0}W_m(r),
  \qquad
  Z(r):=W_0(r)=W^{H_r}.
\]
Set
\[
  Y(r):=\overline{G\cdot\PP(V(r))}\subset\PP(W),
  \qquad
  \Phi(r):=\pi\bigl(Y(r)\cap\PP(W)^{ss}\bigr).
\]
The central quotient-theoretic statement of the paper is the
supporting-hyperplane principle
\[
  \Phi(r)
  =
  \overline{\pi\bigl(\PP(V(r))\cap\PP(W)^{ss}\bigr)}
  =
  \overline{\pi\bigl(\PP(Z(r))\cap\PP(W)^{ss}\bigr)}.
\]
Indeed, if \(f=f_0+f_{>0}\in V(r)\) is semistable, then the
one-parameter-subgroup limit is \(f_0\in Z(r)\), the limit remains
semistable, and
\[
  \pi([f])=\pi([f_0]).
\]
Consequently
\[
  \Phi(r)=\Phi(-r).
\]
Thus the quotient remembers the supporting hyperplane but not the orientation
of its normal.

Returning the opposite normal to the dominant chamber gives the involution
\[
  \iota(r_0,r_1,r_2,r_3,r_4)
  :=(-r_4,-r_3,-r_2,-r_1,-r_0).
\]
On the thirty-eight dominant normals it has seventeen two-cycles,
\[
\begin{aligned}
 &(1,26),(2,27),(3,33),(4,21),(5,24),(6,28),\\
 &(7,18),(8,16),(9,22),(10,34),(11,25),(12,17),\\
 &(15,30),(20,31),(23,38),(29,35),(32,37),
\end{aligned}
\]
and four fixed points \(13,14,19,36\).  Hence
\[
  38=17\cdot2+4
\]
counts thirty-eight oriented maximal half-spaces but only twenty-one
supporting hyperplanes up to coordinate permutation and reversal.  This is
exactly the identification displayed in the second column of
Table~\ref{tab:intro-twenty-one-components}.

\subsection*{Closed orbits, stabilizers, and non-inclusion}

For a general form \(f_k\in V_k\), the limit
\[
  \phi_k
  :=
  \lim_{t\to0}\lambda_{r_k}(t)\cdot f_k
\]
is its weight-zero part and lies in \(W^{H_k}\).  Luna's centralizer
reduction turns the closedness problem for the \(G\)-orbit of \(\phi_k\)
into a closed-orbit problem for
\[
  C_G(H_k)/H_k
\]
on \(W^{H_k}\).  For torus centralizers we use a relative-interior
criterion for the character polytope; for centralizers with nontrivial linear
blocks we combine discriminant and resultant conditions with the
Casimiro--Florentino criterion~\cite{CF12}.  This produces an explicit
polystable normal form \(\nf_k\) for every \(k=1,\ldots,38\), together with
the quotient dimensions used in Table~\ref{tab:intro-twenty-one-components}.
The twenty-one quotient-side representatives selected from these forms are
reproduced in Table~\ref{tab:intro-normal-forms}.

After the supporting-hyperplane identifications, it remains to prove that the
twenty-one closed irreducible families are not further identified and that
none is contained in another.  For a nonzero quintic \(F\), consider
\[
  \mathcal D_F:
  \mathfrak{sl}_5\oplus\C\longrightarrow W,
  \qquad
  (A,c)\longmapsto A\cdot F-cF.
\]
This is represented by a \(126\times25\) matrix and satisfies
\[
  \dim\operatorname{Stab}_G([F])
  =25-\rank\mathcal D_F.
\]
The torus \(H_a\) is always contained in the stabilizer of the normal form in
row \(a\), so \(\rank\mathcal D_{F_a}\leq24\).  Exact integral
specializations and nonzero \(24\times24\) minors prove that the generic rank
is exactly \(24\).  Therefore
\[
  \operatorname{Stab}_G([F_a])^\circ=H_a.
\]
The twenty-one unsigned weight signatures are distinct, hence these tori are
pairwise nonconjugate.

Finally, Luna's \'{e}tale slice theorem gives the closure obstruction used in
the component theorem: if \(\Theta_a\subseteq\Theta_b\), then a conjugate of
the generic connected stabilizer of \(\Theta_b\) must be contained in the
generic connected stabilizer of \(\Theta_a\).  Since both are connected
one-dimensional tori, this inclusion would be an equality, contradicting
nonconjugacy.  Thus all ordered containments between distinct rows of
Table~\ref{tab:intro-twenty-one-components} are excluded simultaneously.

\subsection*{Singularities and the extremal minimal exponent}

The adjective \emph{extremal} refers here to the equality between the global
minimal exponent and the critical value \((n+1)/d\): every general boundary
representative attains that value, and none of its singular strata forces a
smaller one.

For every row of the component catalog, the saturated Jacobian ideal
\[
  J(F_a)
  :=
  \left(
    \frac{\partial F_a}{\partial x_0},\ldots,
    \frac{\partial F_a}{\partial x_4}
  \right):(x_0,\ldots,x_4)^\infty
\]
is computed explicitly.  This determines the scheme structure along the
positive-dimensional singular components and gives the stratification
summarized in Table~\ref{tab:intro-generic-singularities}.  The local analysis
uses the holomorphic splitting lemma, transverse \(A\)- and \(D\)-type
models, Newton nondegeneracy, weighted blow-ups and log canonical thresholds,
Thom--Sebastiani formulas, and Saito's rationality criterion.

At an isolated point, the local equation belongs to one of the eleven
weighted-homogeneous families in Table~\ref{tab:intro-isolated-families}.  Since the sum
of the weights equals the weighted degree, its local minimal exponent is one.
At a positive-dimensional stratum, the generic transverse models have local
minimal exponent greater than or equal to one, and every distinguished point
at which the transverse type changes has local minimal exponent exactly one.
Consequently no singular stratum lowers the global value below one, while at
least one stratum in every component realizes one.  This proves
\[
  \mexp(X_a)=1
  \qquad(a=1,\ldots,21).
\]
The value is the critical quantity
\[
  \frac{n+1}{d}=\frac{4+1}{5}=1
\]
for degree-five hypersurfaces in \(\PP^4\) in the minimal-exponent approach
to GIT stability~\cite{MP20,Par25}.

\subsection*{Relation to earlier work and organization of the paper}

The GIT compactification of quintic threefolds was studied by
Lakhani~\cite{Lak10}.  His four first-level minimal-orbit families become,
in the present quotient-side classification, the four components of largest
dimension, namely \(\Theta_{18},\Theta_{19},\Theta_{20},\Theta_{21}\); the
remaining seventeen components are not contained in their union.  The
precise dictionary, the comparison of the polystability arguments, and the
separation between Lakhani's second-level degeneration loci and the new
irreducible components are given in
Section~\ref{sec:lakhani-comparison}.

We use the standard GIT framework of \cite{MFK94,Dol03,Kir84}, Luna's
results on closed orbits and \'{e}tale slices~\cite{Lun73,Lun75}, and the
convex-geometric form of the Hilbert--Mumford criterion.  The finite support
enumeration is in the spirit of the effective procedures developed
in~\cite{GMMS26}.

Section~\ref{sec:numerical-criterion-maximal-supports} fixes the action
convention, recalls the Hilbert--Mumford criterion, and classifies the
thirty-eight maximal strictly \(T\)-semistable supports.
Section~\ref{sec:boundary-families-sl5} passes to the full \(\SL(5)\)-action,
defines the prequotient and quotient-side families, and proves the boundary
covering.  Section~\ref{sec:quotient-side-inclusions} proves zero-weight
reduction and supporting-hyperplane invariance, reducing the thirty-eight
oriented support labels to twenty-one candidate quotient images.
Section~\ref{sec:normal-form} develops the general closed-orbit criteria and
records one polystable representative for each of the twenty-one quotient
families.  Section~\ref{sec:pairwise-noninclusion} computes the generic
connected stabilizers and applies Luna's slice theorem to prove pairwise
non-inclusion and the component theorem.  Section~\ref{sec:lakhani-comparison}
then compares the resulting classification with Lakhani's.
Section~\ref{sec:singular-loci} states the component-indexed singularity
catalog and the minimal-exponent conclusion.

Appendix~\ref{app:normal-forms} contains the thirty-eight
support-by-support constructions of the polystable normal forms, while
Appendix~\ref{app:singular-analysis} contains the twenty-one
component-by-component calculations of singular schemes, local analytic
types, and minimal exponents.

\section{Numerical criterion and maximal strictly semistable supports}
\label{sec:numerical-criterion-maximal-supports}

\subsection{Numerical criterion}
\label{subsec:numerical-criterion-quintic}

Let
\[
V:=H^0\bigl(\PP^4,\mathcal O_{\PP^4}(1)\bigr)
   =\langle x_0,x_1,x_2,x_3,x_4\rangle_{\C},
\]
and put
\[
W:=\Sym^5V,
\qquad
G:=\SL(V)\simeq\SL(5),
\qquad
\PP(W)=\PP(\Sym^5V).
\]
Thus an element \(0\neq f\in W\) is a quintic form in the variables
\(x_0,\ldots,x_4\), and its projective class defines the quintic
threefold
\[
X_f:=V(f)\subset \PP(V^\vee)\simeq\PP^4.
\]
Let \(T\subset G\) be the diagonal maximal torus with respect to the
ordered basis \(x_0,\ldots,x_4\).

\paragraph*{Action convention.}
Throughout the paper, \(G\) acts on \(W=\Sym^5V\) through the fifth
symmetric power of its standard representation on \(V\).  More
explicitly, if \(g=(g_{ij})\in G\), then
\begin{equation}
\label{eq:quintic-action-convention}
 g\cdot x_j:=\sum_{i=0}^4 g_{ij}x_i,
 \qquad j=0,\ldots,4,
\end{equation}
and this action is extended multiplicatively and linearly to
\(\Sym^5V\).  Via the contragredient action on \(\PP(V^\vee)\), this is
the usual projective change-of-coordinates action on quintic
hypersurfaces.

If one instead writes the action on equations in the pullback form
\[
 (g\star f)(x):=f(g^{-1}x),
\]
then, under the above basis identification,
\[
 g\star f=(g^{-T})\cdot f.
\]
Since \(g\mapsto g^{-T}\) is an automorphism of \(G\), the two
conventions have the same orbits, invariant ring, stable and
semistable loci, and projective GIT quotient.  We use only the action
\(\cdot\) in \eqref{eq:quintic-action-convention} below.

For a vector \(u=(u_0,\ldots,u_4)\in\mathbb Q^5\), set
\[
 \wt(u):=u_0+\cdots+u_4.
\]
We use the notation
\[
 \mathbb Z^5(d):=\{u\in\mathbb Z^5\mid \wt(u)=d\},
 \qquad
 \mathbb Z^5_{\geq0}:=\{u\in\mathbb Z^5\mid u_i\geq0
 \text{ for all }i\},
\]
and
\[
 I:=\mathbb Z^5(5)\cap\mathbb Z^5_{\geq0}.
\]
Thus \(I\) is the set of exponent vectors of degree-five monomials.
For \(u\in I\), write
\[
 x^u:=x_0^{u_0}x_1^{u_1}x_2^{u_2}x_3^{u_3}x_4^{u_4}.
\]
The barycenter of the degree-five simplex is
\[
 \eta:=(1,1,1,1,1)\in I.
\]
For \(r\in\mathbb Q^5\), define
\[
 I(r)_{\geq0}:=\{u\in I\mid r\cdot u\geq0\},
 \qquad
 I(r)_{>0}:=\{u\in I\mid r\cdot u>0\},
\]
and
\[
 I(r)_{=0}:=\{u\in I\mid r\cdot u=0\}.
\]
A nonzero vector \(r\in\mathbb Z^5\) is called \emph{reduced} if
\[
 \gcd(r_0,r_1,r_2,r_3,r_4)=1.
\]

Let
\[
 f=\sum_{u\in I}a_u x^u\in W.
\]
Its support is
\[
 \Supp(f):=\{u\in I\mid a_u\neq0\}.
\]
A one-parameter subgroup of \(T\) is written as
\[
 \lambda_r(t)
 =\diag(t^{r_0},t^{r_1},t^{r_2},t^{r_3},t^{r_4}),
 \qquad
 r=(r_0,\ldots,r_4)\in\mathbb Z^5(0).
\]
By the action convention above,
\begin{equation}
\label{eq:quintic-monomial-weight}
 \lambda_r(t)\cdot x^u
 =t^{r\cdot u}x^u.
\end{equation}
Accordingly, if
\[
 f=\sum_{m\in\mathbb Z} f_m^{(r)},
 \qquad
 f_m^{(r)}
 \in
 \Span_{\C}\{x^u\mid u\in I,\ r\cdot u=m\},
\]
then
\[
 \lambda_r(t)\cdot f
 =\sum_{m\in\mathbb Z}t^m f_m^{(r)}.
\]
In particular, if \(\Supp(f)\subset I(r)_{\geq0}\), then the affine
limit exists and equals the weight-zero part:
\begin{equation}
\label{eq:quintic-zero-weight-limit}
 \lim_{t\to0}\lambda_r(t)\cdot f
 =f_0^{(r)}
 =\sum_{u\in I(r)_{=0}}a_u x^u.
\end{equation}
Thus, with our convention, the zero-weight projection is obtained by
\(\lambda_r(t)\) itself, rather than by \(\lambda_r(t)^{-1}\).

\begin{definition}
Let \(s\subset I\).  We say that \(s\) is \emph{not stable with
respect to \(T\)} if
\[
 s\subset I(r)_{\geq0}
\]
for some nonzero \(r\in\mathbb Z^5(0)\).  We say that \(s\) is
\emph{unstable with respect to \(T\)} if
\[
 s\subset I(r)_{>0}
\]
for some nonzero \(r\in\mathbb Z^5(0)\).

For \(0\neq f\in W\), we say that \(f\) is not stable, respectively
unstable, with respect to \(T\) if \(\Supp(f)\) has the corresponding
property.
\end{definition}

We now record the equivalent convex-geometric formulation.  Let
\[
 \mathcal H_5
 :=\{u\in\mathbb R^5\mid \wt(u)=5\}.
\]
Translation by \(-\eta\) identifies \(\mathcal H_5\) with the real
character space of \(T\), represented by the hyperplane
\[
 \mathbb R^5(0)
 :=\{v\in\mathbb R^5\mid v_0+\cdots+v_4=0\}.
\]
Indeed, the character of the monomial \(x^u\) is represented by
\(u-\eta\), and
\[
 r\cdot(u-\eta)=r\cdot u
\]
for every \(r\in\mathbb Z^5(0)\).  The Hilbert--Mumford numerical
criterion gives
\[
 f\text{ is }T\text{-semistable}
 \quad\Longleftrightarrow\quad
 \eta\in\Conv(\Supp(f)),
\]
and
\[
 f\text{ is }T\text{-stable}
 \quad\Longleftrightarrow\quad
 \eta\in\Int_{\mathcal H_5}\Conv(\Supp(f)),
\]
where \(\Int_{\mathcal H_5}\) denotes the interior in the full
affine hyperplane \(\mathcal H_5\).  By contrast, polystability is an
orbit-closure condition.  The affine torus orbit-closure criterion
(see, for example, \cite{Dol03,PV94}) gives
\[
 [f]\text{ is }T\text{-polystable}
 \quad\Longleftrightarrow\quad
 T\cdot f\text{ is closed in }W
 \quad\Longleftrightarrow\quad
 \eta\in\operatorname{relint}\Conv(\Supp(f)),
\]
where the relative interior is taken in the affine hull of
\(\Conv(\Supp(f))\).  Consequently, \(f\) is strictly
\(T\)-semistable precisely when
\[
 \eta\in\Conv(\Supp(f))
 \qquad\text{but}\qquad
 \eta\notin\Int_{\mathcal H_5}\Conv(\Supp(f)).
\]
The distinction between relative interior and ambient interior will
be important later: a polystable form may have a positive-dimensional
torus stabilizer and hence fail to be stable.

\begin{theorem}[Hilbert--Mumford numerical criterion]
\label{thm:hm-quintic-support-form}
Let \(0\neq f\in W=\Sym^5V\).  The point \([f]\in\PP(W)\), or
equivalently the quintic threefold \(X_f\subset\PP^4\), is not stable
for the \(G\)-action if and only if there exist \(\sigma\in G\) and a
nonzero vector \(r\in\mathbb Z^5(0)\) such that
\[
 \Supp(\sigma\cdot f)\subset I(r)_{\geq0}.
\]
It is unstable for the \(G\)-action if and only if there exist
\(\sigma\in G\) and a nonzero \(r\in\mathbb Z^5(0)\) such that
\[
 \Supp(\sigma\cdot f)\subset I(r)_{>0}.
\]

In particular, \([f]\) is strictly semistable for \(G\) if and only
if
\begin{enumerate}
\item there exists \(\sigma\in G\) such that \(\sigma\cdot f\) is not
stable with respect to the fixed torus \(T\); and
\item for every \(\sigma\in G\), the form \(\sigma\cdot f\) is
semistable with respect to \(T\).
\end{enumerate}
\end{theorem}

\begin{proof}
Every one-parameter subgroup of \(G\) is conjugate to one contained in
the fixed maximal torus \(T\).  After the corresponding change of
basis, it is therefore enough to test the subgroups \(\lambda_r\) with
\(0\neq r\in\mathbb Z^5(0)\).  By
\eqref{eq:quintic-monomial-weight}, the \(\lambda_r\)-weights of the
nonzero monomials of \(\sigma\cdot f\) are exactly the integers
\(r\cdot u\).  The Hilbert--Mumford inequalities are consequently
equivalent to the two displayed support containments.  The final
statement is the conjunction of nonstability and semistability.
\end{proof}

\subsection{Maximal strictly semistable supports for the fixed torus}
\label{subsec:maximal-t-supports-quintic}

We now enumerate the inclusion-maximal supports that can occur for
strictly semistable quintics with respect to the fixed maximal torus.
Let
\[
 \mathcal S
 :=\{I(r)_{\geq0}\mid 0\neq r\in\mathbb Z^5(0)\},
\]
ordered by inclusion.  Although the set of vectors \(r\) is infinite,
\(\mathcal S\) is finite because \(I\) is finite.

For every \(r\in\mathbb Z^5(0)\), one has
\[
 r\cdot\eta=0.
\]
Hence
\[
 \eta\in I(r)_{=0}\subset I(r)_{\geq0}.
\]
It follows that a form whose support is exactly \(I(r)_{\geq0}\) is
\(T\)-semistable, while the defining half-space shows that it is not
\(T\)-stable.  Thus the maximal elements of \(\mathcal S\) are
precisely the maximal full supports of strictly \(T\)-semistable
quintics.

Different vectors may define the same subset of \(I\).  The following
lemma reduces the classification to finitely many supporting
hyperplanes through the barycenter.

\begin{lemma}
\label{lem:maximal-support-normal-quintic}
Let \(I(r)_{\geq0}\) be a maximal element of \(\mathcal S\).  Then
there exist three exponent vectors
\[
 u_1,u_2,u_3\in I
\]
and a reduced vector \(r'\in\mathbb Z^5(0)\) such that
\begin{enumerate}
\item the \(\mathbb Q\)-vector subspace of \(\mathbb Q^5\) spanned by
\(u_1,u_2,u_3,\eta\) has dimension \(4\);
\item \(r'\) is orthogonal to that subspace; and
\item
\[
 I(r)_{\geq0}=I(r')_{\geq0}.
\]
\end{enumerate}
\end{lemma}

\begin{proof}
Put
\[
 S:=I(r)_{\geq0}
\]
and work in the four-dimensional rational vector space
\[
 E:=\mathbb Q^5(0)
 =\{q\in\mathbb Q^5\mid q\cdot\eta=0\}.
\]
Consider the rational polyhedral cone
\[
 K_S
 :=
 \left\{
 q\in E
 \ \middle|\
 \begin{array}{ll}
 q\cdot u\geq0 & (u\in S),\\
 q\cdot v\leq0 & (v\in I\setminus S)
 \end{array}
 \right\}.
\]
The original vector \(r\) belongs to \(K_S\), so this cone is
nonzero.  Moreover, if \(0\neq q\in K_S\) and
\(q\cdot v=0\) for some \(v\in I\setminus S\), then
\[
 S\subsetneq I(q)_{\geq0},
\]
contradicting the maximality of \(S\).  Hence every nonzero
\(q\in K_S\) satisfies
\[
 q\cdot v<0
 \qquad
 (v\in I\setminus S),
\]
and therefore defines exactly the same support:
\[
 I(q)_{\geq0}=S.
\]

The cone \(K_S\) is pointed.  Indeed, if both \(q\) and \(-q\)
belonged to \(K_S\), then \(q\) would vanish on every point of
\(I\); since \(I\) affinely spans \(\mathcal H_5\), this would force
\(q=0\).  Choose an extremal ray of \(K_S\), and let
\(r'\in\mathbb Z^5(0)\) be its primitive integral generator.  Since
\(E\) has dimension \(4\), this ray is cut out by at least three
linearly independent active inequalities.  No inequality indexed by
\(I\setminus S\) can be active, by the strict inequality above.
Consequently, there exist \(u_1,u_2,u_3\in S\) such that
\[
 r'\cdot u_1=r'\cdot u_2=r'\cdot u_3=0
\]
and the three vectors
\[
 u_1-\eta,\qquad u_2-\eta,\qquad u_3-\eta
\]
are linearly independent in \(\mathbb Q^5(0)\).  Equivalently,
\(u_1,u_2,u_3,\eta\) span a four-dimensional subspace of
\(\mathbb Q^5\).  Finally, the preceding argument gives
\[
 I(r')_{\geq0}=S=I(r)_{\geq0},
\]
as required.
\end{proof}

The lemma leads to the following finite enumeration.

\paragraph*{Enumeration algorithm.}
Let \(\mathcal F\) be the finite set of three-element subsets
\[
 \mathbf u=\{u_1,u_2,u_3\}\subset I.
\]
Initialize a provisional collection of supports by
\(\mathcal M':=\varnothing\), and perform the following steps.

\begin{enumerate}
\item[\textbf{Step 1.}]
For each \(\mathbf u\in\mathcal F\), let
\(U_{\mathbf u}\subset\mathbb Q^5\) be the subspace spanned by
\[
 u_1,u_2,u_3,\eta.
\]
If \(\dim_{\mathbb Q}U_{\mathbf u}\neq4\), discard the triple.  If
\(\dim_{\mathbb Q}U_{\mathbf u}=4\), compute a primitive integral
normal vector
\[
 0\neq r\in\mathbb Z^5(0)
\]
to \(U_{\mathbf u}\), determined up to sign.

\item[\textbf{Step 2.}]
For each normal vector obtained in Step~1, add both supports
\[
 I(r)_{\geq0}
 \qquad\text{and}\qquad
 I(-r)_{\geq0}
\]
to \(\mathcal M'\).

\item[\textbf{Step 3.}]
Replace \(\mathcal M'\) by its closure under the Weyl group
\(S_5\), acting by permutation of the coordinates, and delete
repeated supports.  Thus all possible containments
\[
 S\subseteq\pi(S'),
 \qquad
 S,S'\in\mathcal M',\quad \pi\in S_5,
\]
are represented as ordinary containments inside the resulting
collection.

\item[\textbf{Step 4.}]
Delete every support in \(\mathcal M'\) that is properly contained in
another support in \(\mathcal M'\).  Decompose the remaining maximal
supports into \(S_5\)-orbits, choose one representative from each
orbit, and choose its reduced defining normal vector in the dominant
Weyl chamber
\[
 r_0\geq r_1\geq r_2\geq r_3\geq r_4.
\]
\end{enumerate}

Lemma~\ref{lem:maximal-support-normal-quintic} proves that every
maximal element of \(\mathcal S\) occurs in this enumeration.  No
additional barycenter test is needed in the quintic case, because
\(\eta\in I(r)_{=0}\) for every \(r\in\mathbb Z^5(0)\).

\begin{remark}
The computation may be accelerated by pruning the three-element
subsets under the Weyl group, or by deleting duplicate normal
hyperplanes before constructing their supports.  These modifications
change only the running time.  The exhaustive calculation itself uses
exact integer linear algebra, the two choices of sign of every normal
vector, all coordinate permutations, and exact inclusion tests among
subsets of the finite set \(I\).
\end{remark}

\begin{proposition}
\label{prop:quintic-maximal-t-supports}
For \((\dim V,d)=(5,5)\), the maximal \(T\)-strictly semistable
supports are, up to permutation of the coordinates, the thirty-eight
supports
\[
 S_k:=I(r_k)_{\geq0}
 \qquad
 (k=1,\ldots,38)
\]
listed in Table~\ref{tab:maximal-t-semistable-supports}.  Each \(r_k\) is reduced and is written in the dominant
Weyl chamber
\[
 r_{k,0}\geq r_{k,1}\geq r_{k,2}\geq r_{k,3}\geq r_{k,4}.
\]
\end{proposition}

\begin{proof}
There are
\[
 \binom{126}{3}=325{,}500
\]
three-element subsets of \(I\).  Exact integer row reduction shows
that \(309{,}150\) of them satisfy
\[
 \dim_{\mathbb Q}\Span_{\mathbb Q}\{u_1,u_2,u_3,\eta\}=4.
\]
Taking primitive normal vectors, both choices of sign, and all
coordinate permutations produces \(29{,}760\) distinct supports of the
form \(I(r)_{\geq0}\).  Exact inclusion comparison in this
\(S_5\)-closed collection leaves \(3{,}350\) maximal supports, which
decompose into exactly \(38\) orbits under \(S_5\).

The vectors displayed in
Table~\ref{tab:maximal-t-semistable-supports} are the reduced dominant
representatives of these \(38\) orbits.  For every row, direct computation verifies
\[
 \sum_{i=0}^4 r_{k,i}=0,
 \qquad
 \gcd(r_{k,0},\ldots,r_{k,4})=1,
\]
the stated dominant ordering, and the listed value of
\(\lvert I(r_k)_{\geq0}\rvert\).  Maximality is checked against every
coordinate translate of every candidate support; hence no displayed
support is properly contained in another element of \(\mathcal S\).
Since
\[
 \eta\in I(r_k)_{=0}\subset S_k,
\]
a form with support exactly \(S_k\) is \(T\)-semistable, while the
nonzero one-parameter subgroup \(\lambda_{r_k}\) shows that it is not
\(T\)-stable.  Thus every \(S_k\) is a maximal strictly semistable
support.  Conversely,
Lemma~\ref{lem:maximal-support-normal-quintic} shows that every maximal
element of \(\mathcal S\) occurs in the exhaustive enumeration, up to
coordinate permutation.
\end{proof}

\begingroup
\small
\renewcommand{\arraystretch}{1.08}
\begin{longtable}{>{\raggedleft\arraybackslash}p{0.07\textwidth}
                  >{\centering\arraybackslash}p{0.48\textwidth}
                  >{\raggedleft\arraybackslash}p{0.12\textwidth}}
\caption{The thirty-eight maximal strictly \(T\)-semistable monomial supports,
represented by their reduced dominant normal vectors.}
\label{tab:maximal-t-semistable-supports}\\
\toprule
\(k\) & \(r_k\) & \(|S_k|\)\\
\midrule
\endfirsthead
\toprule
\(k\) & \(r_k\) & \(|S_k|\)\\
\midrule
\endhead
\midrule
\multicolumn{3}{r}{\emph{continued on the next page}}\\
\endfoot
\bottomrule
\endlastfoot
1 & \((36, 1, -4, -9, -24)\) & 58 \\
2 & \((27, 2, -3, -8, -18)\) & 59 \\
3 & \((3, 0, 0, -1, -2)\) & 62 \\
4 & \((24, 9, -6, -11, -16)\) & 62 \\
5 & \((21, 6, -4, -9, -14)\) & 62 \\
6 & \((18, 3, -2, -7, -12)\) & 62 \\
7 & \((5, 1, 0, -2, -4)\) & 63 \\
8 & \((4, 2, -1, -2, -3)\) & 63 \\
9 & \((19, 4, -1, -6, -16)\) & 64 \\
10 & \((12, 2, -3, -3, -8)\) & 64 \\
11 & \((14, 4, -1, -6, -11)\) & 65 \\
12 & \((13, 8, -2, -7, -12)\) & 66 \\
13 & \((3, 2, 0, -2, -3)\) & 67 \\
14 & \((4, 1, 0, -1, -4)\) & 67 \\
15 & \((9, 4, -1, -6, -6)\) & 68 \\
16 & \((3, 2, 1, -2, -4)\) & 69 \\
17 & \((12, 7, 2, -8, -13)\) & 69 \\
18 & \((4, 2, 0, -1, -5)\) & 69 \\
19 & \((2, 1, 0, -1, -2)\) & 69 \\
20 & \((8, 3, -2, -2, -7)\) & 69 \\
21 & \((16, 11, 6, -9, -24)\) & 70 \\
22 & \((16, 6, 1, -4, -19)\) & 70 \\
23 & \((4, -1, -1, -1, -1)\) & 70 \\
24 & \((14, 9, 4, -6, -21)\) & 71 \\
25 & \((11, 6, 1, -4, -14)\) & 71 \\
26 & \((24, 9, 4, -1, -36)\) & 73 \\
27 & \((18, 8, 3, -2, -27)\) & 73 \\
28 & \((12, 7, 2, -3, -18)\) & 73 \\
29 & \((6, 1, 1, -4, -4)\) & 73 \\
30 & \((6, 6, 1, -4, -9)\) & 73 \\
31 & \((7, 2, 2, -3, -8)\) & 74 \\
32 & \((3, 3, -2, -2, -2)\) & 75 \\
33 & \((2, 1, 0, 0, -3)\) & 76 \\
34 & \((8, 3, 3, -2, -12)\) & 76 \\
35 & \((4, 4, -1, -1, -6)\) & 76 \\
36 & \((1, 0, 0, 0, -1)\) & 80 \\
37 & \((2, 2, 2, -3, -3)\) & 81 \\
38 & \((1, 1, 1, 1, -4)\) & 91 \\
\end{longtable}
\endgroup

\section{\texorpdfstring{Boundary families for the full \(\SL(5)\)-action}{Boundary families for the full SL(5)-action}}
\label{sec:boundary-families-sl5}

In this section we pass from the maximal \(T\)-strictly semistable
supports
\[
S_k=I(r_k)_{\ge 0},
\qquad
k=1,\ldots,38,
\]
enumerated in Proposition~\ref{prop:quintic-maximal-t-supports} to boundary
families for the full
\(\SL(5)\)-action.  We keep the distinction between prequotient loci and
quotient-side boundary families throughout the section.

Let
\[
G=\SL(5),
\qquad
W=\Sym^5\C^5,
\qquad
\mathcal M^{\mathrm{GIT}}=\mathbb P(W)^{ss}//G,
\]
and let
\[
\pi:\mathbb P(W)^{ss}\longrightarrow \mathcal M^{\mathrm{GIT}}
\]
be the good quotient map.  We write
\[
\Sigma^{ss}:=
\mathbb P(W)^{ss}\setminus \mathbb P(W)^s
\]
for the strictly semistable locus before taking the quotient, and
\[
\partial\mathcal M^{\mathrm{GIT}}
:=
\pi(\Sigma^{ss})
\subset
\mathcal M^{\mathrm{GIT}}
\]
for the GIT boundary in the quotient.

\subsection{The prequotient and quotient-side families}

For each maximal support \(S_k\), set
\[
V_k
:=
\operatorname{Span}_{\C}\{x^u\mid u\in S_k\}
\subset W.
\]
We also use the notation
\[
\mathbb P(V_k)^\circ
:=
\{[f]\in \mathbb P(V_k)\mid \operatorname{Supp}(f)=S_k\}.
\]
Thus \(\mathbb P(V_k)^\circ\) is the open locus in the support family on
which all monomials in \(S_k\) occur with nonzero coefficient.  When we
speak of a general quintic supported on \(S_k\), we mean a point of a
nonempty Zariski-open subset of \(\mathbb P(V_k)^\circ\).

Define
\[
Y_k
:=
\overline{G\cdot \mathbb P(V_k)}
\subset
\mathbb P(W),
\]
where the closure is taken in \(\mathbb P(W)\).  Since
\[
S_k=I(r_k)_{\ge 0},
\]
every point of \(\mathbb P(V_k)\) is not stable with respect to the
one-parameter subgroup \(\lambda_{r_k}\).  Hence every point of
\(G\cdot \mathbb P(V_k)\) is not stable for the \(G\)-action.  As the
stable locus is open, the closure \(Y_k\) is still contained in the
non-stable locus:
\[
Y_k\subset \mathbb P(W)\setminus \mathbb P(W)^s.
\]
Consequently,
\[
Y_k\cap \mathbb P(W)^{ss}
=
Y_k\cap \Sigma^{ss}.
\]

We define the \(k\)-th prequotient boundary family by
\[
\widetilde{\Phi}_k
:=
Y_k\cap \Sigma^{ss}
=
\overline{G\cdot \mathbb P(V_k)}
\cap
\Sigma^{ss}
\subset
\mathbb P(W)^{ss}.
\]
The corresponding quotient-side boundary family is
\[
\Phi_k
:=
\pi(\widetilde{\Phi}_k)
\subset
\partial\mathcal M^{\mathrm{GIT}}.
\]
Thus \(\widetilde{\Phi}_k\) is the \(G\)-saturated prequotient locus
associated with the support \(S_k\), whereas \(\Phi_k\) is its image in
the GIT quotient.

All inclusions among the symbols \(\Phi_k\) are henceforth inclusions
inside the quotient \(\mathcal M^{\mathrm{GIT}}\).  In particular, a
statement such as
\[
\Phi_k\subset \Phi_\ell
\]
is a quotient-side statement.  Similarly, the notation
\[
\dim \Phi_k
\]
always denotes the dimension of the quotient-side family, not the
dimension of the prequotient locus \(\widetilde{\Phi}_k\) inside
\(\mathbb P(W)^{ss}\).

\subsection{Covering of the strictly semistable locus}

The maximal supports from
Proposition~\ref{prop:quintic-maximal-t-supports} cover all strictly semistable
quintics after allowing an \(\SL(5)\)-change of coordinates.  This gives
the prequotient covering statement.

\begin{proposition}[Covering by the maximal-support families]
\label{prop:strictly-semistable-cover}
For the natural action of \(G=\SL(5)\) on \(\mathbb P(W)\), one has
\[
\Sigma^{ss}
=
\widetilde{\Phi}_1\cup\cdots\cup\widetilde{\Phi}_{38}.
\]
Consequently,
\[
\partial\mathcal M^{\mathrm{GIT}}
=
\Phi_1\cup\cdots\cup\Phi_{38}.
\]
\end{proposition}

\begin{proof}
The inclusion
\[
\widetilde{\Phi}_1\cup\cdots\cup\widetilde{\Phi}_{38}
\subset
\Sigma^{ss}
\]
is immediate from the definition of the prequotient families.

Conversely, let \([f]\in \Sigma^{ss}\).  Since \([f]\) is not stable,
the Hilbert--Mumford numerical criterion gives an element
\(\sigma\in G\) and a nonzero one-parameter subgroup
\[
\lambda_r(t)
=
\operatorname{diag}(t^{r_0},\ldots,t^{r_4}),
\qquad
r=(r_0,\ldots,r_4)\in \mathbb Z^5(0),
\]
such that
\[
\operatorname{Supp}(\sigma\cdot f)
\subset
I(r)_{\ge 0}.
\]
Because \([f]\) is semistable for \(G\), the point
\(\sigma\cdot[f]\) is semistable for the fixed maximal torus \(T\).
Equivalently,
\[
\eta=(1,1,1,1,1)
\in
\operatorname{Conv}(\operatorname{Supp}(\sigma\cdot f)).
\]

Choose a maximal element of the finite collection
\[
\{I(s)_{\ge 0}\mid 0\ne s\in \mathbb Z^5(0)\}
\]
which contains \(\operatorname{Supp}(\sigma\cdot f)\).  Since this
maximal support contains the support of a \(T\)-semistable form, its
convex hull contains \(\eta\).  Hence, by
Proposition~\ref{prop:quintic-maximal-t-supports}, it is a
coordinate permutation of one of the thirty-eight supports
\[
S_1,\ldots,S_{38}.
\]
After composing \(\sigma\) with the corresponding coordinate change, we
obtain some \(g\in G\) and some \(k\) such that
\[
g\cdot [f]\in \mathbb P(V_k).
\]
Therefore
\[
[f]\in G\cdot \mathbb P(V_k)\cap \Sigma^{ss}
\subset
\widetilde{\Phi}_k.
\]
This proves
\[
\Sigma^{ss}
\subset
\widetilde{\Phi}_1\cup\cdots\cup\widetilde{\Phi}_{38}.
\]
Applying the quotient map \(\pi\) gives
\[
\partial\mathcal M^{\mathrm{GIT}}
=
\pi(\Sigma^{ss})
=
\Phi_1\cup\cdots\cup\Phi_{38}.
\]
\end{proof}

\section{Quotient-side identifications from supporting hyperplanes}
\label{sec:quotient-side-inclusions}

The thirty-eight labels
\[
  S_k=I(r_k)_{\geq0},
  \qquad
  k=1,\ldots,38,
\]
record oriented half-space supports for the fixed diagonal torus.  The
orientation is relevant before taking the quotient: replacing \(r_k\) by
\(-r_k\) exchanges the two sides of the supporting hyperplane.  The good
quotient, however, retains only the weight-zero degeneration on the
supporting hyperplane.  We make this precise below and then apply it to the
thirty-eight normal vectors of Proposition~\ref{prop:quintic-maximal-t-supports}.

Throughout this section, put
\[
  X:=\mathbb P(W),
  \qquad
  X^{ss}:=\mathbb P(W)^{ss},
\]
and let
\[
  \pi:X^{ss}\longrightarrow\mathcal M^{\mathrm{GIT}}
\]
be the good quotient map.  We use the standard facts that a good quotient
maps closed \(G\)-invariant subsets to closed subsets and is constant on
\(G\)-orbits; see, for example,~\cite{MFK94,Dol03}.

\subsection{The quotient image is determined by the weight-zero subspace}

Let
\[
  0\neq r=(r_0,\ldots,r_4)\in\mathbb Z^5(0),
\]
and write the corresponding weight decomposition as
\[
  W=\bigoplus_{m\in\mathbb Z}W_m(r),
  \qquad
  W_m(r)
  :=\Span_{\mathbb C}\{x^u\mid r\cdot u=m\}.
\]
Set
\[
  V(r):=\bigoplus_{m\geq0}W_m(r)
  =\Span_{\mathbb C}\{x^u\mid r\cdot u\geq0\}
\]
and
\[
  Z(r):=W_0(r)
  =\Span_{\mathbb C}\{x^u\mid r\cdot u=0\}.
\]
If
\[
  H_r:=\lambda_r(\mathbb G_m),
\]
then
\[
  Z(r)=W^{H_r}.
\]
Finally, define
\[
  Y(r):=\overline{G\cdot\mathbb P(V(r))}\subset X
\]
and
\[
  \Phi(r):=\pi\bigl(Y(r)\cap X^{ss}\bigr)
  \subset\mathcal M^{\mathrm{GIT}}.
\]
Every point of \(\mathbb P(V(r))\) is non-stable with respect to
\(\lambda_r\).  Hence \(Y(r)\) is contained in the non-stable locus, and
\[
  Y(r)\cap X^{ss}=Y(r)\cap\Sigma^{ss}.
\]
For the normals in Proposition~\ref{prop:quintic-maximal-t-supports}, one has
\[
  V(r_k)=V_k,
  \qquad
  Y(r_k)=Y_k,
  \qquad
  \Phi(r_k)=\Phi_k.
\]

\begin{lemma}[Zero-weight reduction in the good quotient]
\label{lem:zero-weight-reduction-in-quotient}
For every nonzero \(r\in\mathbb Z^5(0)\), one has
\begin{equation}
\label{eq:zero-weight-quotient-image}
  \Phi(r)
  =
  \overline{
    \pi\bigl(\mathbb P(V(r))\cap X^{ss}\bigr)
  }
  =
  \overline{
    \pi\bigl(\mathbb P(Z(r))\cap X^{ss}\bigr)
  },
\end{equation}
where the closures are taken in \(\mathcal M^{\mathrm{GIT}}\).
\end{lemma}

\begin{proof}
Set
\[
  A(r):=\mathbb P(V(r))\cap X^{ss}.
\]
We first prove
\begin{equation}
\label{eq:prequotient-closure-to-quotient-closure}
  \Phi(r)=\overline{\pi(A(r))}.
\end{equation}
If \(A(r)=\varnothing\), then \(\mathbb P(V(r))\) is contained in the
unstable locus, which is closed and \(G\)-invariant.  Therefore
\(Y(r)\cap X^{ss}=\varnothing\), and
\eqref{eq:prequotient-closure-to-quotient-closure} is immediate.

Assume now that \(A(r)\neq\varnothing\).  Since \(X^{ss}\) is open in
\(X\), the subset \(A(r)\) is a nonempty open, hence dense, subset of the
projective space \(\mathbb P(V(r))\).  Consequently,
\[
  Y(r)\cap X^{ss}
  =
  \overline{G\cdot A(r)}^{\,X^{ss}}.
\]
Indeed, \(G\times A(r)\) is dense in
\(G\times\mathbb P(V(r))\), so the action morphism shows that
\(G\cdot\mathbb P(V(r))\) is contained in the closure of
\(G\cdot A(r)\).

The subset \(Y(r)\cap X^{ss}\) is closed and \(G\)-invariant in
\(X^{ss}\).  Hence its image under the good quotient is closed.  Since it
contains \(A(r)\), we obtain
\[
  \overline{\pi(A(r))}
  \subseteq
  \pi\bigl(Y(r)\cap X^{ss}\bigr).
\]
Conversely, continuity of \(\pi\), together with its \(G\)-invariance,
gives
\[
\begin{aligned}
  \pi\bigl(Y(r)\cap X^{ss}\bigr)
  &\subseteq
  \overline{\pi(G\cdot A(r))} \\
  &=
  \overline{\pi(A(r))}.
\end{aligned}
\]
This proves \eqref{eq:prequotient-closure-to-quotient-closure}.

It remains to replace \(V(r)\) by its weight-zero subspace.  Let
\([f]\in A(r)\), and write
\[
  f=f_0+f_{>0},
  \qquad
  f_0\in Z(r),
  \qquad
  f_{>0}\in\bigoplus_{m>0}W_m(r).
\]
By \eqref{eq:quintic-zero-weight-limit},
\[
  \lim_{t\to0}\lambda_r(t)\cdot f=f_0
\]
in the affine space \(W\).  Since \([f]\) is semistable, there exists a
homogeneous invariant
\[
  F\in\mathbb C[W]^G
\]
of positive degree such that \(F(f)\neq0\).  The invariance of \(F\) gives
\[
  F(f_0)
  =
  \lim_{t\to0}F\bigl(\lambda_r(t)\cdot f\bigr)
  =F(f)
  \neq0.
\]
Thus \(f_0\neq0\) and \([f_0]\in X^{ss}\).

The morphism
\[
  \mathbb G_m\longrightarrow X^{ss},
  \qquad
  t\longmapsto\lambda_r(t)\cdot[f],
\]
extends to \(\mathbb A^1\) with value \([f_0]\) at \(0\).  Its composition
with \(\pi\) is constant on \(\mathbb G_m\).  Since
\(\mathcal M^{\mathrm{GIT}}\) is separated, the composition is constant on
\(\mathbb A^1\), and therefore
\[
  \pi([f])=\pi([f_0]).
\]
It follows that
\[
  \pi\bigl(\mathbb P(V(r))\cap X^{ss}\bigr)
  \subseteq
  \pi\bigl(\mathbb P(Z(r))\cap X^{ss}\bigr).
\]
The reverse inclusion follows from \(Z(r)\subseteq V(r)\).  Hence the two
images are equal.  Combining this equality with
\eqref{eq:prequotient-closure-to-quotient-closure} proves
\eqref{eq:zero-weight-quotient-image}.
\end{proof}

\begin{corollary}[Supporting-hyperplane invariance]
\label{cor:supporting-hyperplane-invariance}
Let \(0\neq r,s\in\mathbb Z^5(0)\).  If there exists \(g\in G\) such that
\[
  Z(s)=gZ(r),
\]
then
\[
  \Phi(s)=\Phi(r).
\]
In particular:
\begin{enumerate}
\item one has
\[
  \Phi(r)=\Phi(-r);
\]
\item if
\[
  H_s=gH_rg^{-1}
\]
for some \(g\in G\), then
\[
  \Phi(s)=\Phi(r).
\]
\end{enumerate}
Thus the quotient-side family is determined by the zero-weight monomial
subspace, up to the \(G\)-action.
\end{corollary}

\begin{proof}
The element \(g\) preserves the semistable locus, and the quotient map
satisfies \(\pi(g\cdot x)=\pi(x)\).  Hence
\[
  \pi\bigl(\mathbb P(Z(s))\cap X^{ss}\bigr)
  =
  \pi\bigl(\mathbb P(Z(r))\cap X^{ss}\bigr).
\]
Taking closures and applying
Lemma~\ref{lem:zero-weight-reduction-in-quotient} gives
\(\Phi(s)=\Phi(r)\).

For the first special case, one has
\[
  Z(-r)=Z(r).
\]
For the second, conjugacy of the one-dimensional subtori gives
\[
  Z(s)=W^{H_s}=gW^{H_r}=gZ(r).
\]
Both assertions therefore follow from the general statement.
\end{proof}

\begin{remark}
Lemma~\ref{lem:zero-weight-reduction-in-quotient} applies to every semistable
point of \(\mathbb P(V(r))\), not only to a general full-support form or to a
chosen normal-form locus.  Thus the quotient-side identifications below do
not require a case-by-case comparison of the closed-orbit normal forms.
\end{remark}

\subsection{The involution on the thirty-eight normal vectors}

Let
\[
  \rho:\quad
  x_0\longmapsto x_4,
  \quad
  x_1\longmapsto x_3,
  \quad
  x_2\longmapsto x_2,
  \quad
  x_3\longmapsto x_1,
  \quad
  x_4\longmapsto x_0.
\]
The underlying permutation is \((0\,4)(1\,3)\), so
\[
  \det(\rho)=1
\]
and hence \(\rho\in G=\SL(5)\).

For
\[
  r=(r_0,r_1,r_2,r_3,r_4),
\]
define
\[
  \iota(r):=(-r_4,-r_3,-r_2,-r_1,-r_0).
\]
Conjugation by \(\rho\) reverses the entries of the weight vector:
\[
  \rho\lambda_r(t)\rho^{-1}
  =
  \lambda_{(r_4,r_3,r_2,r_1,r_0)}(t).
\]
Since a one-parameter subgroup and its inverse have the same image,
\[
  H_{(r_4,r_3,r_2,r_1,r_0)}
  =
  H_{\iota(r)}.
\]
Consequently,
\begin{equation}
\label{eq:iota-invariance-of-phi}
  \Phi(r)=\Phi(\iota(r))
\end{equation}
by Corollary~\ref{cor:supporting-hyperplane-invariance}.

For the dominant normal vectors in
Proposition~\ref{prop:quintic-maximal-t-supports}, the involution \(\iota\)
has the seventeen two-cycles displayed in
Table~\ref{tab:opposite-normal-pairs}.

\begingroup
\small
\renewcommand{\arraystretch}{1.10}
\begin{longtable}{>{\centering\arraybackslash}p{0.12\textwidth}
                  >{\centering\arraybackslash}p{0.36\textwidth}
                  >{\centering\arraybackslash}p{0.36\textwidth}}
\caption{The seventeen opposite-half-space pairs.}
\label{tab:opposite-normal-pairs}\\
\toprule
pair \((k,k')\) & \(r_k\) & \(\iota(r_k)=r_{k'}\)\\
\midrule
\endfirsthead
\toprule
pair \((k,k')\) & \(r_k\) & \(\iota(r_k)=r_{k'}\)\\
\midrule
\endhead
\midrule
\multicolumn{3}{r}{\emph{continued on the next page}}\\
\endfoot
\bottomrule
\endlastfoot
\((1,26)\)  & \((36,1,-4,-9,-24)\)   & \((24,9,4,-1,-36)\) \\
\((2,27)\)  & \((27,2,-3,-8,-18)\)   & \((18,8,3,-2,-27)\) \\
\((3,33)\)  & \((3,0,0,-1,-2)\)      & \((2,1,0,0,-3)\) \\
\((4,21)\)  & \((24,9,-6,-11,-16)\)  & \((16,11,6,-9,-24)\) \\
\((5,24)\)  & \((21,6,-4,-9,-14)\)   & \((14,9,4,-6,-21)\) \\
\((6,28)\)  & \((18,3,-2,-7,-12)\)   & \((12,7,2,-3,-18)\) \\
\((7,18)\)  & \((5,1,0,-2,-4)\)      & \((4,2,0,-1,-5)\) \\
\((8,16)\)  & \((4,2,-1,-2,-3)\)     & \((3,2,1,-2,-4)\) \\
\((9,22)\)  & \((19,4,-1,-6,-16)\)   & \((16,6,1,-4,-19)\) \\
\((10,34)\) & \((12,2,-3,-3,-8)\)    & \((8,3,3,-2,-12)\) \\
\((11,25)\) & \((14,4,-1,-6,-11)\)   & \((11,6,1,-4,-14)\) \\
\((12,17)\) & \((13,8,-2,-7,-12)\)   & \((12,7,2,-8,-13)\) \\
\((15,30)\) & \((9,4,-1,-6,-6)\)     & \((6,6,1,-4,-9)\) \\
\((20,31)\) & \((8,3,-2,-2,-7)\)     & \((7,2,2,-3,-8)\) \\
\((23,38)\) & \((4,-1,-1,-1,-1)\)    & \((1,1,1,1,-4)\) \\
\((29,35)\) & \((6,1,1,-4,-4)\)      & \((4,4,-1,-1,-6)\) \\
\((32,37)\) & \((3,3,-2,-2,-2)\)     & \((2,2,2,-3,-3)\) \\
\end{longtable}
\endgroup

For every row \((k,k')\) of
Table~\ref{tab:opposite-normal-pairs}, equation
\eqref{eq:iota-invariance-of-phi} gives
\begin{equation}
\label{eq:pair-identification-from-iota}
  \Phi_k
  =\Phi(r_k)
  =\Phi(\iota(r_k))
  =\Phi(r_{k'})
  =\Phi_{k'}.
\end{equation}
We record all seventeen equalities in one proposition.

\begin{proposition}[Opposite-half-space identifications]
\label{prop:opposite-half-space-identifications}
One has
\[
\begin{aligned}
  &\Phi_1=\Phi_{26},\quad \Phi_2=\Phi_{27},\quad
   \Phi_3=\Phi_{33},\quad \Phi_4=\Phi_{21},\quad
   \Phi_5=\Phi_{24},\quad \Phi_6=\Phi_{28},\\
  &\Phi_7=\Phi_{18},\quad \Phi_8=\Phi_{16},\quad
   \Phi_9=\Phi_{22},\quad \Phi_{10}=\Phi_{34},\quad
   \Phi_{11}=\Phi_{25},\quad \Phi_{12}=\Phi_{17},\\
  &\Phi_{15}=\Phi_{30},\quad \Phi_{20}=\Phi_{31},\quad
   \Phi_{23}=\Phi_{38},\quad \Phi_{29}=\Phi_{35},\quad
   \Phi_{32}=\Phi_{37}.
\end{aligned}
\]
\end{proposition}

\begin{proof}
Apply \eqref{eq:pair-identification-from-iota} to the seventeen rows of
Table~\ref{tab:opposite-normal-pairs}.
\end{proof}

The four remaining dominant normals are fixed by \(\iota\):
\[
\begin{aligned}
  r_{13}&=(3,2,0,-2,-3),
  &\qquad
  r_{14}&=(4,1,0,-1,-4),\\
  r_{19}&=(2,1,0,-1,-2),
  &
  r_{36}&=(1,0,0,0,-1).
\end{aligned}
\]
Thus the involution on the thirty-eight dominant half-space labels has
seventeen two-cycles and four fixed points.  Equivalently, the thirty-eight
oriented half-space supports arise from twenty-one supporting hyperplanes up
to coordinate permutation and reversal of the normal.

\begin{corollary}[Reduction to twenty-one candidate quotient images]
\label{cor:twenty-one-quotient-side-candidate-cover}
The GIT boundary is covered by the following twenty-one quotient images:
\[
  \Phi_1,\Phi_2,\ldots,\Phi_{15},
  \Phi_{19},\Phi_{20},\Phi_{23},\Phi_{29},\Phi_{32},\Phi_{36}.
\]
Equivalently,
\begin{equation}
\label{eq:boundary-cover-after-identifications}
\begin{aligned}
  \partial\mathcal M^{\mathrm{GIT}}
  ={}&
  \Phi_1\cup\Phi_2\cup\cdots\cup\Phi_{15}
  \cup\Phi_{19}\cup\Phi_{20}\\
  &{}
  \cup\Phi_{23}\cup\Phi_{29}\cup\Phi_{32}\cup\Phi_{36}.
\end{aligned}
\end{equation}
\end{corollary}

\begin{proof}
Proposition~\ref{prop:strictly-semistable-cover} gives the cover by all
thirty-eight support-indexed images.  Replacing the second member of each
two-cycle in Table~\ref{tab:opposite-normal-pairs} by the first member and
applying Proposition~\ref{prop:opposite-half-space-identifications} gives
\eqref{eq:boundary-cover-after-identifications}.
\end{proof}

\begin{remark}
At this stage the twenty-one entries are quotient images selected from the
support-indexed cover.  Section~\ref{sec:normal-form} constructs a nonempty
polystable locus in each entry and proves that these images are closed and
irreducible.  Section~\ref{sec:pairwise-noninclusion} then excludes all
further equalities and containments.
\end{remark}

\section{The twenty-one closed-orbit families}\label{sec:normal-form}

\providecommand{\qtfNFcell}[1]{\(\begin{aligned}[t]#1\end{aligned}\)}
\providecommand{\qtfParamcell}[1]{\begin{tabular}[t]{@{}l@{}}#1\end{tabular}}

The preceding section reduces the boundary cover from thirty-eight oriented
support labels to twenty-one quotient images.  From this point through the
end of the main body we use the component index \(a=1,\ldots,21\).  We choose
\[
  k(a)=1,2,3,4,5,6,7,8,9,10,11,12,13,14,15,19,20,23,29,32,36
\]
in that order, and put
\[
  \Theta_a:=\Phi_{k(a)},\qquad
  F_a:=\nf_{k(a)},\qquad X_a:=V(F_a)\subset\PP^4.
\]
Let \(k'\) be the partner of \(k(a)\) in a two-cycle, and let
\[
  \rho(x_0,x_1,x_2,x_3,x_4)=(x_4,x_3,x_2,x_1,x_0).
\]
Then \(\rho\) identifies the zero-weight spaces,
\[
  Z(r_{k'})=\rho Z(r_{k(a)}),
\]
and Corollary~\ref{cor:supporting-hyperplane-invariance} gives
\(\Phi_{k'}=\Phi_{k(a)}=\Theta_a\).  Thus one selected support label suffices
at the level of quotient images.  This statement does not assert that
\(\rho\) carries the two displayed normal-form formulas coefficientwise into
one another: residual centralizer normalization and a reparametrization may
still be required after applying \(\rho\).  The relation between the paired
normal-form loci and the general closed orbit of \(\Theta_a\) is recorded
precisely after Corollary~\ref{cor:twenty-one-quotient-side-cover}.

\subsection{Centralizer reduction and closed-orbit criteria}

For an original support case \(k\), let \(f_k\) be a general quintic supported
on \(S_k=I(r_k)_{\geq0}\), and set
\[
  \phi_k:=\lim_{t\to0}\lambda_{r_k}(t)\cdot f_k,
  \qquad H_k:=\lambda_{r_k}(\Gm),
  \qquad C_k:=C_G(H_k).
\]
The limit \(\phi_k\) is the weight-zero part of \(f_k\) and belongs to the
fixed space \(W^{H_k}\).  Since \(H_k\) is a torus, \(C_k\) is reductive; since
\(H_k\) acts trivially on \(W^{H_k}\), the action factors through
\(C_k/H_k\).  Affine Luna reduction gives
\begin{equation}
  G\cdot\phi_k\text{ is closed in }W
  \quad\Longleftrightarrow\quad
  C_k\cdot\phi_k\text{ is closed in }W^{H_k}.
  \label{eq:affine-luna-centralizer-reduction}
\end{equation}

When the effective centralizer \(C_k/H_k\) is a torus, we use the standard
relative-interior criterion: for a torus acting linearly on a vector space,
the orbit of a vector is closed if and only if the origin lies in the
relative interior of the convex hull of the characters occurring with
nonzero coefficient.  We abbreviate this test by C-H.

For the non-toric centralizers we use the following criterion.  Let \(R\) be
a complex reductive algebraic group acting on an affine variety \(X\), and
write
\[
  Y(R):=\operatorname{Hom}_{\mathrm{alg.grp}}(\Gm,R).
\]
For \(\eta\in Y(R)\), put
\[
  P_R(\eta):=
  \left\{g\in R\ \middle|\
  \lim_{t\to0}\eta(t)g\eta(t)^{-1}\text{ exists in }R\right\}.
\]
Two one-parameter subgroups \(\eta,\eta'\in Y(R)\) are called
\emph{opposite} if \(P_R(\eta)\cap P_R(\eta')\) is a Levi subgroup of both
parabolic subgroups.  For \(x\in X\), define
\[
  \Lambda_x^R:=
  \left\{\eta\in Y(R)\ \middle|\
  \lim_{t\to0}\eta(t)\cdot x\text{ exists in }X\right\}.
\]
A subset of \(Y(R)\) is \emph{symmetric} if every member has an opposite
member in the subset.

\begin{theorem}[Casimiro--Florentino criterion]
\label{thm:casimiro-florentino-criterion}
Let \(R\) be a complex reductive algebraic group acting on an affine variety
\(X\), and let \(x\in X\).  Then
\[
  R\cdot x\text{ is closed in }X
  \quad\Longleftrightarrow\quad
  \Lambda_x^R\text{ is symmetric}.
\]
\end{theorem}

This is \cite[Theorem~4.9]{CF12}.  The inverse one-parameter subgroup
\(\eta^{-1}(t)=\eta(t)^{-1}\) is opposite to \(\eta\), because
\[
  P_R(\eta)\cap P_R(\eta^{-1})
  =C_R\bigl(\eta(\Gm)\bigr)
\]
is a Levi subgroup of both parabolics.  Hence invariance of
\(\Lambda_x^R\) under \(\eta\mapsto\eta^{-1}\) is sufficient for closedness.
We abbreviate this test by C-F.  In the support calculations of
Appendix~\ref{app:normal-forms}, generic nonvanishing,
discriminant, resultant, and residual-stability conditions combine with
linear balancing identities to prove the required inversion symmetry.

More precisely, in a non-toric support case we start with an arbitrary
\[
  \eta\in\Lambda_{\phi_k}^{C_k},
\]
diagonalize \(\eta\) inside the repeated-weight blocks of \(C_k\), and
translate existence of the affine limit into nonnegativity conditions on the
weights of the monomials occurring in \(\phi_k\).  The genericity conditions,
the determinant-one relation, and the resulting balancing identities force
\(\eta\) to lie in the distinguished torus \(H_k\), which acts trivially on
\(W^{H_k}\).  In each non-toric case the conclusion has the form
\[
  \Lambda_{\phi_k}^{C_k}
  =\{\lambda_{r_k}(t^m)\mid m\in\mathbb Z\}.
\]
This set is invariant under \(m\mapsto-m\), hence symmetric.  It follows that
\(C_k\cdot\phi_k\) is closed in \(W^{H_k}\), and
\eqref{eq:affine-luna-centralizer-reduction} then gives that
\(G\cdot\phi_k\) is closed in \(W\).

After closedness has been established, the residual action of \(C_k\),
together with projective rescaling, is used to eliminate redundant
coefficients and obtain an explicit normal form \(\nf_k\).  Thus every
support calculation in Appendix~\ref{app:normal-forms} has three
steps: extraction of the weight-zero limit, proof of polystability, and
reduction to normal form.  The support-by-support organization is retained in
that appendix because these operations depend on the original thirty-eight
maximal supports; the main body records only one representative for each
quotient component.

\subsection{Component representatives and dimensions}

Table~\ref{tab:component-normal-forms} records one general polystable
weight-zero representative for each quotient image, the closed-orbit test,
and its quotient dimension.  The full support-indexed list and every
case-by-case verification are given in
Appendix~\ref{app:normal-forms}.

Let us put
\[
B_{\lambda}(u,v):=uv(u-v)(u-\lambda v),\qquad
B_{\lambda,\mu}(u,v):=v(u-v)(u-\lambda v)(u-\mu v).
\]
Here \(B_d,C_d,Q_d,L_d,M_d\) denote general forms of the indicated degree,
and \(Q_{2,2}\) denotes a bihomogeneous form of bidegree \((2,2)\).
All parameters are understood to lie in the relevant nonempty Zariski-open
generic locus.

\begingroup
\scriptsize
\renewcommand{\arraystretch}{1.18}
\setlength{\tabcolsep}{0.8pt}
\setlength{\extrarowheight}{1pt}
\begin{longtable}{@{}
  >{\raggedleft\arraybackslash}p{0.027\textwidth}
  >{\centering\arraybackslash}p{0.092\textwidth}
  >{\centering\arraybackslash}p{0.052\textwidth}
  >{\raggedright\arraybackslash}p{0.525\textwidth}
  >{\raggedright\arraybackslash}p{0.230\textwidth}
  >{\raggedleft\arraybackslash}p{0.035\textwidth}
@{}}
\caption{Closed-orbit representatives for the twenty-one quotient components.  The second column records the orbit of the oriented support label under reversal of the supporting half-space; a singleton is a fixed case.  The displayed normal form is attached to the selected representative label, and no coefficientwise identification of paired displayed formulas by coordinate reversal is asserted.  The tests C-H and C-F are the relative-interior and Casimiro--Florentino criteria described in this section.}
\label{tab:component-normal-forms}\\
\toprule
\(a\) & support orbit & Test & Representative \(F_a=\nf_{k(a)}\) & Parameters & \(d_a\)\\
\midrule
\endfirsthead
\toprule
\(a\) & support orbit & Test & Representative \(F_a=\nf_{k(a)}\) & Parameters & \(d_a\)\\
\midrule
\endhead
\midrule
\multicolumn{6}{r}{\emph{continued on the next page}}\\
\endfoot
\bottomrule
\endlastfoot

1 & \(1\leftrightarrow26\) & C-H & \qtfNFcell{
F_{1}={}&x_1^4x_2+x_0x_3^4+x_0x_2^3x_4\\
&+x_0x_1x_2x_3x_4+\alpha x_0^2x_4^3
} & \qtfParamcell{$\alpha\in\Cstar$} & 1\\
\midrule

2 & \(2\leftrightarrow27\) & C-H & \qtfNFcell{
F_{2}={}&x_1^3x_2^2+x_1^4x_3+x_0x_2x_3^3+x_0x_2^3x_4\\
&+\alpha x_0x_1x_2x_3x_4+\beta x_0^2x_4^3
} & \qtfParamcell{$\alpha,\beta\in\Cstar$} & 2\\
\midrule

3 & \(3\leftrightarrow33\) & C-F & \qtfNFcell{
F_{3}={}&B_5(x_1,x_2)+x_0x_2x_3^3\\
&+x_0(x_1^2+x_2^2)x_3x_4+x_0^2x_4^3
} & \qtfParamcell{$B_5$: general\\ binary quintic} & 6\\
\midrule

4 & \(4\leftrightarrow21\) & C-H & \qtfNFcell{
F_{4}={}&x_1^2x_2^3+x_1^3x_3x_4+x_0x_2^4+x_0x_1x_3^3\\
&+\alpha x_0x_1x_2x_3x_4+\beta x_0^2x_4^3
} & \qtfParamcell{$\alpha,\beta\in\Cstar$} & 2\\
\midrule

5 & \(5\leftrightarrow24\) & C-H & \qtfNFcell{
F_{5}={}&x_1^2x_2^3+x_1^3x_3^2+x_1^3x_2x_4+x_0x_2^3x_3\\
&+\alpha x_0x_1x_3^3+\beta x_0x_1x_2x_3x_4+\gamma x_0^2x_4^3
} & \qtfParamcell{$\alpha,\beta,\gamma\in\Cstar$} & 3\\
\midrule

6 & \(6\leftrightarrow28\) & C-H & \qtfNFcell{
F_{6}={}&x_1^2x_2^3+x_1^3x_2x_3+x_1^4x_4+x_0x_2^2x_3^2\\
&+\alpha x_0x_2^3x_4+\beta x_0x_1x_3^3+\gamma x_0x_1x_2x_3x_4\\
&+\delta x_0x_1^2x_4^2+\epsilon x_0^2x_4^3
} & \qtfParamcell{$\alpha,\beta,\gamma,\delta,$\\ $\epsilon\in\Cstar$} & 5\\
\midrule

7 & \(7\leftrightarrow18\) & C-H & \qtfNFcell{
F_{7}={}&x_2^5+x_1^2x_2^2x_3+x_1^4x_4+x_0x_1x_3^3\\
&+\alpha x_0x_1x_2x_3x_4+\beta x_0^2x_3x_4^2
} & \qtfParamcell{$\alpha,\beta\in\Cstar$} & 2\\
\midrule

8 & \(8\leftrightarrow16\) & C-H & \qtfNFcell{
F_{8}={}&x_1^2x_2^2x_3+x_1^3x_4^2+x_0x_2^4+x_0x_1x_3^3\\
&+\alpha x_0x_1x_2x_3x_4+\beta x_0^2x_3x_4^2
} & \qtfParamcell{$\alpha,\beta\in\Cstar$} & 2\\
\midrule

9 & \(9\leftrightarrow22\) & C-H & \qtfNFcell{
F_{9}={}&x_1x_2^4+x_1^2x_2^2x_3+\alpha x_1^3x_3^2+x_1^4x_4\\
&+x_0x_2x_3^3+\beta x_0x_2^3x_4+\gamma x_0x_1x_2x_3x_4\\
&+\delta x_0^2x_3x_4^2
} & \qtfParamcell{$\alpha,\beta,\gamma,\delta\in\Cstar$} & 4\\
\midrule

10 & \(10\leftrightarrow34\) & C-F & \qtfNFcell{
F_{10}={}&x_1^3Q_{\mathbf q}(x_2,x_3)+x_1^4x_4+x_0B_\lambda(x_2,x_3)\\
&+x_0x_1x_4R_{\mathbf r}(x_2,x_3)+\beta x_0x_1^2x_4^2+x_0^2x_4^3
} & \qtfParamcell{$\lambda\in\C\setminus\{0,1\}$;\\ $\mathbf q,\mathbf r\in\C^3$, $\beta\in\C$} & 8\\
\midrule

11 & \(11\leftrightarrow25\) & C-H & \qtfNFcell{
F_{11}={}&x_1x_2^4+\alpha x_1^2x_2^2x_3+\beta x_1^3x_3^2+x_1^3x_2x_4\\
&+\gamma x_0x_2^2x_3^2+x_0x_2^3x_4+\delta x_0x_1x_3^3\\
&+\epsilon x_0x_1x_2x_3x_4+\zeta x_0x_1^2x_4^2+x_0^2x_3x_4^2
} & \qtfParamcell{$\alpha,\beta,\gamma,\delta,$\\ $\epsilon,\zeta\in\Cstar$} & 6\\
\midrule

12 & \(12\leftrightarrow17\) & C-H & \qtfNFcell{
F_{12}={}&x_1x_2^4+\alpha x_1^2x_2x_3^2+\beta x_1^2x_2^2x_4+x_1^3x_4^2\\
&+x_0x_2^3x_3+\gamma x_0x_1x_3^3+\delta x_0x_1x_2x_3x_4\\
&+\epsilon x_0^2x_3^2x_4+x_0^2x_2x_4^2
} & \qtfParamcell{$\alpha,\beta,\gamma,\delta,$\\ $\epsilon\in\Cstar$} & 5\\
\midrule

13 & \(13\) & C-H & \qtfNFcell{
F_{13}={}&x_2^5+\alpha x_1x_2^3x_3+\beta x_1^2x_2x_3^2+x_1^3x_4^2\\
&+x_0x_2^3x_4+\gamma x_0x_1x_2x_3x_4+x_0^2x_3^3\\
&+\delta x_0^2x_2x_4^2
} & \qtfParamcell{$\alpha,\beta,\gamma,\delta\in\Cstar$} & 4\\
\midrule

14 & \(14\) & C-H & \qtfNFcell{
F_{14}={}&x_2^5+\alpha x_1x_2^3x_3+\beta x_1^2x_2x_3^2+x_1^4x_4+x_0x_3^4\\
&+\gamma x_0x_2^3x_4+\delta x_0x_1x_2x_3x_4+x_0^2x_2x_4^2
} & \qtfParamcell{$\alpha,\beta,\gamma,\delta\in\Cstar$} & 4\\
\midrule

15 & \(15\leftrightarrow30\) & C-F & \qtfNFcell{
F_{15}={}&x_1x_2^4+x_1^2x_2^2x_3\\
&+x_1^3(\gamma x_3^2+\beta x_3x_4+x_4^2)+x_0x_2^3x_4\\
&+x_0x_1x_2R(x_3,x_4)+x_0^2B(x_3,x_4)
} & \qtfParamcell{$\gamma,\beta\in\C$;\\
$R\in\Sym^2\langle x_3,x_4\rangle$;\\
$B\in\Sym^3\langle x_3,x_4\rangle$} & 9\\
\midrule

16 & \(19\) & C-H & \qtfNFcell{
F_{16}={}&x_2^5+\alpha x_1x_2^3x_3+\beta x_1^2x_2x_3^2+\gamma x_1^2x_2^2x_4\\
&+x_1^3x_3x_4+\delta x_0x_2^2x_3^2+\epsilon x_0x_2^3x_4+x_0x_1x_3^3\\
&+\zeta x_0x_1x_2x_3x_4+\xi x_0x_1^2x_4^2+x_0^2x_3^2x_4\\
&+\theta x_0^2x_2x_4^2
} & \qtfParamcell{$\alpha,\beta,\gamma,\delta,$\\ $\epsilon,\zeta,\xi,\theta\in\Cstar$} & 8\\
\midrule

17 & \(20\leftrightarrow31\) & C-F & \qtfNFcell{
F_{17}={}&x_0B_{\lambda,\mu}(x_2,x_3)+x_1^2C_{\mathbf c}(x_2,x_3)+x_1^3x_2x_4\\
&+x_0x_1x_4Q_{\mathbf q}(x_2,x_3)+x_0x_1^2x_4^2\\
&+x_0^2x_4^2M_{\mathbf m}(x_2,x_3)
} & \qtfParamcell{$\lambda,\mu$ distinct in\\ $\C\setminus\{0,1\}$;\\ $\mathbf c\in\C^4$, $\mathbf q\in\C^3$,\\ $\mathbf m\in\C^2$} & 11\\
\midrule

18 & \(23\leftrightarrow38\) & C-F & \qtfNFcell{
F_{18}={}&x_0B_4(x_1,x_2,x_3,x_4)
} & \qtfParamcell{$B_4$: smooth quartic\\ surface in $\PP^3$} & 19\\
\midrule

19 & \(29\leftrightarrow35\) & C-F & \qtfNFcell{
F_{19}={}&x_3A_4(x_1,x_2)+x_4C_4(x_1,x_2)\\
&+x_0Q_{2,2}(x_1,x_2;x_3,x_4)+x_0^2x_3x_4(x_3-x_4)
} & \qtfParamcell{$A_4,C_4$: binary quartics;\\ $Q_{2,2}$ of bidegree $(2,2)$} & 15\\
\midrule

20 & \(32\leftrightarrow37\) & C-F & \qtfNFcell{
F_{20}={}&x_0^2C_0(x_2,x_3,x_4)+x_0x_1C_1(x_2,x_3,x_4)\\
&+x_1^2C_2(x_2,x_3,x_4)
} & \qtfParamcell{$C_0,C_1,C_2$: ternary cubics} & 18\\
\midrule

21 & \(36\) & C-F & \qtfNFcell{
F_{21}={}&B_5(x_1,x_2,x_3)+x_0x_4C_3(x_1,x_2,x_3)\\
&+x_0^2x_3x_4^2
} & \qtfParamcell{$B_5,C_3$: ternary forms\\ of degrees $5,3$} & 24\\
\midrule

\end{longtable}
\endgroup

\begin{corollary}[Twenty-one closed irreducible quotient-side families]
\label{cor:twenty-one-quotient-side-cover}
The twenty-one quotient images listed in
Table~\ref{tab:component-normal-forms} are nonempty closed irreducible
subvarieties and cover the GIT boundary:
\[
  \partial\mathcal M^{\mathrm{GIT}}
  =\Theta_1\cup\cdots\cup\Theta_{21},
\]
The subvarieties are indexed by \(\Theta_a=\Phi_{k(a)}\), as fixed above,
and their dimensions are the numbers \(d_a\) in the last column of the table.
\end{corollary}

\begin{proof}
Appendix~\ref{app:normal-forms} constructs, for every oriented
support case, a nonempty open family of weight-zero limits with closed
\(G\)-orbit and reduces it to the displayed normal form.  Thus every selected
\(\Phi_{k(a)}\) is nonempty.  The variety \(Y_{k(a)}\cap X^{ss}\) is a
nonempty irreducible closed \(G\)-invariant subset of \(X^{ss}\); its image
under the good quotient is therefore closed and irreducible.  The covering
statement is Corollary~\ref{cor:twenty-one-quotient-side-candidate-cover},
and the dimension calculations are those in Appendix~\ref{app:normal-forms}.
\end{proof}

\begin{remark}[Paired normal-form loci]
\label{rem:paired-normal-form-loci}
Let \(k'\) be paired with \(k(a)\).  The constructions in
Appendix~\ref{app:normal-forms} start from nonempty open full-support
loci for both labels and reduce their weight-zero limits to the respective
displayed normal forms.  The quotient images of these normal-form loci are
therefore dense constructible subsets of
\(\Phi_{k'}=\Phi_{k(a)}=\Theta_a\).  Since every fiber of a good quotient
contains a unique closed orbit, a general point of \(\Theta_a\) has that
closed orbit represented, up to projective equivalence, by a general member
of the selected family \(F_a\).  No coefficientwise identification of the
two displayed formulas by \(\rho\) is asserted or needed.
\end{remark}

\begin{remark}
This corollary is still a covering statement.  Pairwise non-inclusion, and
therefore the assertion that the \(\Theta_a\) are exactly the irreducible
components, is proved in the next section.
\end{remark}

\section{Generic stabilizers, non-inclusion, and the component theorem}
\label{sec:pairwise-noninclusion}

Table~\ref{tab:component-normal-forms} fixes the representative label
\(k(a)\) and records the dimension \(d_a=\dim\Theta_a\) for every component.
The dimensions are retained as part of the classification, but they are not
needed for pairwise non-inclusion: the stabilizer argument below treats all
ordered pairs simultaneously.

By Corollary~\ref{cor:twenty-one-quotient-side-cover},
\begin{equation}
  \partial\mathcal M^{\mathrm{GIT}}
  =\Theta_1\cup\cdots\cup\Theta_{21}.
  \label{eq:boundary-covered-by-theta}
\end{equation}
We prove that this cover is already the irreducible-component decomposition.

\begin{theorem}[Pairwise non-inclusion and boundary components]
\label{thm:pairwise-noninclusion}
For any two distinct indices \(a,b\in\{1,\ldots,21\}\), one has
\[
  \Theta_a\not\subseteq\Theta_b
\]
inside \(\mathcal M^{\mathrm{GIT}}\).  Consequently,
\(\Theta_1,\ldots,\Theta_{21}\) are precisely the irreducible components of
\(\partial\mathcal M^{\mathrm{GIT}}\); in particular, the strictly semistable
GIT boundary has exactly twenty-one irreducible components.
\end{theorem}

The proof has three ingredients.  First, the generic connected stabilizer
along \(\Theta_a\) is the one-dimensional torus defined by its supporting
hyperplane.  Second, these twenty-one tori are pairwise nonconjugate in
\(G\).  Third, Luna's \'{e}tale slice theorem implies that an inclusion
\(\Theta_a\subseteq\Theta_b\) would force a conjugate of the generic
connected stabilizer of \(\Theta_b\) to be contained in that of
\(\Theta_a\).  Since both are connected one-dimensional tori, such an
inclusion is an equality, yielding the required contradiction.

\subsection{Generic connected stabilizers}
\label{subsec:generic-connected-stabilizers}

For the representative case \(k(a)\) fixed in
Table~\ref{tab:component-normal-forms}, set
\begin{equation}
  r_a:=r_{k(a)},
  \qquad
  H_a:=\lambda_{r_a}(\Gm)\subset G.
  \label{eq:generic-stabilizer-torus}
\end{equation}
Every monomial occurring in the Case \(k(a)\) normal form has
\(r_a\)-weight zero.  Hence \(H_a\) fixes every member of that normal-form
family as a vector, not merely projectively.

Let \(\mathfrak g:=\mathfrak{sl}_5\).  For a nonzero quintic
\(F\in W=\Sym^5\C^5\), consider the linear map
\begin{equation}
  \mathcal D_F:\mathfrak g\oplus\C\longrightarrow W,
  \qquad
  (A,c)\longmapsto A\cdot F-cF,
  \label{eq:projective-stabilizer-map}
\end{equation}
where \(A\cdot F\) is the infinitesimal action associated with the action
convention in \eqref{eq:quintic-action-convention}.  We use the following basis
of \(\mathfrak g\):
\[
  x_i\frac{\partial}{\partial x_j}
  \quad(i\neq j),
  \qquad
  x_i\frac{\partial}{\partial x_i}
  -x_4\frac{\partial}{\partial x_4}
  \quad(0\leq i\leq3),
\]
and append the projective-scaling column \(-F\).  With respect to the
126 degree-five monomials, \(\mathcal D_F\) is represented by a
\(126\times25\) matrix.

\begin{lemma}
\label{lem:projective-lie-stabilizer}
For every nonzero quintic \(F\), projection to the first factor induces an
isomorphism
\[
  \ker(\mathcal D_F)
  \xrightarrow{\ \sim\ }
  \operatorname{Lie}\!\left(\operatorname{Stab}_G([F])\right).
\]
Consequently,
\begin{equation}
  \dim\operatorname{Stab}_G([F])
  =25-\rank(\mathcal D_F).
  \label{eq:stabilizer-dimension-from-rank}
\end{equation}
\end{lemma}

\begin{proof}
An element \(A\in\mathfrak g\) is tangent to the projective stabilizer of
\([F]\) precisely when \(A\cdot F\in\C F\).  This is equivalent to the
existence of \(c\in\C\) such that \(\mathcal D_F(A,c)=0\).  Since
\(F\neq0\), the scalar \(c\) is uniquely determined by \(A\), so projection
to \(\mathfrak g\) is an isomorphism onto the projective stabilizer Lie
algebra.  Over \(\C\), the stabilizer is a smooth algebraic group, and its
dimension equals the dimension of its Lie algebra, giving
\eqref{eq:stabilizer-dimension-from-rank}.
\end{proof}

For a primitive vector \(r\in\mathbb Z^5(0)\), define its unsigned
permutation signature by
\begin{equation}
  \nu(r):=
  \min_{\mathrm{lex}}
  \bigl\{\mathrm{sort}(r),\mathrm{sort}(-r)\bigr\},
  \label{eq:unsigned-weight-signature}
\end{equation}
where \(\mathrm{sort}\) denotes increasing rearrangement.

\begin{proposition}[Generic connected stabilizers]
\label{prop:generic-connected-stabilizers}
For a general closed-orbit normal form representing \(\Theta_a\), one has
\[
  \operatorname{Stab}_G([F])^{\circ}=H_a
  \qquad(1\leq a\leq21).
\]
Equivalently, the generic connected stabilizer along \(\Theta_a\) is the
conjugacy class of \(H_a\).
\end{proposition}

\begin{proof}
Let \(P_a\) be the irreducible coefficient space of the displayed Case
\(k(a)\) normal form before deleting the genericity discriminants, and let
\(F_a(\mathbf u)\) denote the corresponding universal form.  Because all
monomials of \(F_a(\mathbf u)\) have \(r_a\)-weight zero, the infinitesimal
generator
\[
  R_a:=\diag(r_{a,0},\ldots,r_{a,4})\in\mathfrak{sl}_5
\]
satisfies \(R_a\cdot F_a(\mathbf u)=0\) identically on \(P_a\).  Thus
\begin{equation}
  \rank\mathcal D_{F_a(\mathbf u)}\leq24
  \qquad\text{for every }\mathbf u\in P_a.
  \label{eq:universal-rank-upper-bound}
\end{equation}

We now certify that equality occurs.  Fix
\[
  p=1{,}000{,}003.
\]
For each \(a\), the ancillary file
\path{filter2_certificates.json} records the following exact data:
\begin{enumerate}
\item an integral specialization \(F_a^*\) of the Case \(k(a)\) normal form;
\item twenty-four row indices and twenty-four column indices in the integer
      matrix of \(\mathcal D_{F_a^*}\); and
\item the determinant \(\delta_a\) of the selected \(24\times24\) submatrix,
      reduced modulo \(p\).
\end{enumerate}
The residues \(\delta_a\) are listed in
Table~\ref{tab:generic-stabilizer-certificates}, and all are nonzero.  Since
the selected submatrices have integer entries,
\[
  \delta_a\not\equiv0\pmod p
  \quad\Longrightarrow\quad
  \det(C_a)\neq0\text{ in }\mathbb Z.
\]
Therefore the corresponding \(24\times24\) minor of the universal matrix
\(\mathcal D_{F_a(\mathbf u)}\) is a nonzero polynomial on \(P_a\).  Hence
\(\rank\mathcal D_{F_a(\mathbf u)}\geq24\) on a nonempty Zariski-open subset
of \(P_a\).  Together with \eqref{eq:universal-rank-upper-bound}, this gives
\[
  \rank\mathcal D_{F_a(\mathbf u)}=24
\]
for general \(\mathbf u\in P_a\).

The integral test point \(F_a^*\) need not itself satisfy every genericity
condition imposed in the normal-form table.  Indeed, the rank-24 locus is a
nonempty open subset of the irreducible space \(P_a\), while the general
closed-orbit locus constructed case by case in
Appendix~\ref{app:normal-forms} and summarized in
Section~\ref{sec:normal-form} is another nonempty open subset of \(P_a\); the
two open subsets therefore intersect.

By Lemma~\ref{lem:projective-lie-stabilizer}, the general projective
stabilizer has dimension one.  It contains the connected one-dimensional
torus \(H_a\), so its identity component is exactly \(H_a\).  Finally, the
quotient image of the general normal-form locus is a dense constructible
subset of \(\Theta_a\), hence contains a dense open subset; this proves the
assertion about the generic connected stabilizer along \(\Theta_a\).
\end{proof}

\begingroup
\scriptsize
\renewcommand{\arraystretch}{1.10}
\setlength{\tabcolsep}{4pt}
\begin{longtable}{@{}r r l r@{}}
\caption{Exact modular certificates for the generic projective stabilizer computation.  Here \(C_a\) is the selected \(24\times24\) integer submatrix of \(\mathcal D_{F_a^*}\), and \(p=1{,}000{,}003\).}
\label{tab:generic-stabilizer-certificates}\\
\toprule
\(a\) & \(k(a)\) & \(\nu(r_a)\) & \(\det(C_a)\bmod p\)\\
\midrule
\endfirsthead
\toprule
\(a\) & \(k(a)\) & \(\nu(r_a)\) & \(\det(C_a)\bmod p\)\\
\midrule
\endhead
1  & 1  & \((-36,-1,4,9,24)\)   & 28050  \\
2  & 2  & \((-27,-2,3,8,18)\)   & 818759 \\
3  & 3  & \((-3,0,0,1,2)\)      & 341126 \\
4  & 4  & \((-24,-9,6,11,16)\)  & 749212 \\
5  & 5  & \((-21,-6,4,9,14)\)   & 871295 \\
6  & 6  & \((-18,-3,2,7,12)\)   & 515005 \\
7  & 7  & \((-5,-1,0,2,4)\)     & 419273 \\
8  & 8  & \((-4,-2,1,2,3)\)     & 226555 \\
9  & 9  & \((-19,-4,1,6,16)\)   & 925602 \\
10 & 10 & \((-12,-2,3,3,8)\)    & 658896 \\
11 & 11 & \((-14,-4,1,6,11)\)   & 266577 \\
12 & 12 & \((-13,-8,2,7,12)\)   & 910121 \\
13 & 13 & \((-3,-2,0,2,3)\)     & 157960 \\
14 & 14 & \((-4,-1,0,1,4)\)     & 506192 \\
15 & 15 & \((-9,-4,1,6,6)\)     & 882759 \\
16 & 19 & \((-2,-1,0,1,2)\)     & 843397 \\
17 & 20 & \((-8,-3,2,2,7)\)     & 538703 \\
18 & 23 & \((-4,1,1,1,1)\)      & 247964 \\
19 & 29 & \((-6,-1,-1,4,4)\)    & 133397 \\
20 & 32 & \((-3,-3,2,2,2)\)     & 392541 \\
21 & 36 & \((-1,0,0,0,1)\)      & 206888 \\
\bottomrule
\end{longtable}
\endgroup

\begin{remark}[Verification of the certificates]
\label{rem:generic-stabilizer-certificate-verification}
The script \path{filter2_stabilizer.py} reconstructs every specialized
normal form and every \(126\times25\) matrix from the data in
\path{filter2_certificates.json}.  It verifies the weight-zero condition,
the infinitesimal kernel vector, each displayed nonzero modular determinant,
and the pairwise distinct signatures.  The independent script
\path{independent_sympy_check.py} reconstructs the matrices by symbolic
differentiation and verifies
\(\rank_{\mathbb Q}\mathcal D_{F_a^*}=24\) for all twenty-one cases.  No
floating-point rank computation is used.
\end{remark}

\subsection{Conjugacy classes of the generic tori}
\label{subsec:generic-torus-conjugacy}

\begin{lemma}
\label{lem:one-dimensional-torus-conjugacy}
Let \(r,s\in\mathbb Z^5(0)\) be primitive.  The one-dimensional subtori
\(\lambda_r(\Gm)\) and \(\lambda_s(\Gm)\) of \(\SL(5)\) are conjugate if and
only if
\[
  \nu(r)=\nu(s).
\]
\end{lemma}

\begin{proof}
Because \(r\) and \(s\) are primitive, the cocharacters
\(\lambda_r:\Gm\to\lambda_r(\Gm)\) and
\(\lambda_s:\Gm\to\lambda_s(\Gm)\) are isomorphisms onto their images.  A
conjugacy between the two image tori therefore induces an algebraic-group
automorphism of \(\Gm\), necessarily \(t\mapsto t\) or \(t\mapsto t^{-1}\).
The restrictions of the standard five-dimensional representation to the two
tori must consequently have the same multiset of weights, up to a common
sign.  This is precisely the equality \(\nu(r)=\nu(s)\).

Conversely, equality of the signatures gives a permutation of the coordinate
weight spaces carrying the weights of \(r\) to those of \(s\), possibly after
replacing \(s\) by \(-s\).  A permutation matrix realizes this conjugacy in
\(\operatorname{GL}(5)\); if its determinant is not one, multiply it by a
diagonal matrix centralizing the torus and having the inverse determinant.
The resulting conjugating matrix lies in \(\SL(5)\).
\end{proof}

\begin{corollary}
\label{cor:generic-tori-pairwise-nonconjugate}
The tori \(H_1,\ldots,H_{21}\) are pairwise nonconjugate in \(G\).
\end{corollary}

\begin{proof}
The twenty-one signatures in
Table~\ref{tab:generic-stabilizer-certificates} are pairwise distinct, so the
claim follows from Lemma~\ref{lem:one-dimensional-torus-conjugacy}.
\end{proof}

\subsection{The Luna closure obstruction}
\label{subsec:luna-closure-obstruction}

We now isolate the specialization statement needed for the component proof.
The saturation step below ensures that the stabilizer-subgroup property holds
over an actual open neighborhood in the quotient, rather than only on a
chosen slice chart.

\begin{lemma}[Luna closure obstruction]
\label{lem:luna-closure-obstruction}
Let \(Z_1,Z_2\subset\mathcal M^{\mathrm{GIT}}\) be irreducible closed
subvarieties.  Suppose that there are dense open subsets \(U_i\subseteq Z_i\)
and connected reductive subgroups \(K_i\subset G\) such that, for every
\(z\in U_i\), the identity component of the stabilizer of the unique closed
orbit over \(z\) is conjugate to \(K_i\).  If \(Z_1\subseteq Z_2\), then a
conjugate of \(K_2\) is contained in \(K_1\).
\end{lemma}

\begin{proof}
Choose \(z\in U_1\), let \(G\cdot[F]\) be the unique closed orbit over
\(z\), and put
\[
  S:=\operatorname{Stab}_G([F]).
\]
After conjugating \(K_1\), we may assume that
\[
  S^{\circ}=K_1.
\]
Choose a homogeneous invariant \(I\in\C[W]^G\) of positive degree with
\(I(F)\neq0\), and set
\[
  X_I:=\bigl\{[f]\in\PP(W)^{ss}\mid I(f)\neq0\bigr\}.
\]
This is a \(G\)-invariant affine neighborhood of \([F]\).  Its good quotient
\[
  \pi_I:X_I\longrightarrow Q_I:=X_I//G
\]
is an open neighborhood of \(z\) in \(\mathcal M^{\mathrm{GIT}}\).

Apply Luna's \'{e}tale slice theorem at the closed orbit \(G\cdot[F]\)
\cite{Lun73}; compare also \cite{Lun75}.  After shrinking the slice, the
strongly \'{e}tale local model \(G\times^S V\) gives a \(G\)-invariant open
neighborhood \(O\subseteq X_I\) of \(G\cdot[F]\) such that, for every
\(y\in O\), the stabilizer \(G_y\) is conjugate to a subgroup of \(S\).  In
the slice tube, this is the identity
\[
  G_{[g,v]}=gS_vg^{-1}\subseteq gSg^{-1}.
\]

We pass from this equivariant neighborhood to a saturated neighborhood over
the quotient.  Put
\[
  C:=X_I\setminus O.
\]
Then \(C\) is closed and \(G\)-invariant, so \(\pi_I(C)\) is closed in
\(Q_I\).  Moreover,
\[
  z\notin\pi_I(C).
\]
Indeed, if \(y\in C\) satisfied \(\pi_I(y)=z\), then the closure of
\(G\cdot y\) would contain the unique closed orbit \(G\cdot[F]\) in that
fibre.  Since \(C\) is closed and \(G\)-invariant, this would imply
\([F]\in C\), contrary to \([F]\in O\).  Therefore
\[
  Q:=Q_I\setminus\pi_I(C)
\]
is an open neighborhood of \(z\), and
\begin{equation}
  \pi_I^{-1}(Q)\subseteq O.
  \label{eq:saturated-luna-neighborhood}
\end{equation}

Because \(Z_1\subseteq Z_2\), the point \(z\) lies in \(Z_2\).  Hence
\(Q\cap Z_2\) is a nonempty open subset of the irreducible variety \(Z_2\),
and it meets the dense open subset \(U_2\).  Choose
\[
  q\in Q\cap U_2
\]
and let \(G\cdot[E]\) be the unique closed orbit over \(q\).  By
\eqref{eq:saturated-luna-neighborhood}, one has \([E]\in O\), so a conjugate
of \(\operatorname{Stab}_G([E])\) is contained in \(S\).  Since \(q\in U_2\),
the identity component \(\operatorname{Stab}_G([E])^{\circ}\) is conjugate
to \(K_2\).  Consequently, a conjugate of \(K_2\) is contained in \(S\).
Since \(K_2\) is connected, that conjugate is contained in
\[
  S^{\circ}=K_1.
\]
This proves the claim.
\end{proof}

\begin{proof}[Proof of Theorem~\ref{thm:pairwise-noninclusion}]
Suppose that \(\Theta_a\subseteq\Theta_b\) for distinct indices \(a\) and
\(b\).  By Proposition~\ref{prop:generic-connected-stabilizers} and
Lemma~\ref{lem:luna-closure-obstruction}, there exists \(g\in G\) such that
\[
  gH_bg^{-1}\subseteq H_a.
\]
Both sides are connected one-dimensional tori.  A proper closed subgroup of
an irreducible one-dimensional algebraic group has dimension zero, so this
inclusion must be an equality.  Thus \(H_a\) and \(H_b\) are conjugate,
contradicting Corollary~\ref{cor:generic-tori-pairwise-nonconjugate}.  Hence
\[
  \Theta_a\not\subseteq\Theta_b
  \qquad(a\neq b).
\]

It remains to identify the irreducible components.  Let \(C\) be an
irreducible component of \(\partial\mathcal M^{\mathrm{GIT}}\).  Since the
boundary is the finite union in \eqref{eq:boundary-covered-by-theta}, the
irreducible set \(C\) is contained in some \(\Theta_j\).  Since \(\Theta_j\)
is itself an irreducible closed subset of the boundary, the maximality of
\(C\) gives \(C=\Theta_j\).  Conversely, every \(\Theta_i\) is contained in
an irreducible component \(C_i\), and the preceding argument gives
\(C_i=\Theta_j\) for some \(j\).  Thus \(\Theta_i\subseteq\Theta_j\), so
pairwise non-inclusion forces \(i=j\) and hence \(C_i=\Theta_i\).  Therefore
\(\Theta_1,\ldots,\Theta_{21}\) are exactly the twenty-one irreducible
components of \(\partial\mathcal M^{\mathrm{GIT}}\).
\end{proof}

\clearpage
\section{Comparison with Lakhani's classification}\label{sec:lakhani-comparison}

Having proved the complete component theorem, we now compare it with the earlier study of the GIT compactification of quintic threefolds by
Lakhani~\cite{Lak10}.  Since the two papers organize the boundary in rather
different ways, we record the precise dictionary between their families and
separate the results already obtained in \cite{Lak10} from the additional
results proved here.  Lakhani uses \emph{non-stable} for the union of the
unstable and strictly semistable loci.  He first isolates seven families
\(\mathrm{SS1},\ldots,\mathrm{SS7}\), then takes their one-parameter-subgroup
limits to four families
\(\mathrm{MO\!-\!A},\ldots,\mathrm{MO\!-\!D}\), which he calls the first level of minimal
orbits, and finally studies ten further families
\(\mathrm{MO2\!-\!I},\ldots,\mathrm{MO2\!-\!X}\) arising from special degenerations of those
four families.  The present paper instead takes as its primary object the
quotient-side boundary
\[
  \partial\mathcal M^{\mathrm{GIT}}
  =\pi\bigl(\PP(W)^{ss}\setminus\PP(W)^s\bigr)
\]
and distinguishes throughout between a maximal monomial-support family
before taking the quotient and its image in \(\mathcal M^{\mathrm{GIT}}\).
This distinction is essential for comparing the two classifications.

\paragraph*{The dictionary between the two classifications.}
There is first a sign convention to take into account.  If
\(r=(r_0,\ldots,r_4)\) is one of Lakhani's normalized one-parameter
subgroups, then his associated non-stable monomial family is
\(M_{\leq 0}(r)\).  With the action convention used here,
\[
  M_{\leq 0}(r)=I(-r)_{\geq 0}.
\]
After returning \(-r\) to the dominant Weyl chamber by reversing the
coordinates, its dominant representative is
\[
  \iota(r):=(-r_4,-r_3,-r_2,-r_1,-r_0).
\]
We retain Lakhani's notation \(q_{a,b}(U\mid V)\) for a general
bihomogeneous form of bidegree \((a,b)\) in the indicated blocks.  With this
convention, the correspondence is the following.  The equality in the fourth
column is an equality of the closures of the quotient images of the general
closed-orbit loci; it is not a literal equality of the displayed parameter
spaces before coordinate changes and residual centralizer normalization.

\begin{center}
{\footnotesize
\setlength{\tabcolsep}{3pt}
\renewcommand{\arraystretch}{1.22}
\begin{tabular}{@{}
 >{\centering\arraybackslash}p{0.085\textwidth}
 >{\raggedright\arraybackslash}p{0.285\textwidth}
 >{\centering\arraybackslash}p{0.205\textwidth}
 >{\centering\arraybackslash}p{0.225\textwidth}
 >{\centering\arraybackslash}p{0.05\textwidth}
 @{}}
\toprule
Lakhani family
& First-level form
& Lakhani type / present support
& Present quotient component
& \(\dim\) \\
\midrule
\(\mathrm{MO\!-\!A}\)
& \(q_{2,3}(x_0,x_1\mid x_2,x_3,x_4)\)
& \(\mathrm{SS1}\leftrightarrow S_{32}\)\newline
  \(\mathrm{SS3}\leftrightarrow S_{37}\)
& \(\Theta_{20}=\Phi_{32}=\Phi_{37}\)\newline
  normal form \(F_{20}\)
& \(18\) \\
\(\mathrm{MO\!-\!B}\)
& \(q_5+x_0x_4q_3+x_0^2x_4^2q_1\)\newline
  (the \(q_i\) are in \(x_1,x_2,x_3\))
& \(\mathrm{SS5}\leftrightarrow S_{36}\)
& \(\Theta_{21}=\Phi_{36}\)\newline
  normal form \(F_{21}\)
& \(24\) \\
\(\mathrm{MO\!-\!C}\)
& \(q_{1,4}+x_4q_{2,2}+x_4^2q_3\)\newline
  (blocks \((x_0,x_1)\mid(x_2,x_3)\))
& \(\mathrm{SS6}\leftrightarrow S_{29}\)\newline
  \(\mathrm{SS7}\leftrightarrow S_{35}\)
& \(\Theta_{19}=\Phi_{29}=\Phi_{35}\)\newline
  normal form \(F_{19}\)
& \(15\) \\
\(\mathrm{MO\!-\!D}\)
& \(x_0q_4(x_1,x_2,x_3,x_4)\)
& \(\mathrm{SS2}\leftrightarrow S_{23}\)\newline
  \(\mathrm{SS4}\leftrightarrow S_{38}\)
& \(\Theta_{18}=\Phi_{23}=\Phi_{38}\)\newline
  normal form \(F_{18}\)
& \(19\) \\
\bottomrule
\end{tabular}}
\end{center}

The normal-form match is immediate for \(\mathrm{MO\!-\!A}\) and
\(\mathrm{MO\!-\!D}\).  For \(\mathrm{MO\!-\!B}\), a general nonzero linear form
\(q_1\) is sent to a coordinate by the residual linear action.  For
\(\mathrm{MO\!-\!C}\), after the coordinate permutation dictated by \(\iota\), a general
square-free binary cubic is normalized to \(uv(u-v)\).  These are the
normalizations appearing in \(F_{21}\) and \(F_{19}\), respectively.
Consequently, the seven support types occurring in \cite{Lak10} correspond
to
\[
  S_{23},\ S_{29},\ S_{32},\ S_{35},\ S_{36},\ S_{37},\ S_{38}.
\]
Under reversal of the supporting half-space they form the three pairs
\[
  (23,38),\qquad (29,35),\qquad (32,37),
\]
together with the fixed case \(36\).  They therefore give four quotient-side
families.  In the complete component catalog these are exactly the four
components of largest dimension:
\[
  \dim\Theta_{21}=24,\qquad
  \dim\Theta_{18}=19,\qquad
  \dim\Theta_{20}=18,\qquad
  \dim\Theta_{19}=15,
\]
whereas \(\dim\Theta_a\leq 11\) for \(1\leq a\leq 17\).

\paragraph*{Results already obtained in \cite{Lak10}.}
The following parts of the boundary analysis precede the present paper.
\begin{enumerate}
\item
Lakhani developed a monomial-poset and linear-programming analysis of the
Hilbert--Mumford criterion, exhibited the seven families
\(\mathrm{SS1},\ldots,\mathrm{SS7}\), and computed their destabilizing
one-parameter subgroups and associated bad flags.  He also characterized
these families geometrically by, respectively: a double plane; a hyperplane
component; a triple line; a quadruple point; a special triple point whose
tangent cone is the union of a double plane and another hyperplane, together
with a contact-order-five condition; a double line with constant
double-plane tangent cone, together with a contact-order-four condition; and
a triple point on a plane with the incidence and quartic-intersection
conditions stated in his Theorem~2.

\item
By taking the corresponding one-parameter-subgroup limits, Lakhani obtained
the four explicit invariant families
\(\mathrm{MO\!-\!A},\ldots,\mathrm{MO\!-\!D}\) displayed above.  Their
centralizer representations have block types \(2+3\), \(1+3+1\),
\(2+2+1\), and \(1+4\), respectively.  Lakhani applied Luna's criterion to
these representations, proved that a general member of each family has
closed \(\SL(5)\)-orbit, and separated the unstable members from those that
admit a further degeneration.  Thus the equations of these four large
families and their generic closed-orbit property are prior to the present
work.

\item
Lakhani introduced the ten second-level families
\(\mathrm{MO2\!-\!I},\ldots,\mathrm{MO2\!-\!X}\), repeated the closed-orbit analysis on those families, and described a
degeneration diagram inside the union of the four first-level families.  In
particular, he showed that the nonclosed orbits occurring in this recursive
analysis eventually degenerate to
\[
  x_0x_1x_2x_3x_4,
\]
the terminal normal-crossing orbit in his diagram.  He also gave several
partial criteria for the stable locus, which are not used in the present
paper.
\end{enumerate}

\paragraph*{Comparison of the polystability arguments.}
Both papers use the same basic reduction: a non-stable form is first replaced
by its weight-zero one-parameter-subgroup limit, and Luna's criterion then
reduces closedness of the \(\SL(5)\)-orbit to closedness for a centralizer
action on the fixed subspace.  Since the original one-parameter subgroup
\(H\) acts trivially on that fixed subspace, closedness depends on the
effective quotient \(C_G(H)/H\).  Accordingly, Lakhani's phrase ``stable
with respect to the centralizer'' is naturally read as stability for this
effective centralizer action.  At the level of \(G\), the resulting limit is
not stable, because its stabilizer contains \(H\); when its orbit is closed,
it is the polystable representative of the corresponding S-equivalence
class.

The proof presented in \cite{Lak10} is a recursive Hilbert--Mumford
calculation.  For \(\mathrm{MO\!-\!A}\), Lakhani writes a normalized
one-parameter subgroup of the \(\SL(2)\times\SL(3)\)-type centralizer in
block-diagonal form, orders the coefficient monomials by their centralizer
weights, and uses a poset calculation to list five maximal semistable
nonclosed subfamilies and eighteen unstable subfamilies.  His
Proposition~24 states that the former degenerate to second-level families,
that the latter are unstable, and that every remaining member has closed
orbit.  For \(\mathrm{MO\!-\!B}\), \(\mathrm{MO\!-\!C}\), and
\(\mathrm{MO\!-\!D}\), the same procedure yields, respectively, four
semistable nonclosed and four unstable subfamilies, one of each, and four
semistable nonclosed and three unstable subfamilies.  Lakhani also applies
analogous calculations to the non-toric second-level centralizers.  Thus his
proof of generic closedness proceeds by classifying the exceptional
centralizer-nonstable support loci and iterating the construction on their
limits.  An important additional output of this method is the explicit
further-degeneration stratification inside the four large families.

The present paper organizes the closed-orbit step differently.  For each of
the thirty-eight maximal supports it first extracts the full weight-zero
limit \(\phi_k\) and then proves directly that a specified general limit is
polystable.  If the effective centralizer is a torus, closedness is checked
by the relative-interior criterion for the character polytope.  If the
effective centralizer has a nontrivial linear block, the paper uses the
Casimiro--Florentino criterion~\cite{CF12}.  Writing
\(C_k=C_G(H_k)\), the relevant set is
\[
  \Lambda_{\phi_k}^{C_k}
  :=\left\{
       \lambda:\Gm\longrightarrow C_k
       \ \middle|\
       \lim_{t\to0}\lambda(t)\cdot\phi_k\text{ exists}
     \right\}.
\]
On an explicit nonempty open locus, nonvanishing coefficient,
discriminant, and resultant conditions are combined with linear balancing
identities for the one-parameter-subgroup weights.  The case-by-case
analysis proves that \(\Lambda_{\phi_k}^{C_k}\) is invariant under
\(\lambda\mapsto\lambda^{-1}\), which gives closedness by the
Casimiro--Florentino criterion.  The residual centralizer action is then
used to obtain an explicit normal form and compute the quotient dimension.
In particular, the four representatives in the table are the non-toric
cases \(k=23,29,32,36\), so this gives an alternative direct proof of their
generic polystability, together with explicit genericity conditions.

These two approaches should be viewed as complementary rather than as
competing claims of priority.  Lakhani's centralizer Hilbert--Mumford analysis
already establishes the generic closedness of the four first-level families
and, because it classifies their exceptional loci, records many special
degenerations.  The present argument instead gives a uniform direct
polystability certificate for a general weight-zero limit in every one of
the thirty-eight maximal-support cases and retains the parameters needed for
the quotient-side normal forms and dimension calculations.  Conversely, the
present paper does not attempt to reconstruct the full internal
centralizer-stratification diagram for all thirty-eight support families.

\paragraph*{The non-exhaustiveness of the four-family picture.}
The difference between the two classifications is therefore not the validity
of Lakhani's four first-level families, but their exhaustiveness.  The
exhaustive calculation in Proposition~\ref{prop:quintic-maximal-t-supports}
produces thirty-eight maximal \(T\)-strictly semistable supports, and
Proposition~\ref{prop:strictly-semistable-cover} proves that their
\(G\)-saturations cover the entire strictly semistable locus.  The seven
Lakhani support types account for three opposite-half-space pairs and one
fixed point.  The remaining thirty-one supports account for fourteen further
pairs and three further fixed points.  Equivalently,
\[
  38=(3\cdot2+1)+(14\cdot2+3),
  \qquad
  21=4+17.
\]
Thus the exhaustive computation here shows that the seven-family list does
not exhaust the maximal strictly semistable supports and that the four
first-level quotient families do not exhaust
\(\partial\mathcal M^{\mathrm{GIT}}\).

The second-level families
\(\mathrm{MO2\!-\!I},\ldots,\mathrm{MO2\!-\!X}\) should in particular not be confused with the additional components
\(\Theta_1,\ldots,\Theta_{17}\).  The former are special
further-degeneration loci arising inside the union of the four first-level
families.  The latter are additional irreducible components which, by
Theorem~\ref{thm:pairwise-noninclusion}, are not contained in any of
\(\Theta_{18},\ldots,\Theta_{21}\).

\paragraph*{Additional results of the present paper.}
Beyond the four-family picture and its internal degeneration analysis, the
present paper proves the following.
\begin{enumerate}
\item
The Hilbert--Mumford support calculation is made exhaustive: exact integer
linear algebra, both orientations of each candidate supporting hyperplane,
all coordinate permutations, and exact inclusion tests give precisely the
thirty-eight maximal \(T\)-strictly semistable supports.  Their
\(G\)-saturations cover the whole strictly semistable locus.  This yields
thirty-one additional oriented maximal supports to the seven support types
occurring in \cite{Lak10}.

\item
The S-equivalences among support-indexed families are explained by the
uniform zero-weight principle
\[
  \Phi(r)
  =\overline{\pi\bigl(\PP(V(r))\cap\PP(W)^{ss}\bigr)}
  =\overline{\pi\bigl(\PP(Z(r))\cap\PP(W)^{ss}\bigr)}
  =\Phi(-r),
\]
proved in Lemma~\ref{lem:zero-weight-reduction-in-quotient} and
Corollary~\ref{cor:supporting-hyperplane-invariance}.  Hence the quotient
remembers the supporting hyperplane but not the orientation of its normal.
The thirty-eight oriented supports decompose into seventeen two-cycles and
four fixed points, giving the twenty-one closed irreducible quotient-side
families of Corollary~\ref{cor:twenty-one-quotient-side-cover}.  The three
pairs and one fixed point visible in Lakhani's list are exactly the
restriction of this general involution to his seven support types.

\item
For all thirty-eight support cases, the present paper constructs explicit
polystable weight-zero normal forms and computes the corresponding
quotient-side dimensions.  For the resulting twenty-one quotient families,
it determines the generic connected projective stabilizers and proves that
they are pairwise nonconjugate one-dimensional tori.  The exact rank
certificates of Proposition~\ref{prop:generic-connected-stabilizers},
together with Luna's \'{e}tale slice theorem, exclude every containment
\[
  \Theta_a\subseteq\Theta_b
  \qquad(a\neq b).
\]
Theorem~\ref{thm:pairwise-noninclusion} therefore proves that the boundary
has exactly twenty-one irreducible components.  In particular, it proves
that the four Lakhani families are genuine irreducible components in the
complete quotient-side decomposition and that there are seventeen further
components, none of which is contained in one of those four components.

\item
Finally, the present paper gives a component-by-component scheme-theoretic
and local analytic study of the general closed-orbit representatives.  It
determines the saturated Jacobian scheme and the local analytic
stratification in every component, identifies the generic transverse
structures along the positive-dimensional singular loci, proves that
precisely eleven isolated weighted-homogeneous local normal-form families
occur, and identifies their normalized weight systems with eleven rows of
Yonemura's list.  It also proves
\[
  \mexp(X_a)=1=\frac{4+1}{5}
  \qquad(a=1,\ldots,21).
\]
These scheme-theoretic singularity calculations, the occurrence-and-
exhaustion statement for the eleven weight systems, and the minimal-exponent
calculation do not appear in \cite{Lak10}.
\end{enumerate}

In summary, Lakhani first found the four families that become the four
largest components in the present classification and carried out a
substantial recursive closed-orbit and degeneration analysis within their
union.  The present paper places those families in the complete quotient-side
picture: it gives a uniform polystability construction for all maximal
supports, identifies seventeen additional irreducible components, proves that
the resulting twenty-one families are exactly the boundary components, and
determines the generic orbit-theoretic and singularity-theoretic structure of
every component.

\clearpage
\section{Singular schemes, local types, and minimal exponents}
\label{sec:singular-loci}

For each \(a=1,\ldots,21\), let \(k(a)\) be the representative selected in
Table~\ref{tab:component-normal-forms}, let
\[
  F_a:=\nf_{k(a)},\qquad X_a:=V(F_a)\subset\PP^4,
\]
and restrict the parameters to the fixed nonempty Zariski-open generic locus
specified for that row.  We determine the saturated Jacobian scheme
\[
  \Sing(X_a)=V\bigl(J(F_a)\bigr),\qquad
  J(F_a):=
  \left(\frac{\partial F_a}{\partial x_0},\ldots,
        \frac{\partial F_a}{\partial x_4}\right)
  :(x_0,\ldots,x_4)^\infty.
\]

The analysis is now indexed only by quotient components.  If \(k'\) is the
partner of \(k(a)\) in a two-cycle, then \(\rho\) identifies
\(Z(r_{k(a)})\) with \(Z(r_{k'})\), while
Corollary~\ref{cor:supporting-hyperplane-invariance} identifies the quotient
images \(\Phi_{k(a)}\) and \(\Phi_{k'}\).  By
Remark~\ref{rem:paired-normal-form-loci}, the fiber over a general point of
their common quotient family contains a unique closed orbit, represented up
to projective equivalence by a general member of \(F_a\).  This does not
assert that the two
displayed normal-form formulas are carried into one another by \(\rho\)
alone; residual normalization and reparametrization may be needed.  Since
saturated Jacobian schemes, local analytic types, and minimal exponents are
invariant under projective coordinate changes, the single calculation for
\(F_a\) covers both oriented support labels.  The correspondence is recorded
in the support-orbit column of Table~\ref{tab:component-normal-forms}.

The reduced positive-dimensional singular supports are explicit low-degree
configurations: lines and line arrangements, smooth conics, cuspidal plane
curves, planes, and smooth quartic surfaces.  Some components also contain
isolated singular points.  Exactly eleven weighted-homogeneous local
normal-form families occur at those isolated points.  Their defining equations,
integral weights, Milnor numbers, Yonemura numbers, and component occurrences
are collected in Table~\ref{tab:intro-isolated-families}.  We first place the
weight systems in Yonemura's classification and record the precise genericity
conditions, then give the component-indexed singular-support catalog.  The
complete local calculations are deferred to
Appendix~\ref{app:singular-analysis}.

\subsection*{Relation with Yonemura's list of hypersurface simple \(K3\) weights}

Recall first the class of singularities to which Yonemura's result applies.
Following \cite[Definition--Proposition~1.4]{Yon90}, a three-dimensional
singularity is called a \emph{simple \(K3\) singularity} if it is a
Gorenstein purely elliptic singularity of \((0,2)\)-type; equivalently, the
exceptional divisor on any minimal resolution is a normal \(K3\) surface.  In
the hypersurface setting, Yonemura considers isolated singularities defined
by nondegenerate polynomials.  For a nondegenerate hypersurface simple
\(K3\) singularity, the defining polynomial is quasi-homogeneous with a
uniquely determined rational weight
\[
  a=(a_1,a_2,a_3,a_4)\in\mathbb Q_{>0}^4,
  \qquad
  a_1+a_2+a_3+a_4=1;
\]
see \cite[Theorem~1.6 and Section~2]{Yon90}.

For such a rational weight \(a\), put
\[
  T(a):=
  \bigl\{\nu\in\mathbb Z_{\geq0}^4\bigm|a\cdot\nu=1\bigr\},
  \qquad
  C(a):=
  \sum_{\nu\in T(a)}\mathbb R_{\geq0}\nu\subset\mathbb R^4.
\]
Yonemura considers the set
\[
  W_4:=
  \left\{
    a\in\mathbb Q_{>0}^4
    \ \middle|\
    \begin{array}{l}
      a_1+\cdots+a_4=1,\quad
      a_1\geq a_2\geq a_3\geq a_4,\\[2pt]
      (1,1,1,1)\in\operatorname{Int} C(a)
    \end{array}
  \right\}.
\]
Yonemura proves in \cite[Proposition~2.1]{Yon90} that \(W_4\) has
cardinality \(95\), and Table~2.2 of the same paper gives the complete list of these weights together with one
quasi-homogeneous polynomial realizing each weight as a nondegenerate
hypersurface simple \(K3\) singularity.

For an integral weight system
\[
  (w_1,w_2,w_3,w_4;D),
  \qquad
  w_1+w_2+w_3+w_4=D,
\]
let
\[
  a(w;D)
  :=
  \left(
    \frac{w_{\sigma(1)}}D,
    \frac{w_{\sigma(2)}}D,
    \frac{w_{\sigma(3)}}D,
    \frac{w_{\sigma(4)}}D
  \right),
  \qquad
  w_{\sigma(1)}\geq\cdots\geq w_{\sigma(4)},
\]
be its ordered normalized weight.  Whenever \(a(w;D)\) occurs in
Yonemura's Table~2.2, we call the corresponding row index the
\emph{Yonemura number} of the weight system.  We write
\(\Sfam{\mu}{N}\) for a weighted-homogeneous local normal-form family
with Milnor number \(\mu\) whose weight system has Yonemura number \(N\).
Thus the subscript records the Milnor number, while the superscript records
the corresponding row of Yonemura's table.  The eleven weight systems that
occur at isolated points of the quintic GIT boundary have the following
correspondence.

\begin{table}[htbp]
\centering
\scriptsize
\renewcommand{\arraystretch}{1.18}
\setlength{\tabcolsep}{4pt}
\begin{tabular}{@{}llll@{}}
\toprule
Local family & Integral weight system & Ordered normalized weight & Yonemura No.\\
\midrule
\(\Sfam{81}{1}\)
& \((1,1,1,1;4)\)
& \(\left(\frac14,\frac14,\frac14,\frac14\right)\) & \(1\)\\
\(\Sfam{87}{52}\)
& \((7,8,9,12;36)\)
& \(\left(\frac13,\frac14,\frac29,\frac7{36}\right)\) & \(52\)\\
\(\Sfam{88}{84}\)
& \((5,6,7,9;27)\)
& \(\left(\frac13,\frac7{27},\frac29,\frac5{27}\right)\) & \(84\)\\
\(\Sfam{88}{15}\)
& \((3,3,4,5;15)\)
& \(\left(\frac13,\frac4{15},\frac15,\frac15\right)\) & \(15\)\\
\(\Sfam{90}{2}\)
& \((2,3,3,4;12)\)
& \(\left(\frac13,\frac14,\frac14,\frac16\right)\) & \(2\)\\
\(\Sfam{91}{53}\)
& \((3,4,5,6;18)\)
& \(\left(\frac13,\frac5{18},\frac29,\frac16\right)\) & \(53\)\\
\(\Sfam{96}{86}\)
& \((4,5,7,9;25)\)
& \(\left(\frac9{25},\frac7{25},\frac15,\frac4{25}\right)\) & \(86\)\\
\(\Sfam{96}{94}\)
& \((3,4,5,7;19)\)
& \(\left(\frac7{19},\frac5{19},\frac4{19},\frac3{19}\right)\) & \(94\)\\
\(\Sfam{96}{21}\)
& \((1,1,1,2;5)\)
& \(\left(\frac25,\frac15,\frac15,\frac15\right)\) & \(21\)\\
\(\Sfam{100}{3}\)
& \((1,1,2,2;6)\)
& \(\left(\frac13,\frac13,\frac16,\frac16\right)\) & \(3\)\\
\(\Sfam{102}{62}\)
& \((3,4,5,8;20)\)
& \(\left(\frac25,\frac14,\frac15,\frac3{20}\right)\) & \(62\)\\
\bottomrule
\end{tabular}
\caption{The eleven isolated local normal-form families and their Yonemura
numbers.  In the notation \(\Sfam{\mu}{N}\), the subscript is the Milnor
number and the superscript records row~\(N\) of Table~2.2
of~\cite{Yon90}.}
\label{tab:quintic-yonemura-correspondence}
\end{table}

This correspondence is deliberately stated at the level of weight systems.
Yonemura's classification is a classification, in terms of weights, of
nondegenerate hypersurface simple \(K3\) singularities, and the polynomial in
each row of his table is an example realizing that weight.  The row number
does not assert that all parameter values in one of the local families below
are mutually analytically equivalent, or that they are analytically
equivalent to Yonemura's displayed example.  Accordingly, the result of the
present paper is not a new classification of isolated hypersurface
singularities.  It is instead an occurrence-and-exhaustion statement for the
quintic GIT boundary: precisely the following eleven weighted-homogeneous
local normal-form families occur at the isolated singular points of general
closed-orbit representatives.

Every defining equation below is weighted homogeneous of degree \(D\) for
positive weights \((w_1,w_2,w_3,w_4)\) satisfying
\[
  w_1+w_2+w_3+w_4=D.
\]
Once the indicated genericity condition guarantees that the critical point is
isolated, the standard weighted-homogeneous formulas give
\[
  \widetilde\alpha_0
  =
  \frac{w_1+w_2+w_3+w_4}{D}
  =1,
  \qquad
  \mu
  =
  \prod_{i=1}^4\left(\frac{D}{w_i}-1\right);
\]
see \cite{Sai71,Sai94}.  This is the critical value
\[
  \frac{n+1}{d}=\frac{4+1}{5}=1
\]
for quintic hypersurfaces in \(\PP^4\) in the minimal-exponent approach
to GIT stability; see \cite{Par25}.  We therefore use the paper-specific term
\emph{extremal quintic threefold singularities} for these eleven local
normal-form families.  This terminology refers to their occurrence on the
quintic GIT boundary and to the equality \(\widetilde\alpha_0=1\), not to a
claim that their weight systems lie outside the classical Yonemura list.

\begin{definition}
\label{def:isolated-family-81-Y1}
Let \(B_4(X_1,X_2,X_3,X_4)\) be a homogeneous quartic whose projectivization
\(V(B_4)\subset\PP^3\) is smooth.  We denote by
\[
\QH{1,1,1,1;4}{81}(B_4)
\]
the weighted-homogeneous family of isolated hypersurface germs represented by
\[
B_4(X_1,X_2,X_3,X_4)=0.
\]
It is quasi-homogeneous of weighted degree \(4\) with respect to the weights
\[
(w_{X_1},w_{X_2},w_{X_3},w_{X_4})=(1,1,1,1).
\]
Its ordered normalized weight is
\[
\left(\frac14,\frac14,\frac14,\frac14\right),
\]
and hence its Yonemura number is \(1\).
Its Milnor number is \(\mu=81\), and its local minimal exponent is
\[
\widetilde{\alpha}=1.
\]
We abbreviate this weighted-homogeneous local normal-form family by
\[
\Sfam{81}{1}(B_4).
\]
\end{definition}

\begin{definition}
\label{def:isolated-family-87-Y52}
We denote by
\[
\QH{7,8,9,12;36}{87}(\tau)
\]
the weighted-homogeneous family of isolated hypersurface germs represented by
\[
X_1^4X_2+X_3^4+X_2^3X_4+\tau X_1X_2X_3X_4+X_4^3=0,
\]
where
\[
\tau\in\Cstar,
\qquad
\tau^4\neq 256.
\]
It is quasi-homogeneous of weighted degree \(36\) with respect to the weights
\[
(w_{X_1},w_{X_2},w_{X_3},w_{X_4})=(7,8,9,12).
\]
Its ordered normalized weight is
\[
\left(\frac13,\frac14,\frac29,\frac7{36}\right),
\]
and hence its Yonemura number is \(52\).
Its Milnor number is \(\mu=87\), and its local minimal exponent is
\[
\widetilde{\alpha}=1.
\]
We abbreviate this weighted-homogeneous local normal-form family by
\[
\Sfam{87}{52}(\tau).
\]
\end{definition}

\begin{definition}
\label{def:isolated-family-88-Y84}
We denote by
\[
\QH{5,6,7,9;27}{88}(\alpha,\beta)
\]
the weighted-homogeneous family of isolated hypersurface germs represented by
\[
X_1^3X_2^2+X_1^4X_3+X_2X_3^3+X_2^3X_4
+\alpha X_1X_2X_3X_4+\beta X_4^3=0,
\]
where
\[
\alpha,\beta\in\Cstar,
\]
and
\[
\alpha^6-2\alpha^5+\alpha^4
+54\alpha^3\beta-378\alpha^2\beta+576\alpha\beta
+729\beta^2-256\beta\neq 0.
\]
It is quasi-homogeneous of weighted degree \(27\) with respect to the weights
\[
(w_{X_1},w_{X_2},w_{X_3},w_{X_4})=(5,6,7,9).
\]
Its ordered normalized weight is
\[
\left(\frac13,\frac7{27},\frac29,\frac5{27}\right),
\]
and hence its Yonemura number is \(84\).
Its Milnor number is \(\mu=88\), and its local minimal exponent is
\[
\widetilde{\alpha}=1.
\]
We abbreviate this weighted-homogeneous local normal-form family by
\[
\Sfam{88}{84}(\alpha,\beta).
\]
\end{definition}

\begin{definition}
\label{def:isolated-family-88-Y15}
Let \(B_5(X_1,X_2)\) be a binary quintic for which the hypersurface germ below is isolated at the origin.  We denote by
\[
\QH{3,3,4,5;15}{88}(B_5)
\]
the weighted-homogeneous family of isolated hypersurface germs represented by
\[
B_5(X_1,X_2)+X_2X_3^3+(X_1^2+X_2^2)X_3X_4+X_4^3=0.
\]
It is quasi-homogeneous of weighted degree \(15\) with respect to the weights
\[
(w_{X_1},w_{X_2},w_{X_3},w_{X_4})=(3,3,4,5).
\]
Its ordered normalized weight is
\[
\left(\frac13,\frac4{15},\frac15,\frac15\right),
\]
and hence its Yonemura number is \(15\).
Its Milnor number is \(\mu=88\), and its local minimal exponent is
\[
\widetilde{\alpha}=1.
\]
We abbreviate this weighted-homogeneous local normal-form family by
\[
\Sfam{88}{15}(B_5).
\]
\end{definition}

\begin{definition}
\label{def:isolated-family-90-Y2}
Let \(B_4\) be a square-free binary quartic and let \(Q_2,R_2\) be binary
quadratic forms.  Let
\[
\mathcal R_{90}^{\mathrm I}(B_4,Q_2,R_2,\beta)
\]
denote the weighted discriminant condition for the weighted-homogeneous germ below to be
isolated at the origin.  We denote by
\[
\QH{2,3,3,4;12}{90}(B_4,Q_2,R_2,\beta)
\]
the weighted-homogeneous family of isolated hypersurface germs represented by
\[
B_4(X_2,X_3)+X_1^3Q_2(X_2,X_3)+X_1^4X_4
+X_1X_4R_2(X_2,X_3)+\beta X_1^2X_4^2+X_4^3=0,
\]
where
\[
\beta\in\C,
\qquad
\mathcal R_{90}^{\mathrm I}(B_4,Q_2,R_2,\beta)\neq0.
\]
It is quasi-homogeneous of weighted degree \(12\) with respect to the weights
\[
(w_{X_1},w_{X_2},w_{X_3},w_{X_4})=(2,3,3,4).
\]
Its ordered normalized weight is
\[
\left(\frac13,\frac14,\frac14,\frac16\right),
\]
and hence its Yonemura number is \(2\).
Its Milnor number is \(\mu=90\), and its local minimal exponent is
\[
\widetilde{\alpha}=1.
\]
We abbreviate this weighted-homogeneous local normal-form family by
\[
\Sfam{90}{2}(B_4,Q_2,R_2,\beta).
\]
\end{definition}

\begin{definition}
\label{def:isolated-family-91-Y53}
Let
\[
\mathcal R_{91}(\alpha,\beta,\gamma,\delta,\epsilon)
\]
denote the bad-parameter equation for which the weighted-homogeneous germ below is not
isolated at the origin.  We denote by
\[
\QH{3,4,5,6;18}{91}(\alpha,\beta,\gamma,\delta,\epsilon)
\]
the weighted-homogeneous family of isolated hypersurface germs represented by
\[
X_1^2X_2^3+X_1^3X_2X_3+X_1^4X_4+X_2^2X_3^2
+\alpha X_2^3X_4+\beta X_1X_3^3
+\gamma X_1X_2X_3X_4+\delta X_1^2X_4^2+\epsilon X_4^3=0,
\]
where
\[
\alpha,\beta,\gamma,\delta,\epsilon\in\Cstar,
\qquad
\mathcal R_{91}(\alpha,\beta,\gamma,\delta,\epsilon)\neq0.
\]
It is quasi-homogeneous of weighted degree \(18\) with respect to the weights
\[
(w_{X_1},w_{X_2},w_{X_3},w_{X_4})=(3,4,5,6).
\]
Its ordered normalized weight is
\[
\left(\frac13,\frac5{18},\frac29,\frac16\right),
\]
and hence its Yonemura number is \(53\).
Its Milnor number is \(\mu=91\), and its local minimal exponent is
\[
\widetilde{\alpha}=1.
\]
We abbreviate this weighted-homogeneous local normal-form family by
\[
\Sfam{91}{53}(\alpha,\beta,\gamma,\delta,\epsilon).
\]
\end{definition}

\begin{definition}
\label{def:isolated-family-96-Y86}
Let
\[
\mathcal R_{96}(\alpha,\beta)
\]
denote the bad-parameter equation for which the weighted-homogeneous germ below is not
isolated at the origin.  We denote by
\[
\QH{4,5,7,9;25}{96}(\alpha,\beta)
\]
the weighted-homogeneous family of isolated hypersurface germs represented by
\[
X_2^5+X_1^2X_2^2X_3+X_1^4X_4+X_1X_3^3
+\alpha X_1X_2X_3X_4+\beta X_3X_4^2=0,
\]
where
\[
\alpha,\beta\in\Cstar,
\qquad
\mathcal R_{96}(\alpha,\beta)\neq0.
\]
It is quasi-homogeneous of weighted degree \(25\) with respect to the weights
\[
(w_{X_1},w_{X_2},w_{X_3},w_{X_4})=(4,5,7,9).
\]
Its ordered normalized weight is
\[
\left(\frac9{25},\frac7{25},\frac15,\frac4{25}\right),
\]
and hence its Yonemura number is \(86\).
Its Milnor number is \(\mu=96\), and its local minimal exponent is
\[
\widetilde{\alpha}=1.
\]
We abbreviate this weighted-homogeneous local normal-form family by
\[
\Sfam{96}{86}(\alpha,\beta).
\]
\end{definition}

\begin{definition}
\label{def:isolated-family-96-Y94}
Let
\[
\mathcal R^{\mathrm{II}}_{96}(\alpha,\beta,\gamma,\delta)
\]
denote the bad-parameter equation for which the weighted-homogeneous germ below is not
isolated at the origin.  We denote by
\[
\QH{3,4,5,7;19}{96}(\alpha,\beta,\gamma,\delta)
\]
the weighted-homogeneous family of isolated hypersurface germs represented by
\[
\begin{aligned}
&X_1X_2^4+X_1^2X_2^2X_3+
\alpha X_1^3X_3^2+X_1^4X_4+X_2X_3^3+\beta X_2^3X_4\\
&\qquad
+\gamma X_1X_2X_3X_4+
\delta X_3X_4^2=0,
\end{aligned}
\]
where
\[
\alpha,
\beta,
\gamma,
\delta\in\Cstar,
\qquad
\mathcal R^{\mathrm{II}}_{96}(\alpha,\beta,\gamma,\delta)\neq0.
\]
It is quasi-homogeneous of weighted degree \(19\) with respect to the weights
\[
(w_{X_1},w_{X_2},w_{X_3},w_{X_4})=(3,4,5,7).
\]
Its ordered normalized weight is
\[
\left(\frac7{19},\frac5{19},\frac4{19},\frac3{19}\right),
\]
and hence its Yonemura number is \(94\).
Its Milnor number is \(\mu=96\), and its local minimal exponent is
\[
\widetilde{\alpha}=1.
\]
We abbreviate this weighted-homogeneous local normal-form family by
\[
\Sfam{96}{94}(\alpha,\beta,\gamma,\delta).
\]
\end{definition}

\begin{definition}
\label{def:isolated-family-96-Y21}
Let \(B_5\) be a ternary quintic, let \(C_3\) be a ternary cubic, and let \(L_1\)
be a nonzero linear form in \((X_1,X_2,X_3)\).  Let
\[
\mathcal R^{\mathrm{III}}_{96}(B_5,C_3,L_1)
\]
denote the bad-parameter equation for which the weighted-homogeneous germ below is not
isolated at the origin.  We denote by
\[
\QH{1,1,1,2;5}{96}(B_5,C_3,L_1)
\]
the weighted-homogeneous family of isolated hypersurface germs represented by
\[
B_5(X_1,X_2,X_3)+X_4C_3(X_1,X_2,X_3)+X_4^2L_1(X_1,X_2,X_3)=0,
\]
where
\[
\mathcal R^{\mathrm{III}}_{96}(B_5,C_3,L_1)\neq0.
\]
It is quasi-homogeneous of weighted degree \(5\) with respect to the weights
\[
(w_{X_1},w_{X_2},w_{X_3},w_{X_4})=(1,1,1,2).
\]
Its ordered normalized weight is
\[
\left(\frac25,\frac15,\frac15,\frac15\right),
\]
and hence its Yonemura number is \(21\).
Its Milnor number is \(\mu=96\), and its local minimal exponent is
\[
\widetilde{\alpha}=1.
\]
We abbreviate this weighted-homogeneous local normal-form family by
\[
\Sfam{96}{21}(B_5,C_3,L_1).
\]
\end{definition}

\begin{definition}
\label{def:isolated-family-100-Y3}
Let \(A_4,C_4\) be binary quartics in \((X_1,X_2)\), let \(Q_{2,2}\) be a
bihomogeneous form of bidegree \((2,2)\) in \((X_1,X_2)\) and \((X_3,X_4)\),
and let \(B_3\) be a binary cubic in \((X_3,X_4)\).  Let
\[
\mathcal R^{\mathrm I}_{100}(A_4,C_4,Q_{2,2},B_3)
\]
denote the bad-parameter equation for which the weighted-homogeneous germ below is not
isolated at the origin.  We denote by
\[
\QH{1,1,2,2;6}{100}(A_4,C_4,Q_{2,2},B_3)
\]
the weighted-homogeneous family of isolated hypersurface germs represented by
\[
X_3A_4(X_1,X_2)+X_4C_4(X_1,X_2)
+Q_{2,2}(X_1,X_2;X_3,X_4)+B_3(X_3,X_4)=0,
\]
where
\[
\mathcal R^{\mathrm I}_{100}(A_4,C_4,Q_{2,2},B_3)\neq0.
\]
It is quasi-homogeneous of weighted degree \(6\) with respect to the weights
\[
(w_{X_1},w_{X_2},w_{X_3},w_{X_4})=(1,1,2,2).
\]
Its ordered normalized weight is
\[
\left(\frac13,\frac13,\frac16,\frac16\right),
\]
and hence its Yonemura number is \(3\).
Its Milnor number is \(\mu=100\), and its local minimal exponent is
\[
\widetilde{\alpha}=1.
\]
We abbreviate this weighted-homogeneous local normal-form family by
\[
\Sfam{100}{3}(A_4,C_4,Q_{2,2},B_3).
\]
\end{definition}

\begin{definition}
\label{def:isolated-family-102-Y62}
Let
\[
\mathcal R^{\mathrm I}_{102}(\alpha,\beta,\gamma,\delta)
\]
denote the bad-parameter equation for which the weighted-homogeneous germ below is not
isolated at the origin.  We denote by
\[
\QH{3,4,5,8;20}{102}(\alpha,\beta,\gamma,\delta)
\]
the weighted-homogeneous family of isolated hypersurface germs represented by
\[
\begin{aligned}
&X_2^5+
\alpha X_1X_2^3X_3+
\beta X_1^2X_2X_3^2+X_1^4X_4+X_3^4+
\gamma X_2^3X_4 \\
&\qquad
+\delta X_1X_2X_3X_4+X_2X_4^2=0,
\end{aligned}
\]
where
\[
\alpha,
\beta,
\gamma,
\delta\in\Cstar,
\qquad
\mathcal R^{\mathrm I}_{102}(\alpha,\beta,\gamma,\delta)\neq0.
\]
It is quasi-homogeneous of weighted degree \(20\) with respect to the weights
\[
(w_{X_1},w_{X_2},w_{X_3},w_{X_4})=(3,4,5,8).
\]
Its ordered normalized weight is
\[
\left(\frac25,\frac14,\frac15,\frac3{20}\right),
\]
and hence its Yonemura number is \(62\).
Its Milnor number is \(\mu=102\), and its local minimal exponent is
\[
\widetilde{\alpha}=1.
\]
We abbreviate this weighted-homogeneous local normal-form family by
\[
\Sfam{102}{62}(\alpha,\beta,\gamma,\delta).
\]
\end{definition}

\paragraph{Equation-level catalog.}
The defining equations, integral weights, Milnor numbers, Yonemura numbers,
and component occurrences of the eleven families are summarized in
Table~\ref{tab:intro-isolated-families}.  Definitions~\ref{def:isolated-family-81-Y1}--
\ref{def:isolated-family-102-Y62} record the exact nonempty Zariski-open
genericity conditions.

\subsection*{The component-indexed singular-support catalog}

\providecommand{\C}{\mathbb C}
\providecommand{\Cstar}{\mathbb C^{\times}}
\providecommand{\PP}{\mathbf P}
\providecommand{\Sing}{\operatorname{Sing}}
\providecommand{\qtfSLcell}[1]{\begin{tabular}[t]{@{}l@{}}#1\end{tabular}}
\providecommand{\qtfSLmath}[1]{\(\begin{aligned}[t]#1\end{aligned}\)}
\providecommand{\qtfSLnone}{\(\varnothing\)}

\noindent
All rows are for general parameters.  We use
\[
  P_0=(1:0:0:0:0),\qquad P_\infty=(0:0:0:0:1),
\]
and
\[
  B_\lambda(u,v):=uv(u-v)(u-\lambda v),\qquad
  B_{\lambda,\mu}(u,v):=v(u-v)(u-\lambda v)(u-\mu v).
\]
The last column records only isolated singular points; distinguished points
on positive-dimensional strata and their transverse analytic types are
listed in Table~\ref{tab:intro-generic-singularities} and proved in
Appendix~\ref{app:singular-analysis}.  In the notation \(\Sfam{\mu}{N}\),
\(\mu\) is the Milnor number and \(N\) the Yonemura number.  In the row
\(a=19\), \(B^{(29)}_3=x_3x_4(x_3-x_4)\).

\begingroup
\scriptsize
\renewcommand{\arraystretch}{1.18}
\setlength{\tabcolsep}{0.8pt}
\setlength{\extrarowheight}{1pt}
\begin{longtable}{@{}
  >{\raggedleft\arraybackslash}p{0.025\textwidth}
  >{\raggedleft\arraybackslash}p{0.035\textwidth}
  >{\raggedright\arraybackslash}p{0.335\textwidth}
  >{\raggedright\arraybackslash}p{0.335\textwidth}
  >{\raggedright\arraybackslash}p{0.215\textwidth}
@{}}
\caption{Reduced singular supports of the twenty-one quotient-component representatives.  Here \(k=k(a)\) is the selected oriented support case.  For a paired label, coordinate reversal identifies the zero-weight spaces and the quotient images agree, but it need not identify the displayed normal forms coefficientwise.  The general singularity data therefore agree up to projective equivalence.}
\label{tab:component-singular-loci}\\
\toprule
\(a\) & \(k(a)\) & Singular locus & Geometry / incidences & Isolated local family\\
\midrule
\endfirsthead
\toprule
\(a\) & \(k(a)\) & Singular locus & Geometry / incidences & Isolated local family\\
\midrule
\endhead
\midrule
\multicolumn{5}{r}{\emph{continued on the next page}}\\
\endfoot
\bottomrule
\endlastfoot

1 & 1 & \qtfSLmath{\Sing(X_{1})&=C\cup\{P_0\}\\
C&=V(x_0,x_1,x_3^4+x_2^3x_4)} & \qtfSLcell{rational plane quartic;\\ cusp at \(P_\infty\)} & \qtfSLcell{\(P_0:\ \Sfam{87}{52}(\tau)\)}\\
\midrule

2 & 2 & \qtfSLmath{\Sing(X_{2})&=L\cup C\cup\{P_0\}\\
L&=V(x_0,x_1,x_2)\\
C&=V(x_0,x_1,x_3^3+x_2^2x_4)} & \qtfSLcell{\(L\) line; \(C\) cuspidal cubic;\\ \(L\cap C=\{P_\infty\}\)} & \qtfSLcell{\(P_0:\ \Sfam{88}{84}(\alpha,\beta)\)}\\
\midrule

3 & 3 & \qtfSLmath{\Sing(X_{3})&=L\cup\{P_0\}\\
L&=V(x_0,x_1,x_2)} & \qtfSLcell{\(L\) line; special point \(P_\infty\)} & \qtfSLcell{\(P_0:\ \Sfam{88}{15}(B_5)\)}\\
\midrule

4 & 4 & \qtfSLmath{\Sing(X_{4})&=M\sqcup L\\
M&=V(x_2,x_3,x_4)\\
L&=V(x_0,x_1,x_2)} & \qtfSLcell{two disjoint lines} & \qtfSLnone\\
\midrule

5 & 5 & \qtfSLmath{\Sing(X_{5})&=N\cup L\cup M\\
N&=V(x_2,x_3,x_4)\\
L&=V(x_0,x_1,x_2)\\
M&=V(x_0,x_1,x_3)} & \qtfSLcell{three lines; \(L\cap M=\{P_\infty\}\);\\ \(N\cap L=N\cap M=\varnothing\)} & \qtfSLnone\\
\midrule

6 & 6 & \qtfSLmath{\Sing(X_{6})&=L\cup C\cup\{P_0\}\\
L&=V(x_0,x_1,x_2)\\
C&=V(x_0,x_1,x_3^2+\alpha x_2x_4)} & \qtfSLcell{\(L\) line; \(C\) smooth conic;\\ \(L\cap C=\{P_\infty\}\)} & \qtfSLcell{\(P_0:\ \Sfam{91}{53}\)\\
\((\alpha,\beta,\gamma,\delta,\epsilon)\)}\\
\midrule

7 & 7 & \qtfSLmath{\Sing(X_{7})&=L\cup\{P_0\}\\
L&=V(x_0,x_1,x_2)} & \qtfSLcell{\(L\) line; special point \(P_\infty\)} & \qtfSLcell{\(P_0:\ \Sfam{96}{86}(\alpha,\beta)\)}\\
\midrule

8 & 8 & \qtfSLmath{\Sing(X_{8})&=N\cup L\\
N&=V(x_2,x_3,x_4)\\
L&=V(x_0,x_1,x_2)} & \qtfSLcell{two disjoint lines} & \qtfSLnone\\
\midrule

9 & 9 & \qtfSLmath{\Sing(X_{9})&=L\cup\{P_0\}\\
L&=V(x_0,x_1,x_2)} & \qtfSLcell{\(L\) line; special point \(P_\infty\)} & \qtfSLcell{\(P_0:\ \Sfam{96}{94}(\alpha,\beta,\gamma,\delta)\)}\\
\midrule

10 & 10 & \qtfSLmath{\Sing(X_{10})&=C_\lambda\cup\{P_0\}\\
C_\lambda&=V(x_0,x_1,B_\lambda(x_2,x_3))} & \qtfSLcell{\(C_\lambda=L_0\cup L_\infty\cup L_1\cup L_\lambda\);\\ four lines through \(P_\infty\)} & \qtfSLcell{\(P_0:\ \Sfam{90}{2}\)\\
\((B_\lambda,Q_{\mathbf q},R_{\mathbf r},\beta)\)}\\
\midrule

11 & 11 & \qtfSLmath{\Sing(X_{11})&=N\cup L\\
N&=V(x_2,x_3,x_4)\\
L&=V(x_0,x_1,x_2)} & \qtfSLcell{two disjoint lines} & \qtfSLnone\\
\midrule

12 & 12 & \qtfSLmath{\Sing(X_{12})&=N\cup L\\
N&=V(x_2,x_3,x_4)\\
L&=V(x_0,x_1,x_2)} & \qtfSLcell{two disjoint lines} & \qtfSLnone\\
\midrule

13 & 13 & \qtfSLmath{\Sing(X_{13})&=N\cup L\\
N&=V(x_2,x_3,x_4)\\
L&=V(x_0,x_1,x_2)} & \qtfSLcell{two disjoint lines} & \qtfSLnone\\
\midrule

14 & 14 & \qtfSLmath{\Sing(X_{14})&=\{P_0,P_\infty\}} & \qtfSLcell{two isolated points} & \qtfSLcell{\(P_0,P_\infty:\)\\
\(\Sfam{102}{62}(\alpha,\beta,\gamma,\delta)\)}\\
\midrule

15 & 15 & \qtfSLmath{\Sing(X_{15})&=N\cup L\\
N&=V(x_2,x_3,x_4)\\
L&=V(x_0,x_1,x_2)} & \qtfSLcell{two disjoint lines} & \qtfSLnone\\
\midrule

16 & 19 & \qtfSLmath{\Sing(X_{16})&=N\cup L\\
N&=V(x_2,x_3,x_4)\\
L&=V(x_0,x_1,x_2)} & \qtfSLcell{two disjoint lines} & \qtfSLnone\\
\midrule

17 & 20 & \qtfSLmath{\Sing(X_{17})&=N\cup C_{\lambda,\mu}\\
N&=V(x_2,x_3,x_4)\\
C_{\lambda,\mu}&=V(x_0,x_1,B_{\lambda,\mu}(x_2,x_3))} & \qtfSLcell{\(N\) line; \(C_{\lambda,\mu}\) four lines\\ through \(P_\infty\); disjoint from \(N\)} & \qtfSLnone\\
\midrule

18 & 23 & \qtfSLmath{\Sing(X_{18})&=S\cup\{P_0\}\\
S&=V(x_0,B_4)} & \qtfSLcell{\(S\) smooth quartic surface} & \qtfSLcell{\(P_0:\ \Sfam{81}{1}(B_4)\)}\\
\midrule

19 & 29 & \qtfSLmath{\Sing(X_{19})&=L\cup\{P_0\}\\
L&=V(x_0,x_1,x_2)} & \qtfSLcell{\(L\) line} & \qtfSLcell{\(P_0:\ \Sfam{100}{3}\)\\
\((A_4,C_4,Q_{2,2},B^{(29)}_3)\)}\\
\midrule

20 & 32 & \qtfSLmath{\Sing(X_{20})&=N\cup\Pi\\
N&=V(x_2,x_3,x_4)\\
\Pi&=V(x_0,x_1)} & \qtfSLcell{\(N\) line; \(\Pi\simeq\PP^2\);\\ \(N\cap\Pi=\varnothing\)} & \qtfSLnone\\
\midrule

21 & 36 & \qtfSLmath{\Sing(X_{21})&=\{P_0,P_\infty\}} & \qtfSLcell{two isolated points} & \qtfSLcell{\(P_0,P_\infty:\)\\
\(\Sfam{96}{21}(B_5,C_3,X_3)\)}\\
\midrule

\end{longtable}
\endgroup

\begin{theorem}[Component-indexed singularity and minimal-exponent catalog]
\label{thm:component-singularity-catalog}
For a general closed-orbit representative \(X_a\) of every component
\(\Theta_a\), the reduced singular support is the one in
Table~\ref{tab:component-singular-loci}, and the full local analytic
stratification is the one in Table~\ref{tab:intro-generic-singularities}.
The isolated points belong to precisely the eleven weighted-homogeneous
families of Table~\ref{tab:intro-isolated-families}; no other isolated
local family occurs for a general boundary representative.  Moreover,
\[
  \mexp(X_a)=1=\frac{4+1}{5},\qquad a=1,\ldots,21.
\]
\end{theorem}

\begin{proof}
The saturated-Jacobian, local analytic, and minimal-exponent calculations are
carried out component by component in
Appendix~\ref{app:singular-analysis}.  For every paired support label \(k'\),
coordinate reversal identifies its zero-weight space with that of the chosen
representative, and Corollary~\ref{cor:supporting-hyperplane-invariance}
identifies their quotient images.
Remark~\ref{rem:paired-normal-form-loci} therefore shows that the paired
normal-form loci represent the same general closed orbit up to projective
equivalence.  Hence the twenty-one calculations cover
all thirty-eight original oriented support labels; no coefficientwise
identification of the displayed normal forms by \(\rho\) is used.
\end{proof}

\clearpage
\appendix
\section{Support-by-support construction of the closed-orbit normal forms}
\label{app:normal-forms}

This appendix supplies the thirty-eight support-indexed calculations
underlying Section~\ref{sec:normal-form}.  For each maximal strictly
\(T\)-semistable support
\[
  S_k=I(r_k)_{\geq0},\qquad k=1,\ldots,38,
\]
we extract the weight-zero limit \(\phi_k\), apply the centralizer reduction
and the C-H or C-F criterion developed in
Section~\ref{sec:normal-form}, and use the residual centralizer action to
obtain the explicit normal form \(\nf_k\).  The calculations are necessarily
indexed by all thirty-eight oriented supports; the quotient-side reduction
to twenty-one representatives has already been carried out in
Sections~\ref{sec:quotient-side-inclusions} and~\ref{sec:normal-form}.

For reference, the full support-indexed catalog is reproduced first.  Its
paired rows determine the same quotient component, but are retained here so
that every case calculation can be checked directly against its original
support label.

Let us put
\[
B_{\lambda}(u,v):=uv(u-v)(u-\lambda v),\qquad
B_{\lambda,\mu}(u,v):=v(u-v)(u-\lambda v)(u-\mu v).
\]
Here \(B_d,C_d,Q_d,L_d,M_d\) denote general forms of the indicated degree,
and \(Q_{2,2}\) denotes a bihomogeneous form of bidegree \((2,2)\).
All parameters are understood to lie in the relevant nonempty Zariski-open
generic locus.

\begingroup
\scriptsize
\renewcommand{\arraystretch}{1.18}
\setlength{\tabcolsep}{1pt}
\setlength{\extrarowheight}{1pt}
\begin{longtable}{@{}
  >{\raggedleft\arraybackslash}p{0.030\textwidth}
  >{\centering\arraybackslash}p{0.060\textwidth}
  >{\raggedright\arraybackslash}p{0.595\textwidth}
  >{\raggedright\arraybackslash}p{0.245\textwidth}
  >{\raggedleft\arraybackslash}p{0.035\textwidth}
@{}}
\caption{Closed-orbit normal forms for the \(38\) quintic-threefold boundary families.}\label{tab:quintic-normal-forms}\\
\toprule
\(k\) & Test & Normal form & Parameters & Dim\\
\midrule
\endfirsthead
\toprule
\(k\) & Test & Normal form & Parameters & Dim\\
\midrule
\endhead
\midrule
\multicolumn{5}{r}{\emph{continued on the next page}}\\
\endfoot
\bottomrule
\endlastfoot

1 & C-H & \qtfNFcell{
\nf_1={}&x_1^4x_2+x_0x_3^4+x_0x_2^3x_4\\
&+x_0x_1x_2x_3x_4+\alpha x_0^2x_4^3
} & \qtfParamcell{$\alpha\in\Cstar$} & 1\\
\midrule

2 & C-H & \qtfNFcell{
\nf_2={}&x_1^3x_2^2+x_1^4x_3+x_0x_2x_3^3+x_0x_2^3x_4\\
&+\alpha x_0x_1x_2x_3x_4+\beta x_0^2x_4^3
} & \qtfParamcell{$\alpha,\beta\in\Cstar$} & 2\\
\midrule

3 & C-F & \qtfNFcell{
\nf_3={}&B_5(x_1,x_2)+x_0x_2x_3^3\\
&+x_0(x_1^2+x_2^2)x_3x_4+x_0^2x_4^3
} & \qtfParamcell{$B_5$: general\\ binary quintic} & 6\\
\midrule

4 & C-H & \qtfNFcell{
\nf_4={}&x_1^2x_2^3+x_1^3x_3x_4+x_0x_2^4+x_0x_1x_3^3\\
&+\alpha x_0x_1x_2x_3x_4+\beta x_0^2x_4^3
} & \qtfParamcell{$\alpha,\beta\in\Cstar$} & 2\\
\midrule

5 & C-H & \qtfNFcell{
\nf_5={}&x_1^2x_2^3+x_1^3x_3^2+x_1^3x_2x_4+x_0x_2^3x_3\\
&+\alpha x_0x_1x_3^3+\beta x_0x_1x_2x_3x_4+\gamma x_0^2x_4^3
} & \qtfParamcell{$\alpha,\beta,\gamma\in\Cstar$} & 3\\
\midrule

6 & C-H & \qtfNFcell{
\nf_6={}&x_1^2x_2^3+x_1^3x_2x_3+x_1^4x_4+x_0x_2^2x_3^2\\
&+\alpha x_0x_2^3x_4+\beta x_0x_1x_3^3+\gamma x_0x_1x_2x_3x_4\\
&+\delta x_0x_1^2x_4^2+\epsilon x_0^2x_4^3
} & \qtfParamcell{$\alpha,\beta,\gamma,\delta,$\\ $\epsilon\in\Cstar$} & 5\\
\midrule

7 & C-H & \qtfNFcell{
\nf_7={}&x_2^5+x_1^2x_2^2x_3+x_1^4x_4+x_0x_1x_3^3\\
&+\alpha x_0x_1x_2x_3x_4+\beta x_0^2x_3x_4^2
} & \qtfParamcell{$\alpha,\beta\in\Cstar$} & 2\\
\midrule

8 & C-H & \qtfNFcell{
\nf_8={}&x_1^2x_2^2x_3+x_1^3x_4^2+x_0x_2^4+x_0x_1x_3^3\\
&+\alpha x_0x_1x_2x_3x_4+\beta x_0^2x_3x_4^2
} & \qtfParamcell{$\alpha,\beta\in\Cstar$} & 2\\
\midrule

9 & C-H & \qtfNFcell{
\nf_9={}&x_1x_2^4+x_1^2x_2^2x_3+\alpha x_1^3x_3^2+x_1^4x_4\\
&+x_0x_2x_3^3+\beta x_0x_2^3x_4+\gamma x_0x_1x_2x_3x_4\\
&+\delta x_0^2x_3x_4^2
} & \qtfParamcell{$\alpha,\beta,\gamma,\delta\in\Cstar$} & 4\\
\midrule

10 & C-F & \qtfNFcell{
\nf_{10}={}&x_1^3Q_{\mathbf q}(x_2,x_3)+x_1^4x_4+x_0B_\lambda(x_2,x_3)\\
&+x_0x_1x_4R_{\mathbf r}(x_2,x_3)+\beta x_0x_1^2x_4^2+x_0^2x_4^3
} & \qtfParamcell{$\lambda\in\C\setminus\{0,1\}$;\\ $\mathbf q,\mathbf r\in\C^3$, $\beta\in\C$} & 8\\
\midrule

11 & C-H & \qtfNFcell{
\nf_{11}={}&x_1x_2^4+\alpha x_1^2x_2^2x_3+\beta x_1^3x_3^2+x_1^3x_2x_4\\
&+\gamma x_0x_2^2x_3^2+x_0x_2^3x_4+\delta x_0x_1x_3^3\\
&+\epsilon x_0x_1x_2x_3x_4+\zeta x_0x_1^2x_4^2+x_0^2x_3x_4^2
} & \qtfParamcell{$\alpha,\beta,\gamma,\delta,$\\ $\epsilon,\zeta\in\Cstar$} & 6\\
\midrule

12 & C-H & \qtfNFcell{
\nf_{12}={}&x_1x_2^4+\alpha x_1^2x_2x_3^2+\beta x_1^2x_2^2x_4+x_1^3x_4^2\\
&+x_0x_2^3x_3+\gamma x_0x_1x_3^3+\delta x_0x_1x_2x_3x_4\\
&+\epsilon x_0^2x_3^2x_4+x_0^2x_2x_4^2
} & \qtfParamcell{$\alpha,\beta,\gamma,\delta,$\\ $\epsilon\in\Cstar$} & 5\\
\midrule

13 & C-H & \qtfNFcell{
\nf_{13}={}&x_2^5+\alpha x_1x_2^3x_3+\beta x_1^2x_2x_3^2+x_1^3x_4^2\\
&+x_0x_2^3x_4+\gamma x_0x_1x_2x_3x_4+x_0^2x_3^3\\
&+\delta x_0^2x_2x_4^2
} & \qtfParamcell{$\alpha,\beta,\gamma,\delta\in\Cstar$} & 4\\
\midrule

14 & C-H & \qtfNFcell{
\nf_{14}={}&x_2^5+\alpha x_1x_2^3x_3+\beta x_1^2x_2x_3^2+x_1^4x_4+x_0x_3^4\\
&+\gamma x_0x_2^3x_4+\delta x_0x_1x_2x_3x_4+x_0^2x_2x_4^2
} & \qtfParamcell{$\alpha,\beta,\gamma,\delta\in\Cstar$} & 4\\
\midrule

15 & C-F & \qtfNFcell{
\nf_{15}={}&x_1x_2^4+x_1^2x_2^2x_3\\
&+x_1^3(\gamma x_3^2+\beta x_3x_4+x_4^2)+x_0x_2^3x_4\\
&+x_0x_1x_2R(x_3,x_4)+x_0^2B(x_3,x_4)
} & \qtfParamcell{$\gamma,\beta\in\C$;\\
$R\in\Sym^2\langle x_3,x_4\rangle$;\\
$B\in\Sym^3\langle x_3,x_4\rangle$} & 9\\
\midrule

16 & C-H & \qtfNFcell{
\nf_{16}={}&x_2^4x_4+x_1x_2^2x_3^2+x_1^3x_3x_4+x_0x_1x_2x_3x_4\\
&+\alpha x_0^2x_3^3+\beta x_0^2x_1x_4^2
} & \qtfParamcell{$\alpha,\beta\in\Cstar$} & 2\\
\midrule

17 & C-H & \qtfNFcell{
\nf_{17}={}&x_2^4x_3+x_1x_2^3x_4+\alpha x_1^2x_2x_3^2+x_1^3x_3x_4\\
&+\beta x_0x_2^2x_3^2+\gamma x_0x_1x_2x_3x_4+\delta x_0x_1^2x_4^2\\
&+x_0^2x_3^3+\epsilon x_0^2x_2x_4^2
} & \qtfParamcell{$\alpha,\beta,\gamma,\delta,$\\ $\epsilon\in\Cstar$} & 5\\
\midrule

18 & C-H & \qtfNFcell{
\nf_{18}={}&x_2^5+x_1x_2^2x_3^2+x_1^3x_3x_4+x_0x_3^4\\
&+\alpha x_0x_1x_2x_3x_4+\beta x_0^2x_1x_4^2
} & \qtfParamcell{$\alpha,\beta\in\Cstar$} & 2\\
\midrule

19 & C-H & \qtfNFcell{
\nf_{19}={}&x_2^5+\alpha x_1x_2^3x_3+\beta x_1^2x_2x_3^2+\gamma x_1^2x_2^2x_4\\
&+x_1^3x_3x_4+\delta x_0x_2^2x_3^2+\epsilon x_0x_2^3x_4+x_0x_1x_3^3\\
&+\zeta x_0x_1x_2x_3x_4+\xi x_0x_1^2x_4^2+x_0^2x_3^2x_4\\
&+\theta x_0^2x_2x_4^2
} & \qtfParamcell{$\alpha,\beta,\gamma,\delta,$\\ $\epsilon,\zeta,\xi,\theta\in\Cstar$} & 8\\
\midrule

20 & C-F & \qtfNFcell{
\nf_{20}={}&x_0B_{\lambda,\mu}(x_2,x_3)+x_1^2C_{\mathbf c}(x_2,x_3)+x_1^3x_2x_4\\
&+x_0x_1x_4Q_{\mathbf q}(x_2,x_3)+x_0x_1^2x_4^2\\
&+x_0^2x_4^2M_{\mathbf m}(x_2,x_3)
} & \qtfParamcell{$\lambda,\mu$ distinct in\\ $\C\setminus\{0,1\}$;\\ $\mathbf c\in\C^4$, $\mathbf q\in\C^3$,\\ $\mathbf m\in\C^2$} & 11\\
\midrule

21 & C-H & \qtfNFcell{
\nf_{21}={}&x_2^3x_3^2+x_2^4x_4+x_1^3x_3x_4+x_0x_1x_3^3\\
&+\alpha x_0x_1x_2x_3x_4+\beta x_0^3x_4^2
} & \qtfParamcell{$\alpha,\beta\in\Cstar$} & 2\\
\midrule

22 & C-H & \qtfNFcell{
\nf_{22}={}&x_2^4x_3+x_1x_2^2x_3^2+\alpha x_1^2x_3^3+x_1^3x_2x_4\\
&+x_0x_3^4+\beta x_0x_2^3x_4+\gamma x_0x_1x_2x_3x_4\\
&+\delta x_0^2x_1x_4^2
} & \qtfParamcell{$\alpha,\beta,\gamma,\delta\in\Cstar$} & 4\\
\midrule

23 & C-F & \qtfNFcell{
\nf_{23}={}&x_0B_4(x_1,x_2,x_3,x_4)
} & \qtfParamcell{$B_4$: smooth quartic\\ surface in $\PP^3$} & 19\\
\midrule

24 & C-H & \qtfNFcell{
\nf_{24}={}&x_2^3x_3^2+x_1x_2^3x_4+\alpha x_1^2x_3^3+x_1^3x_3x_4\\
&+x_0x_2x_3^3+\beta x_0x_1x_2x_3x_4+\gamma x_0^3x_4^2
} & \qtfParamcell{$\alpha,\beta,\gamma\in\Cstar$} & 3\\
\midrule

25 & C-H & \qtfNFcell{
\nf_{25}={}&x_2^4x_3+\alpha x_1x_2^2x_3^2+\beta x_1^2x_3^3+x_1^2x_2^2x_4\\
&+x_1^3x_3x_4+x_0x_2x_3^3+\gamma x_0x_2^3x_4\\
&+\delta x_0x_1x_2x_3x_4+\epsilon x_0^2x_3^2x_4+\zeta x_0^2x_1x_4^2
} & \qtfParamcell{$\alpha,\beta,\gamma,\delta,$\\ $\epsilon,\zeta\in\Cstar$} & 6\\
\midrule

26 & C-H & \qtfNFcell{
\nf_{26}={}&x_2x_3^4+x_1^4x_4+x_0x_2^3x_4\\
&+\alpha x_0x_1x_2x_3x_4+x_0^3x_4^2
} & \qtfParamcell{$\alpha\in\Cstar$} & 1\\
\midrule

27 & C-H & \qtfNFcell{
\nf_{27}={}&x_2^2x_3^3+x_1x_3^4+x_1^3x_2x_4+x_0x_2^3x_4\\
&+\alpha x_0x_1x_2x_3x_4+\beta x_0^3x_4^2
} & \qtfParamcell{$\alpha,\beta\in\Cstar$} & 2\\
\midrule

28 & C-H & \qtfNFcell{
\nf_{28}={}&x_2^3x_3^2+x_1x_2x_3^3+x_1^2x_2^2x_4+\beta x_1^3x_3x_4\\
&+x_0x_3^4+\alpha x_0x_2^3x_4+\gamma x_0x_1x_2x_3x_4\\
&+\delta x_0^2x_3^2x_4+\epsilon x_0^3x_4^2
} & \qtfParamcell{$\alpha,\beta,\gamma,\delta,$\\ $\epsilon\in\Cstar$} & 5\\
\midrule

29 & C-F & \qtfNFcell{
\nf_{29}={}&x_3A_4(x_1,x_2)+x_4C_4(x_1,x_2)\\
&+x_0Q_{2,2}(x_1,x_2;x_3,x_4)+x_0^2x_3x_4(x_3-x_4)
} & \qtfParamcell{$A_4,C_4$: binary quartics;\\ $Q_{2,2}$ of bidegree $(2,2)$} & 15\\
\midrule

30 & C-F & \qtfNFcell{
\nf_{30}={}&x_2^4x_3+x_2^2x_3^2L_1(x_0,x_1)+x_2^3x_4L_2(x_0,x_1)\\
&+x_3^3Q_1(x_0,x_1)+x_2x_3x_4Q_2(x_0,x_1)\\
&+x_0x_1(x_0-x_1)x_4^2
} & \qtfParamcell{$L_1,L_2$: binary linear;\\ $Q_1,Q_2$: binary quadratic} & 9\\
\midrule

31 & C-F & \qtfNFcell{
\nf_{31}={}&x_4B_\lambda(x_1,x_2)+x_3^2C_3(x_1,x_2)+x_0x_3^3L_1(x_1,x_2)\\
&+x_0x_3x_4Q_2(x_1,x_2)+x_0^2x_3^2x_4\\
&+x_0^2x_4^2M_1(x_1,x_2)
} & \qtfParamcell{$\lambda\in\C\setminus\{0,1\}$;\\ $C_3,L_1,Q_2,M_1$\\ of degrees $3,1,2,1$} & 11\\
\midrule

32 & C-F & \qtfNFcell{
\nf_{32}={}&x_0^2C_0(x_2,x_3,x_4)+x_0x_1C_1(x_2,x_3,x_4)\\
&+x_1^2C_2(x_2,x_3,x_4)
} & \qtfParamcell{$C_0,C_1,C_2$: ternary cubics} & 18\\
\midrule

33 & C-F & \qtfNFcell{
\nf_{33}={}&B_5(x_2,x_3)+x_1^3x_2x_4\\
&+x_0x_1x_4(x_2^2+x_3^2)+x_0^3x_4^2
} & \qtfParamcell{$B_5$: general\\ binary quintic} & 6\\
\midrule

34 & C-F & \qtfNFcell{
\nf_{34}={}&x_4B_\lambda(x_1,x_2)+x_3^3Q_2(x_1,x_2)+x_0x_3^4\\
&+x_0x_3x_4R_2(x_1,x_2)+\beta x_0^2x_3^2x_4+x_0^3x_4^2
} & \qtfParamcell{$\lambda\in\C\setminus\{0,1\}$;\\ $Q_2,R_2$: binary quadratic;\\ $\beta\in\C$} & 8\\
\midrule

35 & C-F & \qtfNFcell{
\nf_{35}={}&x_0A_4(x_2,x_3)+x_1C_4(x_2,x_3)\\
&+x_4Q_{2,2}(x_0,x_1;x_2,x_3)+x_4^2x_0x_1(x_0-x_1)
} & \qtfParamcell{$A_4,C_4$: binary quartics;\\ $Q_{2,2}$ of bidegree $(2,2)$} & 15\\
\midrule

36 & C-F & \qtfNFcell{
\nf_{36}={}&B_5(x_1,x_2,x_3)+x_0x_4C_3(x_1,x_2,x_3)\\
&+x_0^2x_3x_4^2
} & \qtfParamcell{$B_5,C_3$: ternary forms\\ of degrees $5,3$} & 24\\
\midrule

37 & C-F & \qtfNFcell{
\nf_{37}={}&x_3^2D_0(x_0,x_1,x_2)+x_3x_4D_1(x_0,x_1,x_2)\\
&+x_4^2D_2(x_0,x_1,x_2)
} & \qtfParamcell{$D_0,D_1,D_2$: ternary cubics} & 18\\
\midrule

38 & C-F & \qtfNFcell{
\nf_{38}={}&x_4B_4(x_0,x_1,x_2,x_3)
} & \qtfParamcell{$B_4$: smooth quartic\\ surface in $\PP^3$} & 19\\

\end{longtable}
\endgroup

Unless stated otherwise, ``general'' means that the coefficients belong to
the specified nonempty Zariski-open subset of the relevant parameter space.
The nonvanishing, discriminant, resultant, and residual-stability conditions
used in the closed-orbit and normalization arguments are recorded in the
individual cases.

\StartCase{1}

\subsubsection*{1-PS limit}
The witness vector and the corresponding one-parameter subgroup are
\[
  r_1=(36,1,-4,-9,-24),
  \qquad
  \lambda_1(t)=\diag(t^{36},t,t^{-4},t^{-9},t^{-24}),
  \quad t\in\Gm.
\]
Since the entries of \(r_1\) sum to zero, \(\lambda_1\) is a one-parameter
subgroup of \(G=\SL(5)\).  Let \(f_1\) be a general quintic supported on
\[
  S_1=I(r_1)_{\geq0}.
\]
With the action convention of \eqref{eq:quintic-monomial-weight}, its affine
one-parameter-subgroup limit is the weight-zero part
\[
  \phi_1:=\lim_{t\to0}\lambda_1(t)\cdot f_1.
\]
The degree-five weight-zero equations have precisely the exponent vectors
\[
  (0,4,1,0,0),\quad
  (1,0,0,4,0),\quad
  (1,0,3,0,1),\quad
  (1,1,1,1,1),\quad
  (2,0,0,0,3).
\]
Consequently,
\[
  \phi_1
  =a_1x_1^4x_2
  +a_2x_0x_3^4
  +a_3x_0x_2^3x_4
  +a_4x_0x_1x_2x_3x_4
  +a_5x_0^2x_4^3,
  \qquad
  a_i\in\C^\times.
\]

Set \(H_1=\lambda_1(\Gm)\).  The five weights of \(H_1\) on
\(\langle x_0,\ldots,x_4\rangle\) are pairwise distinct, and their multiset
is not invariant under multiplication by \(-1\).  Hence
\[
  C_G(H_1)=N_G(H_1)=T,
\]
and
\[
  W^{H_1}
  =\left\langle
    x_1^4x_2,
    x_0x_3^4,
    x_0x_2^3x_4,
    x_0x_1x_2x_3x_4,
    x_0^2x_4^3
  \right\rangle.
\]

\subsubsection*{Polystability}
Let \(v_1,\ldots,v_5\) be the five exponent vectors above, with
\(v_4=(1,1,1,1,1)\), and set
\[
  \delta_i:=v_i-v_4
  \qquad (i=1,2,3,5).
\]
Explicitly,
\[
\begin{aligned}
  \delta_1&=(-1,3,0,-1,-1),
  &\delta_2&=(0,-1,-1,3,-1),\\
  \delta_3&=(0,-1,2,-1,0),
  &\delta_5&=(1,-1,-1,-1,2).
\end{aligned}
\]
The vectors \(\delta_1,\delta_2,\delta_3\) are linearly independent and
\[
  \delta_1+\delta_2+\delta_3+\delta_5=0.
\]
Therefore
\[
  0\in
  \operatorname{relint}
  \Conv\{\delta_1,\delta_2,\delta_3,\delta_5\}.
\]
Since every coefficient of \(\phi_1\) is nonzero, the convex-hull criterion
shows that \((T/H_1)\cdot\phi_1\) is closed in \(W^{H_1}\).  By affine Luna
reduction and \(N_G(H_1)=T\), the orbit \(G\cdot\phi_1\) is closed in \(W\).
Thus \([\phi_1]\) is polystable.

\subsubsection*{Normal form}
On the full-support chart of \(\PP(W^{H_1})\), rescale projectively so that
\(a_4=1\), and put
\[
  b_i:=\frac{a_i}{a_4}
  \qquad (i=1,2,3,5).
\]
The relative characters \(\delta_1,\delta_2,\delta_3\) form an integral basis
of \(X^*(T/H_1)\).  Hence the character map
\[
  T/H_1\xrightarrow{\sim}(\C^\times)^3,
  \qquad
  [t]\longmapsto
  \bigl(\delta_1(t),\delta_2(t),\delta_3(t)\bigr),
\]
is an isomorphism.  Choose \(t\in T\) such that
\[
  \delta_i(t)=b_i^{-1}
  \qquad (i=1,2,3).
\]
This normalizes the first four coefficients to \(1\).  Since
\[
  \delta_5=-(\delta_1+\delta_2+\delta_3),
\]
the remaining coefficient is
\[
  \alpha
  =b_1b_2b_3b_5
  =\frac{a_1a_2a_3a_5}{a_4^4}
  \in\C^\times.
\]
Thus every general limit in Case~\(1\) is equivalent to the unique normal
form
\[
  \nf_1(\alpha)
  =x_1^4x_2
  +x_0x_3^4
  +x_0x_2^3x_4
  +x_0x_1x_2x_3x_4
  +\alpha x_0^2x_4^3,
  \qquad
  \alpha\in\C^\times.
\]
The parameter \(\alpha\) is the complete invariant of the full-support
\(T\)-orbits, so the corresponding quotient-side family is one-dimensional.
Thus
\[
  \dim \Phi_1=1.
\]
No further restriction on \(\alpha\) is imposed at the normal-form stage;
every member of this family is polystable.

\StartCase{2}

\subsubsection*{1-PS limit}
The witness vector and the corresponding one-parameter subgroup are
\[
  r_2=(27,2,-3,-8,-18),
  \qquad
  \lambda_2(t)=\diag(t^{27},t^2,t^{-3},t^{-8},t^{-18}),
  \quad t\in\Gm.
\]
Since the entries of \(r_2\) sum to zero, \(\lambda_2\) is a one-parameter
subgroup of \(G=\SL(5)\).  Let \(f_2\) be a general quintic supported on
\[
  S_2=I(r_2)_{\geq0}.
\]
With the action convention of \eqref{eq:quintic-monomial-weight}, its affine
one-parameter-subgroup limit is the weight-zero part
\[
  \phi_2:=\lim_{t\to0}\lambda_2(t)\cdot f_2.
\]
The degree-five weight-zero equations have precisely the exponent vectors
\[
\begin{gathered}
  (0,3,2,0,0),\quad
  (0,4,0,1,0),\quad
  (1,0,1,3,0),\\
  (1,0,3,0,1),\quad
  (1,1,1,1,1),\quad
  (2,0,0,0,3).
\end{gathered}
\]
Consequently,
\[
\begin{aligned}
  \phi_2={}&a_1x_1^3x_2^2+a_2x_1^4x_3
  +a_3x_0x_2x_3^3\\
  &+a_4x_0x_2^3x_4+a_5x_0x_1x_2x_3x_4
  +a_6x_0^2x_4^3,
  \qquad a_i\in\C^\times.
\end{aligned}
\]

Set \(H_2=\lambda_2(\Gm)\).  The five weights of \(H_2\) on
\(\langle x_0,\ldots,x_4\rangle\) are pairwise distinct, and their multiset
is not invariant under multiplication by \(-1\).  Hence
\[
  C_G(H_2)=N_G(H_2)=T,
\]
and
\[
  W^{H_2}
  =\left\langle
    x_1^3x_2^2,
    x_1^4x_3,
    x_0x_2x_3^3,
    x_0x_2^3x_4,
    x_0x_1x_2x_3x_4,
    x_0^2x_4^3
  \right\rangle.
\]

\subsubsection*{Polystability}
Let \(v_1,\ldots,v_6\) be the six exponent vectors above, with
\(v_5=(1,1,1,1,1)\), and set
\[
  \delta_i:=v_i-v_5
  \qquad (i=1,2,3,4,6).
\]
Explicitly,
\[
\begin{aligned}
  \delta_1&=(-1,2,1,-1,-1),
  &\delta_2&=(-1,3,-1,0,-1),\\
  \delta_3&=(0,-1,0,2,-1),
  &\delta_4&=(0,-1,2,-1,0),\\
  \delta_6&=(1,-1,-1,-1,2).
\end{aligned}
\]
These characters span the three-dimensional character space of \(T/H_2\)
and satisfy the positive relation
\[
  \delta_1+\delta_2+2\delta_3+\delta_4+2\delta_6=0.
\]
Therefore
\[
  0\in
  \operatorname{relint}
  \Conv\{\delta_1,\delta_2,\delta_3,\delta_4,\delta_6\}.
\]
Since every coefficient of \(\phi_2\) is nonzero, the convex-hull criterion
shows that \((T/H_2)\cdot\phi_2\) is closed in \(W^{H_2}\).  By affine Luna
reduction and \(N_G(H_2)=T\), the orbit \(G\cdot\phi_2\) is closed in \(W\).
Thus \([\phi_2]\) is polystable.

\subsubsection*{Normal form}
On the full-support chart of \(\PP(W^{H_2})\), use the fourth monomial as the
projective reference and set
\[
  \varepsilon_i:=v_i-v_4
  \qquad (i=1,2,3,5,6).
\]
A direct lattice calculation gives
\[
  X^*(T/H_2)
  =\mathbb Z\varepsilon_1
  \oplus\mathbb Z\varepsilon_2
  \oplus\mathbb Z\varepsilon_3.
\]
Hence
\[
  T/H_2\xrightarrow{\sim}(\C^\times)^3,
  \qquad
  [t]\longmapsto
  \bigl(\varepsilon_1(t),\varepsilon_2(t),\varepsilon_3(t)\bigr).
\]
Projectively rescale so that \(a_4=1\), put
\[
  b_i:=\frac{a_i}{a_4}
  \qquad (i=1,2,3,5,6),
\]
and choose \(t\in T\) such that
\[
  \varepsilon_i(t)=b_i^{-1}
  \qquad (i=1,2,3).
\]
This normalizes the first four coefficients to \(1\).  Since
\[
  \varepsilon_5=-\varepsilon_1+\varepsilon_2,
  \qquad
  \varepsilon_6=-4\varepsilon_1+3\varepsilon_2-\varepsilon_3,
\]
the two remaining coefficients are
\[
  \alpha
  =\frac{a_1a_5}{a_2a_4},
  \qquad
  \beta
  =\frac{a_1^4a_3a_6}{a_2^3a_4^3}.
\]
Thus every general limit in Case~\(2\) is equivalent to the unique normal
form
\[
\begin{aligned}
  \nf_2(\alpha,\beta)={}&x_1^3x_2^2+x_1^4x_3
  +x_0x_2x_3^3\\
  &+x_0x_2^3x_4+\alpha x_0x_1x_2x_3x_4
  +\beta x_0^2x_4^3,
  \qquad
  (\alpha,\beta)\in(\C^\times)^2.
\end{aligned}
\]
The pair \((\alpha,\beta)\) is the complete invariant of the full-support
\(T\)-orbits.  Since
\[
  \dim\PP(W^{H_2})^\circ=5,
  \qquad
  \dim(T/H_2)=3,
\]
the corresponding quotient-side family is two-dimensional.  Equivalently,
\[
  \dim\Phi_2=5-3=2.
\]
No additional condition on \((\alpha,\beta)\) is imposed at the normal-form
stage; every member of this family is polystable.

\StartCase{3}

\subsubsection*{1-PS limit}
The witness vector and the corresponding one-parameter subgroup are
\[
  r_3=(3,0,0,-1,-2),
  \qquad
  \lambda_3(t)=\diag(t^3,1,1,t^{-1},t^{-2}),
  \quad t\in\Gm.
\]
Since the entries of \(r_3\) sum to zero, \(\lambda_3\) is a one-parameter
subgroup of \(G=\SL(5)\).  Let \(f_3\) be a general quintic supported on
\[
  S_3=I(r_3)_{\geq0}.
\]
With the action convention of \eqref{eq:quintic-monomial-weight}, its affine
one-parameter-subgroup limit is the weight-zero part
\[
  \phi_3:=\lim_{t\to0}\lambda_3(t)\cdot f_3.
\]
Put \(U:=\langle x_1,x_2\rangle\).  Solving the degree-five weight-zero
equations gives
\[
  W^{H_3}
  =\operatorname{Sym}^5U
  \oplus x_0x_3^3U
  \oplus x_0x_3x_4\operatorname{Sym}^2U
  \oplus \C x_0^2x_4^3,
  \qquad
  \dim W^{H_3}=12,
\]
where \(H_3=\lambda_3(\Gm)\).  Hence a general limit has the form
\[
  \phi_3
  =B_5(x_1,x_2)
  +x_0x_3^3L_1(x_1,x_2)
  +x_0x_3x_4Q_2(x_1,x_2)
  +\kappa x_0^2x_4^3,
\]
where \(B_5\), \(L_1\), and \(Q_2\) are binary forms of degrees \(5\),
\(1\), and \(2\), respectively.  We work on the nonempty Zariski-open locus
\[
  \kappa\neq0,
  \qquad
  \operatorname{Disc}(B_5)\neq0,
  \qquad
  \operatorname{Disc}(Q_2)\neq0,
  \qquad
  \operatorname{Res}(L_1,Q_2)\neq0.
\]

The weights of \(H_3\) on \(\langle x_0,\ldots,x_4\rangle\) are
\(3,0,0,-1,-2\).  Therefore
\[
  C_G(H_3)
  =
  \left\{
    \diag(\tau_0)\oplus A\oplus\diag(\tau_3)\oplus\diag(\tau_4)
    \ \middle|\
    A\in\operatorname{GL}(2),\ 
    \tau_0\tau_3\tau_4\det A=1
  \right\}.
\]
In particular,
\[
  \dim C_G(H_3)=6,
\]
and \(H_3\) acts trivially on \(W^{H_3}\).

\subsubsection*{Polystability}
Set \(R:=C_G(H_3)\).  By affine Luna reduction, it is enough to prove that
\(R\cdot\phi_3\) is closed in \(W^{H_3}\).  Let
\(\lambda\in\Lambda_{\phi_3}^{R}\).  After conjugation in the
\(\operatorname{GL}(U)\)-block, choose an eigenbasis \(X,Y\) of \(U\) and write the
weights of \(\lambda\) as
\[
  (\alpha,s+n,s-n,\beta,\gamma),
  \qquad
  \sigma:=\alpha+2s+\beta+\gamma=0.
\]
Put
\[
  \ell:=\alpha+s+3\beta,
  \qquad
  \omega:=2\alpha+3\gamma.
\]
The nonvanishing of the relative invariants above, together with the
nonnegativity of all weights occurring in \(\phi_3\), gives
\[
  \operatorname{Disc}(B_5)\neq0\Longrightarrow s\geq0,
  \qquad
  \operatorname{Res}(L_1,Q_2)\neq0\Longrightarrow \ell\geq0,
  \qquad
  \kappa\neq0\Longrightarrow \omega\geq0.
\]
Since
\[
  5s+\ell+\omega=3\sigma=0,
\]
we obtain
\[
  s=\ell=\omega=0.
\]
The condition \(\operatorname{Disc}(B_5)\neq0\) then forces \(n=0\): if \(n\neq0\), the
nonnegative-weight condition makes \(B_5\) divisible by the cube of one of
the eigenvectors, contradicting square-freeness.  Consequently every
one-parameter subgroup in \(\Lambda_{\phi_3}^{R}\) has weights
\[
  (3m,0,0,-m,-2m)
  \qquad (m\in\mathbb Z),
\]
and hence
\[
  \Lambda_{\phi_3}^{R}
  =\{\lambda_3(t^m)\mid m\in\mathbb Z\}.
\]
This set is symmetric.  The Casimiro--Florentino criterion therefore shows
that \(\phi_3\) is polystable for the \(R\)-action.  By affine Luna
reduction, \(G\cdot\phi_3\) is closed in \(W\), and hence \([\phi_3]\) is
polystable.

\subsubsection*{Normal form}
Because \(Q_2\) is square-free and \(L_1\) is not a factor of \(Q_2\), the
divisor of \(L_1Q_2\) on \(\PP(U)\cong\PP^1\) consists of three distinct
points.  The \(\PGL(U)\)-part of \(C_G(H_3)\) sends these points to the roots
of
\[
  x_2(x_1^2+x_2^2).
\]
Thus we may take
\[
  L_1=x_2,
  \qquad
  Q_2=x_1^2+x_2^2.
\]
The remaining diagonal action and projective rescaling normalize the
coefficients of the other three blocks to \(1\).  Every general limit in
Case~\(3\) is therefore equivalent to a member of the normal-form family
\[
  \nf_3(B_5)
  =B_5(x_1,x_2)
  +x_0x_2x_3^3
  +x_0(x_1^2+x_2^2)x_3x_4
  +x_0^2x_4^3,
\]
where \(B_5\) ranges over a nonempty Zariski-open subset of the space of
binary quintics, in particular \(\operatorname{Disc}(B_5)\neq0\).  The residual stabilizer
is finite for a general \(B_5\).  Since
\[
  \dim W^{H_3}=12,
  \qquad
  \dim\bigl(C_G(H_3)/H_3\bigr)=5,
\]
the corresponding quotient-side family has dimension
\[
  \dim\Phi_3
  =12-1-5
  =6.
\]

\StartCase{4}

\subsubsection*{1-PS limit}
The witness vector and the corresponding one-parameter subgroup are
\[
  r_4=(24,9,-6,-11,-16),
  \qquad
  \lambda_4(t)=\diag(t^{24},t^9,t^{-6},t^{-11},t^{-16}),
  \quad t\in\Gm.
\]
Since the entries of \(r_4\) sum to zero, \(\lambda_4\) is a one-parameter
subgroup of \(G=\SL(5)\).  Let \(f_4\) be a general quintic supported on
\[
  S_4=I(r_4)_{\geq0}.
\]
With the action convention of \eqref{eq:quintic-monomial-weight}, its affine
one-parameter-subgroup limit is the weight-zero part
\[
  \phi_4:=\lim_{t\to0}\lambda_4(t)\cdot f_4.
\]
The degree-five weight-zero equations have precisely the exponent vectors
\[
\begin{gathered}
  (0,2,3,0,0),\quad
  (0,3,0,1,1),\quad
  (1,0,4,0,0),\\
  (1,1,0,3,0),\quad
  (1,1,1,1,1),\quad
  (2,0,0,0,3).
\end{gathered}
\]
Consequently,
\[
\begin{aligned}
  \phi_4={}&a_1x_1^2x_2^3+a_2x_1^3x_3x_4
  +a_3x_0x_2^4+a_4x_0x_1x_3^3\\
  &+a_5x_0x_1x_2x_3x_4+a_6x_0^2x_4^3,
  \qquad a_i\in\C^\times.
\end{aligned}
\]

Set \(H_4=\lambda_4(\Gm)\).  The five weights of \(H_4\) on
\(\langle x_0,\ldots,x_4\rangle\) are pairwise distinct, and their multiset
is not invariant under multiplication by \(-1\).  Hence
\[
  C_G(H_4)=N_G(H_4)=T,
\]
and
\[
  W^{H_4}
  =\left\langle
    x_1^2x_2^3,
    x_1^3x_3x_4,
    x_0x_2^4,
    x_0x_1x_3^3,
    x_0x_1x_2x_3x_4,
    x_0^2x_4^3
  \right\rangle,
  \qquad
  \dim W^{H_4}=6.
\]

\subsubsection*{Polystability}
Let \(v_1,\ldots,v_6\) be the six exponent vectors above, with
\(v_5=(1,1,1,1,1)\), and set
\[
  \delta_i:=v_i-v_5
  \qquad (i=1,2,3,4,6).
\]
Explicitly,
\[
\begin{aligned}
  \delta_1&=(-1,1,2,-1,-1),
  &\delta_2&=(-1,2,-1,0,0),\\
  \delta_3&=(0,-1,3,-1,-1),
  &\delta_4&=(0,0,-1,2,-1),\\
  \delta_6&=(1,-1,-1,-1,2).
\end{aligned}
\]
These characters span the three-dimensional character space of \(T/H_4\)
and satisfy the positive relation
\[
  \delta_1+\delta_2+\delta_3+2\delta_4+2\delta_6=0.
\]
Therefore
\[
  0\in
  \operatorname{relint}
  \Conv\{\delta_1,\delta_2,\delta_3,\delta_4,\delta_6\}.
\]
Since every coefficient of \(\phi_4\) is nonzero, the convex-hull criterion
shows that \((T/H_4)\cdot\phi_4\) is closed in \(W^{H_4}\).  By affine Luna
reduction and \(N_G(H_4)=T\), the orbit \(G\cdot\phi_4\) is closed in \(W\).
Thus \([\phi_4]\) is polystable.

\subsubsection*{Normal form}
On the full-support chart of \(\PP(W^{H_4})\), use the fourth monomial as the
projective reference and set
\[
  \varepsilon_i:=v_i-v_4
  \qquad (i=1,2,3,5,6).
\]
A direct lattice calculation shows that
\[
  X^*(T/H_4)
  =\mathbb Z\varepsilon_1
  \oplus\mathbb Z\varepsilon_2
  \oplus\mathbb Z\varepsilon_3,
\]
while
\[
  \varepsilon_5=-\varepsilon_1+\varepsilon_2+\varepsilon_3,
  \qquad
  \varepsilon_6=-4\varepsilon_1+3\varepsilon_2+3\varepsilon_3.
\]
Thus
\[
  T/H_4\xrightarrow{\sim}(\C^\times)^3,
  \qquad
  [t]\longmapsto
  \bigl(\varepsilon_1(t),\varepsilon_2(t),\varepsilon_3(t)\bigr).
\]
Projectively rescale so that \(a_4=1\), and use the effective torus action to
normalize \(a_1,a_2,a_3\) to \(1\).  The two remaining coefficients are
\[
  \alpha=\frac{a_1a_5}{a_2a_3},
  \qquad
  \beta=\frac{a_1^4a_4a_6}{a_2^3a_3^3}.
\]
Hence every general limit in Case~\(4\) is equivalent to the unique normal
form
\[
\begin{aligned}
  \nf_4(\alpha,\beta)={}&x_1^2x_2^3+x_1^3x_3x_4
  +x_0x_2^4+x_0x_1x_3^3\\
  &+\alpha x_0x_1x_2x_3x_4
  +\beta x_0^2x_4^3,
  \qquad
  (\alpha,\beta)\in(\C^\times)^2.
\end{aligned}
\]
The pair \((\alpha,\beta)\) is the complete invariant of the full-support
\(T\)-orbits.  Since
\[
  \dim\PP(W^{H_4})^\circ=5,
  \qquad
  \dim(T/H_4)=3,
\]
the corresponding quotient-side family has dimension
\[
  \dim\Phi_4=5-3=2.
\]
No additional condition on \((\alpha,\beta)\) is imposed at the normal-form
stage; every member of this family is polystable.

\StartCase{5}

\subsubsection*{1-PS limit}
The witness vector and the corresponding one-parameter subgroup are
\[
  r_5=(21,6,-4,-9,-14),
  \qquad
  \lambda_5(t)=\diag(t^{21},t^6,t^{-4},t^{-9},t^{-14}),
  \quad t\in\Gm.
\]
Since the entries of \(r_5\) sum to zero, \(\lambda_5\) is a one-parameter
subgroup of \(G=\SL(5)\).  Let \(f_5\) be a general quintic supported on
\[
  S_5=I(r_5)_{\geq0}.
\]
With the action convention of \eqref{eq:quintic-monomial-weight}, its affine
one-parameter-subgroup limit is the weight-zero part
\[
  \phi_5:=\lim_{t\to0}\lambda_5(t)\cdot f_5.
\]
The degree-five weight-zero equations have precisely the exponent vectors
\[
\begin{gathered}
  (0,2,3,0,0),\quad
  (0,3,0,2,0),\quad
  (0,3,1,0,1),\quad
  (1,0,3,1,0),\\
  (1,1,0,3,0),\quad
  (1,1,1,1,1),\quad
  (2,0,0,0,3).
\end{gathered}
\]
Consequently,
\[
\begin{aligned}
  \phi_5={}&a_1x_1^2x_2^3+a_2x_1^3x_3^2+a_3x_1^3x_2x_4
  +a_4x_0x_2^3x_3\\
  &+a_5x_0x_1x_3^3+a_6x_0x_1x_2x_3x_4
  +a_7x_0^2x_4^3,
  \qquad a_i\in\C^\times.
\end{aligned}
\]

Set \(H_5=\lambda_5(\Gm)\).  The five weights of \(H_5\) on
\(\langle x_0,\ldots,x_4\rangle\) are pairwise distinct, and their multiset
is not invariant under multiplication by \(-1\).  Hence
\[
  C_G(H_5)=N_G(H_5)=T,
\]
and
\[
\begin{aligned}
  W^{H_5}=\bigl\langle{}&x_1^2x_2^3,
  x_1^3x_3^2,
  x_1^3x_2x_4,
  x_0x_2^3x_3,\\
  &x_0x_1x_3^3,
  x_0x_1x_2x_3x_4,
  x_0^2x_4^3\bigr\rangle,
  \qquad
  \dim W^{H_5}=7.
\end{aligned}
\]

\subsubsection*{Polystability}
Let \(v_1,\ldots,v_7\) be the seven exponent vectors above, with
\(v_6=(1,1,1,1,1)\), and set
\[
  \delta_i:=v_i-v_6
  \qquad (i=1,\ldots,7).
\]
Thus \(\delta_6=0\), and the six nonzero characters span the
three-dimensional character space of \(T/H_5\).  They satisfy the positive
relation
\[
  \delta_1+\delta_2+\delta_3+2\delta_4+2\delta_5+3\delta_7=0.
\]
Therefore
\[
  0\in\operatorname{relint}
  \Conv\{\delta_1,\delta_2,\delta_3,\delta_4,\delta_5,\delta_7\}.
\]
Since every coefficient of \(\phi_5\) is nonzero, the convex-hull criterion
shows that \((T/H_5)\cdot\phi_5\) is closed in \(W^{H_5}\).  By affine Luna
reduction and \(N_G(H_5)=T\), the orbit \(G\cdot\phi_5\) is closed in \(W\).
Thus \([\phi_5]\) is polystable.

\subsubsection*{Normal form}
On the full-support chart of \(\PP(W^{H_5})\), use the fourth monomial as the
projective reference and put
\[
  \varepsilon_i:=v_i-v_4
  \qquad (i=1,2,3,5,6,7).
\]
A direct character-lattice calculation gives
\[
  X^*(T/H_5)
  =\mathbb Z\varepsilon_1
  \oplus\mathbb Z\varepsilon_2
  \oplus\mathbb Z\varepsilon_3,
\]
together with
\[
  \varepsilon_5=-\varepsilon_1+\varepsilon_2,
  \qquad
  \varepsilon_6=-\varepsilon_1+\varepsilon_3,
  \qquad
  \varepsilon_7=-3\varepsilon_1-\varepsilon_2+3\varepsilon_3.
\]
Projectively rescale so that \(a_4=1\), and use the effective torus action to
normalize \(a_1,a_2,a_3\) to \(1\).  The three remaining coefficients are
\[
  \alpha=\frac{a_1a_5}{a_2a_4},
  \qquad
  \beta=\frac{a_1a_6}{a_3a_4},
  \qquad
  \gamma=\frac{a_1^3a_2a_7}{a_3^3a_4^2}.
\]
Hence every general limit in Case~\(5\) is equivalent to the unique normal
form
\[
\begin{aligned}
  \nf_5(\alpha,\beta,\gamma)={}&x_1^2x_2^3+x_1^3x_3^2
  +x_1^3x_2x_4+x_0x_2^3x_3\\
  &+\alpha x_0x_1x_3^3
  +\beta x_0x_1x_2x_3x_4
  +\gamma x_0^2x_4^3,
\end{aligned}
\]
where
\[
  (\alpha,\beta,\gamma)\in(\C^\times)^3.
\]
The triple \((\alpha,\beta,\gamma)\) is the complete invariant of the
full-support \(T\)-orbits.  Since
\[
  \dim\PP(W^{H_5})^\circ=6,
  \qquad
  \dim(T/H_5)=3,
\]
the corresponding quotient-side family has dimension
\[
  \dim\Phi_5=6-3=3.
\]
No additional condition on \((\alpha,\beta,\gamma)\) is imposed at the
normal-form stage; every member of this family is polystable.

\StartCase{6}

\subsubsection*{1-PS limit}
The witness vector and the corresponding one-parameter subgroup are
\[
  r_6=(18,3,-2,-7,-12),
  \qquad
  \lambda_6(t)=\diag(t^{18},t^3,t^{-2},t^{-7},t^{-12}),
  \quad t\in\Gm.
\]
Since the entries of \(r_6\) sum to zero, \(\lambda_6\) is a one-parameter
subgroup of \(G=\SL(5)\).  Let \(f_6\) be a general quintic supported on
\[
  S_6=I(r_6)_{\geq0}.
\]
With the action convention of \eqref{eq:quintic-monomial-weight}, its affine
one-parameter-subgroup limit is the weight-zero part
\[
  \phi_6:=\lim_{t\to0}\lambda_6(t)\cdot f_6.
\]
The degree-five weight-zero equations have precisely the exponent vectors
\[
\begin{gathered}
  (0,2,3,0,0),\quad
  (0,3,1,1,0),\quad
  (0,4,0,0,1),\quad
  (1,0,2,2,0),\quad
  (1,0,3,0,1),\\
  (1,1,0,3,0),\quad
  (1,1,1,1,1),\quad
  (1,2,0,0,2),\quad
  (2,0,0,0,3).
\end{gathered}
\]
Consequently,
\[
\begin{aligned}
  \phi_6={}&a_1x_1^2x_2^3+a_2x_1^3x_2x_3+a_3x_1^4x_4
  +a_4x_0x_2^2x_3^2+a_5x_0x_2^3x_4\\
  &+a_6x_0x_1x_3^3+a_7x_0x_1x_2x_3x_4
  +a_8x_0x_1^2x_4^2+a_9x_0^2x_4^3,
  \qquad a_i\in\C^\times.
\end{aligned}
\]

Set \(H_6=\lambda_6(\Gm)\).  The five weights of \(H_6\) on
\(\langle x_0,\ldots,x_4\rangle\) are pairwise distinct, and their multiset
is not invariant under multiplication by \(-1\).  Hence
\[
  C_G(H_6)=N_G(H_6)=T.
\]
Moreover,
\[
\begin{aligned}
  W^{H_6}=\bigl\langle{}&x_1^2x_2^3,
  x_1^3x_2x_3,
  x_1^4x_4,
  x_0x_2^2x_3^2,
  x_0x_2^3x_4,\\
  &x_0x_1x_3^3,
  x_0x_1x_2x_3x_4,
  x_0x_1^2x_4^2,
  x_0^2x_4^3\bigr\rangle,
  \qquad
  \dim W^{H_6}=9.
\end{aligned}
\]

\subsubsection*{Polystability}
Let \(v_1,\ldots,v_9\) be the nine exponent vectors above, with
\(v_7=(1,1,1,1,1)\), and set
\[
  \delta_i:=v_i-v_7
  \qquad (i=1,\ldots,9).
\]
Thus \(\delta_7=0\).  The eight nonzero characters span the
three-dimensional character space of \(T/H_6\) and satisfy the positive
relation
\[
  \delta_1+\delta_2+\delta_3+3\delta_4+\delta_5
  +2\delta_6+\delta_8+3\delta_9=0.
\]
Therefore
\[
  0\in\operatorname{relint}
  \Conv\{\delta_1,\delta_2,\delta_3,\delta_4,
          \delta_5,\delta_6,\delta_8,\delta_9\}.
\]
Since every coefficient of \(\phi_6\) is nonzero, the convex-hull criterion
shows that \((T/H_6)\cdot\phi_6\) is closed in \(W^{H_6}\).  By affine Luna
reduction and \(C_G(H_6)=T\), the orbit \(G\cdot\phi_6\) is closed in \(W\).
Thus \([\phi_6]\) is polystable.

\subsubsection*{Normal form}
On the full-support chart
\[
  \PP(W^{H_6})^\circ
  =\{[a_1:\cdots:a_9]\mid a_i\neq0\}
  \cong(\C^\times)^8,
\]
the effective torus \(T/H_6\) has dimension \(3\).  Combining its action
with overall projective rescaling, one may normalize the first four
coefficients to \(1\).  The affine relations
\[
\begin{aligned}
  v_5&=v_1-2v_2+v_3+v_4,
  &v_6&=-v_1+v_2+v_4,\\
  v_7&=-v_2+v_3+v_4,
  &v_8&=-2v_2+2v_3+v_4,\\
  v_9&=-4v_2+3v_3+2v_4
\end{aligned}
\]
give the five remaining quotient parameters
\[
  \alpha=\frac{a_2^2a_5}{a_1a_3a_4},
  \qquad
  \beta=\frac{a_1a_6}{a_2a_4},
  \qquad
  \gamma=\frac{a_2a_7}{a_3a_4},
\]
\[
  \delta=\frac{a_2^2a_8}{a_3^2a_4},
  \qquad
  \epsilon=\frac{a_2^4a_9}{a_3^3a_4^2}.
\]
Hence every general limit in Case~\(6\) is equivalent to the unique normal
form
\[
\begin{aligned}
  \nf_6(\alpha,\beta,\gamma,\delta,\epsilon)={}&
  x_1^2x_2^3+x_1^3x_2x_3+x_1^4x_4+x_0x_2^2x_3^2
  +\alpha x_0x_2^3x_4\\
  &+\beta x_0x_1x_3^3
  +\gamma x_0x_1x_2x_3x_4
  +\delta x_0x_1^2x_4^2
  +\epsilon x_0^2x_4^3,
\end{aligned}
\]
where
\[
  (\alpha,\beta,\gamma,\delta,\epsilon)\in(\C^\times)^5.
\]
These five parameters form a complete invariant of the full-support
\(T\)-orbits.  Since
\[
  \dim\PP(W^{H_6})^\circ=8,
  \qquad
  \dim(T/H_6)=3,
\]
and Luna reduction induces a finite morphism to the ambient GIT quotient,
the corresponding quotient-side family has dimension
\[
  \dim\Phi_6=8-3=5.
\]
No additional condition on
\((\alpha,\beta,\gamma,\delta,\epsilon)\) is imposed at the normal-form
stage; every member of this family is polystable.

\StartCase{7}

\subsubsection*{1-PS limit}
The witness vector and the corresponding one-parameter subgroup are
\[
  r_7=(5,1,0,-2,-4),
  \qquad
  \lambda_7(t)=\diag(t^5,t,1,t^{-2},t^{-4}),
  \quad t\in\Gm.
\]
Since the entries of \(r_7\) sum to zero, \(\lambda_7\) is a one-parameter
subgroup of \(G=\SL(5)\).  Let \(f_7\) be a general quintic supported on
\[
  S_7=I(r_7)_{\geq0}.
\]
With the action convention of \eqref{eq:quintic-monomial-weight}, its affine
one-parameter-subgroup limit is the weight-zero part
\[
  \phi_7:=\lim_{t\to0}\lambda_7(t)\cdot f_7.
\]
The degree-five weight-zero equations have precisely the exponent vectors
\[
  (0,0,5,0,0),\quad
  (0,2,2,1,0),\quad
  (0,4,0,0,1),\quad
  (1,1,0,3,0),\quad
  (1,1,1,1,1),\quad
  (2,0,0,1,2).
\]
Consequently,
\[
\begin{aligned}
  \phi_7={}&a_1x_2^5+a_2x_1^2x_2^2x_3+a_3x_1^4x_4
  +a_4x_0x_1x_3^3\\
  &+a_5x_0x_1x_2x_3x_4+a_6x_0^2x_3x_4^2,
  \qquad a_i\in\C^\times.
\end{aligned}
\]

Set \(H_7=\lambda_7(\Gm)\).  The five weights of \(H_7\) on
\(\langle x_0,\ldots,x_4\rangle\) are pairwise distinct, and their multiset
is not invariant under multiplication by \(-1\).  Hence
\[
  C_G(H_7)=N_G(H_7)=T.
\]
Moreover,
\[
\begin{aligned}
  W^{H_7}=\bigl\langle{}&x_2^5,
  x_1^2x_2^2x_3,
  x_1^4x_4,
  x_0x_1x_3^3,\\
  &x_0x_1x_2x_3x_4,
  x_0^2x_3x_4^2\bigr\rangle,
  \qquad
  \dim W^{H_7}=6.
\end{aligned}
\]

\subsubsection*{Polystability}
Let \(v_1,\ldots,v_6\) be the six exponent vectors above, with
\(v_5=(1,1,1,1,1)\), and set
\[
  \delta_i:=v_i-v_5
  \qquad (i=1,\ldots,6).
\]
Thus \(\delta_5=0\).  The five nonzero characters span the
three-dimensional character space of \(T/H_7\) and satisfy the positive
relation
\[
  \delta_1+2\delta_2+\delta_3+\delta_4+4\delta_6=0.
\]
Therefore
\[
  0\in\operatorname{relint}
  \Conv\{\delta_1,\delta_2,\delta_3,\delta_4,\delta_6\}.
\]
Since every coefficient of \(\phi_7\) is nonzero, the convex-hull criterion
shows that \((T/H_7)\cdot\phi_7\) is closed in \(W^{H_7}\).  By affine Luna
reduction and \(C_G(H_7)=T\), the orbit \(G\cdot\phi_7\) is closed in \(W\).
Thus \([\phi_7]\) is polystable.

\subsubsection*{Normal form}
On the full-support chart
\[
  \PP(W^{H_7})^\circ
  =\{[a_1:\cdots:a_6]\mid a_i\neq0\}
  \cong(\C^\times)^5,
\]
the effective torus \(T/H_7\) has dimension \(3\).  Combining its action
with overall projective rescaling, one may normalize the first four
coefficients to \(1\).  The affine relations
\[
  v_5=v_1-2v_2+v_3+v_4,
  \qquad
  v_6=2v_1-5v_2+2v_3+2v_4
\]
give the two quotient parameters
\[
  \alpha=\frac{a_2^2a_5}{a_1a_3a_4},
  \qquad
  \beta=\frac{a_2^5a_6}{a_1^2a_3^2a_4^2}.
\]
Hence every general limit in Case~\(7\) is equivalent to a normal form
\[
\begin{aligned}
  \nf_7(\alpha,\beta)={}&x_2^5+x_1^2x_2^2x_3+x_1^4x_4
  +x_0x_1x_3^3\\
  &+\alpha x_0x_1x_2x_3x_4
  +\beta x_0^2x_3x_4^2,
\end{aligned}
\]
where
\[
  (\alpha,\beta)\in(\C^\times)^2.
\]
The parameters \(\alpha\) and \(\beta\) form coordinates on the
full-support torus quotient.  Since
\[
  \dim\PP(W^{H_7})^\circ=5,
  \qquad
  \dim(T/H_7)=3,
\]
the corresponding quotient-side family has dimension
\[
  \dim\Phi_7=5-3=2.
\]
No additional condition on \((\alpha,\beta)\) is imposed at the normal-form
stage; every member of this family is polystable.

\StartCase{8}

\subsubsection*{1-PS limit}
The witness vector and the corresponding one-parameter subgroup are
\[
  r_8=(4,2,-1,-2,-3),
  \qquad
  \lambda_8(t)=\diag(t^4,t^2,t^{-1},t^{-2},t^{-3}),
  \quad t\in\Gm.
\]
Since the entries of \(r_8\) sum to zero, \(\lambda_8\) is a one-parameter
subgroup of \(G=\SL(5)\).  Let \(f_8\) be a general quintic supported on
\[
  S_8=I(r_8)_{\geq0}.
\]
With the action convention of \eqref{eq:quintic-monomial-weight}, its affine
one-parameter-subgroup limit is the weight-zero part
\[
  \phi_8:=\lim_{t\to0}\lambda_8(t)\cdot f_8.
\]
The degree-five weight-zero equations have precisely the exponent vectors
\[
  (0,2,2,1,0),\quad
  (0,3,0,0,2),\quad
  (1,0,4,0,0),\quad
  (1,1,0,3,0),\quad
  (1,1,1,1,1),\quad
  (2,0,0,1,2).
\]
Consequently,
\[
\begin{aligned}
  \phi_8={}&a_1x_1^2x_2^2x_3+a_2x_1^3x_4^2+a_3x_0x_2^4
  +a_4x_0x_1x_3^3\\
  &+a_5x_0x_1x_2x_3x_4+a_6x_0^2x_3x_4^2,
  \qquad a_i\in\C^\times.
\end{aligned}
\]

Set \(H_8=\lambda_8(\Gm)\).  The five weights of \(H_8\) on
\(\langle x_0,\ldots,x_4\rangle\) are pairwise distinct, and their multiset
is not invariant under multiplication by \(-1\).  Hence
\[
  C_G(H_8)=N_G(H_8)=T.
\]
Moreover,
\[
\begin{aligned}
  W^{H_8}=\bigl\langle{}&x_1^2x_2^2x_3,
  x_1^3x_4^2,
  x_0x_2^4,
  x_0x_1x_3^3,\\
  &x_0x_1x_2x_3x_4,
  x_0^2x_3x_4^2\bigr\rangle,
  \qquad
  \dim W^{H_8}=6.
\end{aligned}
\]

\subsubsection*{Polystability}
Let \(v_1,\ldots,v_6\) be the six exponent vectors above, with
\(v_5=(1,1,1,1,1)\), and set
\[
  \delta_i:=v_i-v_5
  \qquad (i=1,\ldots,6).
\]
Thus \(\delta_5=0\).  The five nonzero characters span the
three-dimensional character space of \(T/H_8\) and satisfy the positive
relation
\[
  \delta_1+\delta_2+\delta_3+\delta_4+2\delta_6=0.
\]
Therefore
\[
  0\in\operatorname{relint}
  \Conv\{\delta_1,\delta_2,\delta_3,\delta_4,\delta_6\}.
\]
Since every coefficient of \(\phi_8\) is nonzero, the convex-hull criterion
shows that \((T/H_8)\cdot\phi_8\) is closed in \(W^{H_8}\).  By affine Luna
reduction and \(C_G(H_8)=T\), the orbit \(G\cdot\phi_8\) is closed in \(W\).
Thus \([\phi_8]\) is polystable.

\subsubsection*{Normal form}
On the full-support chart
\[
  \PP(W^{H_8})^\circ
  =\{[a_1:\cdots:a_6]\mid a_i\neq0\}
  \cong(\C^\times)^5,
\]
the effective torus \(T/H_8\) has dimension \(3\).  Combining its action
with overall projective rescaling, one may normalize the first four
coefficients to \(1\).  The affine relations
\[
  2v_5=-v_1+v_2+v_3+v_4,
  \qquad
  v_6=-2v_1+v_2+v_3+v_4
\]
show that the remaining coefficients may be parametrized by
\[
  \alpha^2=\frac{a_1a_5^2}{a_2a_3a_4},
  \qquad
  \beta=\frac{a_1^2a_6}{a_2a_3a_4}.
\]
After choosing \(\alpha\), every general limit in Case~\(8\) is equivalent
to a normal form
\[
\begin{aligned}
  \nf_8(\alpha,\beta)={}&x_1^2x_2^2x_3+x_1^3x_4^2+x_0x_2^4
  +x_0x_1x_3^3\\
  &+\alpha x_0x_1x_2x_3x_4
  +\beta x_0^2x_3x_4^2,
\end{aligned}
\]
where
\[
  (\alpha,\beta)\in(\C^\times)^2.
\]
Since
\[
  \dim\PP(W^{H_8})^\circ=5,
  \qquad
  \dim(T/H_8)=3,
\]
the corresponding quotient-side family has dimension
\[
  \dim\Phi_8=5-3=2.
\]
No additional condition on \((\alpha,\beta)\) is imposed at the normal-form
stage; every member of this family is polystable.

\StartCase{9}

\subsubsection*{1-PS limit}
The witness vector and the corresponding one-parameter subgroup are
\[
  r_9=(19,4,-1,-6,-16),
  \qquad
  \lambda_9(t)=\diag(t^{19},t^4,t^{-1},t^{-6},t^{-16}),
  \quad t\in\Gm.
\]
Since the entries of \(r_9\) sum to zero, \(\lambda_9\) is a one-parameter
subgroup of \(G=\SL(5)\).  Let \(f_9\) be a general quintic supported on
\[
  S_9=I(r_9)_{\geq0}.
\]
With the action convention of \eqref{eq:quintic-monomial-weight}, its affine
one-parameter-subgroup limit is the weight-zero part
\[
  \phi_9:=\lim_{t\to0}\lambda_9(t)\cdot f_9.
\]
The degree-five weight-zero equations have precisely the exponent vectors
\[
\begin{gathered}
  (0,1,4,0,0),\quad
  (0,2,2,1,0),\quad
  (0,3,0,2,0),\quad
  (0,4,0,0,1),\\
  (1,0,1,3,0),\quad
  (1,0,3,0,1),\quad
  (1,1,1,1,1),\quad
  (2,0,0,1,2).
\end{gathered}
\]
Consequently,
\[
\begin{aligned}
  \phi_9={}&a_1x_1x_2^4+a_2x_1^2x_2^2x_3+a_3x_1^3x_3^2
  +a_4x_1^4x_4\\
  &+a_5x_0x_2x_3^3+a_6x_0x_2^3x_4
  +a_7x_0x_1x_2x_3x_4+a_8x_0^2x_3x_4^2,
  \qquad a_i\in\C^\times.
\end{aligned}
\]

Set \(H_9=\lambda_9(\Gm)\).  The five weights of \(H_9\) on
\(\langle x_0,\ldots,x_4\rangle\) are pairwise distinct, and their multiset
is not invariant under multiplication by \(-1\).  Hence
\[
  C_G(H_9)=N_G(H_9)=T.
\]
Moreover,
\[
\begin{aligned}
  W^{H_9}=\bigl\langle{}&x_1x_2^4,
  x_1^2x_2^2x_3,
  x_1^3x_3^2,
  x_1^4x_4,\\
  &x_0x_2x_3^3,
  x_0x_2^3x_4,
  x_0x_1x_2x_3x_4,
  x_0^2x_3x_4^2\bigr\rangle,
  \qquad
  \dim W^{H_9}=8.
\end{aligned}
\]

\subsubsection*{Polystability}
Let \(v_1,\ldots,v_8\) be the eight exponent vectors above, with
\(v_7=(1,1,1,1,1)\), and set
\[
  \delta_i:=v_i-v_7
  \qquad (i=1,\ldots,8).
\]
Thus \(\delta_7=0\).  The seven nonzero characters span the
three-dimensional character space of \(T/H_9\) and satisfy the positive
relation
\[
  \delta_1+\delta_2+\delta_3+\delta_4
  +\delta_5+\delta_6+4\delta_8=0.
\]
Therefore
\[
  0\in\operatorname{relint}
  \Conv\{\delta_1,\delta_2,\delta_3,\delta_4,
          \delta_5,\delta_6,\delta_8\}.
\]
Since every coefficient of \(\phi_9\) is nonzero, the convex-hull criterion
shows that \((T/H_9)\cdot\phi_9\) is closed in \(W^{H_9}\).  By affine Luna
reduction and \(C_G(H_9)=T\), the orbit \(G\cdot\phi_9\) is closed in \(W\).
Thus \([\phi_9]\) is polystable.

\subsubsection*{Normal form}
On the full-support chart
\[
  \PP(W^{H_9})^\circ
  =\{[a_1:\cdots:a_8]\mid a_i\neq0\}
  \cong(\C^\times)^7,
\]
the effective torus \(T/H_9\) has dimension \(3\).  The character
differences
\[
\begin{aligned}
  v_2-v_1&=(0,1,-2,1,0),\\
  v_4-v_1&=(0,3,-4,0,1),\\
  v_5-v_1&=(1,-1,-3,3,0)
\end{aligned}
\]
have rank \(3\) in the character space of \(T/H_9\).  Combining the torus
action with overall projective rescaling, one may therefore normalize the
coefficients of
\[
  x_1x_2^4,\qquad
  x_1^2x_2^2x_3,\qquad
  x_1^4x_4,\qquad
  x_0x_2x_3^3
\]
to \(1\).  Every general limit in Case~\(9\) is consequently equivalent to
\[
\begin{aligned}
  \nf_9(\alpha,\beta,\gamma,\delta)={}&x_1x_2^4
  +x_1^2x_2^2x_3+\alpha x_1^3x_3^2+x_1^4x_4
  +x_0x_2x_3^3\\
  &+\beta x_0x_2^3x_4
  +\gamma x_0x_1x_2x_3x_4
  +\delta x_0^2x_3x_4^2,
\end{aligned}
\]
where
\[
  (\alpha,\beta,\gamma,\delta)\in(\C^\times)^4.
\]
Since
\[
  \dim\PP(W^{H_9})^\circ=7,
  \qquad
  \dim(T/H_9)=3,
\]
the corresponding quotient-side family has dimension
\[
  \dim\Phi_9=7-3=4.
\]
No additional condition on \((\alpha,\beta,\gamma,\delta)\) is imposed at
the normal-form stage; every member of this family is polystable.

\StartCase{10}

\subsubsection*{1-PS limit}
The witness vector and the corresponding one-parameter subgroup are
\[
  r_{10}=(12,2,-3,-3,-8),
  \qquad
  \lambda_{10}(t)=\diag(t^{12},t^2,t^{-3},t^{-3},t^{-8}),
  \quad t\in\Gm.
\]
Since the entries of \(r_{10}\) sum to zero, \(\lambda_{10}\) is a
one-parameter subgroup of \(G=\SL(5)\).  Let \(f_{10}\) be a general
quintic supported on
\[
  S_{10}=I(r_{10})_{\geq0}.
\]
With the action convention of \eqref{eq:quintic-monomial-weight}, its affine
one-parameter-subgroup limit is the weight-zero part
\[
  \phi_{10}:=\lim_{t\to0}\lambda_{10}(t)\cdot f_{10}.
\]
Put
\[
  U:=\langle x_2,x_3\rangle.
\]
Solving the degree-five weight-zero equations gives
\[
\begin{aligned}
  W^{H_{10}}={}&x_1^3\operatorname{Sym}^2U
  \oplus \C x_1^4x_4
  \oplus x_0\operatorname{Sym}^4U
  \oplus x_0x_1x_4\operatorname{Sym}^2U\\
  &\oplus \C x_0x_1^2x_4^2
  \oplus \C x_0^2x_4^3,
  \qquad
  \dim W^{H_{10}}=14,
\end{aligned}
\]
where \(H_{10}=\lambda_{10}(\Gm)\).  Hence a general limit has the form
\[
\begin{aligned}
  \phi_{10}={}&x_1^3Q_2(x_2,x_3)
  +\kappa x_1^4x_4
  +x_0B_4(x_2,x_3)\\
  &+x_0x_1x_4R_2(x_2,x_3)
  +\beta x_0x_1^2x_4^2
  +\epsilon x_0^2x_4^3,
\end{aligned}
\]
where \(Q_2,R_2\) are binary quadratics and \(B_4\) is a binary quartic.
For the closed-orbit construction we work on the nonempty Zariski-open locus
\[
  \kappa\epsilon\neq0,
  \qquad
  \operatorname{Disc}(B_4)\neq0,
  \qquad
  \operatorname{Res}(B_4,Q_2)\neq0.
\]
No nonvanishing condition on \(\beta\) is required.

The weight \(-3\) of \(H_{10}\) has multiplicity two on \(U\).  Therefore
\[
  C_G(H_{10})
  =
  \left\{
    \diag(\mu_0,\mu_1)\oplus A\oplus\diag(\mu_4)
    \ \middle|\
    A\in\operatorname{GL}(U),\ 
    \mu_0\mu_1\mu_4\det A=1
  \right\}.
\]
In particular,
\[
  \dim C_G(H_{10})=6,
  \qquad
  \dim\bigl(C_G(H_{10})/H_{10}\bigr)=5.
\]

\subsubsection*{Polystability}
Set \(R:=C_G(H_{10})\).  By affine Luna reduction, it is enough to prove
that \(R\cdot\phi_{10}\) is closed in \(W^{H_{10}}\).  Let
\(\rho\in\Lambda_{\phi_{10}}^R\).  After conjugation in the
\(\operatorname{GL}(U)\)-block, choose an eigenbasis \(X,Y\) of \(U\) and
write the weights of \(\rho\) as
\[
  (a,b,s+n,s-n,c),
  \qquad
  a+b+2s+c=0.
\]
Define
\[
  p_B:=a+4s,
  \qquad
  p_Q:=3b+2s,
  \qquad
  p_\kappa:=4b+c,
  \qquad
  p_\epsilon:=2a+3c.
\]
The nonnegative-weight condition, together with the indicated relative
invariants, gives
\[
\begin{aligned}
  \operatorname{Disc}(B_4)\neq0&\Longrightarrow p_B\geq0,\\
  \operatorname{Res}(B_4,Q_2)\neq0
    &\Longrightarrow p_B+2p_Q\geq0,\\
  \kappa\neq0&\Longrightarrow p_\kappa\geq0,\\
  \epsilon\neq0&\Longrightarrow p_\epsilon\geq0.
\end{aligned}
\]
These quantities satisfy the balancing identity
\[
  5p_B+(p_B+2p_Q)+2p_\kappa+4p_\epsilon
  =14(a+b+2s+c)=0.
\]
Hence
\[
  p_B=p_Q=p_\kappa=p_\epsilon=0.
\]
Since \(p_B=0\), the weights of the monomials in \(B_4(X,Y)\) are
\((4-2j)n\), \(j=0,\ldots,4\).  If \(n\neq0\), the nonnegative-weight
condition forces \(B_4\) to have a multiple root, contradicting
\(\operatorname{Disc}(B_4)\neq0\).  Thus \(n=0\).

Solving the remaining equalities gives
\[
  (a,b,s,s,c)=m(12,2,-3,-3,-8)
  \qquad (m\in\mathbb Z).
\]
Consequently,
\[
  \Lambda_{\phi_{10}}^R
  =\{\lambda_{10}(t^m)\mid m\in\mathbb Z\}.
\]
This rank-one set is symmetric.  The Casimiro--Florentino criterion therefore
shows that \(R\cdot\phi_{10}\) is closed in \(W^{H_{10}}\).  By affine Luna
reduction, \(G\cdot\phi_{10}\) is closed in \(W\), and hence
\([\phi_{10}]\) is polystable.

\subsubsection*{Normal form}
Because \(B_4\) is square-free, the \(\operatorname{GL}(U)\)-action puts it
in the form
\[
  B_\lambda(u,v):=uv(u-v)(u-\lambda v),
  \qquad
  \lambda\in\C\setminus\{0,1\}.
\]
Write
\[
  Q_{\mathbf q}(u,v)=q_0u^2+q_1uv+q_2v^2,
  \qquad
  R_{\mathbf r}(u,v)=r_0u^2+r_1uv+r_2v^2.
\]
Using the remaining diagonal factors in \(C_G(H_{10})\), together with
overall projective rescaling, we normalize the coefficients of
\[
  x_0B_\lambda(x_2,x_3),
  \qquad
  x_1^4x_4,
  \qquad
  x_0^2x_4^3
\]
to \(1\).  Every general limit in Case~\(10\) is therefore equivalent to a
member of the normal-form family
\[
\begin{aligned}
  \nf_{10}(\lambda,\mathbf q,\mathbf r,\beta)
  ={}&x_1^3Q_{\mathbf q}(x_2,x_3)
  +x_1^4x_4
  +x_0B_\lambda(x_2,x_3)\\
  &+x_0x_1x_4R_{\mathbf r}(x_2,x_3)
  +\beta x_0x_1^2x_4^2
  +x_0^2x_4^3,
\end{aligned}
\]
where
\[
  \lambda\in\C\setminus\{0,1\},
  \qquad
  (\mathbf q,\mathbf r,\beta)\in\C^3\times\C^3\times\C,
  \qquad
  \operatorname{Res}(B_\lambda,Q_{\mathbf q})\neq0.
\]
The residual stabilizer is finite for a general parameter point.  Therefore
\[
  \dim\Phi_{10}
  =\dim\PP(W^{H_{10}})-\dim\bigl(C_G(H_{10})/H_{10}\bigr)
  =13-5
  =8.
\]

\StartCase{11}

\subsubsection*{1-PS limit}
The witness vector and the corresponding one-parameter subgroup are
\[
  r_{11}=(14,4,-1,-6,-11),
  \qquad
  \lambda_{11}(t)=\diag(t^{14},t^4,t^{-1},t^{-6},t^{-11}),
  \quad t\in\Gm.
\]
Since the entries of \(r_{11}\) sum to zero, \(\lambda_{11}\) is a
one-parameter subgroup of \(G=\SL(5)\).  Let \(f_{11}\) be a general
quintic supported on
\[
  S_{11}=I(r_{11})_{\geq0}.
\]
With the action convention of \eqref{eq:quintic-monomial-weight}, its affine
one-parameter-subgroup limit is the weight-zero part
\[
  \phi_{11}:=\lim_{t\to0}\lambda_{11}(t)\cdot f_{11}.
\]
The degree-five weight-zero equations have precisely the exponent vectors
\[
\begin{gathered}
  (0,1,4,0,0),\quad
  (0,2,2,1,0),\quad
  (0,3,0,2,0),\quad
  (0,3,1,0,1),\quad
  (1,0,2,2,0),\\
  (1,0,3,0,1),\quad
  (1,1,0,3,0),\quad
  (1,1,1,1,1),\quad
  (1,2,0,0,2),\quad
  (2,0,0,1,2).
\end{gathered}
\]
Consequently,
\[
\begin{aligned}
  \phi_{11}={}&a_1x_1x_2^4
  +a_2x_1^2x_2^2x_3
  +a_3x_1^3x_3^2
  +a_4x_1^3x_2x_4
  +a_5x_0x_2^2x_3^2\\
  &+a_6x_0x_2^3x_4
  +a_7x_0x_1x_3^3
  +a_8x_0x_1x_2x_3x_4
  +a_9x_0x_1^2x_4^2
  +a_{10}x_0^2x_3x_4^2,
\end{aligned}
\]
where \(a_i\in\C^\times\) for all \(i\).

Set \(H_{11}=\lambda_{11}(\Gm)\).  The five weights of \(H_{11}\) on
\(\langle x_0,\ldots,x_4\rangle\) are pairwise distinct, and their multiset
is not invariant under multiplication by \(-1\).  Hence
\[
  C_G(H_{11})=N_G(H_{11})=T.
\]
Moreover,
\[
\begin{aligned}
  W^{H_{11}}=\bigl\langle{}&x_1x_2^4,
  x_1^2x_2^2x_3,
  x_1^3x_3^2,
  x_1^3x_2x_4,
  x_0x_2^2x_3^2,\\
  &x_0x_2^3x_4,
  x_0x_1x_3^3,
  x_0x_1x_2x_3x_4,
  x_0x_1^2x_4^2,
  x_0^2x_3x_4^2\bigr\rangle,
  \qquad
  \dim W^{H_{11}}=10.
\end{aligned}
\]

\subsubsection*{Polystability}
Let \(v_1,\ldots,v_{10}\) be the ten exponent vectors above, with
\(v_8=(1,1,1,1,1)\), and set
\[
  \delta_i:=v_i-v_8
  \qquad (i=1,\ldots,10).
\]
Thus \(\delta_8=0\), and
\[
  \delta_2=\frac12\delta_1+\frac12\delta_3.
\]
The eight characters
\[
  \delta_1,\ \delta_3,\ \delta_4,\ \delta_5,
  \ \delta_6,\ \delta_7,\ \delta_9,\ \delta_{10}
\]
span the three-dimensional character space of \(T/H_{11}\) and satisfy the
positive relation
\[
  \delta_1+\delta_3+\delta_4+\delta_5+\delta_6+\delta_7
  +\delta_9+3\delta_{10}=0.
\]
Therefore
\[
  0\in\operatorname{relint}
  \Conv\{\delta_1,\delta_3,\delta_4,\delta_5,
          \delta_6,\delta_7,\delta_9,\delta_{10}\}.
\]
Adding \(\delta_2\) and \(\delta_8\) does not change this conclusion.  Since
all ten coefficients of \(\phi_{11}\) are nonzero, the convex-hull criterion
shows that \((T/H_{11})\cdot\phi_{11}\) is closed in \(W^{H_{11}}\).  By
affine Luna reduction and \(C_G(H_{11})=T\), the orbit
\(G\cdot\phi_{11}\) is closed in \(W\).  Thus \([\phi_{11}]\) is
polystable.

\subsubsection*{Normal form}
On the full-support chart
\[
  \PP(W^{H_{11}})^\circ
  =\{[a_1:\cdots:a_{10}]\mid a_i\neq0\}
  \cong(\C^\times)^9,
\]
the effective torus \(T/H_{11}\) has dimension \(3\).  The character
differences
\[
\begin{aligned}
  v_4-v_1&=(0,2,-3,0,1),\\
  v_6-v_1&=(1,-1,-1,0,1),\\
  v_{10}-v_1&=(2,-1,-4,1,2)
\end{aligned}
\]
have rank \(3\) in the character space of \(T/H_{11}\).  Combining the
torus action with overall projective rescaling, one may therefore normalize
the coefficients of
\[
  x_1x_2^4,\qquad
  x_1^3x_2x_4,\qquad
  x_0x_2^3x_4,\qquad
  x_0^2x_3x_4^2
\]
to \(1\).  Every general limit in Case~\(11\) is consequently equivalent to
a member of the normal-form family
\[
\begin{aligned}
  \nf_{11}(\alpha,\beta,\gamma,\delta,\varepsilon,\zeta)
  ={}&x_1x_2^4
  +\alpha x_1^2x_2^2x_3
  +\beta x_1^3x_3^2
  +x_1^3x_2x_4
  +\gamma x_0x_2^2x_3^2\\
  &+x_0x_2^3x_4
  +\delta x_0x_1x_3^3
  +\varepsilon x_0x_1x_2x_3x_4
  +\zeta x_0x_1^2x_4^2
  +x_0^2x_3x_4^2,
\end{aligned}
\]
where
\[
  (\alpha,\beta,\gamma,\delta,\varepsilon,\zeta)
  \in(\C^\times)^6.
\]
Since
\[
  \dim\PP(W^{H_{11}})^\circ=9,
  \qquad
  \dim(T/H_{11})=3,
\]
the corresponding quotient-side family has dimension
\[
  \dim\Phi_{11}=9-3=6.
\]
No additional condition on the six parameters is imposed at the normal-form
stage; every full-support member of this family is polystable.

\StartCase{12}

\subsubsection*{1-PS limit}
The witness vector and the corresponding one-parameter subgroup are
\[
  r_{12}=(13,8,-2,-7,-12),
  \qquad
  \lambda_{12}(t)=\diag(t^{13},t^8,t^{-2},t^{-7},t^{-12}),
  \quad t\in\Gm.
\]
Since the entries of \(r_{12}\) sum to zero, \(\lambda_{12}\) is a
one-parameter subgroup of \(G=\SL(5)\).  Let \(f_{12}\) be a general
quintic supported on
\[
  S_{12}=I(r_{12})_{\geq0}.
\]
With the action convention of \eqref{eq:quintic-monomial-weight}, its affine
one-parameter-subgroup limit is the weight-zero part
\[
  \phi_{12}:=\lim_{t\to0}\lambda_{12}(t)\cdot f_{12}.
\]
The degree-five weight-zero equations have precisely the exponent vectors
\[
\begin{gathered}
  (0,1,4,0,0),\quad
  (0,2,1,2,0),\quad
  (0,2,2,0,1),\quad
  (0,3,0,0,2),\quad
  (1,0,3,1,0),\\
  (1,1,0,3,0),\quad
  (1,1,1,1,1),\quad
  (2,0,0,2,1),\quad
  (2,0,1,0,2).
\end{gathered}
\]
Consequently,
\[
\begin{aligned}
  \phi_{12}={}&a_1x_1x_2^4
  +a_2x_1^2x_2x_3^2
  +a_3x_1^2x_2^2x_4
  +a_4x_1^3x_4^2
  +a_5x_0x_2^3x_3\\
  &+a_6x_0x_1x_3^3
  +a_7x_0x_1x_2x_3x_4
  +a_8x_0^2x_3^2x_4
  +a_9x_0^2x_2x_4^2,
\end{aligned}
\]
where \(a_i\in\C^\times\) for all \(i\).

Set \(H_{12}=\lambda_{12}(\Gm)\).  The five weights of \(H_{12}\) on
\(\langle x_0,\ldots,x_4\rangle\) are pairwise distinct, and their multiset
is not invariant under multiplication by \(-1\).  Hence
\[
  C_G(H_{12})=N_G(H_{12})=T.
\]
Moreover,
\[
\begin{aligned}
  W^{H_{12}}=\bigl\langle{}&x_1x_2^4,
  x_1^2x_2x_3^2,
  x_1^2x_2^2x_4,
  x_1^3x_4^2,
  x_0x_2^3x_3,\\
  &x_0x_1x_3^3,
  x_0x_1x_2x_3x_4,
  x_0^2x_3^2x_4,
  x_0^2x_2x_4^2\bigr\rangle,
  \qquad
  \dim W^{H_{12}}=9.
\end{aligned}
\]

\subsubsection*{Polystability}
Let \(v_1,\ldots,v_9\) be the nine exponent vectors above, with
\(v_7=(1,1,1,1,1)\), and set
\[
  \delta_i:=v_i-v_7
  \qquad (i=1,\ldots,9).
\]
Thus \(\delta_7=0\).  The eight nonzero characters span the
three-dimensional character space of \(T/H_{12}\) and satisfy the positive
relation
\[
  \delta_1+\delta_2+2\delta_3+2\delta_4
  +\delta_5+\delta_6+4\delta_8+2\delta_9=0.
\]
Therefore
\[
  0\in\operatorname{relint}
  \Conv\{\delta_1,\delta_2,\delta_3,\delta_4,
          \delta_5,\delta_6,\delta_8,\delta_9\}.
\]
Adding \(\delta_7=0\) does not change this conclusion.  Since all nine
coefficients of \(\phi_{12}\) are nonzero, the convex-hull criterion shows
that \((T/H_{12})\cdot\phi_{12}\) is closed in \(W^{H_{12}}\).  By affine
Luna reduction and \(C_G(H_{12})=T\), the orbit
\(G\cdot\phi_{12}\) is closed in \(W\).  Thus \([\phi_{12}]\) is
polystable.

\subsubsection*{Normal form}
On the full-support chart
\[
  \PP(W^{H_{12}})^\circ
  =\{[a_1:\cdots:a_9]\mid a_i\neq0\}
  \cong(\C^\times)^8,
\]
the effective torus \(T/H_{12}\) has dimension \(3\).  The character
differences
\[
\begin{aligned}
  v_4-v_1&=(0,2,-4,0,2),\\
  v_5-v_1&=(1,-1,-1,1,0),\\
  v_9-v_1&=(2,-1,-3,0,2)
\end{aligned}
\]
have rank \(3\) in the character space of \(T/H_{12}\).  Combining the
torus action with overall projective rescaling, one may therefore normalize
the coefficients of
\[
  x_1x_2^4,\qquad
  x_1^3x_4^2,\qquad
  x_0x_2^3x_3,\qquad
  x_0^2x_2x_4^2
\]
to \(1\).  Every general limit in Case~\(12\) is consequently equivalent to
a member of the normal-form family
\[
\begin{aligned}
  \nf_{12}(\alpha,\beta,\gamma,\delta,\varepsilon)
  ={}&x_1x_2^4
  +\alpha x_1^2x_2x_3^2
  +\beta x_1^2x_2^2x_4
  +x_1^3x_4^2
  +x_0x_2^3x_3\\
  &+\gamma x_0x_1x_3^3
  +\delta x_0x_1x_2x_3x_4
  +\varepsilon x_0^2x_3^2x_4
  +x_0^2x_2x_4^2,
\end{aligned}
\]
where
\[
  (\alpha,\beta,\gamma,\delta,\varepsilon)
  \in(\C^\times)^5.
\]
Since
\[
  \dim\PP(W^{H_{12}})^\circ=8,
  \qquad
  \dim(T/H_{12})=3,
\]
the corresponding quotient-side family has dimension
\[
  \dim\Phi_{12}=8-3=5.
\]
No additional condition on the five parameters is imposed at the normal-form
stage; every full-support member of this family is polystable.

\StartCase{13}

\subsubsection*{1-PS limit}
The witness vector and the corresponding one-parameter subgroup are
\[
  r_{13}=(3,2,0,-2,-3),
  \qquad
  \lambda_{13}(t)=\diag(t^3,t^2,1,t^{-2},t^{-3}),
  \quad t\in\Gm.
\]
Since the entries of \(r_{13}\) sum to zero, \(\lambda_{13}\) is a
one-parameter subgroup of \(G=\SL(5)\).  Let \(f_{13}\) be a general
quintic supported on
\[
  S_{13}=I(r_{13})_{\geq0}.
\]
With the action convention of \eqref{eq:quintic-monomial-weight}, its affine
one-parameter-subgroup limit is the weight-zero part
\[
  \phi_{13}:=\lim_{t\to0}\lambda_{13}(t)\cdot f_{13}.
\]
The degree-five weight-zero equations have precisely the exponent vectors
\[
\begin{gathered}
  (0,0,5,0,0),\quad
  (0,1,3,1,0),\quad
  (0,2,1,2,0),\quad
  (0,3,0,0,2),\\
  (1,0,3,0,1),\quad
  (1,1,1,1,1),\quad
  (2,0,0,3,0),\quad
  (2,0,1,0,2).
\end{gathered}
\]
Consequently,
\[
\begin{aligned}
  \phi_{13}={}&a_1x_2^5
  +a_2x_1x_2^3x_3
  +a_3x_1^2x_2x_3^2
  +a_4x_1^3x_4^2
  +a_5x_0x_2^3x_4\\
  &+a_6x_0x_1x_2x_3x_4
  +a_7x_0^2x_3^3
  +a_8x_0^2x_2x_4^2,
\end{aligned}
\]
where \(a_i\in\C^\times\) for all \(i\).

Set \(H_{13}=\lambda_{13}(\Gm)\).  The five weights of \(H_{13}\) on
\(\langle x_0,\ldots,x_4\rangle\) are pairwise distinct, and hence
\[
  C_G(H_{13})=T.
\]
Moreover,
\[
\begin{aligned}
  W^{H_{13}}=\bigl\langle{}&x_2^5,
  x_1x_2^3x_3,
  x_1^2x_2x_3^2,
  x_1^3x_4^2,
  x_0x_2^3x_4,\\
  &x_0x_1x_2x_3x_4,
  x_0^2x_3^3,
  x_0^2x_2x_4^2\bigr\rangle,
  \qquad
  \dim W^{H_{13}}=8.
\end{aligned}
\]

\subsubsection*{Polystability}
Let \(v_1,\ldots,v_8\) be the eight exponent vectors above, with
\(v_6=(1,1,1,1,1)\), and set
\[
  \delta_i:=v_i-v_6
  \qquad (i=1,\ldots,8).
\]
Thus \(\delta_6=0\).  The seven nonzero characters span the
three-dimensional character space of \(T/H_{13}\) and satisfy the positive
relation
\[
  \delta_1+\delta_2+\delta_3+4\delta_4
  +\delta_5+4\delta_7+3\delta_8=0.
\]
Therefore
\[
  0\in\operatorname{relint}
  \Conv\{\delta_1,\delta_2,\delta_3,\delta_4,
          \delta_5,\delta_7,\delta_8\}.
\]
Adding \(\delta_6=0\) does not change this conclusion.  Since all eight
coefficients of \(\phi_{13}\) are nonzero, the convex-hull criterion shows
that \((T/H_{13})\cdot\phi_{13}\) is closed in \(W^{H_{13}}\).  By affine
Luna reduction and \(C_G(H_{13})=T\), the orbit \(G\cdot\phi_{13}\) is
closed in \(W\).  Thus \([\phi_{13}]\) is polystable.

\subsubsection*{Normal form}
On the full-support chart
\[
  \PP(W^{H_{13}})^\circ
  =\{[a_1:\cdots:a_8]\mid a_i\neq0\}
  \cong(\C^\times)^7,
\]
the effective torus \(T/H_{13}\) has dimension \(3\).  After adjoining
overall projective rescaling, the characters of the coefficients of
\[
  x_2^5,\qquad
  x_1^3x_4^2,\qquad
  x_0x_2^3x_4,\qquad
  x_0^2x_3^3
\]
are independent.  These four coefficients may therefore be normalized to
\(1\).  Every general limit in Case~\(13\) is consequently equivalent to a
member of the normal-form family
\[
\begin{aligned}
  \nf_{13}(\alpha,\beta,\gamma,\delta)
  ={}&x_2^5
  +\alpha x_1x_2^3x_3
  +\beta x_1^2x_2x_3^2
  +x_1^3x_4^2
  +x_0x_2^3x_4\\
  &+\gamma x_0x_1x_2x_3x_4
  +x_0^2x_3^3
  +\delta x_0^2x_2x_4^2,
\end{aligned}
\]
where
\[
  (\alpha,\beta,\gamma,\delta)\in(\C^\times)^4.
\]
Since
\[
  \dim\PP(W^{H_{13}})^\circ=7,
  \qquad
  \dim(T/H_{13})=3,
\]
the corresponding quotient-side family has dimension
\[
  \dim\Phi_{13}=7-3=4.
\]
No additional condition on the four parameters is imposed at the normal-form
stage; every full-support member of this family is polystable.

\StartCase{14}

\subsubsection*{1-PS limit}
The witness vector and the corresponding one-parameter subgroup are
\[
  r_{14}=(4,1,0,-1,-4),
  \qquad
  \lambda_{14}(t)=\diag(t^4,t,1,t^{-1},t^{-4}),
  \quad t\in\Gm.
\]
Since the entries of \(r_{14}\) sum to zero, \(\lambda_{14}\) is a
one-parameter subgroup of \(G=\SL(5)\).  Let \(f_{14}\) be a general
quintic supported on
\[
  S_{14}=I(r_{14})_{\geq0}.
\]
With the action convention of \eqref{eq:quintic-monomial-weight}, its affine
one-parameter-subgroup limit is the weight-zero part
\[
  \phi_{14}:=\lim_{t\to0}\lambda_{14}(t)\cdot f_{14}.
\]
The degree-five weight-zero equations have precisely the exponent vectors
\[
\begin{gathered}
  (0,0,5,0,0),\quad
  (0,1,3,1,0),\quad
  (0,2,1,2,0),\quad
  (0,4,0,0,1),\\
  (1,0,0,4,0),\quad
  (1,0,3,0,1),\quad
  (1,1,1,1,1),\quad
  (2,0,1,0,2).
\end{gathered}
\]
Consequently,
\[
\begin{aligned}
  \phi_{14}={}&a_1x_2^5
  +a_2x_1x_2^3x_3
  +a_3x_1^2x_2x_3^2
  +a_4x_1^4x_4
  +a_5x_0x_3^4\\
  &+a_6x_0x_2^3x_4
  +a_7x_0x_1x_2x_3x_4
  +a_8x_0^2x_2x_4^2,
\end{aligned}
\]
where \(a_i\in\C^\times\) for all \(i\).

Set \(H_{14}=\lambda_{14}(\Gm)\).  The five weights of \(H_{14}\) on
\(\langle x_0,\ldots,x_4\rangle\) are pairwise distinct.  Hence
\[
  C_G(H_{14})=T.
\]
Moreover,
\[
\begin{aligned}
  W^{H_{14}}=\bigl\langle{}&x_2^5,
  x_1x_2^3x_3,
  x_1^2x_2x_3^2,
  x_1^4x_4,
  x_0x_3^4,\\
  &x_0x_2^3x_4,
  x_0x_1x_2x_3x_4,
  x_0^2x_2x_4^2\bigr\rangle,
  \qquad
  \dim W^{H_{14}}=8.
\end{aligned}
\]

\subsubsection*{Polystability}
Let \(v_1,\ldots,v_8\) be the eight exponent vectors above, with
\(v_7=(1,1,1,1,1)\), and set
\[
  \delta_i:=v_i-v_7
  \qquad (i=1,\ldots,8).
\]
Thus \(\delta_7=0\).  The seven nonzero characters span the
three-dimensional character space of \(T/H_{14}\) and satisfy the positive
relation
\[
  \delta_1+\delta_2+\delta_3+4\delta_4
  +4\delta_5+\delta_6+7\delta_8=0.
\]
Therefore
\[
  0\in\operatorname{relint}
  \Conv\{\delta_1,\delta_2,\delta_3,\delta_4,
          \delta_5,\delta_6,\delta_8\}.
\]
Adding \(\delta_7=0\) does not change this conclusion.  Since all eight
coefficients of \(\phi_{14}\) are nonzero, the convex-hull criterion shows
that \((T/H_{14})\cdot\phi_{14}\) is closed in \(W^{H_{14}}\).  By affine
Luna reduction and \(C_G(H_{14})=T\), the orbit \(G\cdot\phi_{14}\) is
closed in \(W\).  Thus \([\phi_{14}]\) is polystable.

\subsubsection*{Normal form}
On the full-support chart
\[
  \PP(W^{H_{14}})^\circ
  =\{[a_1:\cdots:a_8]\mid a_i\neq0\}
  \cong(\C^\times)^7,
\]
the effective torus \(T/H_{14}\) has dimension \(3\).  The character
differences
\[
\begin{aligned}
  v_4-v_1&=(0,4,-5,0,1),\\
  v_5-v_1&=(1,0,-5,4,0),\\
  v_8-v_1&=(2,0,-4,0,2)
\end{aligned}
\]
have rank \(3\) in the character space of \(T/H_{14}\).  Combining the
torus action with overall projective rescaling, one may therefore normalize
the coefficients of
\[
  x_2^5,\qquad
  x_1^4x_4,\qquad
  x_0x_3^4,\qquad
  x_0^2x_2x_4^2
\]
to \(1\).  Every general limit in Case~\(14\) is consequently equivalent to
a member of the normal-form family
\[
\begin{aligned}
  \nf_{14}(\alpha,\beta,\gamma,\delta)
  ={}&x_2^5
  +\alpha x_1x_2^3x_3
  +\beta x_1^2x_2x_3^2
  +x_1^4x_4
  +x_0x_3^4\\
  &+\gamma x_0x_2^3x_4
  +\delta x_0x_1x_2x_3x_4
  +x_0^2x_2x_4^2,
\end{aligned}
\]
where
\[
  (\alpha,\beta,\gamma,\delta)\in(\C^\times)^4.
\]
Since
\[
  \dim\PP(W^{H_{14}})^\circ=7,
  \qquad
  \dim(T/H_{14})=3,
\]
the corresponding quotient-side family has dimension
\[
  \dim\Phi_{14}=7-3=4.
\]
No additional condition on the four parameters is imposed at the normal-form
stage; every full-support member of this family is polystable.

\StartCase{15}

\subsubsection*{1-PS limit}
The witness vector and the corresponding one-parameter subgroup are
\[
  r_{15}=(9,4,-1,-6,-6),
  \qquad
  \lambda_{15}(t)=\diag(t^9,t^4,t^{-1},t^{-6},t^{-6}),
  \quad t\in\Gm.
\]
Since the entries of \(r_{15}\) sum to zero, \(\lambda_{15}\) is a
one-parameter subgroup of \(G=\SL(5)\).  Let \(f_{15}\) be a general
quintic supported on
\[
  S_{15}=I(r_{15})_{\geq0}.
\]
With the action convention of \eqref{eq:quintic-monomial-weight}, its affine
one-parameter-subgroup limit is the weight-zero part
\[
  \phi_{15}:=\lim_{t\to0}\lambda_{15}(t)\cdot f_{15}.
\]
Put
\[
  U:=\langle x_3,x_4\rangle.
\]
Solving the degree-five weight-zero equations gives
\[
\begin{aligned}
  W^{H_{15}}={}&\langle x_1x_2^4\rangle
  \oplus x_1^2x_2^2\operatorname{Sym}^1U
  \oplus x_1^3\operatorname{Sym}^2U
  \oplus x_0x_2^3\operatorname{Sym}^1U\\
  &\oplus x_0x_1x_2\operatorname{Sym}^2U
  \oplus x_0^2\operatorname{Sym}^3U,
  \qquad
  \dim W^{H_{15}}=15,
\end{aligned}
\]
where \(H_{15}=\lambda_{15}(\Gm)\).  Thus a general limit can be written as
\[
\begin{aligned}
  \phi_{15}={}&\kappa x_1x_2^4
  +x_1^2x_2^2L_1(x_3,x_4)
  +x_1^3Q_1(x_3,x_4)\\
  &+x_0x_2^3L_2(x_3,x_4)
  +x_0x_1x_2Q_2(x_3,x_4)
  +x_0^2B_3(x_3,x_4),
\end{aligned}
\]
with \(L_1,L_2\) linear, \(Q_1,Q_2\) quadratic, and \(B_3\) cubic in
\(x_3,x_4\).  We work on the nonempty open locus
\[
  \kappa\neq0,
  \qquad
  L_1\wedge L_2\neq0,
  \qquad
  \operatorname{Disc}(Q_1)\neq0,
  \qquad
  \operatorname{Disc}(B_3)\neq0.
\]

The weight \(-6\) occurs twice on \(U\), so the centralizer is non-toric:
\[
  C_G(H_{15})
  =
  \left\{
    \diag(\mu_0,\mu_1,\mu_2)\oplus A
    \ \middle|\
    A\in\GL(U),\ 
    \mu_0\mu_1\mu_2\det A=1
  \right\}.
\]
Hence
\[
  \dim C_G(H_{15})=6,
  \qquad
  \dim\bigl(C_G(H_{15})/H_{15}\bigr)=5.
\]

\subsubsection*{Polystability}
Set \(C:=C_G(H_{15})\).  By affine Luna reduction, it is enough to prove
that \(C\cdot\phi_{15}\) is closed in \(W^{H_{15}}\).  Let
\(\lambda\in\Lambda_{\phi_{15}}^C\).  After conjugation in the
\(\GL(U)\)-block, its weights may be written as
\[
  \wt_\lambda(x_0,x_1,x_2,X,Y)
  =(a,b,c,s+n,s-n),
  \qquad
  a+b+c+2s=0,
\]
for a suitable eigenbasis \(X,Y\) of \(U\).

Introduce the averaged weights
\[
  \nu:=b+4c,
  \qquad
  q:=3b+2s,
  \qquad
  p:=2a+3s.
\]
The nonvanishing of \(\kappa\), \(\operatorname{Disc}(Q_1)\), and
\(\operatorname{Disc}(B_3)\), together with the nonnegativity of all
occurring monomial weights, gives
\[
  \nu\geq0,
  \qquad
  q\geq0,
  \qquad
  p\geq0.
\]
On the other hand,
\[
  \nu+q+2p=4(a+b+c+2s)=0.
\]
Therefore
\[
  \nu=q=p=0.
\]
Since \(p=0\), a nonzero relative weight \(n\) would force the binary cubic
\(B_3\) to be divisible by \(X^2\) or by \(Y^2\), contradicting
\(\operatorname{Disc}(B_3)\neq0\).  Hence \(n=0\).  Solving the remaining
weight equations gives
\[
  (a,b,c,s)=(9m,4m,-m,-6m)
\]
for some \(m\in\mathbb Z\).  Consequently,
\[
  \Lambda_{\phi_{15}}^C
  =\{\lambda_{15}(t^m)\mid m\in\mathbb Z\}.
\]
This set is symmetric under \(m\mapsto-m\).  The Casimiro--Florentino
criterion therefore shows that \(C\cdot\phi_{15}\) is closed in
\(W^{H_{15}}\).  Luna reduction then implies that \(G\cdot\phi_{15}\) is
closed in \(W\), so \([\phi_{15}]\) is polystable.

\subsubsection*{Normal form}
Assume that \(L_1,L_2\) are linearly independent and write
\[
  Q_1=eL_1^2+fL_1L_2+hL_2^2.
\]
The coefficient
\[
  J(\phi_{15}):=\kappa e
\]
is invariant under the centralizer action and projective rescaling.  Thus
this coefficient cannot in general be normalized away; equivalently, in
monomial coordinates the ratio
\[
  \frac{a_1a_4}{a_2^2}
\]
is invariant.

Use \(\GL(U)\) to send \(L_1\) and \(L_2\) to \(x_3\) and \(x_4\), and
then use the diagonal factors of \(C\), together with projective rescaling,
to normalize the coefficients of
\[
  x_1x_2^4,
  \qquad
  x_1^2x_2^2x_3,
  \qquad
  x_0x_2^3x_4
\]
to \(1\).  On the dense open locus \(h\neq0\), the residual one-parameter
subgroup
\[
  \lambda_*(t)=\diag(t,1,1,1,t^{-1})
\]
normalizes the coefficient of \(x_1^3x_4^2\) to \(1\).  After renaming the
remaining coefficients, every general limit is equivalent to a member of
\[
\begin{aligned}
  \nf_{15}(Q,R,B)={}&x_1x_2^4
  +x_1^2x_2^2x_3
  +x_1^3Q(x_3,x_4)
  +x_0x_2^3x_4\\
  &+x_0x_1x_2R(x_3,x_4)
  +x_0^2B(x_3,x_4),
\end{aligned}
\]
where
\[
  Q=\gamma x_3^2+\beta x_3x_4+x_4^2,
\]
\[
  R=\rho_0x_3^2+\rho_1x_3x_4+\rho_2x_4^2,
\]
and
\[
  B=\sigma_0x_3^3
  +\sigma_1x_3^2x_4
  +\sigma_2x_3x_4^2
  +\sigma_3x_4^3.
\]
The nine parameters
\[
  (\gamma,\beta,\rho_0,\rho_1,\rho_2,
   \sigma_0,\sigma_1,\sigma_2,\sigma_3)
\]
range over the nonempty open subset of \(\C^9\) on which
\[
  \beta^2-4\gamma\neq0,
  \qquad
  \operatorname{Disc}(B)\neq0.
\]
The root choices in the normalization are finite, and the remaining
\(\lambda_*\)-ambiguity is finite.  Hence this slice maps generically
finitely onto a dense subset of \(\Phi_{15}\).  Since
\[
  \dim\PP(W^{H_{15}})=14,
  \qquad
  \dim\bigl(C_G(H_{15})/H_{15}\bigr)=5,
\]
we obtain
\[
  \dim\Phi_{15}=14-5=9.
\]
No condition arising from the later singular-locus analysis is imposed at
the normal-form stage.

\StartCase{16}

\subsubsection*{1-PS limit}
The witness vector and the corresponding one-parameter subgroup are
\[
  r_{16}=(3,2,1,-2,-4),
  \qquad
  \lambda_{16}(t)=\diag(t^3,t^2,t,t^{-2},t^{-4}),
  \quad t\in\Gm.
\]
Since the entries of \(r_{16}\) sum to zero, \(\lambda_{16}\) is a
one-parameter subgroup of \(G=\SL(5)\).  Let \(f_{16}\) be a general
quintic supported on
\[
  S_{16}=I(r_{16})_{\geq0}.
\]
With the action convention of \eqref{eq:quintic-monomial-weight}, its affine
one-parameter-subgroup limit is the weight-zero part
\[
  \phi_{16}:=\lim_{t\to0}\lambda_{16}(t)\cdot f_{16}.
\]
The degree-five weight-zero equations have precisely the exponent vectors
\[
\begin{gathered}
  (0,0,4,0,1),\quad
  (0,1,2,2,0),\quad
  (0,3,0,1,1),\quad
  (1,1,1,1,1),\\
  (2,0,0,3,0),\quad
  (2,1,0,0,2).
\end{gathered}
\]
Consequently,
\[
\begin{aligned}
  \phi_{16}={}&a_1x_2^4x_4
  +a_2x_1x_2^2x_3^2
  +a_3x_1^3x_3x_4
  +a_4x_0x_1x_2x_3x_4\\
  &+a_5x_0^2x_3^3
  +a_6x_0^2x_1x_4^2,
\end{aligned}
\]
where \(a_i\in\C^\times\) for all \(i\).

Set \(H_{16}=\lambda_{16}(\Gm)\).  The five weights of \(H_{16}\) on
\(\langle x_0,\ldots,x_4\rangle\) are pairwise distinct.  Hence
\[
  C_G(H_{16})=T.
\]
Moreover,
\[
\begin{aligned}
  W^{H_{16}}=\bigl\langle{}&x_2^4x_4,
  x_1x_2^2x_3^2,
  x_1^3x_3x_4,
  x_0x_1x_2x_3x_4,\\
  &x_0^2x_3^3,
  x_0^2x_1x_4^2\bigr\rangle,
  \qquad
  \dim W^{H_{16}}=6.
\end{aligned}
\]

\subsubsection*{Polystability}
Let \(v_1,\ldots,v_6\) be the six exponent vectors above, with
\(v_4=(1,1,1,1,1)\), and set
\[
  \delta_i:=v_i-v_4
  \qquad (i=1,\ldots,6).
\]
Thus \(\delta_4=0\).  The five nonzero characters span the
three-dimensional character space of \(T/H_{16}\) and satisfy the positive
relation
\[
  \delta_1+\delta_2+\delta_3+\delta_5+2\delta_6=0.
\]
Therefore
\[
  0\in\operatorname{relint}
  \Conv\{\delta_1,\delta_2,\delta_3,\delta_5,\delta_6\}.
\]
Adding \(\delta_4=0\) does not change this conclusion.  Since all six
coefficients of \(\phi_{16}\) are nonzero, the convex-hull criterion shows
that \((T/H_{16})\cdot\phi_{16}\) is closed in \(W^{H_{16}}\).  By affine
Luna reduction and \(C_G(H_{16})=T\), the orbit \(G\cdot\phi_{16}\) is
closed in \(W\).  Thus \([\phi_{16}]\) is polystable.

\subsubsection*{Normal form}
On the full-support chart
\[
  \PP(W^{H_{16}})^\circ
  =\{[a_1:\cdots:a_6]\mid a_i\neq0\}
  \cong(\C^\times)^5,
\]
the effective torus \(T/H_{16}\) has dimension \(3\).  The character
differences
\[
\begin{aligned}
  v_2-v_1&=(0,1,-2,2,-1),\\
  v_3-v_1&=(0,3,-4,1,0),\\
  v_4-v_1&=(1,1,-3,1,0)
\end{aligned}
\]
have rank \(3\) in its character space.  Combining the torus action with
overall projective rescaling, one may therefore normalize the coefficients
of
\[
  x_2^4x_4,
  \qquad
  x_1x_2^2x_3^2,
  \qquad
  x_1^3x_3x_4,
  \qquad
  x_0x_1x_2x_3x_4
\]
to \(1\).  Every general limit in Case~\(16\) is consequently equivalent to
a member of the normal-form family
\[
\begin{aligned}
  \nf_{16}(\alpha,\beta)
  ={}&x_2^4x_4
  +x_1x_2^2x_3^2
  +x_1^3x_3x_4
  +x_0x_1x_2x_3x_4\\
  &+\alpha x_0^2x_3^3
  +\beta x_0^2x_1x_4^2,
\end{aligned}
\]
where
\[
  (\alpha,\beta)\in(\C^\times)^2.
\]
Since
\[
  \dim\PP(W^{H_{16}})^\circ=5,
  \qquad
  \dim(T/H_{16})=3,
\]
the corresponding quotient-side family has dimension
\[
  \dim\Phi_{16}=5-3=2.
\]
No additional condition on \(\alpha\) or \(\beta\) is imposed at the
normal-form stage; every full-support member of this family is polystable.

\StartCase{17}

\subsubsection*{1-PS limit}
The witness vector and the corresponding one-parameter subgroup are
\[
  r_{17}=(12,7,2,-8,-13),
  \qquad
  \lambda_{17}(t)=\diag(t^{12},t^7,t^2,t^{-8},t^{-13}),
  \quad t\in\Gm.
\]
Since the entries of \(r_{17}\) sum to zero, \(\lambda_{17}\) is a
one-parameter subgroup of \(G=\SL(5)\).  Let \(f_{17}\) be a general
quintic supported on
\[
  S_{17}=I(r_{17})_{\geq0}.
\]
With the action convention of \eqref{eq:quintic-monomial-weight}, its affine
one-parameter-subgroup limit is the weight-zero part
\[
  \phi_{17}:=\lim_{t\to0}\lambda_{17}(t)\cdot f_{17}.
\]
The degree-five weight-zero equations have precisely the exponent vectors
\[
\begin{gathered}
  (0,0,4,1,0),\quad
  (0,1,3,0,1),\quad
  (0,2,1,2,0),\quad
  (0,3,0,1,1),\quad
  (1,0,2,2,0),\\
  (1,1,1,1,1),\quad
  (1,2,0,0,2),\quad
  (2,0,0,3,0),\quad
  (2,0,1,0,2).
\end{gathered}
\]
Consequently,
\[
\begin{aligned}
  \phi_{17}={}&a_1x_2^4x_3
  +a_2x_1x_2^3x_4
  +a_3x_1^2x_2x_3^2
  +a_4x_1^3x_3x_4
  +a_5x_0x_2^2x_3^2\\
  &+a_6x_0x_1x_2x_3x_4
  +a_7x_0x_1^2x_4^2
  +a_8x_0^2x_3^3
  +a_9x_0^2x_2x_4^2,
\end{aligned}
\]
where \(a_i\in\C^\times\) for all \(i\).

Set \(H_{17}=\lambda_{17}(\Gm)\).  The five weights of \(H_{17}\) on
\(\langle x_0,\ldots,x_4\rangle\) are pairwise distinct.  Hence
\[
  C_G(H_{17})=T.
\]
Moreover,
\[
\begin{aligned}
  W^{H_{17}}=\bigl\langle{}&x_2^4x_3,
  x_1x_2^3x_4,
  x_1^2x_2x_3^2,
  x_1^3x_3x_4,
  x_0x_2^2x_3^2,\\
  &x_0x_1x_2x_3x_4,
  x_0x_1^2x_4^2,
  x_0^2x_3^3,
  x_0^2x_2x_4^2\bigr\rangle,
  \qquad
  \dim W^{H_{17}}=9.
\end{aligned}
\]

\subsubsection*{Polystability}
Let \(v_1,\ldots,v_9\) be the nine exponent vectors above, with
\(v_6=(1,1,1,1,1)\), and set
\[
  \delta_i:=v_i-v_6
  \qquad (i=1,\ldots,9).
\]
Thus \(\delta_6=0\).  The characters \(\delta_i\) span the
three-dimensional character space of \(T/H_{17}\) and satisfy the positive
relation
\[
  \delta_1+\delta_2+\delta_3+\delta_4+\delta_5+\delta_6
  +3\delta_7+2\delta_8+2\delta_9=0.
\]
Therefore
\[
  0\in\operatorname{relint}\Conv\{\delta_1,\ldots,\delta_9\}.
\]
Since all nine coefficients of \(\phi_{17}\) are nonzero, the convex-hull
criterion shows that \((T/H_{17})\cdot\phi_{17}\) is closed in
\(W^{H_{17}}\).  By affine Luna reduction and \(C_G(H_{17})=T\), the orbit
\(G\cdot\phi_{17}\) is closed in \(W\).  Thus \([\phi_{17}]\) is
polystable.

\subsubsection*{Normal form}
On the full-support chart
\[
  \PP(W^{H_{17}})^\circ
  =\{[a_1:\cdots:a_9]\mid a_i\neq0\}
  \cong(\C^\times)^8,
\]
the effective torus \(T/H_{17}\) has dimension \(3\).  The character
differences
\[
\begin{aligned}
  v_2-v_1&=(0,1,-1,-1,1),\\
  v_4-v_1&=(0,3,-4,0,1),\\
  v_8-v_1&=(2,0,-4,2,0)
\end{aligned}
\]
have rank \(3\) in its character space.  Combining the torus action with
overall projective rescaling, one may therefore normalize the coefficients
of
\[
  x_2^4x_3,
  \qquad
  x_1x_2^3x_4,
  \qquad
  x_1^3x_3x_4,
  \qquad
  x_0^2x_3^3
\]
to \(1\).  Every general limit in Case~\(17\) is consequently equivalent to
a member of the normal-form family
\[
\begin{aligned}
  \nf_{17}(\alpha,\beta,\gamma,\delta,\varepsilon)
  ={}&x_2^4x_3
  +x_1x_2^3x_4
  +\alpha x_1^2x_2x_3^2
  +x_1^3x_3x_4
  +\beta x_0x_2^2x_3^2\\
  &+\gamma x_0x_1x_2x_3x_4
  +\delta x_0x_1^2x_4^2
  +x_0^2x_3^3
  +\varepsilon x_0^2x_2x_4^2,
\end{aligned}
\]
where
\[
  (\alpha,\beta,\gamma,\delta,\varepsilon)
  \in(\C^\times)^5.
\]
Since
\[
  \dim\PP(W^{H_{17}})^\circ=8,
  \qquad
  \dim(T/H_{17})=3,
\]
the corresponding quotient-side family has dimension
\[
  \dim\Phi_{17}=8-3=5.
\]
No additional condition on the five parameters is imposed at the
normal-form stage; every full-support member of this family is polystable.

\StartCase{18}

\subsubsection*{1-PS limit}
The witness vector and the corresponding one-parameter subgroup are
\[
  r_{18}=(4,2,0,-1,-5),
  \qquad
  \lambda_{18}(t)=\diag(t^4,t^2,1,t^{-1},t^{-5}),
  \quad t\in\Gm.
\]
Since the entries of \(r_{18}\) sum to zero, \(\lambda_{18}\) is a
one-parameter subgroup of \(G=\SL(5)\).  Let \(f_{18}\) be a general
quintic supported on
\[
  S_{18}=I(r_{18})_{\geq0}.
\]
With the action convention of \eqref{eq:quintic-monomial-weight}, its affine
one-parameter-subgroup limit is the weight-zero part
\[
  \phi_{18}:=\lim_{t\to0}\lambda_{18}(t)\cdot f_{18}.
\]
The degree-five weight-zero equations have precisely the exponent vectors
\[
\begin{gathered}
  (0,0,5,0,0),\quad
  (0,1,2,2,0),\quad
  (0,3,0,1,1),\quad
  (1,0,0,4,0),\quad
  (1,1,1,1,1),\quad
  (2,1,0,0,2).
\end{gathered}
\]
Consequently,
\[
  \phi_{18}
  =a_1x_2^5
  +a_2x_1x_2^2x_3^2
  +a_3x_1^3x_3x_4
  +a_4x_0x_3^4
  +a_5x_0x_1x_2x_3x_4
  +a_6x_0^2x_1x_4^2,
\]
where \(a_i\in\C^\times\) for all \(i\).

Set \(H_{18}=\lambda_{18}(\Gm)\).  The five weights of \(H_{18}\) on
\(\langle x_0,\ldots,x_4\rangle\) are pairwise distinct.  Moreover, their
multiset is not invariant under multiplication by \(-1\).  Hence
\[
  C_G(H_{18})=N_G(H_{18})=T.
\]
The fixed subspace is
\[
  W^{H_{18}}
  =\left\langle
  x_2^5,
  x_1x_2^2x_3^2,
  x_1^3x_3x_4,
  x_0x_3^4,
  x_0x_1x_2x_3x_4,
  x_0^2x_1x_4^2
  \right\rangle,
  \qquad
  \dim W^{H_{18}}=6.
\]

\subsubsection*{Polystability}
Let \(v_1,\ldots,v_6\) be the six exponent vectors above, with
\(v_5=(1,1,1,1,1)\), and put
\[
  \delta_i:=v_i-v_5
  \qquad (i=1,\ldots,6).
\]
Thus \(\delta_5=0\).  These characters span the three-dimensional character
space of \(T/H_{18}\) and satisfy the positive relation
\[
  \delta_1+\delta_2+\delta_3+\delta_4+\delta_5+3\delta_6=0.
\]
Therefore
\[
  0\in\operatorname{relint}\Conv\{\delta_1,\ldots,\delta_6\}.
\]
Since all six coefficients of \(\phi_{18}\) are nonzero, the convex-hull
criterion shows that \((T/H_{18})\cdot\phi_{18}\) is closed in
\(W^{H_{18}}\).  By affine Luna reduction and
\(C_G(H_{18})=T\), the orbit \(G\cdot\phi_{18}\) is closed in \(W\).
Thus \([\phi_{18}]\) is polystable.

\subsubsection*{Normal form}
On the full-support chart
\[
  U_{18}:=\PP(W^{H_{18}})^\circ
  =\{[a_1:\cdots:a_6]\mid a_i\neq0\}
  \cong(\C^\times)^5,
\]
the effective torus \(T/H_{18}\) has dimension \(3\).  Combining its action
with overall projective rescaling, one may normalize
\[
  a_1=a_2=a_3=a_4=1.
\]
Indeed, if \(w_i=(v_i,1)\) denotes the coefficient character for the action
of \(T\times\Gm\), then \(w_1,\ldots,w_4\) are independent and
\[
  w_5=w_1-2w_2+w_3+w_4,
  \qquad
  w_6=2w_1-5w_2+2w_3+2w_4.
\]
The two remaining invariant coefficients are therefore
\[
  \alpha=\frac{a_2^2a_5}{a_1a_3a_4},
  \qquad
  \beta=\frac{a_2^5a_6}{a_1^2a_3^2a_4^2}.
\]
Thus every general limit in Case~\(18\) is equivalent to a unique member of
the normal-form family
\[
  \nf_{18}(\alpha,\beta)
  =x_2^5
  +x_1x_2^2x_3^2
  +x_1^3x_3x_4
  +x_0x_3^4
  +\alpha x_0x_1x_2x_3x_4
  +\beta x_0^2x_1x_4^2,
\]
where
\[
  (\alpha,\beta)\in(\C^\times)^2.
\]
The two displayed character relations form a basis of the relation lattice,
so
\[
  U_{18}/(T/H_{18})\cong(\C^\times)^2
\]
with coordinates \((\alpha,\beta)\).  Moreover, for a full-support point
\(x\in W^{H_{18}}\), the rank computation gives \(G_x^\circ=H_{18}\).
Consequently, any element of \(G\) carrying one full-support point of
\(W^{H_{18}}\) to another lies in
\(N_G(H_{18})=T\).  Hence the full \(G\)-quotient introduces no further
identifications on this locus.

It follows that
\[
  \dim\Phi_{18}
  =\dim U_{18}-\dim(T/H_{18})
  =5-3
  =2.
\]
No additional condition on \((\alpha,\beta)\) is imposed at the normal-form
stage; every full-support member of this family is polystable.

\StartCase{19}

\subsubsection*{1-PS limit}
The witness vector and the corresponding one-parameter subgroup are
\[
  r_{19}=(2,1,0,-1,-2),
  \qquad
  \lambda_{19}(t)=\diag(t^2,t,1,t^{-1},t^{-2}),
  \quad t\in\Gm.
\]
Since the entries of \(r_{19}\) sum to zero, \(\lambda_{19}\) is a
one-parameter subgroup of \(G=\SL(5)\).  Let \(f_{19}\) be a general
quintic supported on
\[
  S_{19}=I(r_{19})_{\geq0}.
\]
With the action convention of \eqref{eq:quintic-monomial-weight}, its affine
one-parameter-subgroup limit is the weight-zero part
\[
  \phi_{19}:=\lim_{t\to0}\lambda_{19}(t)\cdot f_{19}.
\]
The degree-five weight-zero equations have precisely the exponent vectors
\[
\begin{gathered}
  (0,0,5,0,0),\quad
  (0,1,3,1,0),\quad
  (0,2,1,2,0),\quad
  (0,2,2,0,1),\\
  (0,3,0,1,1),\quad
  (1,0,2,2,0),\quad
  (1,0,3,0,1),\quad
  (1,1,0,3,0),\\
  (1,1,1,1,1),\quad
  (1,2,0,0,2),\quad
  (2,0,0,2,1),\quad
  (2,0,1,0,2).
\end{gathered}
\]
Consequently,
\[
\begin{aligned}
  \phi_{19}={}&
   a_1x_2^5
  +a_2x_1x_2^3x_3
  +a_3x_1^2x_2x_3^2
  +a_4x_1^2x_2^2x_4
  +a_5x_1^3x_3x_4\\
  &+a_6x_0x_2^2x_3^2
  +a_7x_0x_2^3x_4
  +a_8x_0x_1x_3^3
  +a_9x_0x_1x_2x_3x_4\\
  &+a_{10}x_0x_1^2x_4^2
  +a_{11}x_0^2x_3^2x_4
  +a_{12}x_0^2x_2x_4^2,
\end{aligned}
\]
where \(a_i\in\C^\times\) for all \(i\).

Set \(H_{19}=\lambda_{19}(\Gm)\).  Its five weights on
\(\langle x_0,\ldots,x_4\rangle\) are pairwise distinct, and hence
\[
  C_G(H_{19})=T.
\]
The fixed subspace is
\[
\begin{aligned}
  W^{H_{19}}=\langle{}&
  x_2^5,
  x_1x_2^3x_3,
  x_1^2x_2x_3^2,
  x_1^2x_2^2x_4,
  x_1^3x_3x_4,\\
  &x_0x_2^2x_3^2,
  x_0x_2^3x_4,
  x_0x_1x_3^3,
  x_0x_1x_2x_3x_4,\\
  &x_0x_1^2x_4^2,
  x_0^2x_3^2x_4,
  x_0^2x_2x_4^2
  \rangle,
  \qquad
  \dim W^{H_{19}}=12.
\end{aligned}
\]

\subsubsection*{Polystability}
Let \(v_1,\ldots,v_{12}\) be the exponent vectors above, with
\(v_9=(1,1,1,1,1)\), and put
\[
  \delta_i:=v_i-v_9
  \qquad (i=1,\ldots,12).
\]
Thus \(\delta_9=0\).  The characters \(\delta_i\) span the
three-dimensional character space of \(T/H_{19}\) and satisfy the positive
relation
\[
\begin{aligned}
  &\delta_1+\delta_2+\delta_3+\delta_4+\delta_5+\delta_6
  +\delta_7+\delta_8+\delta_9\\
  &\hspace{3cm}+4\delta_{10}+4\delta_{11}+\delta_{12}=0.
\end{aligned}
\]
Therefore
\[
  0\in\operatorname{relint}
  \Conv\{\delta_1,\ldots,\delta_{12}\}.
\]
Since all twelve coefficients of \(\phi_{19}\) are nonzero, the
convex-hull criterion shows that
\((T/H_{19})\cdot\phi_{19}\) is closed in \(W^{H_{19}}\).  By affine Luna
reduction and \(C_G(H_{19})=T\), the orbit \(G\cdot\phi_{19}\) is closed in
\(W\).  Hence \([\phi_{19}]\) is polystable.

\subsubsection*{Normal form}
The full-support chart is
\[
  \PP(W^{H_{19}})^\circ
  =\{[a_1:\cdots:a_{12}]\mid a_i\neq0\}
  \cong(\C^\times)^{11}.
\]
The effective torus \(T/H_{19}\) has dimension \(3\).  The character
differences associated with the coefficients of
\[
  x_2^5,
  \qquad
  x_1^3x_3x_4,
  \qquad
  x_0x_1x_3^3,
  \qquad
  x_0^2x_3^2x_4
\]
have full rank after including overall projective rescaling.  We may
therefore normalize
\[
  a_1=a_5=a_8=a_{11}=1.
\]
Renaming the remaining coefficients gives the normal form
\[
\begin{aligned}
  \nf_{19}(\alpha,\beta,\gamma,\delta,
  \varepsilon,\zeta,\xi,\theta)
  ={}&x_2^5
  +\alpha x_1x_2^3x_3
  +\beta x_1^2x_2x_3^2
  +\gamma x_1^2x_2^2x_4\\
  &+x_1^3x_3x_4
  +\delta x_0x_2^2x_3^2
  +\varepsilon x_0x_2^3x_4
  +x_0x_1x_3^3\\
  &+\zeta x_0x_1x_2x_3x_4
  +\xi x_0x_1^2x_4^2
  +x_0^2x_3^2x_4
  +\theta x_0^2x_2x_4^2,
\end{aligned}
\]
where
\[
  (\alpha,\beta,\gamma,\delta,
  \varepsilon,\zeta,\xi,\theta)
  \in(\C^\times)^8.
\]
Thus eight parameters remain after diagonal torus scaling and projective
rescaling.  Accordingly,
\[
  \dim\Phi_{19}
  =\dim\PP(W^{H_{19}})^\circ-\dim(T/H_{19})
  =11-3
  =8.
\]
No additional parameter condition is imposed at the normal-form stage; every
full-support member of this family is polystable.

\StartCase{20}

\subsubsection*{1-PS limit}
The witness vector and the corresponding one-parameter subgroup are
\[
  r_{20}=(8,3,-2,-2,-7),
  \qquad
  \lambda_{20}(t)=\diag(t^8,t^3,t^{-2},t^{-2},t^{-7}),
  \quad t\in\Gm.
\]
Since the entries of \(r_{20}\) sum to zero, \(\lambda_{20}\) is a
one-parameter subgroup of \(G=\SL(5)\).  Let \(f_{20}\) be a general
quintic supported on
\[
  S_{20}=I(r_{20})_{\geq0}.
\]
With the action convention of \eqref{eq:quintic-monomial-weight}, its affine
one-parameter-subgroup limit is the weight-zero part
\[
  \phi_{20}:=\lim_{t\to0}\lambda_{20}(t)\cdot f_{20}.
\]
Put
\[
  U:=\langle x_2,x_3\rangle,
  \qquad
  H_{20}:=\lambda_{20}(\Gm).
\]
Solving the degree-five weight-zero equations gives
\[
\begin{aligned}
  W^{H_{20}}={}&x_1^2\operatorname{Sym}^3U
  \oplus x_1^3x_4\operatorname{Sym}^1U
  \oplus x_0\operatorname{Sym}^4U
  \oplus x_0x_1x_4\operatorname{Sym}^2U\\
  &\oplus\langle x_0x_1^2x_4^2\rangle
  \oplus x_0^2x_4^2\operatorname{Sym}^1U,
  \qquad
  \dim W^{H_{20}}=17.
\end{aligned}
\]
Thus a general limit can be written as
\[
\begin{aligned}
  \phi_{20}={}&x_1^2C_3(x_2,x_3)
  +x_1^3x_4L_1(x_2,x_3)
  +x_0B_4(x_2,x_3)\\
  &+x_0x_1x_4Q_2(x_2,x_3)
  +\kappa x_0x_1^2x_4^2
  +x_0^2x_4^2M_1(x_2,x_3),
\end{aligned}
\]
where \(C_3\), \(B_4\), and \(Q_2\) are binary forms of degrees
\(3\), \(4\), and \(2\), respectively, and \(L_1,M_1\) are binary
linear forms.  We work on the nonempty open locus
\[
  \kappa\neq0,
  \qquad
  \operatorname{Disc}(B_4)\neq0,
  \qquad
  \operatorname{Disc}(C_3)\neq0,
\]
\[
  \operatorname{Res}(B_4,L_1)\neq0,
  \qquad
  \operatorname{Res}(B_4,M_1)\neq0.
\]

The weight \(-2\) occurs twice on \(U\), so the centralizer is non-toric:
\[
  C_G(H_{20})
  =
  \left\{
    \diag(\mu_0,\mu_1)\oplus A\oplus\diag(\mu_4)
    \ \middle|\
    A\in\GL(U),\
    \mu_0\mu_1\mu_4\det A=1
  \right\}.
\]
Hence
\[
  \dim C_G(H_{20})=6,
  \qquad
  \dim\bigl(C_G(H_{20})/H_{20}\bigr)=5.
\]

\subsubsection*{Polystability}
Set \(R:=C_G(H_{20})\).  By affine Luna reduction, it is enough to prove
that \(R\cdot\phi_{20}\) is closed in \(W^{H_{20}}\).  Let
\(\rho\in\Lambda_{\phi_{20}}^R\).  After conjugation in the
\(\GL(U)\)-block, its weights may be written as
\[
  \wt_\rho(x_0,x_1,X,Y,x_4)
  =(a,b,s+n,s-n,c),
  \qquad
  a+b+2s+c=0,
\]
for a suitable eigenbasis \(X,Y\) of \(U\).

Introduce the averaged block weights
\[
  p_B:=a+4s,
  \qquad
  p_C:=2b+3s,
  \qquad
  p_L:=3b+s+c,
\]
\[
  e:=a+2b+2c,
  \qquad
  p_M:=2a+s+2c.
\]
The nonvanishing of the discriminants, resultants, and \(\kappa\), together
with nonnegativity of all occurring monomial weights, gives
\[
  p_C\geq0,
  \qquad
  p_B\geq0,
  \qquad
  e\geq0,
\]
\[
  p_B+4p_L\geq0,
  \qquad
  p_B+4p_M\geq0.
\]
These quantities satisfy the balancing identity
\[
  4p_C+2p_B+e+(p_B+4p_L)+2(p_B+4p_M)
  =22(a+b+2s+c)=0.
\]
Consequently,
\[
  p_B=p_C=p_L=p_M=e=0.
\]
Since \(p_B=0\), a nonzero relative weight \(n\) would force the binary
quartic \(B_4\) to be divisible by \(X^2\) or by \(Y^2\), contradicting
\(\operatorname{Disc}(B_4)\neq0\).  Hence \(n=0\).  Solving the remaining weight
equations gives
\[
  (a,b,s,c)=m(8,3,-2,-7)
\]
for some \(m\in\mathbb Z\).  Therefore
\[
  \Lambda_{\phi_{20}}^R
  =\{\lambda_{20}(t^m)\mid m\in\mathbb Z\}.
\]
This set is symmetric under \(m\mapsto-m\).  The Casimiro--Florentino
criterion shows that \(R\cdot\phi_{20}\) is closed in \(W^{H_{20}}\).
Luna reduction then implies that \(G\cdot\phi_{20}\) is closed in \(W\),
so \([\phi_{20}]\) is polystable.

\subsubsection*{Normal form}
On the dense open set where \(B_4\) has four distinct roots and the root of
\(L_1\) is distinct from them, use \(\GL(U)\) to put
\[
  L_1=x_2
\]
and
\[
  B_4(x_2,x_3)=B_{\lambda,\mu}(x_2,x_3)
  :=x_3(x_2-x_3)(x_2-\lambda x_3)(x_2-\mu x_3),
\]
where
\[
  \lambda\mu(\lambda-1)(\mu-1)(\lambda-\mu)\neq0.
\]
Write
\[
  C_{\mathbf c}(x_2,x_3)
  =c_0x_2^3+c_1x_2^2x_3+c_2x_2x_3^2+c_3x_3^3,
\]
\[
  Q_{\mathbf q}(x_2,x_3)
  =q_0x_2^2+q_1x_2x_3+q_2x_3^2,
\]
and
\[
  M_{\mathbf m}(x_2,x_3)=m_0x_2+m_1x_3.
\]
Using the diagonal factors in \(C_G(H_{20})\), together with projective
rescaling, normalize the coefficients of
\[
  x_0B_{\lambda,\mu}(x_2,x_3),
  \qquad
  x_1^3x_2x_4,
  \qquad
  x_0x_1^2x_4^2
\]
to \(1\).  Every general limit is then equivalent to a member of
\[
\begin{aligned}
  \nf_{20}(\lambda,\mu,\mathbf c,\mathbf q,\mathbf m)
  ={}&x_0B_{\lambda,\mu}(x_2,x_3)
  +x_1^2C_{\mathbf c}(x_2,x_3)
  +x_1^3x_2x_4\\
  &+x_0x_1x_4Q_{\mathbf q}(x_2,x_3)
  +x_0x_1^2x_4^2
  +x_0^2x_4^2M_{\mathbf m}(x_2,x_3).
\end{aligned}
\]
The parameters range over the nonempty open subset on which
\[
  \lambda\mu(\lambda-1)(\mu-1)(\lambda-\mu)\neq0,
  \qquad
  \operatorname{Disc}(C_{\mathbf c})\neq0,
\]
\[
  \operatorname{Res}(B_{\lambda,\mu},M_{\mathbf m})\neq0.
\]
There are
\[
  2+4+3+2=11
\]
parameters.  Since
\[
  \dim\PP(W^{H_{20}})=16,
  \qquad
  \dim\bigl(C_G(H_{20})/H_{20}\bigr)=5,
\]
we obtain
\[
  \dim\Phi_{20}=16-5=11.
\]
No condition arising from the later singular-locus analysis is imposed at
the normal-form stage.

\StartCase{21}

\subsubsection*{1-PS limit}
The witness vector and the corresponding one-parameter subgroup are
\[
  r_{21}=(16,11,6,-9,-24),
  \qquad
  \lambda_{21}(t)=\diag(t^{16},t^{11},t^6,t^{-9},t^{-24}),
  \quad t\in\Gm.
\]
Since the entries of \(r_{21}\) sum to zero, \(\lambda_{21}\) is a
one-parameter subgroup of \(G=\SL(5)\).  Let \(f_{21}\) be a general
quintic supported on
\[
  S_{21}=I(r_{21})_{\geq0}.
\]
With the action convention of \eqref{eq:quintic-monomial-weight}, its affine
one-parameter-subgroup limit is the weight-zero part
\[
  \phi_{21}:=\lim_{t\to0}\lambda_{21}(t)\cdot f_{21}.
\]
Put
\[
  H_{21}:=\lambda_{21}(\Gm).
\]
Solving the degree-five weight-zero equations gives
\[
\begin{aligned}
  W^{H_{21}}=\langle{}&
  x_2^3x_3^2,
  \ x_2^4x_4,
  \ x_1^3x_3x_4,
  \ x_0x_1x_3^3,\\
  &x_0x_1x_2x_3x_4,
  \ x_0^3x_4^2
  \rangle,
  \qquad
  \dim W^{H_{21}}=6.
\end{aligned}
\]
Thus a general limit is
\[
\begin{aligned}
  \phi_{21}={}&
  a_1x_2^3x_3^2
  +a_2x_2^4x_4
  +a_3x_1^3x_3x_4
  +a_4x_0x_1x_3^3\\
  &+a_5x_0x_1x_2x_3x_4
  +a_6x_0^3x_4^2,
  \qquad
  a_i\in\C^\times.
\end{aligned}
\]
The five weights of \(H_{21}\) on \(\langle x_0,\ldots,x_4\rangle\) are
pairwise distinct.  Hence the centralizer is the maximal torus:
\[
  C_G(H_{21})=T.
\]

\subsubsection*{Polystability}
Let \(v_1,\ldots,v_6\) be the exponent vectors of the six monomials above,
ordered as displayed, and put
\[
  \eta=(1,1,1,1,1),
  \qquad
  \delta_i:=v_i-\eta.
\]
The effective torus is \(T/H_{21}\).  Its nonzero character weights satisfy
\[
  \delta_1+\delta_2+2\delta_3+\delta_4+2\delta_6=0,
\]
equivalently,
\[
  v_1+v_2+2v_3+v_4+2v_6=7\eta.
\]
All coefficients in this relation are positive, and
\[
  \operatorname{rank}\langle\delta_1,\ldots,\delta_6\rangle=3
  =\dim(T/H_{21}).
\]
Therefore
\[
  0\in\operatorname{relint}
  \operatorname{Conv}\{\delta_1,\ldots,\delta_6\}.
\]
Since all coefficients of \(\phi_{21}\) are nonzero, the affine
convex-hull criterion shows that \((T/H_{21})\cdot\phi_{21}\) is closed in
\(W^{H_{21}}\).  Affine Luna reduction then implies that
\[
  G\cdot\phi_{21}
\]
is closed in \(W\).  Hence \([\phi_{21}]\) is polystable.

\subsubsection*{Normal form}
The full-support chart is
\[
  \PP(W^{H_{21}})^\circ
  =\{[a_1:\cdots:a_6]\mid a_i\neq0\}
  \cong(\C^\times)^5.
\]
The effective torus \(T/H_{21}\) has dimension \(3\), and the rank
calculation above shows that its general stabilizer is finite.  The character
differences of the first four monomials have rank \(3\), so diagonal torus
scaling together with projective rescaling normalizes their coefficients to
\(1\).  Every general limit is therefore equivalent to a member of
\[
\begin{aligned}
  \nf_{21}(\alpha,\beta)
  ={}&x_2^3x_3^2
  +x_2^4x_4
  +x_1^3x_3x_4
  +x_0x_1x_3^3\\
  &+\alpha x_0x_1x_2x_3x_4
  +\beta x_0^3x_4^2,
  \qquad
  (\alpha,\beta)\in(\C^\times)^2.
\end{aligned}
\]
Consequently,
\[
  \dim\Phi_{21}
  =\dim\PP(W^{H_{21}})^\circ-\dim(T/H_{21})
  =5-3
  =2.
\]
No condition arising from the later singular-locus analysis is imposed at
the normal-form stage.

\StartCase{22}

\subsubsection*{1-PS limit}
The witness vector and the corresponding one-parameter subgroup are
\[
  r_{22}=(16,6,1,-4,-19),
  \qquad
  \lambda_{22}(t)=\diag(t^{16},t^6,t,t^{-4},t^{-19}),
  \quad t\in\Gm.
\]
Since the entries of \(r_{22}\) sum to zero, \(\lambda_{22}\) is a
one-parameter subgroup of \(G=\SL(5)\).  Let \(f_{22}\) be a general
quintic supported on
\[
  S_{22}=I(r_{22})_{\geq0}.
\]
With the action convention of \eqref{eq:quintic-monomial-weight}, its affine
one-parameter-subgroup limit is the weight-zero part
\[
  \phi_{22}:=\lim_{t\to0}\lambda_{22}(t)\cdot f_{22}.
\]
Put
\[
  H_{22}:=\lambda_{22}(\Gm).
\]
Solving the degree-five weight-zero equations gives
\[
\begin{aligned}
  W^{H_{22}}=\langle{}&
  x_2^4x_3,
  \ x_1x_2^2x_3^2,
  \ x_1^2x_3^3,
  \ x_1^3x_2x_4,\\
  &x_0x_3^4,
  \ x_0x_2^3x_4,
  \ x_0x_1x_2x_3x_4,
  \ x_0^2x_1x_4^2
  \rangle,
  \qquad
  \dim W^{H_{22}}=8.
\end{aligned}
\]
Thus a general limit is
\[
\begin{aligned}
  \phi_{22}={}&
  a_1x_2^4x_3
  +a_2x_1x_2^2x_3^2
  +a_3x_1^2x_3^3
  +a_4x_1^3x_2x_4\\
  &+a_5x_0x_3^4
  +a_6x_0x_2^3x_4
  +a_7x_0x_1x_2x_3x_4
  +a_8x_0^2x_1x_4^2,
  \qquad
  a_i\in\C^\times.
\end{aligned}
\]
The five weights of \(H_{22}\) on \(\langle x_0,\ldots,x_4\rangle\) are
pairwise distinct.  Hence the centralizer is the maximal torus:
\[
  C_G(H_{22})=T.
\]

\subsubsection*{Polystability}
Let \(v_1,\ldots,v_8\) be the exponent vectors of the eight monomials above,
ordered as displayed, and put
\[
  \eta=(1,1,1,1,1),
  \qquad
  \delta_i:=v_i-\eta.
\]
The effective torus is \(T/H_{22}\).  Its character weights satisfy the
positive relation
\[
  4\delta_1+5\delta_2+5\delta_3+4\delta_4
  +4\delta_5+5\delta_6+5\delta_7+18\delta_8=0,
\]
equivalently,
\[
  4v_1+5v_2+5v_3+4v_4+4v_5+5v_6+5v_7+18v_8=50\eta.
\]
Moreover,
\[
  \operatorname{rank}\langle\delta_1,\ldots,\delta_8\rangle=3
  =\dim(T/H_{22}).
\]
Therefore
\[
  0\in\operatorname{relint}
  \operatorname{Conv}\{\delta_1,\ldots,\delta_8\}.
\]
Since all coefficients of \(\phi_{22}\) are nonzero, the affine
convex-hull criterion shows that \((T/H_{22})\cdot\phi_{22}\) is closed in
\(W^{H_{22}}\).  Affine Luna reduction then implies that
\[
  G\cdot\phi_{22}
\]
is closed in \(W\).  Hence \([\phi_{22}]\) is polystable.

\subsubsection*{Normal form}
The full-support chart is
\[
  \PP(W^{H_{22}})^\circ
  =\{[a_1:\cdots:a_8]\mid a_i\neq0\}
  \cong(\C^\times)^7.
\]
The effective torus \(T/H_{22}\) has dimension \(3\), and the rank
calculation above shows that its general stabilizer is finite.  The three
character differences obtained from
\[
  x_1x_2^2x_3^2,
  \qquad
  x_1^3x_2x_4,
  \qquad
  x_0x_3^4
\]
relative to \(x_2^4x_3\) have rank \(3\).  Hence diagonal torus scaling,
together with projective rescaling, normalizes the coefficients of these four
monomials to \(1\).  Every general limit is therefore equivalent to a member
of
\[
\begin{aligned}
  \nf_{22}(\alpha,\beta,\gamma,\delta)
  ={}&x_2^4x_3
  +x_1x_2^2x_3^2
  +\alpha x_1^2x_3^3
  +x_1^3x_2x_4
  +x_0x_3^4\\
  &+\beta x_0x_2^3x_4
  +\gamma x_0x_1x_2x_3x_4
  +\delta x_0^2x_1x_4^2,
\end{aligned}
\]
where
\[
  (\alpha,\beta,\gamma,\delta)\in(\C^\times)^4.
\]
Consequently,
\[
  \dim\Phi_{22}
  =\dim\PP(W^{H_{22}})^\circ-\dim(T/H_{22})
  =7-3
  =4.
\]
No condition arising from the later singular-locus analysis is imposed at
the normal-form stage.

\StartCase{23}

\subsubsection*{1-PS limit}
The witness vector and the corresponding one-parameter subgroup are
\[
  r_{23}=(4,-1,-1,-1,-1),
  \qquad
  \lambda_{23}(t)=\diag(t^4,t^{-1},t^{-1},t^{-1},t^{-1}),
  \quad t\in\Gm.
\]
Since the entries of \(r_{23}\) sum to zero, \(\lambda_{23}\) is a
one-parameter subgroup of \(G=\SL(5)\).  Let \(f_{23}\) be a general
quintic supported on
\[
  S_{23}=I(r_{23})_{\geq0}.
\]
For a degree-five monomial \(x^u\), one has
\[
  r_{23}\cdot u
  =4u_0-u_1-u_2-u_3-u_4
  =5u_0-5.
\]
Hence
\[
  I(r_{23})_{\geq0}=\{u\in I\mid u_0\geq1\},
  \qquad
  I(r_{23})_{=0}=\{u\in I\mid u_0=1\}.
\]
Therefore the affine one-parameter-subgroup limit is
\[
  \phi_{23}
  :=\lim_{t\to0}\lambda_{23}(t)\cdot f_{23}
  =x_0B_4(x_1,x_2,x_3,x_4),
\]
where, with
\[
  U:=\langle x_1,x_2,x_3,x_4\rangle,
\]
one has \(B_4\in\operatorname{Sym}^4U\).  We work on the nonempty
Zariski-open locus on which \([B_4]\in\PP(\operatorname{Sym}^4U)\) is
stable for the \(\SL(U)\)-action.

Put
\[
  H_{23}:=\lambda_{23}(\Gm).
\]
The weight \(-1\) occurs with multiplicity four on \(U\), so
\[
  C_G(H_{23})
  =
  \left\{
    \diag(\mu_0)\oplus A
    \ \middle|\
    A\in\operatorname{GL}(U),\ \mu_0\det A=1
  \right\}.
\]
Thus the centralizer is non-toric, and
\[
  W^{H_{23}}=x_0\operatorname{Sym}^4U,
  \qquad
  \dim W^{H_{23}}=35.
\]
Moreover, \(H_{23}\) acts trivially on \(W^{H_{23}}\), and the effective
residual action is
\[
  C_G(H_{23})/H_{23}\simeq\PGL(U).
\]

\subsubsection*{Polystability}
Set
\[
  R:=C_G(H_{23}).
\]
By affine Luna reduction, it is enough to prove that
\[
  R\cdot\phi_{23}
\]
is closed in \(W^{H_{23}}\).  We apply the Casimiro--Florentino
criterion.

Let \(\lambda\in\Lambda_{\phi_{23}}^R\).  After conjugation in \(R\), we
may assume that \(\lambda\) is diagonal on \(U\).  In a corresponding
eigenbasis \(X_1,\ldots,X_4\), write
\[
  \wt_\lambda(x_0,X_1,X_2,X_3,X_4)
  =(a,b_1,b_2,b_3,b_4),
\]
so that
\[
  a+b_1+b_2+b_3+b_4=0.
\]
Define
\[
  c_i:=4b_i+a
  \qquad(1\leq i\leq4).
\]
Then \(c_1+\cdots+c_4=0\), and hence
\[
  \nu(t):=\diag(t^{c_1},t^{c_2},t^{c_3},t^{c_4})
\]
is a one-parameter subgroup of \(\SL(U)\).  For every quartic monomial
\(X^u=X_1^{u_1}\cdots X_4^{u_4}\) one has
\[
  4\,\wt_\lambda(x_0X^u)
  =u_1c_1+u_2c_2+u_3c_3+u_4c_4.
\]
Since \(\lambda\in\Lambda_{\phi_{23}}^R\), every monomial occurring in
\(B_4\) has nonnegative \(\nu\)-weight.  The stability of \([B_4]\) for
\(\SL(U)\) therefore forces \(\nu\) to be trivial.  Thus
\[
  c_1=c_2=c_3=c_4=0,
\]
and consequently there is an integer \(m\) such that
\[
  (a,b_1,b_2,b_3,b_4)=m(4,-1,-1,-1,-1).
\]
Hence
\[
  \Lambda_{\phi_{23}}^R
  =\{\lambda_{23}(t^m)\mid m\in\mathbb Z\}.
\]
This set is symmetric, since it contains the inverse one-parameter subgroup
for every one of its elements.  The Casimiro--Florentino criterion shows
that \(R\cdot\phi_{23}\) is closed in \(W^{H_{23}}\).  Affine Luna
reduction then gives that
\[
  G\cdot\phi_{23}
\]
is closed in \(W\).  Hence \([\phi_{23}]\) is polystable.

\subsubsection*{Normal form}
The closed-orbit normal form is
\[
  \nf_{23}(B_4)
  =x_0B_4(x_1,x_2,x_3,x_4),
\]
where \([B_4]\) belongs to the stable locus
\[
  \PP(\operatorname{Sym}^4U)^s.
\]
Under the effective centralizer action, two such normal forms determine the
same quotient point precisely when their quartic forms determine the same
point of
\[
  \PP(\operatorname{Sym}^4U)^s/\PGL(U).
\]
Since a stable quartic has finite stabilizer for the effective
\(\PGL(U)\)-action,
\[
  \dim\PP(\operatorname{Sym}^4U)=34,
  \qquad
  \dim\PGL(U)=15.
\]
Consequently,
\[
  \dim\Phi_{23}=34-15=19.
\]
No smoothness condition on \(V(B_4)\subset\PP^3\), or any other condition
arising from the later singular-locus analysis, is imposed at the
normal-form stage.

\StartCase{24}

\subsubsection*{1-PS limit}
The witness vector and the corresponding one-parameter subgroup are
\[
  r_{24}=(14,9,4,-6,-21),
  \qquad
  \lambda_{24}(t)=\diag(t^{14},t^9,t^4,t^{-6},t^{-21}),
  \quad t\in\Gm.
\]
Since the entries of \(r_{24}\) sum to zero, \(\lambda_{24}\) is a
one-parameter subgroup of \(G=\SL(5)\).  Let \(f_{24}\) be a general
quintic supported on
\[
  S_{24}=I(r_{24})_{\geq0}.
\]
Its affine one-parameter-subgroup limit is the weight-zero part
\[
  \phi_{24}:=\lim_{t\to0}\lambda_{24}(t)\cdot f_{24}.
\]
Put
\[
  H_{24}:=\lambda_{24}(\Gm).
\]
Solving the degree-five weight-zero equations gives
\[
\begin{aligned}
  W^{H_{24}}=\langle{}&
  x_2^3x_3^2,
  \ x_1x_2^3x_4,
  \ x_1^2x_3^3,
  \ x_1^3x_3x_4,\\
  &x_0x_2x_3^3,
  \ x_0x_1x_2x_3x_4,
  \ x_0^3x_4^2
  \rangle,
  \qquad
  \dim W^{H_{24}}=7.
\end{aligned}
\]
Thus a general limit is
\[
\begin{aligned}
  \phi_{24}={}&
  a_1x_2^3x_3^2
  +a_2x_1x_2^3x_4
  +a_3x_1^2x_3^3
  +a_4x_1^3x_3x_4\\
  &+a_5x_0x_2x_3^3
  +a_6x_0x_1x_2x_3x_4
  +a_7x_0^3x_4^2,
  \qquad
  a_i\in\C^\times.
\end{aligned}
\]
The five weights of \(H_{24}\) on \(\langle x_0,\ldots,x_4\rangle\) are
pairwise distinct.  Hence
\[
  C_G(H_{24})=T.
\]

\subsubsection*{Polystability}
Let \(v_1,\ldots,v_7\) be the exponent vectors of the seven monomials above,
ordered as displayed, and put
\[
  \eta=(1,1,1,1,1),
  \qquad
  \delta_i:=v_i-\eta.
\]
The effective torus is \(T/H_{24}\).  Its character weights satisfy the
positive relation
\[
  \delta_1+2\delta_2+\delta_3+2\delta_4
  +\delta_5+\delta_6+3\delta_7=0,
\]
equivalently,
\[
  v_1+2v_2+v_3+2v_4+v_5+v_6+3v_7=11\eta.
\]
Moreover,
\[
  \operatorname{rank}\langle\delta_1,\ldots,\delta_7\rangle
  =3
  =\dim(T/H_{24}).
\]
Therefore
\[
  0\in\operatorname{relint}
  \operatorname{Conv}\{\delta_1,\ldots,\delta_7\}.
\]
Since all coefficients of \(\phi_{24}\) are nonzero, the affine
convex-hull criterion shows that \((T/H_{24})\cdot\phi_{24}\) is closed in
\(W^{H_{24}}\).  The quotient \(N_G(H_{24})/C_G(H_{24})\) is finite, so the
normalizer orbit is also closed.  Affine Luna reduction then implies that
\[
  G\cdot\phi_{24}
\]
is closed in \(W\).  Hence \([\phi_{24}]\) is polystable.

\subsubsection*{Normal form}
The full-support chart is
\[
  \PP(W^{H_{24}})^\circ
  =\{[a_1:\cdots:a_7]\mid a_i\neq0\}
  \cong(\C^\times)^6.
\]
The effective torus \(T/H_{24}\) has dimension \(3\), and the rank
calculation above shows that its general stabilizer is finite.  Using the
diagonal torus together with projective rescaling, we normalize the
coefficients of
\[
  x_2^3x_3^2,
  \qquad
  x_1x_2^3x_4,
  \qquad
  x_1^3x_3x_4,
  \qquad
  x_0x_2x_3^3
\]
to be equal to \(1\).  Every general limit is therefore equivalent to a
member of
\[
\begin{aligned}
  \nf_{24}(\alpha,\beta,\gamma)
  ={}&x_2^3x_3^2
  +x_1x_2^3x_4
  +\alpha x_1^2x_3^3
  +x_1^3x_3x_4\\
  &+x_0x_2x_3^3
  +\beta x_0x_1x_2x_3x_4
  +\gamma x_0^3x_4^2,
\end{aligned}
\]
where
\[
  (\alpha,\beta,\gamma)\in(\C^\times)^3.
\]
Consequently,
\[
  \dim\Phi_{24}
  =\dim\PP(W^{H_{24}})^\circ-\dim(T/H_{24})
  =6-3
  =3.
\]
No condition arising from the later singular-locus analysis is imposed at
the normal-form stage.

\StartCase{25}

\subsubsection*{1-PS limit}
The witness vector and the corresponding one-parameter subgroup are
\[
  r_{25}=(11,6,1,-4,-14),
  \qquad
  \lambda_{25}(t)=\diag(t^{11},t^6,t,t^{-4},t^{-14}),
  \quad t\in\Gm.
\]
Since the entries of \(r_{25}\) sum to zero, \(\lambda_{25}\) is a
one-parameter subgroup of \(G=\SL(5)\).  Let \(f_{25}\) be a general
quintic supported on
\[
  S_{25}=I(r_{25})_{\geq0}.
\]
Its affine one-parameter-subgroup limit is the weight-zero part
\[
  \phi_{25}:=\lim_{t\to0}\lambda_{25}(t)\cdot f_{25}.
\]
Put
\[
  H_{25}:=\lambda_{25}(\Gm).
\]
Solving the degree-five weight-zero equations gives
\[
\begin{aligned}
  W^{H_{25}}=\langle{}&
  x_2^4x_3,
  \ x_1x_2^2x_3^2,
  \ x_1^2x_3^3,
  \ x_1^2x_2^2x_4,
  \ x_1^3x_3x_4,\\
  &x_0x_2x_3^3,
  \ x_0x_2^3x_4,
  \ x_0x_1x_2x_3x_4,
  \ x_0^2x_3^2x_4,
  \ x_0^2x_1x_4^2
  \rangle,
  \qquad
  \dim W^{H_{25}}=10.
\end{aligned}
\]
Thus a general limit is
\[
\begin{aligned}
  \phi_{25}={}&
  a_1x_2^4x_3
  +a_2x_1x_2^2x_3^2
  +a_3x_1^2x_3^3
  +a_4x_1^2x_2^2x_4
  +a_5x_1^3x_3x_4\\
  &+a_6x_0x_2x_3^3
  +a_7x_0x_2^3x_4
  +a_8x_0x_1x_2x_3x_4
  +a_9x_0^2x_3^2x_4
  +a_{10}x_0^2x_1x_4^2,
\end{aligned}
\]
where \(a_i\in\C^\times\).  The five weights of \(H_{25}\) on
\(\langle x_0,\ldots,x_4\rangle\) are pairwise distinct, and hence
\[
  C_G(H_{25})=T.
\]

\subsubsection*{Polystability}
Let \(v_1,\ldots,v_{10}\) be the exponent vectors of the ten monomials
above, ordered as displayed, and put
\[
  \eta=(1,1,1,1,1),
  \qquad
  \delta_i:=v_i-\eta.
\]
The effective torus is \(T/H_{25}\).  Its character weights satisfy the
positive relation
\[
  \delta_1+\delta_2+\delta_3+\delta_4+\delta_5
  +\delta_6+\delta_7+\delta_8+\delta_9+4\delta_{10}=0,
\]
equivalently,
\[
  v_1+v_2+v_3+v_4+v_5+v_6+v_7+v_8+v_9+4v_{10}=13\eta.
\]
Moreover,
\[
  \operatorname{rank}\langle\delta_1,\ldots,\delta_{10}\rangle
  =3
  =\dim(T/H_{25}).
\]
Therefore
\[
  0\in\operatorname{relint}
  \operatorname{Conv}\{\delta_1,\ldots,\delta_{10}\}.
\]
Since all coefficients of \(\phi_{25}\) are nonzero, the affine
convex-hull criterion shows that \((T/H_{25})\cdot\phi_{25}\) is closed in
\(W^{H_{25}}\).  Affine Luna reduction then implies that
\[
  G\cdot\phi_{25}
\]
is closed in \(W\).  Hence \([\phi_{25}]\) is polystable.

\subsubsection*{Normal form}
The full-support chart is
\[
  \PP(W^{H_{25}})^\circ
  =\{[a_1:\cdots:a_{10}]\mid a_i\neq0\}
  \cong(\C^\times)^9.
\]
The effective torus \(T/H_{25}\) has dimension \(3\), and the rank
calculation above shows that its general stabilizer is finite.  Using the
diagonal torus together with projective rescaling, we normalize the
coefficients of
\[
  x_2^4x_3,
  \qquad
  x_1^2x_2^2x_4,
  \qquad
  x_1^3x_3x_4,
  \qquad
  x_0x_2x_3^3
\]
to be equal to \(1\).  Every general limit is therefore equivalent to a
member of
\[
\begin{aligned}
  \nf_{25}(\alpha,\beta,\gamma,\delta,\varepsilon,\zeta)
  ={}&x_2^4x_3
  +\alpha x_1x_2^2x_3^2
  +\beta x_1^2x_3^3
  +x_1^2x_2^2x_4\\
  &+x_1^3x_3x_4
  +x_0x_2x_3^3
  +\gamma x_0x_2^3x_4
  +\delta x_0x_1x_2x_3x_4\\
  &+\varepsilon x_0^2x_3^2x_4
  +\zeta x_0^2x_1x_4^2,
\end{aligned}
\]
where
\[
  (\alpha,\beta,\gamma,\delta,\varepsilon,\zeta)
  \in(\C^\times)^6.
\]
Consequently,
\[
  \dim\Phi_{25}
  =\dim\PP(W^{H_{25}})^\circ-\dim(T/H_{25})
  =9-3
  =6.
\]
No condition arising from the later singular-locus analysis is imposed at
the normal-form stage.

\StartCase{26}

\subsubsection*{1-PS limit}
The witness vector and the corresponding one-parameter subgroup are
\[
  r_{26}=(24,9,4,-1,-36),
  \qquad
  \lambda_{26}(t)=\diag(t^{24},t^9,t^4,t^{-1},t^{-36}),
  \quad t\in\Gm.
\]
Since the entries of \(r_{26}\) sum to zero, \(\lambda_{26}\) is a
one-parameter subgroup of \(G=\SL(5)\).  Let \(f_{26}\) be a general
quintic supported on
\[
  S_{26}=I(r_{26})_{\geq0}.
\]
Its affine one-parameter-subgroup limit is the weight-zero part
\[
  \phi_{26}:=\lim_{t\to0}\lambda_{26}(t)\cdot f_{26}.
\]
Put
\[
  H_{26}:=\lambda_{26}(\Gm).
\]
Solving the degree-five weight-zero equations gives
\[
  W^{H_{26}}
  =\left\langle
  x_2x_3^4,
  \ x_1^4x_4,
  \ x_0x_2^3x_4,
  \ x_0x_1x_2x_3x_4,
  \ x_0^3x_4^2
  \right\rangle,
  \qquad
  \dim W^{H_{26}}=5.
\]
Thus a general limit is
\[
  \phi_{26}
  =a_1x_2x_3^4
  +a_2x_1^4x_4
  +a_3x_0x_2^3x_4
  +a_4x_0x_1x_2x_3x_4
  +a_5x_0^3x_4^2,
\]
where \(a_i\in\C^\times\).  The five weights of \(H_{26}\) on
\(\langle x_0,\ldots,x_4\rangle\) are pairwise distinct, and hence
\[
  C_G(H_{26})=T.
\]

\subsubsection*{Polystability}
Let \(v_1,\ldots,v_5\) be the exponent vectors of the five monomials
above, ordered as displayed, and put
\[
  \eta=(1,1,1,1,1),
  \qquad
  \delta_i:=v_i-\eta.
\]
Here \(\delta_4=0\), and the four nonzero characters satisfy
\[
  \delta_1+\delta_2+\delta_3+\delta_5=0,
\]
equivalently,
\[
  v_1+v_2+v_3+v_5=4\eta.
\]
Moreover,
\[
  \operatorname{rank}\langle
  \delta_1,\delta_2,\delta_3,\delta_5
  \rangle
  =3
  =\dim(T/H_{26}).
\]
Therefore
\[
  0\in\operatorname{relint}
  \operatorname{Conv}\{\delta_1,\delta_2,\delta_3,\delta_5\}.
\]
Since all five coefficients of \(\phi_{26}\) are nonzero, the affine
convex-hull criterion shows that
\[
  (T/H_{26})\cdot\phi_{26}
\]
is closed in \(W^{H_{26}}\).  Affine Luna reduction and
\(C_G(H_{26})=T\) then imply that \(G\cdot\phi_{26}\) is closed in
\(W\).  Hence \([\phi_{26}]\) is polystable.

\subsubsection*{Normal form}
The full-support chart is
\[
  \PP(W^{H_{26}})^\circ
  =\{[a_1:\cdots:a_5]\mid a_i\neq0\}
  \cong(\C^\times)^4.
\]
The effective torus \(T/H_{26}\) has dimension \(3\), and the rank
calculation above shows that its general stabilizer is finite.  Using the
diagonal torus together with projective rescaling, we normalize the
coefficients of
\[
  x_2x_3^4,
  \qquad
  x_1^4x_4,
  \qquad
  x_0x_2^3x_4,
  \qquad
  x_0^3x_4^2
\]
to be equal to \(1\).  Every general limit is therefore equivalent to a
member of
\[
  \nf_{26}(\alpha)
  =x_2x_3^4
  +x_1^4x_4
  +x_0x_2^3x_4
  +\alpha x_0x_1x_2x_3x_4
  +x_0^3x_4^2,
  \qquad
  \alpha\in\C^\times.
\]
The character relation above shows that
\[
  \frac{a_1a_2a_3a_5}{a_4^4}
\]
is invariant under diagonal torus scaling and projective rescaling.  In
the displayed normal form it equals \(\alpha^{-4}\); hence \(\alpha\) is
determined up to the residual finite fourth-root ambiguity.

Consequently,
\[
  \dim\Phi_{26}
  =\dim\PP(W^{H_{26}})^\circ-\dim(T/H_{26})
  =4-3
  =1.
\]
No condition arising from the later singular-locus analysis is imposed at
the normal-form stage.

\StartCase{27}

\subsubsection*{1-PS limit}
The witness vector and the corresponding one-parameter subgroup are
\[
  r_{27}=(18,8,3,-2,-27),
  \qquad
  \lambda_{27}(t)=\diag(t^{18},t^8,t^3,t^{-2},t^{-27}),
  \quad t\in\Gm.
\]
Since the entries of \(r_{27}\) sum to zero, \(\lambda_{27}\) is a
one-parameter subgroup of \(G=\SL(5)\).  Let \(f_{27}\) be a general
quintic supported on
\[
  S_{27}=I(r_{27})_{\geq0}.
\]
Its affine one-parameter-subgroup limit is the weight-zero part
\[
  \phi_{27}:=\lim_{t\to0}\lambda_{27}(t)\cdot f_{27}.
\]
Put
\[
  H_{27}:=\lambda_{27}(\Gm).
\]
Solving the degree-five weight-zero equations gives
\[
  W^{H_{27}}
  =\left\langle
  x_2^2x_3^3,
  \ x_1x_3^4,
  \ x_1^3x_2x_4,
  \ x_0x_2^3x_4,
  \ x_0x_1x_2x_3x_4,
  \ x_0^3x_4^2
  \right\rangle,
  \qquad
  \dim W^{H_{27}}=6.
\]
Thus a general limit is
\[
  \phi_{27}
  =a_1x_2^2x_3^3
  +a_2x_1x_3^4
  +a_3x_1^3x_2x_4
  +a_4x_0x_2^3x_4
  +a_5x_0x_1x_2x_3x_4
  +a_6x_0^3x_4^2,
\]
where \(a_i\in\C^\times\).  The five weights of \(H_{27}\) on
\(\langle x_0,\ldots,x_4\rangle\) are pairwise distinct, and hence
\[
  C_G(H_{27})=T.
\]

\subsubsection*{Polystability}
Let \(v_1,\ldots,v_6\) be the exponent vectors of the six monomials
above, ordered as displayed, and put
\[
  \eta=(1,1,1,1,1),
  \qquad
  \delta_i:=v_i-\eta.
\]
Here \(\delta_5=0\), and the five nonzero characters satisfy
\[
  \delta_1+\delta_2+2\delta_3+\delta_4+2\delta_6=0,
\]
equivalently,
\[
  v_1+v_2+2v_3+v_4+2v_6=7\eta.
\]
Moreover,
\[
  \operatorname{rank}\langle
  \delta_1,\delta_2,\delta_3,\delta_4,\delta_6
  \rangle
  =3
  =\dim(T/H_{27}).
\]
Therefore
\[
  0\in\operatorname{relint}
  \operatorname{Conv}\{\delta_1,\delta_2,\delta_3,\delta_4,\delta_6\}.
\]
Since all six coefficients of \(\phi_{27}\) are nonzero, the affine
convex-hull criterion shows that
\[
  (T/H_{27})\cdot\phi_{27}
\]
is closed in \(W^{H_{27}}\).  Affine Luna reduction and
\(C_G(H_{27})=T\) then imply that \(G\cdot\phi_{27}\) is closed in
\(W\).  Hence \([\phi_{27}]\) is polystable.

\subsubsection*{Normal form}
The full-support chart is
\[
  \PP(W^{H_{27}})^\circ
  =\{[a_1:\cdots:a_6]\mid a_i\neq0\}
  \cong(\C^\times)^5.
\]
The effective torus \(T/H_{27}\) has dimension \(3\), and the rank
calculation above shows that its general stabilizer is finite.  Using the
diagonal torus together with projective rescaling, we normalize the
coefficients of
\[
  x_2^2x_3^3,
  \qquad
  x_1x_3^4,
  \qquad
  x_1^3x_2x_4,
  \qquad
  x_0x_2^3x_4
\]
to be equal to \(1\).  Every general limit is therefore equivalent to a
member of
\[
  \nf_{27}(\alpha,\beta)
  =x_2^2x_3^3
  +x_1x_3^4
  +x_1^3x_2x_4
  +x_0x_2^3x_4
  +\alpha x_0x_1x_2x_3x_4
  +\beta x_0^3x_4^2,
\]
where
\[
  (\alpha,\beta)\in(\C^\times)^2.
\]
Consequently,
\[
  \dim\Phi_{27}
  =\dim\PP(W^{H_{27}})^\circ-\dim(T/H_{27})
  =5-3
  =2.
\]
No condition arising from the later singular-locus analysis is imposed at
the normal-form stage.

\StartCase{28}

\subsubsection*{1-PS limit}
The witness vector and the corresponding one-parameter subgroup are
\[
  r_{28}=(12,7,2,-3,-18),
  \qquad
  \lambda_{28}(t)=\diag(t^{12},t^7,t^2,t^{-3},t^{-18}),
  \quad t\in\Gm.
\]
Since the entries of \(r_{28}\) sum to zero, \(\lambda_{28}\) is a
one-parameter subgroup of \(G=\SL(5)\).  Let \(f_{28}\) be a general
quintic supported on
\[
  S_{28}=I(r_{28})_{\geq0}.
\]
Its affine one-parameter-subgroup limit is the weight-zero part
\[
  \phi_{28}:=\lim_{t\to0}\lambda_{28}(t)\cdot f_{28}.
\]
Put
\[
  H_{28}:=\lambda_{28}(\Gm).
\]
Solving the degree-five weight-zero equations gives
\[
  \begin{aligned}
  W^{H_{28}}=\langle{}&
  x_2^3x_3^2,
  \ x_1x_2x_3^3,
  \ x_1^2x_2^2x_4,
  \ x_1^3x_3x_4,
  \ x_0x_3^4,\\
  &x_0x_2^3x_4,
  \ x_0x_1x_2x_3x_4,
  \ x_0^2x_3^2x_4,
  \ x_0^3x_4^2
  \rangle,
  \qquad
  \dim W^{H_{28}}=9.
  \end{aligned}
\]
Thus a general limit is
\[
  \begin{aligned}
  \phi_{28}={}&
  a_1x_2^3x_3^2
  +a_2x_1x_2x_3^3
  +a_3x_1^2x_2^2x_4
  +a_4x_1^3x_3x_4
  +a_5x_0x_3^4\\
  &+a_6x_0x_2^3x_4
  +a_7x_0x_1x_2x_3x_4
  +a_8x_0^2x_3^2x_4
  +a_9x_0^3x_4^2,
  \end{aligned}
\]
where \(a_i\in\C^\times\).  The five weights of \(H_{28}\) on
\(\langle x_0,\ldots,x_4\rangle\) are pairwise distinct, and hence
\[
  C_G(H_{28})=T.
\]

\subsubsection*{Polystability}
Let \(v_1,\ldots,v_9\) be the exponent vectors of the nine monomials
above, ordered as displayed, and put
\[
  \eta=(1,1,1,1,1),
  \qquad
  \delta_i:=v_i-\eta.
\]
Here \(\delta_7=0\), and the eight nonzero characters satisfy
\[
  \delta_1+\delta_2+3\delta_3+2\delta_4+\delta_5+\delta_6
  +\delta_8+3\delta_9=0,
\]
equivalently,
\[
  v_1+v_2+3v_3+2v_4+v_5+v_6+v_8+3v_9=13\eta.
\]
Moreover,
\[
  \operatorname{rank}\langle
  \delta_1,\delta_2,\delta_3,\delta_4,
  \delta_5,\delta_6,\delta_8,\delta_9
  \rangle
  =3
  =\dim(T/H_{28}).
\]
Therefore
\[
  0\in\operatorname{relint}
  \operatorname{Conv}
  \{\delta_1,\delta_2,\delta_3,\delta_4,
    \delta_5,\delta_6,\delta_8,\delta_9\}.
\]
Since all nine coefficients of \(\phi_{28}\) are nonzero, the affine
convex-hull criterion shows that
\[
  (T/H_{28})\cdot\phi_{28}
\]
is closed in \(W^{H_{28}}\).  Affine Luna reduction and
\(C_G(H_{28})=T\) then imply that \(G\cdot\phi_{28}\) is closed in
\(W\).  Hence \([\phi_{28}]\) is polystable.

\subsubsection*{Normal form}
The full-support chart is
\[
  \PP(W^{H_{28}})^\circ
  =\{[a_1:\cdots:a_9]\mid a_i\neq0\}
  \cong(\C^\times)^8.
\]
The effective torus \(T/H_{28}\) has dimension \(3\), and the rank
calculation above shows that its general stabilizer is finite.  Using the
diagonal torus together with projective rescaling, we normalize the
coefficients of
\[
  x_2^3x_3^2,
  \qquad
  x_1x_2x_3^3,
  \qquad
  x_1^2x_2^2x_4,
  \qquad
  x_0x_3^4
\]
to be equal to \(1\).  Every general limit is therefore equivalent to a
member of
\[
  \begin{aligned}
  \nf_{28}(\alpha,\beta,\gamma,\delta,\varepsilon)
  ={}&x_2^3x_3^2
  +x_1x_2x_3^3
  +x_1^2x_2^2x_4
  +\beta x_1^3x_3x_4
  +x_0x_3^4\\
  &+\alpha x_0x_2^3x_4
  +\gamma x_0x_1x_2x_3x_4
  +\delta x_0^2x_3^2x_4
  +\varepsilon x_0^3x_4^2,
  \end{aligned}
\]
where
\[
  (\alpha,\beta,\gamma,\delta,\varepsilon)
  \in(\C^\times)^5.
\]
Consequently,
\[
  \dim\Phi_{28}
  =\dim\PP(W^{H_{28}})^\circ-\dim(T/H_{28})
  =8-3
  =5.
\]
No condition arising from the later singular-locus analysis is imposed at
the normal-form stage.

\StartCase{29}

\subsubsection*{1-PS limit}
The witness vector and the corresponding one-parameter subgroup are
\[
  r_{29}=(6,1,1,-4,-4),
  \qquad
  \lambda_{29}(t)=\diag(t^6,t,t,t^{-4},t^{-4}),
  \quad t\in\Gm.
\]
Since the entries of \(r_{29}\) sum to zero, \(\lambda_{29}\) is a
one-parameter subgroup of \(G=\SL(5)\).  Let \(f_{29}\) be a general
quintic supported on
\[
  S_{29}=I(r_{29})_{\geq0}.
\]
Its affine one-parameter-subgroup limit is the weight-zero part
\[
  \phi_{29}:=\lim_{t\to0}\lambda_{29}(t)\cdot f_{29}.
\]
Put
\[
  U:=\langle x_1,x_2\rangle,
  \qquad
  V:=\langle x_3,x_4\rangle,
  \qquad
  H_{29}:=\lambda_{29}(\Gm).
\]
Solving the degree-five weight-zero equations gives
\[
  W^{H_{29}}
  =
  \operatorname{Sym}^4U\otimes V
  \oplus
  x_0\operatorname{Sym}^2U\otimes\operatorname{Sym}^2V
  \oplus
  x_0^2\operatorname{Sym}^3V,
\]
and hence
\[
  \dim W^{H_{29}}=23.
\]
Thus a general limit has the form
\[
  \phi_{29}
  =x_3A_4(x_1,x_2)+x_4C_4(x_1,x_2)
  +x_0Q_{2,2}(x_1,x_2;x_3,x_4)
  +x_0^2B_3(x_3,x_4),
\]
where
\[
  A_4,C_4\in\operatorname{Sym}^4U,
  \qquad
  Q_{2,2}\in\operatorname{Sym}^2U\otimes\operatorname{Sym}^2V,
  \qquad
  B_3\in\operatorname{Sym}^3V.
\]
We work on the nonempty Zariski-open locus
\[
  \operatorname{Res}(A_4,C_4)\neq0,
  \qquad
  \operatorname{Disc}(B_3)\neq0.
\]
The first condition says that the pencil spanned by \(A_4\) and \(C_4\)
is base-point-free on \(\PP(U)\), while the second says that \(B_3\) is
square-free.

The weights \(1\) and \(-4\) each occur with multiplicity two.  Therefore
the centralizer is non-toric:
\[
  C_G(H_{29})
  =
  \left\{
    \diag(\mu_0)\oplus A\oplus B
    \ \middle|\
    A\in\GL(U),\ B\in\GL(V),\
    \mu_0\det A\det B=1
  \right\}.
\]
Consequently,
\[
  \dim C_G(H_{29})=8,
  \qquad
  \dim\bigl(C_G(H_{29})/H_{29}\bigr)=7.
\]

\subsubsection*{Polystability}
Set
\[
  R:=C_G(H_{29}).
\]
By affine Luna reduction, it is enough to prove that
\(R\cdot\phi_{29}\) is closed in \(W^{H_{29}}\).  Let
\(\rho\in\Lambda_{\phi_{29}}^R\).  After conjugation in the two
non-abelian blocks, the weights of \(\rho\) may be written as
\[
  \wt_\rho(x_0,X,Y,Z_1,Z_2)
  =(a,s+n,s-n,t+v,t-v),
\]
where \(X,Y\) and \(Z_1,Z_2\) are eigenbases of \(U\) and \(V\), and
\[
  a+2s+2t=0.
\]
Define the averaged block weights
\[
  p_F:=4s+t,
  \qquad
  p_B:=2a+3t.
\]
The nonvanishing of \(\operatorname{Res}(A_4,C_4)\), together with
nonnegativity of the occurring monomial weights, gives
\[
  p_F\geq0.
\]
Similarly, \(\operatorname{Disc}(B_3)\neq0\) gives
\[
  p_B\geq0.
\]
These quantities satisfy
\[
  p_F+p_B
  =(4s+t)+(2a+3t)
  =2(a+2s+2t)
  =0.
\]
Hence
\[
  p_F=p_B=0.
\]
Since \(p_B=0\), a nonzero relative weight \(v\) would force the binary
cubic \(B_3\) to have a multiple root.  Thus \(v=0\).  Since
\(p_F=0\), a nonzero relative weight \(n\) would force both \(A_4\) and
\(C_4\) to be divisible by the same square of a linear form, contradicting
\(\operatorname{Res}(A_4,C_4)\neq0\).  Therefore \(n=0\).

Solving the remaining equations gives
\[
  (a,s,t)=m(6,1,-4)
\]
for some \(m\in\mathbb Z\).  Consequently,
\[
  \Lambda_{\phi_{29}}^R
  =\{\lambda_{29}(t^m)\mid m\in\mathbb Z\}.
\]
This set is symmetric under \(m\mapsto-m\).  The
Casimiro--Florentino criterion therefore shows that
\(R\cdot\phi_{29}\) is closed in \(W^{H_{29}}\).  Luna reduction then
implies that
\[
  G\cdot\phi_{29}
\]
is closed in \(W\), so \([\phi_{29}]\) is polystable.

The same calculation shows that the identity component of the stabilizer in
\(R\) is exactly \(H_{29}\).  Hence the stabilizer for the effective action
of \(R/H_{29}\) is finite on this open locus.

\subsubsection*{Normal form}
Since \(B_3\) is square-free, the \(\GL(V)\)-action normalizes it to
\[
  B^{(29)}_3(x_3,x_4):=x_3x_4(x_3-x_4).
\]
Thus every general limit is equivalent to a member of
\[
\begin{aligned}
  \nf_{29}(A_4,C_4,Q_{2,2})
  ={}&x_3A_4(x_1,x_2)+x_4C_4(x_1,x_2)\\
  &+x_0Q_{2,2}(x_1,x_2;x_3,x_4)
  +x_0^2x_3x_4(x_3-x_4),
\end{aligned}
\]
where
\[
  A_4,C_4\in\operatorname{Sym}^4U,
  \qquad
  Q_{2,2}\in\operatorname{Sym}^2U\otimes\operatorname{Sym}^2V,
  \qquad
  \operatorname{Res}(A_4,C_4)\neq0.
\]
The residual action consists of the \(\GL(U)\)-action on
\((x_1,x_2)\), together with the finite stabilizer of
\(x_3x_4(x_3-x_4)\) in the \(V\)-variables.  Accordingly, the normal-form
parameter space has dimension
\[
  5+5+9-\dim\GL(U)=19-4=15.
\]
Equivalently, since the effective centralizer stabilizer is finite,
\[
  \dim\PP(W^{H_{29}})=22,
  \qquad
  \dim\bigl(C_G(H_{29})/H_{29}\bigr)=7,
\]
and therefore
\[
  \dim\Phi_{29}=22-7=15.
\]
No condition arising from the later singular-locus or minimal-exponent
analysis is imposed at the normal-form stage.

\StartCase{30}

\subsubsection*{1-PS limit}
The witness vector and the corresponding one-parameter subgroup are
\[
  r_{30}=(6,6,1,-4,-9),
  \qquad
  \lambda_{30}(t)=\diag(t^6,t^6,t,t^{-4},t^{-9}),
  \quad t\in\Gm.
\]
Since the entries of \(r_{30}\) sum to zero, \(\lambda_{30}\) is a
one-parameter subgroup of \(G=\SL(5)\).  Let \(f_{30}\) be a general
quintic supported on
\[
  S_{30}=I(r_{30})_{\geq0}.
\]
Its affine one-parameter-subgroup limit is the weight-zero part
\[
  \phi_{30}:=\lim_{t\to0}\lambda_{30}(t)\cdot f_{30}.
\]
Put
\[
  U:=\langle x_0,x_1\rangle,
  \qquad
  H_{30}:=\lambda_{30}(\Gm).
\]
Solving the degree-five weight-zero equations gives
\[
\begin{aligned}
  W^{H_{30}}={}&\langle x_2^4x_3\rangle
  \oplus x_2^2x_3^2U
  \oplus x_2^3x_4U\\
  &\oplus x_3^3\operatorname{Sym}^2U
  \oplus x_2x_3x_4\operatorname{Sym}^2U
  \oplus x_4^2\operatorname{Sym}^3U,
\end{aligned}
\]
and hence
\[
  \dim W^{H_{30}}=15.
\]
Thus a general limit can be written as
\[
\begin{aligned}
  \phi_{30}={}&\kappa x_2^4x_3
  +x_2^2x_3^2L_1(x_0,x_1)
  +x_2^3x_4L_2(x_0,x_1)\\
  &+x_3^3Q_1(x_0,x_1)
  +x_2x_3x_4Q_2(x_0,x_1)
  +x_4^2C_3(x_0,x_1),
\end{aligned}
\]
where \(L_1,L_2\in\operatorname{Sym}^1U\),
\(Q_1,Q_2\in\operatorname{Sym}^2U\), and
\(C_3\in\operatorname{Sym}^3U\).  We work on the nonempty Zariski-open
locus
\[
  \kappa\neq0,
  \qquad
  \operatorname{Disc}(Q_1)\neq0,
  \qquad
  \operatorname{Disc}(C_3)\neq0.
\]

The weight \(6\) occurs twice on \(U\), so the centralizer is non-toric:
\[
  C_G(H_{30})
  =
  \left\{
    A\oplus\diag(\mu_2,\mu_3,\mu_4)
    \ \middle|\
    A\in\GL(U),\
    \det A\,\mu_2\mu_3\mu_4=1
  \right\}.
\]
Consequently,
\[
  \dim C_G(H_{30})=6,
  \qquad
  \dim\bigl(C_G(H_{30})/H_{30}\bigr)=5.
\]

\subsubsection*{Polystability}
Set
\[
  R:=C_G(H_{30}).
\]
By affine Luna reduction, it is enough to prove that
\(R\cdot\phi_{30}\) is closed in \(W^{H_{30}}\).  Let
\(\rho\in\Lambda_{\phi_{30}}^R\).  After conjugation in the
\(\GL(U)\)-block, its weights may be written as
\[
  \wt_\rho(X,Y,x_2,x_3,x_4)
  =(s+n,s-n,b,c,d),
  \qquad
  2s+b+c+d=0,
\]
for a suitable eigenbasis \(X,Y\) of \(U\).

Introduce the averaged weights
\[
  \omega_0:=4b+c,
  \qquad
  \omega_Q:=2s+3c,
  \qquad
  \omega_C:=3s+2d.
\]
The nonvanishing of \(\kappa\), \(\operatorname{Disc}(Q_1)\), and
\(\operatorname{Disc}(C_3)\), together with nonnegativity of the occurring
monomial weights, gives
\[
  \omega_0\geq0,
  \qquad
  \omega_Q\geq0,
  \qquad
  \omega_C\geq0.
\]
These quantities satisfy the balancing identity
\[
  \omega_0+\omega_Q+2\omega_C
  =4(2s+b+c+d)=0.
\]
Hence
\[
  \omega_0=\omega_Q=\omega_C=0.
\]
Since \(\omega_C=0\), a nonzero relative weight \(n\) would force the
binary cubic \(C_3\) to have a multiple root, contradicting
\(\operatorname{Disc}(C_3)\neq0\).  Thus \(n=0\).  Solving the remaining
weight equations gives
\[
  (s,b,c,d)=m(6,1,-4,-9)
\]
for some \(m\in\mathbb Z\).  Therefore
\[
  \Lambda_{\phi_{30}}^R
  =\{\lambda_{30}(t^m)\mid m\in\mathbb Z\}.
\]
This set is symmetric under \(m\mapsto-m\).  The
Casimiro--Florentino criterion therefore shows that
\(R\cdot\phi_{30}\) is closed in \(W^{H_{30}}\).  Luna reduction then
implies that
\[
  G\cdot\phi_{30}
\]
is closed in \(W\), so \([\phi_{30}]\) is polystable.  The same weight
calculation shows that the stabilizer for the effective action of
\(R/H_{30}\) is finite on this open locus.

\subsubsection*{Normal form}
Since \(C_3\) is square-free, the \(\GL(U)\)-action normalizes it to
\[
  C_3(x_0,x_1)=x_0x_1(x_0-x_1).
\]
Using the remaining diagonal factors in \(C_G(H_{30})\), together with
projective rescaling, normalize the coefficient of \(x_2^4x_3\) to
\(1\).  Every general limit is then equivalent to a member of
\[
\begin{aligned}
  \nf_{30}(L_1,L_2,Q_1,Q_2)
  ={}&x_2^4x_3
  +x_2^2x_3^2L_1(x_0,x_1)
  +x_2^3x_4L_2(x_0,x_1)\\
  &+x_3^3Q_1(x_0,x_1)
  +x_2x_3x_4Q_2(x_0,x_1)
  +x_0x_1(x_0-x_1)x_4^2,
\end{aligned}
\]
where
\[
  L_1,L_2\in\operatorname{Sym}^1U,
  \qquad
  Q_1,Q_2\in\operatorname{Sym}^2U,
  \qquad
  \operatorname{Disc}(Q_1)\neq0.
\]
The displayed coefficient space has dimension
\[
  2+2+3+3=10.
\]
One residual diagonal \(\Gm\)-action preserves this form; on an affine
parameter chart it may be used to normalize one additional nonzero
coefficient.  Thus the effective normal-form parameter count is
\[
  10-1=9.
\]
Equivalently, since the effective centralizer stabilizer is finite,
\[
  \dim\PP(W^{H_{30}})=14,
  \qquad
  \dim\bigl(C_G(H_{30})/H_{30}\bigr)=5,
\]
and therefore
\[
  \dim\Phi_{30}=14-5=9.
\]
No condition arising from the later singular-locus or minimal-exponent
analysis is imposed at the normal-form stage.

\StartCase{31}

\subsubsection*{1-PS limit}
The witness vector and the corresponding one-parameter subgroup are
\[
  r_{31}=(7,2,2,-3,-8),
  \qquad
  \lambda_{31}(t)=\diag(t^7,t^2,t^2,t^{-3},t^{-8}),
  \quad t\in\Gm.
\]
Since the entries of \(r_{31}\) sum to zero, \(\lambda_{31}\) is a
one-parameter subgroup of \(G=\SL(5)\).  Let \(f_{31}\) be a general
quintic supported on
\[
  S_{31}=I(r_{31})_{\geq0}.
\]
Its affine one-parameter-subgroup limit is the weight-zero part
\[
  \phi_{31}:=\lim_{t\to0}\lambda_{31}(t)\cdot f_{31}.
\]
Put
\[
  U:=\langle x_1,x_2\rangle,
  \qquad
  H_{31}:=\lambda_{31}(\Gm).
\]
Solving the degree-five weight-zero equations gives
\[
\begin{aligned}
  W^{H_{31}}={}&
  x_4\operatorname{Sym}^4U
  \oplus x_3^2\operatorname{Sym}^3U
  \oplus x_0x_3^3U\\
  &\oplus x_0x_3x_4\operatorname{Sym}^2U
  \oplus\langle x_0^2x_3^2x_4\rangle
  \oplus x_0^2x_4^2U,
\end{aligned}
\]
and hence
\[
  \dim W^{H_{31}}=17.
\]
Thus a general limit can be written as
\[
\begin{aligned}
  \phi_{31}={}&x_4B_4(x_1,x_2)
  +x_3^2C_3(x_1,x_2)
  +x_0x_3^3L_1(x_1,x_2)\\
  &+x_0x_3x_4Q_2(x_1,x_2)
  +\kappa x_0^2x_3^2x_4
  +x_0^2x_4^2M_1(x_1,x_2),
\end{aligned}
\]
where
\[
  B_4\in\operatorname{Sym}^4U,
  \quad
  C_3\in\operatorname{Sym}^3U,
  \quad
  L_1,M_1\in U,
  \quad
  Q_2\in\operatorname{Sym}^2U.
\]
We work on the nonempty Zariski-open locus
\[
  \kappa\neq0,
  \qquad
  \operatorname{Disc}(B_4)\neq0,
  \qquad
  \operatorname{Res}(B_4,C_3)\neq0,
  \qquad
  M_1\neq0.
\]

The weight \(2\) occurs twice on \(U\), so the centralizer is non-toric:
\[
  C_G(H_{31})
  =
  \left\{
    \diag(\mu_0)\oplus A\oplus\diag(\mu_3,\mu_4)
    \ \middle|\
    A\in\GL(U),\
    \mu_0\mu_3\mu_4\det A=1
  \right\}.
\]
Consequently,
\[
  \dim C_G(H_{31})=6,
  \qquad
  \dim\bigl(C_G(H_{31})/H_{31}\bigr)=5.
\]

\subsubsection*{Polystability}
Set
\[
  R:=C_G(H_{31}).
\]
By affine Luna reduction, it is enough to prove that
\(R\cdot\phi_{31}\) is closed in \(W^{H_{31}}\).  Let
\(\rho\in\Lambda_{\phi_{31}}^R\).  After conjugation in the
\(\GL(U)\)-block, its weights may be written as
\[
  \wt_\rho(x_0,X,Y,x_3,x_4)
  =(a,s+n,s-n,b,c),
  \qquad
  a+2s+b+c=0,
\]
for a suitable eigenbasis \(X,Y\) of \(U\).

Introduce the averaged weights
\[
  p_B:=4s+c,
  \qquad
  p_C:=3s+2b,
  \qquad
  p_M:=2a+s+2c,
  \qquad
  d:=2a+2b+c.
\]
The conditions \(\operatorname{Disc}(B_4)\neq0\) and \(\kappa\neq0\),
together with nonnegativity of the occurring monomial weights, give
\[
  p_B\geq0,
  \qquad
  d\geq0.
\]
They satisfy the balancing identity
\[
  p_B+d
  =(4s+c)+(2a+2b+c)
  =2(a+2s+b+c)
  =0.
\]
Hence
\[
  p_B=d=0.
\]
Since \(p_B=0\), a nonzero relative weight \(n\) would force the binary
quartic \(B_4\) to have a multiple root.  Thus
\[
  n=0.
\]
The nonvanishing of \(\operatorname{Res}(B_4,C_3)\) now gives
\[
  3p_B+4p_C\geq0,
\]
and therefore \(p_C\geq0\).  Since \(M_1\neq0\) and \(n=0\), one also
has \(p_M\geq0\).  The equalities \(p_B=d=0\) imply
\[
  p_M=-p_C.
\]
Consequently,
\[
  p_C=p_M=0.
\]
Solving the remaining weight equations gives
\[
  (a,s,b,c)=m(7,2,-3,-8)
\]
for some \(m\in\mathbb Z\).  Therefore
\[
  \Lambda_{\phi_{31}}^R
  =\{\lambda_{31}(t^m)\mid m\in\mathbb Z\}.
\]
This set is symmetric under \(m\mapsto-m\).  The
Casimiro--Florentino criterion therefore shows that
\(R\cdot\phi_{31}\) is closed in \(W^{H_{31}}\).  Luna reduction then
implies that
\[
  G\cdot\phi_{31}
\]
is closed in \(W\), so \([\phi_{31}]\) is polystable.

\subsubsection*{Normal form}
Since \(B_4\) is square-free, the \(\GL(U)\)-action normalizes it to
\[
  B_\lambda(u,v):=uv(u-v)(u-\lambda v),
  \qquad
  \lambda\in\C\setminus\{0,1\}.
\]
Using the remaining diagonal factors in \(C_G(H_{31})\), together with
projective rescaling, normalize the coefficient of
\(x_0^2x_3^2x_4\) to \(1\).  Every general limit is then equivalent to a
member of
\[
\begin{aligned}
  \nf_{31}(\lambda,C_3,L_1,Q_2,M_1)
  ={}&x_4B_\lambda(x_1,x_2)
  +x_3^2C_3(x_1,x_2)
  +x_0x_3^3L_1(x_1,x_2)\\
  &+x_0x_3x_4Q_2(x_1,x_2)
  +x_0^2x_3^2x_4
  +x_0^2x_4^2M_1(x_1,x_2),
\end{aligned}
\]
where
\[
  C_3\in\operatorname{Sym}^3U,
  \qquad
  L_1,M_1\in U,
  \qquad
  Q_2\in\operatorname{Sym}^2U,
\]
and
\[
  \operatorname{Res}(B_\lambda,C_3)\neq0,
  \qquad
  M_1\neq0.
\]

The displayed coefficient space has dimension
\[
  1+4+2+3+2=12.
\]
There remains a one-dimensional diagonal action, represented by
\[
  x_0\longmapsto\tau^{-1}x_0,
  \qquad
  x_3\longmapsto\tau x_3,
\]
with \(x_1,x_2,x_4\) fixed.  It acts by
\[
  (C_3,L_1,Q_2,M_1)
  \longmapsto
  (\tau^2C_3,\tau^2L_1,Q_2,\tau^{-2}M_1).
\]
Thus the effective normal-form parameter count is
\[
  12-1=11.
\]
Equivalently,
\[
  \dim\PP(W^{H_{31}})=16,
  \qquad
  \dim\bigl(C_G(H_{31})/H_{31}\bigr)=5,
\]
and therefore
\[
  \dim\Phi_{31}=16-5=11.
\]
No condition arising from the later singular-locus or minimal-exponent
analysis is imposed at the normal-form stage.

\StartCase{32}

\subsubsection*{1-PS limit}
The witness vector and the corresponding one-parameter subgroup are
\[
  r_{32}=(3,3,-2,-2,-2),
  \qquad
  \lambda_{32}(t)=\diag(t^3,t^3,t^{-2},t^{-2},t^{-2}),
  \quad t\in\Gm.
\]
Since the entries of \(r_{32}\) sum to zero, \(\lambda_{32}\) is a
one-parameter subgroup of \(G=\SL(5)\).  Let \(f_{32}\) be a general
quintic supported on
\[
  S_{32}=I(r_{32})_{\geq0}.
\]
Its affine one-parameter-subgroup limit is the weight-zero part
\[
  \phi_{32}:=\lim_{t\to0}\lambda_{32}(t)\cdot f_{32}.
\]
Put
\[
  V:=\langle x_0,x_1\rangle,
  \qquad
  U:=\langle x_2,x_3,x_4\rangle,
  \qquad
  H_{32}:=\lambda_{32}(\Gm).
\]
For a degree-five monomial, the weight-zero equations are equivalent to
\[
  u_0+u_1=2,
  \qquad
  u_2+u_3+u_4=3.
\]
Hence
\[
  W^{H_{32}}=\operatorname{Sym}^2V\otimes\operatorname{Sym}^3U,
  \qquad
  \dim W^{H_{32}}=3\cdot10=30,
\]
and a general limit has the form
\[
  \phi_{32}
  =x_0^2C_0(x_2,x_3,x_4)
  +x_0x_1C_1(x_2,x_3,x_4)
  +x_1^2C_2(x_2,x_3,x_4),
\]
where \(C_0,C_1,C_2\in\operatorname{Sym}^3U\).

The two weights of \(\lambda_{32}\) have multiplicities two and three.
Therefore the centralizer is non-toric:
\[
  C_G(H_{32})
  =
  \left\{
    A\oplus B
    \ \middle|\
    A\in\GL(V),\ B\in\GL(U),\ \det A\det B=1
  \right\}.
\]
The kernel of the natural map
\[
  C_G(H_{32})\longrightarrow\PGL(V)\times\PGL(U)
\]
is precisely \(H_{32}\).  Thus
\[
  C_G(H_{32})/H_{32}
  \simeq\PGL(V)\times\PGL(U),
  \qquad
  \dim\bigl(C_G(H_{32})/H_{32}\bigr)=3+8=11.
\]

The stable locus for the effective residual action is nonempty.  Indeed,
if \(\zeta\) is a primitive cube root of unity, then
\[
\begin{aligned}
  \phi_*={}&X_0^2(Y_0^3+Y_1^3+Y_2^3)
  +X_0X_1Y_0Y_1Y_2\\
  &+X_1^2(Y_0^3+\zeta Y_1^3+\zeta^2Y_2^3)
\end{aligned}
\]
has zero moment map for
\(\operatorname{SU}(V)\times\operatorname{SU}(U)\), and its projective
stabilizer is finite.  Hence \([\phi_*]\) is stable by the
Kempf--Ness criterion.  We therefore take \([\phi_{32}]\) in the
nonempty open locus
\[
  \PP\bigl(\operatorname{Sym}^2V\otimes\operatorname{Sym}^3U\bigr)^s.
\]

\subsubsection*{Polystability}
Set
\[
  R:=C_G(H_{32}).
\]
By affine Luna reduction, it is enough to prove that
\(R\cdot\phi_{32}\) is closed in \(W^{H_{32}}\).  Let
\(\rho\in\Lambda_{\phi_{32}}^R\).  After conjugation in the
\(\GL(V)\)- and \(\GL(U)\)-blocks, the weights of \(\rho\) may be
written as
\[
  \wt_\rho(X_0,X_1,Y_0,Y_1,Y_2)
  =(a+s,a-s,b+p,b+q,b-p-q),
\]
with determinant-one condition
\[
  2a+3b=0.
\]
For a monomial
\[
  X_0^{2-i}X_1^iY_0^jY_1^kY_2^\ell,
  \qquad
  i\in\{0,1,2\},
  \qquad
  j+k+\ell=3,
\]
the scalar part cancels, and its \(\rho\)-weight is
\[
  (2-2i)s+jp+kq-\ell(p+q).
\]
Thus the existence of the affine limit means that every occurring
monomial has nonnegative weight for the residual one-parameter subgroup
of \(\SL(V)\times\SL(U)\).  Since \([\phi_{32}]\) is stable for this
action, the Hilbert--Mumford criterion forces
\[
  s=p=q=0.
\]
Consequently the weights of \(\rho\) are
\[
  (a,a,b,b,b),
  \qquad
  2a+3b=0.
\]
Their integrality gives an integer \(m\) such that
\[
  a=3m,
  \qquad
  b=-2m,
\]
and therefore
\[
  \rho(t)=\lambda_{32}(t^m).
\]
Conversely, every \(\lambda_{32}(t^m)\) fixes \(W^{H_{32}}\) pointwise.
Hence
\[
  \Lambda_{\phi_{32}}^R
  =\{\lambda_{32}(t^m)\mid m\in\mathbb Z\}.
\]
This set is symmetric under \(m\mapsto-m\).  The
Casimiro--Florentino criterion therefore shows that
\(R\cdot\phi_{32}\) is closed in \(W^{H_{32}}\).  Luna reduction then
implies that
\[
  G\cdot\phi_{32}
\]
is closed in \(W\), so \([\phi_{32}]\) is polystable.

\subsubsection*{Normal form}
The normal-form family is the stable locus in the fixed representation:
\[
  \nf_{32}(C_0,C_1,C_2)
  =x_0^2C_0(x_2,x_3,x_4)
  +x_0x_1C_1(x_2,x_3,x_4)
  +x_1^2C_2(x_2,x_3,x_4),
\]
where
\[
  C_0,C_1,C_2\in\operatorname{Sym}^3U,
  \qquad
  [\nf_{32}(C_0,C_1,C_2)]
  \in
  \PP\bigl(\operatorname{Sym}^2V\otimes\operatorname{Sym}^3U\bigr)^s.
\]
Equivalently, the dense quotient-side normal-form locus is
\[
  \PP\bigl(\operatorname{Sym}^2V\otimes\operatorname{Sym}^3U\bigr)^s
  \big/
  \bigl(\PGL(V)\times\PGL(U)\bigr).
\]
Since stable points have finite stabilizer for the effective action,
\[
  \dim\PP(W^{H_{32}})=29,
  \qquad
  \dim\bigl(C_G(H_{32})/H_{32}\bigr)=11,
\]
and therefore
\[
  \dim\Phi_{32}=29-11=18.
\]
No condition arising from the later singular-locus or minimal-exponent
analysis is imposed at the normal-form stage.

\StartCase{33}

\subsubsection*{1-PS limit}
The witness vector is
\[
 r_{33}=(2,1,0,0,-3),
\]
and the associated one-parameter subgroup is
\[
 \lambda_{33}(t)=\diag(t^2,t,1,1,t^{-3}),
 \qquad t\in\Gm.
\]
Since the entries of \(r_{33}\) sum to zero, \(\lambda_{33}\) is a
one-parameter subgroup of \(G=\SL(5)\).

Let \(f_{33}\) be a general quintic supported on
\[
 S_{33}=I(r_{33})_{\geq0}
 =\{u\in I\mid 2u_0+u_1-3u_4\geq0\}.
\]
Its affine one-parameter-subgroup limit is the weight-zero part
\[
 \phi_{33}:=\lim_{t\to0}\lambda_{33}(t)\cdot f_{33}.
\]
The weight-zero equations are
\[
 2u_0+u_1-3u_4=0,
 \qquad
 u_0+u_1+u_2+u_3+u_4=5,
\]
and give
\[
 (u_0,u_1,u_4)=(0,0,0),\ (0,3,1),\ (1,1,1),\ (3,0,2).
\]
Writing
\[
 U:=\langle x_2,x_3\rangle,
\]
we obtain
\[
 \phi_{33}
 =B_5(x_2,x_3)
 +x_1^3x_4L_1(x_2,x_3)
 +x_0x_1x_4Q_2(x_2,x_3)
 +\kappa x_0^3x_4^2,
\]
where
\[
 B_5\in\Sym^5U,
 \qquad
 L_1\in U,
 \qquad
 Q_2\in\Sym^2U,
 \qquad
 \kappa\in\C.
\]
We work on the nonempty Zariski-open locus
\[
 \kappa\neq0,
 \qquad
 \operatorname{Disc}(B_5)\neq0,
 \qquad
 \operatorname{Disc}(Q_2)\neq0,
 \qquad
 \operatorname{Res}(L_1,Q_2)\neq0.
\]

Set \(H:=H_{33}=\lambda_{33}(\Gm)\).  The zero weight has multiplicity two
on \(U\), and hence
\[
 C_G(H)
 =
 \left\{
 \diag(\mu_0,\mu_1)\oplus A\oplus\diag(\mu_4)
 \ \middle|\
 A\in\GL(U),\ \mu_0\mu_1\mu_4\det A=1
 \right\}.
\]
Thus \(C_G(H)\) is non-toric and
\[
 \dim C_G(H)=6.
\]
Moreover,
\[
 W^H
 =\Sym^5U
 \oplus x_1^3x_4U
 \oplus x_0x_1x_4\Sym^2U
 \oplus\langle x_0^3x_4^2\rangle,
 \qquad
 \dim W^H=12.
\]
Since \(H\) acts trivially on \(W^H\),
\[
 \dim\bigl(C_G(H)/H\bigr)=5.
\]

\subsubsection*{Polystability}
Put \(R:=C_G(H)\).  By Luna's centralizer reduction, it is enough to
prove that \(R\cdot\phi_{33}\) is closed in \(W^H\).

Let \(\lambda\in\Lambda_{\phi_{33}}^R\).  After conjugation in \(R\), we
may assume that \(\lambda\) is diagonal on the \(\GL(U)\)-block.  In a
corresponding eigenbasis \(X,Y\) of \(U\), write its weights as
\[
 \wt_\lambda(x_0,x_1,X,Y,x_4)=(a,b,s+n,s-n,c),
\]
so that
\[
 S:=a+b+2s+c=0.
\]
Set
\[
 \ell:=3b+s+c,
 \qquad
 d:=3a+2c.
\]
The nonvanishing of \(\operatorname{Disc}(B_5)\) implies \(s\geq0\):
every monomial in the discriminant has total weight \(40s\).  Similarly,
\(\operatorname{Res}(L_1,Q_2)\neq0\) gives
\[
 2\ell+S\geq0,
\]
and hence \(\ell\geq0\).  Finally, the occurring monomial
\(\kappa x_0^3x_4^2\) gives \(d\geq0\).  These quantities satisfy
\[
 5s+\ell+d=3S=0.
\]
Consequently,
\[
 s=\ell=d=0.
\]

It remains to eliminate the relative weight \(n\).  Since \(s=0\), the
six coefficient weights of the binary quintic are
\[
 (5-2j)n,
 \qquad j=0,\ldots,5.
\]
If \(n>0\), then \(B_5\) is divisible by \(X^3\); if \(n<0\), it is
divisible by \(Y^3\).  Either case contradicts
\(\operatorname{Disc}(B_5)\neq0\).  Hence \(n=0\).

Solving
\[
 S=0,
 \qquad
 \ell=3b+c=0,
 \qquad
 d=3a+2c=0
\]
gives
\[
 (a,b,s,c)=m(2,1,0,-3)
 \qquad
 (m\in\mathbb Z).
\]
Therefore
\[
 \Lambda_{\phi_{33}}^R
 =\{\lambda_{33}(t^m)\mid m\in\mathbb Z\}.
\]
This set is symmetric in the sense of the Casimiro--Florentino criterion
\cite{CF12}.  Hence \(R\cdot\phi_{33}\) is closed in \(W^H\), and Luna's
reduction gives that \(G\cdot\phi_{33}\) is closed in \(W\).  Thus
\([\phi_{33}]\) is polystable.

\subsubsection*{Normal form and the quotient parameters}
Because \(Q_2\) has two distinct roots and the root of \(L_1\) is disjoint
from them, the divisor of \(L_1Q_2\) consists of three distinct points of
\(\PP(U)\).  Using the \(\PGL(U)\)-action and the remaining block scalings,
we may normalize
\[
 L_1=x_2,
 \qquad
 Q_2=x_2^2+x_3^2.
\]
Using the remaining diagonal and projective scalings, we also normalize
\(\kappa=1\).  A convenient normal form is therefore
\[
 \nf_{33}(B_5)
 =B_5(x_2,x_3)
 +x_1^3x_2x_4
 +x_0x_1x_4(x_2^2+x_3^2)
 +x_0^3x_4^2,
\]
where
\[
 B_5\in\Sym^5U,
 \qquad
 \operatorname{Disc}(B_5)\neq0.
\]
The stabilizer in \(\PGL(U)\) of the normalized three-point divisor is
finite, so this normal form has no residual continuous redundancy.  Since
\[
 \dim\PP(W^H)=11,
 \qquad
 \dim\bigl(C_G(H)/H\bigr)=5,
\]
the corresponding quotient-side family has dimension
\[
 \dim\Phi_{33}=11-5=6.
\]

\StartCase{34}

\subsubsection*{1-PS limit}
The witness vector and the corresponding one-parameter subgroup are
\[
  r_{34}=(8,3,3,-2,-12),
  \qquad
  \lambda_{34}(t)=\diag(t^8,t^3,t^3,t^{-2},t^{-12}),
  \quad t\in\Gm.
\]
Since the entries of \(r_{34}\) sum to zero, \(\lambda_{34}\) is a
one-parameter subgroup of \(G=\SL(5)\).  Let \(f_{34}\) be a general
quintic supported on
\[
  S_{34}=I(r_{34})_{\geq0}.
\]
Its affine one-parameter-subgroup limit is the weight-zero part
\[
  \phi_{34}:=\lim_{t\to0}\lambda_{34}(t)\cdot f_{34}.
\]
Put
\[
  U:=\langle x_1,x_2\rangle,
  \qquad
  H_{34}:=\lambda_{34}(\Gm).
\]
Solving the degree-five weight-zero equations gives
\[
\begin{aligned}
  W^{H_{34}}={}&
  x_4\operatorname{Sym}^4U
  \oplus x_3^3\operatorname{Sym}^2U
  \oplus\langle x_0x_3^4\rangle\\
  &\oplus x_0x_3x_4\operatorname{Sym}^2U
  \oplus\langle x_0^2x_3^2x_4\rangle
  \oplus\langle x_0^3x_4^2\rangle,
\end{aligned}
\]
and hence
\[
  \dim W^{H_{34}}=14.
\]
Thus a general limit can be written as
\[
\begin{aligned}
  \phi_{34}={}&x_4B_4(x_1,x_2)
  +x_3^3Q_2(x_1,x_2)
  +\kappa x_0x_3^4\\
  &+x_0x_3x_4R_2(x_1,x_2)
  +\delta x_0^2x_3^2x_4
  +\rho x_0^3x_4^2,
\end{aligned}
\]
where
\[
  B_4\in\operatorname{Sym}^4U,
  \qquad
  Q_2,R_2\in\operatorname{Sym}^2U.
\]
We work on the nonempty Zariski-open locus
\[
  \kappa\rho\neq0,
  \qquad
  \operatorname{Disc}(B_4)\neq0,
  \qquad
  \operatorname{Res}(B_4,Q_2)\neq0.
\]

The weight \(3\) occurs twice on \(U\), so the centralizer is non-toric:
\[
  C_G(H_{34})
  =
  \left\{
    \diag(\mu_0)\oplus A\oplus\diag(\mu_3,\mu_4)
    \ \middle|\
    A\in\GL(U),\
    \mu_0\mu_3\mu_4\det A=1
  \right\}.
\]
Consequently,
\[
  \dim C_G(H_{34})=6,
  \qquad
  \dim\bigl(C_G(H_{34})/H_{34}\bigr)=5.
\]

\subsubsection*{Polystability}
Set
\[
  R:=C_G(H_{34}).
\]
By affine Luna reduction, it is enough to prove that
\(R\cdot\phi_{34}\) is closed in \(W^{H_{34}}\).  Let
\(\sigma\in\Lambda_{\phi_{34}}^R\).  After conjugation in the
\(\GL(U)\)-block, its weights may be written as
\[
  \wt_\sigma(x_0,X,Y,x_3,x_4)
  =(a,s+n,s-n,b,c),
  \qquad
  a+2s+b+c=0,
\]
for a suitable eigenbasis \(X,Y\) of \(U\).

Introduce the averaged weights
\[
  p_B:=4s+c,
  \qquad
  p_Q:=2s+3b,
  \qquad
  p_\kappa:=a+4b,
  \qquad
  p_\rho:=3a+2c.
\]
The conditions
\[
  \operatorname{Disc}(B_4)\neq0,
  \qquad
  \operatorname{Res}(B_4,Q_2)\neq0,
  \qquad
  \kappa\rho\neq0
\]
and nonnegativity of the occurring monomial weights give
\[
  p_B\geq0,
  \qquad
  p_B+2p_Q\geq0,
  \qquad
  p_\kappa\geq0,
  \qquad
  p_\rho\geq0.
\]
These quantities satisfy the balancing identity
\[
  3p_B+(p_B+2p_Q)+p_\kappa+3p_\rho
  =10(a+2s+b+c)
  =0.
\]
Hence
\[
  p_B=p_Q=p_\kappa=p_\rho=0.
\]
Since \(p_B=0\), a nonzero relative weight \(n\) would force the
binary quartic \(B_4\) to have a multiple root.  Therefore
\[
  n=0.
\]
Solving the remaining weight equations gives
\[
  (a,s,b,c)=m(8,3,-2,-12)
\]
for some \(m\in\mathbb Z\).  Consequently,
\[
  \Lambda_{\phi_{34}}^R
  =\{\lambda_{34}(t^m)\mid m\in\mathbb Z\}.
\]
This set is symmetric under \(m\mapsto-m\).  The
Casimiro--Florentino criterion therefore shows that
\(R\cdot\phi_{34}\) is closed in \(W^{H_{34}}\).  Luna reduction then
implies that
\[
  G\cdot\phi_{34}
\]
is closed in \(W\), so \([\phi_{34}]\) is polystable.

\subsubsection*{Normal form}
Since \(B_4\) is square-free, the \(\GL(U)\)-action normalizes it to
\[
  B_\lambda(u,v):=uv(u-v)(u-\lambda v),
  \qquad
  \lambda\in\C\setminus\{0,1\}.
\]
Using the remaining diagonal factors in \(C_G(H_{34})\), together with
projective rescaling, normalize the coefficients of
\[
  x_0x_3^4,
  \qquad
  x_0^3x_4^2
\]
to \(1\).  Every general limit is then equivalent to a member of
\[
\begin{aligned}
  \nf_{34}(\lambda,Q_2,R_2,\beta)
  ={}&x_4B_\lambda(x_1,x_2)
  +x_3^3Q_2(x_1,x_2)
  +x_0x_3^4\\
  &+x_0x_3x_4R_2(x_1,x_2)
  +\beta x_0^2x_3^2x_4
  +x_0^3x_4^2,
\end{aligned}
\]
where
\[
  Q_2,R_2\in\operatorname{Sym}^2U,
  \qquad
  \beta\in\C,
  \qquad
  \operatorname{Res}(B_\lambda,Q_2)\neq0.
\]
The cross-ratio \(\lambda\) is defined up to the finite anharmonic
action, which does not affect the dimension.  The displayed parameter
count is
\[
  1+3+3+1=8.
\]
Equivalently,
\[
  \dim\PP(W^{H_{34}})=13,
  \qquad
  \dim\bigl(C_G(H_{34})/H_{34}\bigr)=5,
\]
and therefore
\[
  \dim\Phi_{34}=13-5=8.
\]
No condition arising from the later singular-locus or minimal-exponent
analysis is imposed at the normal-form stage.

\StartCase{35}

\subsubsection*{1-PS limit}
The witness vector and the corresponding one-parameter subgroup are
\[
  r_{35}=(4,4,-1,-1,-6),
  \qquad
  \lambda_{35}(t)=\diag(t^4,t^4,t^{-1},t^{-1},t^{-6}),
  \quad t\in\Gm.
\]
Since the entries of \(r_{35}\) sum to zero, \(\lambda_{35}\) is a
one-parameter subgroup of \(G=\SL(5)\).  Let \(f_{35}\) be a general
quintic supported on
\[
  S_{35}=I(r_{35})_{\geq0}.
\]
Its affine one-parameter-subgroup limit is the weight-zero part
\[
  \phi_{35}:=\lim_{t\to0}\lambda_{35}(t)\cdot f_{35}.
\]
Put
\[
  V:=\langle x_0,x_1\rangle,
  \qquad
  U:=\langle x_2,x_3\rangle,
  \qquad
  H_{35}:=\lambda_{35}(\Gm).
\]
For a degree-five monomial, set
\[
  a=u_0+u_1,
  \qquad
  b=u_2+u_3,
  \qquad
  c=u_4.
\]
The degree and weight-zero equations give
\[
  a=c+1,
  \qquad
  b=4-2c,
  \qquad
  c=0,1,2.
\]
Consequently,
\[
  W^{H_{35}}
  =
  V\otimes\operatorname{Sym}^4U
  \oplus
  x_4\operatorname{Sym}^2V\otimes\operatorname{Sym}^2U
  \oplus
  x_4^2\operatorname{Sym}^3V,
\]
and hence
\[
  \dim W^{H_{35}}=23.
\]
Thus a general limit has the form
\[
\begin{aligned}
  \phi_{35}={}&x_0A_4(x_2,x_3)+x_1C_4(x_2,x_3)\\
  &+x_4Q_{2,2}(x_0,x_1;x_2,x_3)
  +x_4^2C_3(x_0,x_1),
\end{aligned}
\]
where
\[
  A_4,C_4\in\operatorname{Sym}^4U,
  \qquad
  Q_{2,2}\in\operatorname{Sym}^2V\otimes\operatorname{Sym}^2U,
  \qquad
  C_3\in\operatorname{Sym}^3V.
\]
We work on the nonempty Zariski-open locus
\[
  \operatorname{Res}(A_4,C_4)\neq0,
  \qquad
  \operatorname{Disc}(C_3)\neq0.
\]
The first condition says that the pencil spanned by \(A_4\) and \(C_4\)
is base-point-free on \(\PP(U)\), while the second says that \(C_3\) is
square-free.

The weights \(4\) and \(-1\) each occur with multiplicity two.  Therefore
the centralizer is non-toric:
\[
  C_G(H_{35})
  =
  \left\{
    A\oplus B\oplus\diag(\mu_4)
    \ \middle|\
    A\in\GL(V),\ B\in\GL(U),\
    \mu_4\det A\det B=1
  \right\}.
\]
Consequently,
\[
  \dim C_G(H_{35})=8,
  \qquad
  \dim\bigl(C_G(H_{35})/H_{35}\bigr)=7.
\]

\subsubsection*{Polystability}
Set
\[
  R:=C_G(H_{35}).
\]
By affine Luna reduction, it is enough to prove that
\(R\cdot\phi_{35}\) is closed in \(W^{H_{35}}\).  Let
\(\rho\in\Lambda_{\phi_{35}}^R\).  After conjugation in the two
non-abelian blocks, the weights of \(\rho\) may be written as
\[
  \wt_\rho(X,Y,Z,W,x_4)
  =(s+n,s-n,t+v,t-v,a),
\]
where \(X,Y\) and \(Z,W\) are eigenbases of \(V\) and \(U\), and
\[
  2s+2t+a=0.
\]
Define the averaged block weights
\[
  p_F:=s+4t,
  \qquad
  p_B:=3s+2a.
\]
The nonvanishing of \(\operatorname{Res}(A_4,C_4)\), together with
nonnegativity of the occurring monomial weights, gives
\[
  p_F\geq0.
\]
Similarly, \(\operatorname{Disc}(C_3)\neq0\) gives
\[
  p_B\geq0.
\]
These quantities satisfy
\[
  p_F+p_B
  =(s+4t)+(3s+2a)
  =2(2s+2t+a)
  =0.
\]
Hence
\[
  p_F=p_B=0.
\]
Since \(p_B=0\), a nonzero relative weight \(n\) would force the binary
cubic \(C_3\) to have a multiple root.  Thus \(n=0\).  Since
\(p_F=0\), a nonzero relative weight \(v\) would force both \(A_4\) and
\(C_4\) to be divisible by the same square of a linear form, contradicting
\(\operatorname{Res}(A_4,C_4)\neq0\).  Therefore \(v=0\).

Solving the remaining equations gives
\[
  (s,t,a)=m(4,-1,-6)
\]
for some \(m\in\mathbb Z\).  Consequently,
\[
  \Lambda_{\phi_{35}}^R
  =\{\lambda_{35}(t^m)\mid m\in\mathbb Z\}.
\]
This set is symmetric under \(m\mapsto-m\).  The
Casimiro--Florentino criterion therefore shows that
\(R\cdot\phi_{35}\) is closed in \(W^{H_{35}}\).  Luna reduction then
implies that
\[
  G\cdot\phi_{35}
\]
is closed in \(W\), so \([\phi_{35}]\) is polystable.  The same weight
calculation shows that the identity component of the stabilizer in \(R\)
is \(H_{35}\); hence the effective stabilizer is finite on this open
locus.

\subsubsection*{Normal form}
Since \(C_3\) is square-free, the \(\GL(V)\)-action normalizes it to
\[
  B_3(x_0,x_1):=x_0x_1(x_0-x_1).
\]
Thus every general limit is equivalent to a member of
\[
\begin{aligned}
  \nf_{35}(A_4,C_4,Q_{2,2})
  ={}&x_0A_4(x_2,x_3)+x_1C_4(x_2,x_3)\\
  &+x_4Q_{2,2}(x_0,x_1;x_2,x_3)
  +x_4^2x_0x_1(x_0-x_1),
\end{aligned}
\]
where
\[
  A_4,C_4\in\operatorname{Sym}^4U,
  \qquad
  Q_{2,2}\in\operatorname{Sym}^2V\otimes\operatorname{Sym}^2U,
  \qquad
  \operatorname{Res}(A_4,C_4)\neq0.
\]
The residual continuous action is the \(\GL(U)\)-action on
\((x_2,x_3)\); the stabilizer of \(x_0x_1(x_0-x_1)\) in the
\(V\)-variables is finite.  Accordingly, the normal-form parameter space
has dimension
\[
  5+5+9-\dim\GL(U)=19-4=15.
\]
Equivalently,
\[
  \dim\PP(W^{H_{35}})=22,
  \qquad
  \dim\bigl(C_G(H_{35})/H_{35}\bigr)=7,
\]
and therefore
\[
  \dim\Phi_{35}=22-7=15.
\]
No condition arising from the later singular-locus or minimal-exponent
analysis is imposed at the normal-form stage.

\StartCase{36}

\subsubsection*{1-PS limit}
The witness vector and the corresponding one-parameter subgroup are
\[
  r_{36}=(1,0,0,0,-1),
  \qquad
  \lambda_{36}(t)=\diag(t,1,1,1,t^{-1}),
  \quad t\in\Gm.
\]
Since the entries of \(r_{36}\) sum to zero, \(\lambda_{36}\) is a
one-parameter subgroup of \(G=\SL(5)\).  Let \(f_{36}\) be a general
quintic supported on
\[
  S_{36}=I(r_{36})_{\geq0}.
\]
Its affine one-parameter-subgroup limit is the weight-zero part
\[
  \phi_{36}:=\lim_{t\to0}\lambda_{36}(t)\cdot f_{36}.
\]
For a degree-five monomial \(x^u\), the weight-zero equations are
\[
  u_0=u_4,
  \qquad
  2u_0+u_1+u_2+u_3=5.
\]
Hence \(u_0=u_4\in\{0,1,2\}\).  Put
\[
  U:=\langle x_1,x_2,x_3\rangle,
  \qquad
  H_{36}:=\lambda_{36}(\Gm).
\]
Then
\[
  W^{H_{36}}
  =
  \operatorname{Sym}^5U
  \oplus x_0x_4\operatorname{Sym}^3U
  \oplus x_0^2x_4^2U,
\]
and therefore
\[
  \dim W^{H_{36}}=21+10+3=34.
\]
Thus a general limit has the form
\[
  \phi_{36}
  =
  B_5(x_1,x_2,x_3)
  +x_0x_4C_3(x_1,x_2,x_3)
  +x_0^2x_4^2L_1(x_1,x_2,x_3),
\]
where
\[
  B_5\in\operatorname{Sym}^5U,
  \qquad
  C_3\in\operatorname{Sym}^3U,
  \qquad
  L_1\in U.
\]
We work on the nonempty Zariski-open locus
\[
  B_5\neq0,
  \qquad
  L_1\neq0,
  \qquad
  \operatorname{Disc}(C_3)\neq0.
\]
In particular, \(V(C_3)\subset\PP(U)\simeq\PP^2\) is a smooth plane
cubic.

The weight zero occurs with multiplicity three on \(U\).  Hence the
centralizer is non-toric:
\[
  C_G(H_{36})
  =
  \left\{
    \diag(\mu_0)\oplus A\oplus\diag(\mu_4)
    \ \middle|\
    A\in\GL(U),\ \mu_0\mu_4\det A=1
  \right\}.
\]
Consequently,
\[
  \dim C_G(H_{36})=10,
  \qquad
  \dim\bigl(C_G(H_{36})/H_{36}\bigr)=9.
\]

\subsubsection*{Polystability}
Set
\[
  R:=C_G(H_{36}).
\]
By affine Luna reduction, it is enough to prove that
\(R\cdot\phi_{36}\) is closed in \(W^{H_{36}}\).  Let
\(\rho\in\Lambda_{\phi_{36}}^R\).  After conjugation in the
\(\GL(U)\)-block, its weights may be written as
\[
  \wt_\rho(x_0,X_1,X_2,X_3,x_4)
  =(a,s+p,s+q,s-p-q,b),
\]
where \(X_1,X_2,X_3\) is an eigenbasis of \(U\), and
\[
  a+3s+b=0.
\]
Since \(\operatorname{Disc}(C_3)\neq0\), the cubic \([C_3]\) is stable
for the \(\SL(U)\)-action.  The Hilbert--Mumford criterion therefore
forces the residual \(\SL(U)\)-weights to vanish:
\[
  p=q=0.
\]
All monomials in the nonzero quintic block \(B_5\) then have weight
\(5s\), so
\[
  5s\geq0.
\]
All monomials in the nonzero block \(x_0^2x_4^2L_1\) have weight
\[
  2a+2b+s=-5s,
\]
and hence
\[
  -5s\geq0.
\]
Thus
\[
  s=0,
  \qquad
  a+b=0.
\]
It follows that, for some \(m\in\mathbb Z\),
\[
  \rho(t)
  =\diag(t^m,1,1,1,t^{-m})
  =\lambda_{36}(t^m).
\]
Consequently,
\[
  \Lambda_{\phi_{36}}^R
  =\{\lambda_{36}(t^m)\mid m\in\mathbb Z\}.
\]
This set is symmetric under \(m\mapsto-m\).  The
Casimiro--Florentino criterion therefore shows that
\(R\cdot\phi_{36}\) is closed in \(W^{H_{36}}\).  Luna reduction then
implies that
\[
  G\cdot\phi_{36}
\]
is closed in \(W\), and hence \([\phi_{36}]\) is polystable.

\subsubsection*{Normal form}
Since \(L_1\neq0\), the \(\GL(U)\)-action allows us to choose
coordinates such that
\[
  L_1=x_3.
\]
Using the remaining diagonal scalings and the overall projective
scaling, normalize the coefficient of \(x_0^2x_3x_4^2\) to \(1\).
A convenient generic normal form is
\[
  \nf_{36}(B_5,C_3)
  =
  B_5(x_1,x_2,x_3)
  +x_0x_4C_3(x_1,x_2,x_3)
  +x_0^2x_3x_4^2,
\]
where
\[
  B_5\in\operatorname{Sym}^5U,
  \qquad
  C_3\in\operatorname{Sym}^3U,
  \qquad
  B_5\neq0,
  \qquad
  \operatorname{Disc}(C_3)\neq0.
\]
After fixing \(L_1=x_3\), the pair \((B_5,C_3)\) is understood modulo
the residual subgroup preserving \(\langle x_3\rangle\) in
\(C_G(H_{36})/H_{36}\).  Since
\[
  \dim\PP(W^{H_{36}})=33,
  \qquad
  \dim\bigl(C_G(H_{36})/H_{36}\bigr)=9,
\]
the corresponding quotient-side family has dimension
\[
  \dim\Phi_{36}=33-9=24.
\]
No condition arising from the later singular-locus or minimal-exponent
analysis is imposed at the normal-form stage.

\StartCase{37}

\subsubsection*{1-PS limit}
The witness vector and the corresponding one-parameter subgroup are
\[
  r_{37}=(2,2,2,-3,-3),
  \qquad
  \lambda_{37}(t)=\diag(t^2,t^2,t^2,t^{-3},t^{-3}),
  \quad t\in\Gm.
\]
Since the entries of \(r_{37}\) sum to zero, \(\lambda_{37}\) is a
one-parameter subgroup of \(G=\SL(5)\).  Let \(f_{37}\) be a general
quintic supported on
\[
  S_{37}=I(r_{37})_{\geq0}.
\]
Its affine one-parameter-subgroup limit is the weight-zero part
\[
  \phi_{37}:=\lim_{t\to0}\lambda_{37}(t)\cdot f_{37}.
\]
For a degree-five monomial \(x^u\), the weight-zero equations are
equivalent to
\[
  u_0+u_1+u_2=3,
  \qquad
  u_3+u_4=2.
\]
Put
\[
  U:=\langle x_0,x_1,x_2\rangle,
  \qquad
  V:=\langle x_3,x_4\rangle,
  \qquad
  H_{37}:=\lambda_{37}(\Gm).
\]
Hence
\[
  W^{H_{37}}=\operatorname{Sym}^3U\otimes\operatorname{Sym}^2V,
  \qquad
  \dim W^{H_{37}}=10\cdot3=30,
\]
and a general limit has the form
\[
  \phi_{37}
  =x_3^2D_0(x_0,x_1,x_2)
  +x_3x_4D_1(x_0,x_1,x_2)
  +x_4^2D_2(x_0,x_1,x_2),
\]
where
\[
  D_0,D_1,D_2\in\operatorname{Sym}^3U.
\]

The two weights of \(\lambda_{37}\) have multiplicities three and two.
Therefore the centralizer is non-toric:
\[
  C_G(H_{37})
  =
  \left\{
    A\oplus B
    \ \middle|\
    A\in\GL(U),\ B\in\GL(V),\ \det A\det B=1
  \right\}.
\]
The kernel of the natural map
\[
  C_G(H_{37})\longrightarrow\PGL(U)\times\PGL(V)
\]
is precisely \(H_{37}\).  Thus
\[
  C_G(H_{37})/H_{37}
  \simeq\PGL(U)\times\PGL(V),
  \qquad
  \dim\bigl(C_G(H_{37})/H_{37}\bigr)=8+3=11.
\]
We work on the nonempty Zariski-open locus
\[
  [\phi_{37}]
  \in
  \PP\bigl(\operatorname{Sym}^3U\otimes\operatorname{Sym}^2V\bigr)^s
\]
of points stable for the natural \(\SL(U)\times\SL(V)\)-action.

\subsubsection*{Polystability}
Set
\[
  R:=C_G(H_{37}).
\]
By affine Luna reduction, it is enough to prove that
\(R\cdot\phi_{37}\) is closed in \(W^{H_{37}}\).  Let
\(\rho\in\Lambda_{\phi_{37}}^R\).  After conjugation in the
\(\GL(U)\)- and \(\GL(V)\)-blocks, the weights of \(\rho\) may be
written as
\[
  \operatorname{wt}_{\rho}(X_0,X_1,X_2,Y_0,Y_1)
  =
  (a+p,a+q,a-p-q,b+s,b-s),
\]
with determinant-one condition
\[
  3a+2b=0.
\]
For a monomial
\[
  X_0^{i_0}X_1^{i_1}X_2^{i_2}Y_0^{2-j}Y_1^j,
  \qquad
  i_0+i_1+i_2=3,
  \qquad
  j\in\{0,1,2\},
\]
the scalar part cancels, and its \(\rho\)-weight is
\[
  i_0p+i_1q-i_2(p+q)+(2-2j)s.
\]
Thus the existence of the affine limit means that every occurring
monomial has nonnegative weight for the residual one-parameter subgroup
of \(\SL(U)\times\SL(V)\).  Since \([\phi_{37}]\) is stable for this
action, the Hilbert--Mumford criterion forces
\[
  p=q=s=0.
\]
Consequently the weights of \(\rho\) are
\[
  (a,a,a,b,b),
  \qquad
  3a+2b=0.
\]
Their integrality gives an integer \(m\) such that
\[
  a=2m,
  \qquad
  b=-3m,
\]
and therefore
\[
  \rho(t)=\lambda_{37}(t^m).
\]
Conversely, every \(\lambda_{37}(t^m)\) fixes \(W^{H_{37}}\) pointwise.
Hence
\[
  \Lambda_{\phi_{37}}^R
  =\{\lambda_{37}(t^m)\mid m\in\mathbb Z\}.
\]
This set is symmetric under \(m\mapsto-m\).  The
Casimiro--Florentino criterion therefore shows that
\(R\cdot\phi_{37}\) is closed in \(W^{H_{37}}\).  Luna reduction then
implies that
\[
  G\cdot\phi_{37}
\]
is closed in \(W\), and hence \([\phi_{37}]\) is polystable.

\subsubsection*{Normal form}
The normal-form family is the stable locus in the fixed representation:
\[
  \nf_{37}(D_0,D_1,D_2)
  =x_3^2D_0(x_0,x_1,x_2)
  +x_3x_4D_1(x_0,x_1,x_2)
  +x_4^2D_2(x_0,x_1,x_2),
\]
where
\[
  D_0,D_1,D_2\in\operatorname{Sym}^3U,
  \qquad
  [\nf_{37}(D_0,D_1,D_2)]
  \in
  \PP\bigl(\operatorname{Sym}^3U\otimes\operatorname{Sym}^2V\bigr)^s.
\]
Equivalently, the dense quotient-side normal-form locus is
\[
  \PP\bigl(\operatorname{Sym}^3U\otimes\operatorname{Sym}^2V\bigr)^s
  \big/
  \bigl(\PGL(U)\times\PGL(V)\bigr).
\]
Since stable points have finite stabilizer for the effective action,
\[
  \dim\PP(W^{H_{37}})=29,
  \qquad
  \dim\bigl(C_G(H_{37})/H_{37}\bigr)=11,
\]
and therefore
\[
  \dim\Phi_{37}=29-11=18.
\]
No condition arising from the later singular-locus or minimal-exponent
analysis is imposed at the normal-form stage.

\StartCase{38}

\subsubsection*{1-PS limit}
The witness vector and the corresponding one-parameter subgroup are
\[
  r_{38}=(1,1,1,1,-4),
  \qquad
  \lambda_{38}(t)=\diag(t,t,t,t,t^{-4}),
  \quad t\in\Gm.
\]
Since the entries of \(r_{38}\) sum to zero, \(\lambda_{38}\) is a
one-parameter subgroup of \(G=\SL(5)\).  Let \(f_{38}\) be a general
quintic supported on
\[
  S_{38}=I(r_{38})_{\geq0}.
\]
Its affine one-parameter-subgroup limit is the weight-zero part
\[
  \phi_{38}:=\lim_{t\to0}\lambda_{38}(t)\cdot f_{38}.
\]
For a degree-five monomial \(x^u\), one has
\[
  r_{38}\cdot u
  =u_0+u_1+u_2+u_3-4u_4
  =5-5u_4.
\]
Hence the weight-zero monomials are exactly those with \(u_4=1\).  Put
\[
  U:=\langle x_0,x_1,x_2,x_3\rangle,
  \qquad
  H_{38}:=\lambda_{38}(\Gm).
\]
Then
\[
  W^{H_{38}}=x_4\operatorname{Sym}^4U,
  \qquad
  \dim W^{H_{38}}=35,
\]
and a general limit has the form
\[
  \phi_{38}=x_4B_4(x_0,x_1,x_2,x_3),
  \qquad
  B_4\in\operatorname{Sym}^4U.
\]

The weight \(1\) of \(\lambda_{38}\) has multiplicity four on \(U\).
Therefore the centralizer is non-toric:
\[
  C_G(H_{38})
  =
  \left\{
    A\oplus\diag(\mu_4)
    \ \middle|\
    A\in\GL(U),\ \mu_4\det A=1
  \right\}.
\]
The kernel of the natural map
\[
  C_G(H_{38})\longrightarrow\PGL(U)
\]
is precisely \(H_{38}\).  Thus
\[
  C_G(H_{38})/H_{38}\simeq\PGL(U),
  \qquad
  \dim\bigl(C_G(H_{38})/H_{38}\bigr)=15.
\]
We work on the nonempty Zariski-open locus
\[
  [B_4]\in\PP(\operatorname{Sym}^4U)^s
\]
of quartics stable for the natural \(\SL(U)\)-action.

\subsubsection*{Polystability}
Set
\[
  R:=C_G(H_{38}).
\]
By affine Luna reduction, it is enough to prove that
\(R\cdot\phi_{38}\) is closed in \(W^{H_{38}}\).  Let
\(\rho\in\Lambda_{\phi_{38}}^R\).  After conjugation in the
\(\GL(U)\)-block, the weights of \(\rho\) may be written as
\[
  \operatorname{wt}_{\rho}(X_0,X_1,X_2,X_3,x_4)
  =(b_0,b_1,b_2,b_3,a),
\]
with determinant-one condition
\[
  b_0+b_1+b_2+b_3+a=0.
\]
Define the residual weights
\[
  \widetilde b_i:=b_i+\frac{a}{4}
  \qquad (i=0,1,2,3).
\]
Then
\[
  \widetilde b_0+\widetilde b_1+\widetilde b_2+\widetilde b_3=0.
\]
For a monomial
\[
  x_4X_0^{u_0}X_1^{u_1}X_2^{u_2}X_3^{u_3},
  \qquad
  u_0+u_1+u_2+u_3=4,
\]
its \(\rho\)-weight is
\[
  u_0\widetilde b_0+u_1\widetilde b_1
  +u_2\widetilde b_2+u_3\widetilde b_3.
\]
Thus the existence of the affine limit means that every occurring
monomial of \(B_4\) has nonnegative weight for the residual
one-parameter subgroup of \(\SL(U)\).  Since \([B_4]\) is stable, the
Hilbert--Mumford criterion forces
\[
  \widetilde b_0=\widetilde b_1
  =\widetilde b_2=\widetilde b_3=0.
\]
Consequently
\[
  b_0=b_1=b_2=b_3=-\frac{a}{4}.
\]
Integrality gives an integer \(m\) such that
\[
  b_0=b_1=b_2=b_3=m,
  \qquad
  a=-4m,
\]
and hence
\[
  \rho(t)=\lambda_{38}(t^m).
\]
Conversely, every \(\lambda_{38}(t^m)\) fixes \(W^{H_{38}}\) pointwise.
Therefore
\[
  \Lambda_{\phi_{38}}^R
  =\{\lambda_{38}(t^m)\mid m\in\mathbb Z\}.
\]
This set is symmetric under \(m\mapsto-m\).  The
Casimiro--Florentino criterion therefore shows that
\(R\cdot\phi_{38}\) is closed in \(W^{H_{38}}\).  Luna reduction then
implies that
\[
  G\cdot\phi_{38}
\]
is closed in \(W\), and hence \([\phi_{38}]\) is polystable.

\subsubsection*{Normal form}
The normal-form family is
\[
  \nf_{38}(B_4)=x_4B_4(x_0,x_1,x_2,x_3),
\]
where
\[
  [B_4]\in\PP(\operatorname{Sym}^4U)^s.
\]
Equivalently, the dense quotient-side normal-form locus is
\[
  \PP(\operatorname{Sym}^4U)^s/\PGL(U).
\]
Since stable quartics have finite stabilizer for the effective action,
\[
  \dim\PP(W^{H_{38}})=34,
  \qquad
  \dim\bigl(C_G(H_{38})/H_{38}\bigr)=15,
\]
and therefore
\[
  \dim\Phi_{38}=34-15=19.
\]
No condition arising from the later singular-locus or minimal-exponent
analysis is imposed at the normal-form stage.

\clearpage
\section{Component-by-component singularity and minimal-exponent calculations}
\label{app:singular-analysis}

This appendix proves the component-indexed catalog stated in
Section~\ref{sec:singular-loci}.  For each \(a=1,\ldots,21\), let \(k(a)\) be
the representative fixed in Table~\ref{tab:component-normal-forms}, let
\[
  F_a:=\nf_{k(a)},\qquad X_a:=V(F_a)\subset\PP^4,
\]
and work on the corresponding nonempty Zariski-open generic parameter locus.
The calculation is performed once for each quotient component, rather than
once for each of the thirty-eight oriented supports.  For the partner \(k'\)
of a chosen label \(k(a)\), coordinate reversal identifies the corresponding
zero-weight spaces and the two labels have the same quotient image.  By
Remark~\ref{rem:paired-normal-form-loci}, their general closed orbits are
therefore projectively equivalent, although the two displayed normal-form
formulas need not be related by \(\rho\) without further residual
normalization and reparametrization.  The saturated Jacobian scheme, local
analytic equivalence class, and minimal exponent are invariant under that
projective equivalence.

For every \(a\), we determine
\[
  \Sing(X_a)=V\bigl(J(F_a)\bigr),\qquad
  J(F_a):=
  \left(\frac{\partial F_a}{\partial x_0},\ldots,
        \frac{\partial F_a}{\partial x_4}\right)
  :(x_0,\ldots,x_4)^\infty,
\]
decompose its support into irreducible strata, identify the relevant local
analytic models, and compute
\[
  \mexp(X_a)=\min_{p\in\Sing(X_a)}\mexp_p(X_a).
\]
The outcome in every component is
\[
  \mexp(X_a)=1=\frac{4+1}{5}.
\]

\subsection{\texorpdfstring{The component \(\Theta_1\)}{The component Theta 1}}
\label{subsec:singular-minimal-theta-1}

By Table~\ref{tab:component-normal-forms},
\[
  \Theta_1=\Phi_1=\Phi_{26},
  \qquad
  k(1)=1.
\]
A general closed orbit in \(\Theta_1\) is represented by
\[
  F_1(\alpha):=\nf_1(\alpha)
  =x_1^4x_2+x_0x_3^4+x_0x_2^3x_4
   +x_0x_1x_2x_3x_4+\alpha x_0^2x_4^3,
  \qquad
  \alpha\in\C^\times,
\]
and we set
\[
  X_1(\alpha):=V(F_1(\alpha))\subset\PP^4.
\]

\paragraph*{Singular locus.}
For readability, write \(F=F_1(\alpha)\).  Its partial derivatives are
\begin{align*}
  \frac{\partial F}{\partial x_0}
  &=x_3^4+x_2^3x_4+x_1x_2x_3x_4+2\alpha x_0x_4^3,\\
  \frac{\partial F}{\partial x_1}
  &=x_2(x_0x_3x_4+4x_1^3),\\
  \frac{\partial F}{\partial x_2}
  &=x_1^4+x_0x_1x_3x_4+3x_0x_2^2x_4,\\
  \frac{\partial F}{\partial x_3}
  &=x_0(x_1x_2x_4+4x_3^3),\\
  \frac{\partial F}{\partial x_4}
  &=x_0(x_1x_2x_3+x_2^3+3\alpha x_0x_4^2).
\end{align*}
On the hyperplane \(x_0=0\), the equation
\(\partial F/\partial x_2=0\) gives \(x_1=0\) on the reduced support, and
we obtain the plane quartic
\[
  C:=V(x_0,x_1,x_3^4+x_2^3x_4).
\]
Its normalization is
\[
  [s:t]\longmapsto[0:0:s^4:s^3t:-t^4],
\]
so \(C\) is irreducible and rational.  Its unique singular point is
\[
  Q=(0:0:0:0:1),
\]
and, in the chart \(x_4=1\), the curve germ \((C,Q)\) has equation
\(y^3+z^4=0\).  Thus \(Q\) is a \((3,4)\)-cusp.

On the chart \(x_0=1\), write
\[
  (u,v,z,w)=(x_1,x_2,x_3,x_4).
\]
If \(v=0\), the Jacobian equations give \(w=z=u=0\), yielding
\[
  P=(1:0:0:0:0).
\]
If \(v\neq0\), they imply
\[
  zw=-4u^3,
  \qquad
  v^2w=u^4,
  \qquad
  u^8=256v^7,
  \qquad
  \alpha u^8=v^7.
\]
Consequently this branch occurs only for \(\alpha=1/256\).  At that value an
additional rational singular curve appears.  Hence, on the nonempty open set
\[
  \alpha\in\C^\times\setminus\left\{\frac1{256}\right\},
\]
the reduced singular locus is
\[
  \Sing(X_1(\alpha))_{\mathrm{red}}=C\sqcup\{P\}.
\]
We impose this genericity condition from now on.

\paragraph*{The isolated point \(P\).}
The local equation at \(P\) is
\[
  f_P=u^4v+z^4+v^3w+uvzw+\alpha w^3.
\]
It has an isolated critical point and is weighted homogeneous of degree
\(36\) with weights
\[
  \wt(u,v,z,w)=(7,8,9,12).
\]
After a diagonal rescaling, this germ is analytically equivalent to
\[
  U^4V+Z^4+V^3W+\tau UVZW+W^3,
  \qquad
  \tau\in\C^\times,
  \qquad
  \tau^4=\alpha^{-1}.
\]
The genericity condition \(\alpha\neq1/256\) is equivalent to
\(\tau^4\neq256\).  Hence, in the Yonemura-number notation of
Section~\ref{sec:singular-loci},
\[
  (X_1(\alpha),P)
  \sim_{\mathrm{an}}
  \Sfam{87}{52}(\tau).
\]
In particular, the isolated singularity has Milnor number \(87\) and
Yonemura number \(52\).
Therefore
\[
  \mexp_P(X_1(\alpha))
  =\frac{7+8+9+12}{36}
  =1.
\]

\paragraph*{Smooth points of \(C\).}
Every point of \(C\setminus\{Q\}\) lies in the chart \(x_2=1\).  Put
\[
  \eta=x_4+x_3^4.
\]
Then
\[
  F=x_1^4+x_0\Theta,
  \qquad
  \Theta
  =\eta+x_1x_3(\eta-x_3^4)
   +\alpha x_0(\eta-x_3^4)^3.
\]
Since \(\partial\Theta/\partial\eta=1\) along \(C\), the function
\(\eta'=\Theta\) is an analytic coordinate, and the local equation becomes
\[
  x_0\eta'+x_1^4=0
\]
together with the smooth parameter \(x_3\).  Thus the transverse
singularity is of type \(A_3\), and
\[
  \mexp_p(X_1(\alpha))
  =\frac{2+2+1}{4}
  =\frac54
  \qquad
  \bigl(p\in C\setminus\{Q\}\bigr).
\]
The transverse Jacobian ideal is
\((\eta',x_0,x_1^3)\); in particular, the saturated Jacobian scheme is
nonreduced along \(C\), with generic transverse length \(3\).

\paragraph*{The cusp point \(Q\).}
In the chart \(x_4=1\), with
\[
  (u,v,y,z)=(x_0,x_1,x_2,x_3),
\]
the local equation is
\[
  f_Q=\alpha u^2+u(y^3+z^4+vyz)+v^4y.
\]
Completing the square in \(u\) gives
\[
  f_Q=\alpha {u'}^2+g(v,y,z),
\]
where
\[
  g(v,y,z)
  =v^4y-\frac1{4\alpha}\bigl(y^3+z^4+vyz\bigr)^2.
\]
The germ \(g\) is weighted homogeneous of degree \(24\) with weights
\[
  \wt(v,y,z)=(5,4,3).
\]
A direct solution of \(\nabla g=0\) gives
\[
  \Sing(V(g))_{\mathrm{red}}=V(v,y^3+z^4).
\]
At every nonzero point of this curve, \(g\) is analytically equivalent to
\(V^4-H^2\) together with one smooth parameter, and hence has local log
canonical threshold \(3/4\).  Away from the critical locus the threshold is
\(1\).  Thus \(\bigl(\C^3,\tfrac12V(g)\bigr)\) is log canonical away from
the origin.  The weighted-blow-up criterion gives
\[
  \operatorname{lct}_0(g)
  =\frac{5+4+3}{24}
  =\frac12;
\]
see \cite[Proposition~3.2]{Pae24}.  Since
\[
  \operatorname{lct}_0(g)=\min\{\mexp_0(g),1\},
\]
it follows that
\[
  \mexp_0(g)=\frac12.
\]
Applying the Euler-homogeneous Thom--Sebastiani theorem
\cite[Theorem~0.8]{Sai94}, we obtain
\[
  \mexp_Q(X_1(\alpha))
  =\mexp_0(\alpha {u'}^2)+\mexp_0(g)
  =\frac12+\frac12
  =1.
\]

\paragraph*{Global minimal exponent.}
These strata exhaust the singular locus.  Therefore
\[
  \mexp(X_1(\alpha))
  =\min\left\{1,1,\frac54\right\}
  =1.
\]
Thus a general closed-orbit representative of \(\Theta_1\) has global
minimal exponent
\[
  \mexp(X_1(\alpha))=1=\frac{4+1}{5}.
\]

\subsection{\texorpdfstring{The component \(\Theta_2\)}{The component Theta 2}}
\label{subsec:singular-minimal-theta-2}

By Table~\ref{tab:component-normal-forms},
\[
  \Theta_2=\Phi_2=\Phi_{27},
  \qquad
  k(2)=2.
\]
A general closed orbit in \(\Theta_2\) is represented by
\[
  \begin{aligned}
  F_2(\alpha,\beta):=\nf_2(\alpha,\beta)
  ={}&x_1^3x_2^2+x_1^4x_3+x_0x_2x_3^3+x_0x_2^3x_4\\
    &+\alpha x_0x_1x_2x_3x_4+\beta x_0^2x_4^3,
  \end{aligned}
  \qquad
  (\alpha,\beta)\in(\C^\times)^2,
\]
and we set
\[
  X_2(\alpha,\beta):=V(F_2(\alpha,\beta))\subset\PP^4.
\]
The required genericity condition is
\[
  \mathcal D_2(\alpha,\beta)\neq0,
\]
where
\[
  \mathcal D_2(\alpha,\beta)
  =\alpha^6-2\alpha^5+\alpha^4
   +54\alpha^3\beta-378\alpha^2\beta+576\alpha\beta
   +729\beta^2-256\beta.
\]

\paragraph*{Singular locus.}
For readability, write \(F=F_2(\alpha,\beta)\).  Its partial derivatives are
\begin{align*}
  \frac{\partial F}{\partial x_0}
  &=x_2x_3^3+x_2^3x_4+\alpha x_1x_2x_3x_4
    +2\beta x_0x_4^3,\\
  \frac{\partial F}{\partial x_1}
  &=3x_1^2x_2^2+4x_1^3x_3+\alpha x_0x_2x_3x_4,\\
  \frac{\partial F}{\partial x_2}
  &=2x_1^3x_2+x_0x_3^3+3x_0x_2^2x_4
    +\alpha x_0x_1x_3x_4,\\
  \frac{\partial F}{\partial x_3}
  &=x_1^4+3x_0x_2x_3^2+\alpha x_0x_1x_2x_4,\\
  \frac{\partial F}{\partial x_4}
  &=x_0\bigl(x_2^3+\alpha x_1x_2x_3+3\beta x_0x_4^2\bigr).
\end{align*}
On the hyperplane \(x_0=0\), these equations reduce on the support to
\[
  x_1=0,
  \qquad
  x_2(x_3^3+x_2^2x_4)=0.
\]
Hence the singular support in this hyperplane is \(L\cup C\), where
\[
  L:=V(x_0,x_1,x_2)
\]
is a line and
\[
  C:=V(x_0,x_1,x_3^3+x_2^2x_4)
\]
is an irreducible cuspidal plane cubic.  A parametrization of \(C\) is
\[
  [s:t]\longmapsto[0:0:s^3:-s^2t:t^3].
\]
Its unique cusp is
\[
  Q=(0:0:0:0:1),
  \qquad
  L\cap C=\{Q\}.
\]

On the chart \(x_0=1\), put
\[
  (u,v,z,w)=(x_1,x_2,x_3,x_4).
\]
The affine equation is
\[
  f=u^3v^2+u^4z+vz^3+v^3w+\alpha uvzw+\beta w^3.
\]
If one of \(u,v,z,w\) vanishes, the Jacobian equations force
\(u=v=z=w=0\), giving the isolated point
\[
  P=(1:0:0:0:0).
\]
If \(uvzw\neq0\), multiplying the four affine Jacobian equations by
\(u,v,z,w\), respectively, and eliminating the occurring monomials yields a
parameter \(t\) satisfying
\[
  \alpha(t-1)^2=1-4t,
  \qquad
  \beta(t-1)^6=t^2.
\]
Eliminating \(t\) gives \(\mathcal D_2(\alpha,\beta)=0\).  Thus, under the
genericity condition above, there are no further affine singular points, and
\[
  \Sing(X_2(\alpha,\beta))_{\mathrm{red}}
  =(L\cup C)\sqcup\{P\}.
\]

\paragraph*{The isolated point \(P\).}
The local equation \(f\) is weighted homogeneous of degree \(27\) with
weights
\[
  \wt(u,v,z,w)=(5,6,7,9).
\]
The condition \(\mathcal D_2(\alpha,\beta)\neq0\) ensures that its critical
point is isolated.  With
\[
  (X_1,X_2,X_3,X_4)=(u,v,z,w),
\]
the local equation is precisely the defining normal form of the family
\(\Sfam{88}{84}(\alpha,\beta)\).  Hence, in the Yonemura-number notation of
Section~\ref{sec:singular-loci},
\[
  (X_2(\alpha,\beta),P)
  \sim_{\mathrm{an}}
  \Sfam{88}{84}(\alpha,\beta).
\]
In particular, the isolated singularity has Milnor number \(88\) and
Yonemura number \(84\).
Therefore
\[
  \mexp_P(X_2(\alpha,\beta))
  =\frac{5+6+7+9}{27}
  =1.
\]

\paragraph*{Smooth points of the line \(L\).}
Let \(p\in L\setminus\{Q\}\).  Since \(x_3(p)\neq0\), work in the chart
\(x_3=1\) and write
\[
  u=x_0,
  \qquad
  v=x_1,
  \qquad
  y=x_2,
  \qquad
  x_4=q+\zeta,
\]
where \(q=x_4(p)/x_3(p)\).  The quadratic part in the transverse variables
\((u,y)\) is
\[
  uy+\beta q^3u^2,
\]
which is nondegenerate for every \(q\), while
\[
  F(0,v,0,\zeta)=v^4.
\]
By the holomorphic splitting lemma, the germ is analytically equivalent to
\[
  U_1^2+U_2^2+v^4
\]
together with the smooth parameter \(\zeta\).  Thus the transverse type is
\(A_3\), and
\[
  \mexp_p(X_2(\alpha,\beta))
  =\frac12+\frac12+\frac14
  =\frac54
  \qquad
  (p\in L\setminus\{Q\}).
\]
The generic transverse Jacobian ideal is \((U_1,U_2,v^3)\), so the saturated
Jacobian scheme has generic transverse length \(3\) along \(L\).

\paragraph*{Smooth points of the cubic \(C\).}
Let \(p\in C\setminus\{Q\}\).  Then \(x_2(p)\neq0\).  On the chart
\(x_2=1\), put
\[
  \eta=x_4+x_3^3.
\]
The local equation becomes
\[
  F=x_1^3+x_1^4x_3+x_0\eta
    +\alpha x_0x_1x_3(\eta-x_3^3)
    +\beta x_0^2(\eta-x_3^3)^3.
\]
The quadratic form in \((x_0,\eta)\) has rank \(2\).  Its critical section is
\[
  x_0=0,
  \qquad
  \eta=\frac{\alpha x_1x_3^4}{1+\alpha x_1x_3},
\]
and the restriction of \(F\) to this section is
\[
  x_1^3(1+x_1x_3).
\]
Since the second factor is a unit, the transverse germ is analytically
equivalent to
\[
  U_1^2+U_2^2+x_1^3.
\]
Hence the transverse type is \(A_2\), and
\[
  \mexp_p(X_2(\alpha,\beta))
  =\frac12+\frac12+\frac13
  =\frac43
  \qquad
  (p\in C\setminus\{Q\}).
\]
The generic transverse Jacobian ideal is \((U_1,U_2,x_1^2)\), so the
saturated Jacobian scheme has generic transverse length \(2\) along \(C\).

\paragraph*{The intersection point \(Q\).}
In the chart \(x_4=1\), with
\[
  (u,v,y,z)=(x_0,x_1,x_2,x_3),
\]
the local equation is
\[
  f_Q
  =\beta u^2+u(y^3+yz^3+\alpha vyz)+v^3y^2+v^4z.
\]
Completing the square in \(u\) gives
\[
  f_Q=\beta {u'}^2+g(v,y,z),
\]
where
\[
  g(v,y,z)
  =v^3y^2+v^4z
   -\frac1{4\beta}\bigl(y^3+yz^3+\alpha vyz\bigr)^2.
\]
The germ \(g\) is weighted homogeneous of degree \(18\) with weights
\[
  \wt(v,y,z)=(4,3,2).
\]
The face with zero \(v\)-exponent is proportional to
\(y^2(y^2+z^3)^2\), so \(g\) is Newton degenerate.  We therefore use the
weighted-cone criterion instead.  Its reduced critical locus is the union
\[
  \Sigma_L=V(v,y),
  \qquad
  \Sigma_C=V(v,y^2+z^3).
\]
At a nonzero point of \(\Sigma_L\), the transverse type is
\(Y^2+V^4\), whose log canonical threshold is \(3/4\); at a nonzero point
of \(\Sigma_C\), it is \(Y^2+V^3\), whose log canonical threshold is
\(5/6\).  Away from these curves the divisor \(V(g)\) is smooth.  Hence
\(\bigl(\C^3,\tfrac12V(g)\bigr)\) is log canonical away from the origin.
The weighted-cone criterion \cite[Proposition~3.2]{Pae24} gives
\[
  \operatorname{lct}_0(g)
  =\frac{4+3+2}{18}
  =\frac12.
\]
By the general relation \cite{MP20},
\[
  \operatorname{lct}_0(g)=\min\{\mexp_0(g),1\},
\]
so
\[
  \mexp_0(g)=\frac12.
\]
The Euler-homogeneous Thom--Sebastiani theorem
\cite[Proposition~0.7 and Theorem~0.8]{Sai94}, together with
\cite[Theorem~1.2]{CDNM24}, then yields
\[
  \mexp_Q(X_2(\alpha,\beta))
  =\mexp_0(\beta {u'}^2)+\mexp_0(g)
  =\frac12+\frac12
  =1.
\]

\paragraph*{Global minimal exponent.}
The preceding strata exhaust the singular locus.  Therefore
\[
  \mexp(X_2(\alpha,\beta))
  =\min\left\{1,1,\frac54,\frac43\right\}
  =1
  =\frac{4+1}{5}.
\]
Thus a general closed-orbit representative of \(\Theta_2\) has global
minimal exponent equal to \(1\).

\subsection{\texorpdfstring{The component \(\Theta_3\)}{The component Theta 3}}
\label{subsec:singular-minimal-theta-3}

By Table~\ref{tab:component-normal-forms},
\[
  \Theta_3=\Phi_3=\Phi_{33},
  \qquad
  k(3)=3.
\]
Write
\[
  B_5(x_1,x_2)
  =b_0x_1^5+b_1x_1^4x_2+b_2x_1^3x_2^2
   +b_3x_1^2x_2^3+b_4x_1x_2^4+b_5x_2^5.
\]
A general closed orbit in \(\Theta_3\) is represented by
\[
  \begin{aligned}
  F_3(B_5):=\nf_3(B_5)
  ={}&B_5(x_1,x_2)+x_0x_2x_3^3\\
     &+x_0(x_1^2+x_2^2)x_3x_4+x_0^2x_4^3,
  \end{aligned}
\]
and we set
\[
  X_3(B_5):=V(F_3(B_5))\subset\PP^4.
\]
We work on the nonempty Zariski-open set
\[
  b_0b_2\operatorname{Disc}(B_5)\mathcal A_3(B_5)\neq0,
\]
where, writing \(B_{5,i}=\partial B_5/\partial x_i\),
\[
  \mathcal A_3(B_5)
  :=\operatorname{Res}_{\rho}\!\left(
    9B_{5,1}(\rho,1)+2\rho(\rho^2+1)^2,
    27B_{5,2}(\rho,1)+(\rho^2+1)^2(5-\rho^2)
  \right).
\]

\paragraph*{Singular locus.}
On the hyperplane \(x_0=0\), the Jacobian equations include
\[
  B_{5,1}(x_1,x_2)=B_{5,2}(x_1,x_2)=0.
\]
Since \(B_5\) is square-free, these equations force \(x_1=x_2=0\).
Thus the singular support in this hyperplane is the line
\[
  L:=V(x_0,x_1,x_2).
\]

On the chart \(x_0=1\), put
\[
  (u,v,z,w)=(x_1,x_2,x_3,x_4).
\]
The affine equation is
\[
  f=B_5(u,v)+vz^3+(u^2+v^2)zw+w^3,
\]
and its critical-point equations are
\begin{align*}
  B_{5,1}(u,v)+2uzw&=0,\\
  B_{5,2}(u,v)+z^3+2vzw&=0,\\
  3vz^2+(u^2+v^2)w&=0,\\
  (u^2+v^2)z+3w^2&=0.
\end{align*}
If \(z=0\) or \(w=0\), these equations give
\(u=v=z=w=0\), hence the isolated point
\[
  P=(1:0:0:0:0).
\]
If \(zw\neq0\), then \(v\neq0\).  Using the weighted action of weights
\((3,3,4,5)\), normalize \(v=1\) and put \(\rho=u/v\).  The last two
critical equations yield
\[
  zw=\frac{(\rho^2+1)^2}{9},
  \qquad
  z^3=-\frac{(\rho^2+1)^3}{27},
\]
while the first two become the two equations occurring in the definition of
\(\mathcal A_3(B_5)\).  Hence \(\mathcal A_3(B_5)\neq0\) excludes all such
points.  Therefore
\[
  \Sing(X_3(B_5))_{\mathrm{red}}=L\sqcup\{P\}.
\]
The distinguished point on \(L\) is
\[
  Q=(0:0:0:0:1).
\]

\paragraph*{The isolated point \(P\).}
The local equation at \(P\) is the weighted-homogeneous polynomial
\[
  f_P=B_5(u,v)+vz^3+(u^2+v^2)zw+w^3
\]
of degree \(15\) with weights
\[
  \wt(u,v,z,w)=(3,3,4,5).
\]
The genericity assumptions make its critical point isolated.  With
\[
  (X_1,X_2,X_3,X_4)=(u,v,z,w),
\]
the local equation is precisely the defining normal form of the family
\(\Sfam{88}{15}(B_5)\).  Hence, in the Yonemura-number notation of
Section~\ref{sec:singular-loci},
\[
  (X_3(B_5),P)
  \sim_{\mathrm{an}}
  \Sfam{88}{15}(B_5).
\]
In particular, the isolated singularity has Milnor number \(88\) and
Yonemura number \(15\).
Therefore
\[
  \mexp_P(X_3(B_5))
  =\frac{3+3+4+5}{15}
  =1.
\]

\paragraph*{Points of \(L\setminus\{Q\}\).}
Let \(p\in L\setminus\{Q\}\).  In the chart \(x_3=1\), write
\[
  u=x_0,
  \qquad
  v=x_1,
  \qquad
  y=x_2,
  \qquad
  x_4=q+\zeta.
\]
Then
\[
  F_3
  =B_5(v,y)+uy+u(v^2+y^2)(q+\zeta)+u^2(q+\zeta)^3.
\]
The quadratic part in \((u,y)\) is \(uy+q^3u^2\), which has rank \(2\).
The splitting lemma therefore gives
\[
  F_3\sim_{\mathrm{an}}U_1^2+U_2^2+g(v,\zeta).
\]
Solving the critical-section equations in \((u,y)\) gives
\[
  y=-(q+\zeta)v^2+O(v^4),
  \qquad
  u=O(v^4),
\]
and consequently
\[
  g(v,\zeta)=b_0v^5+O(v^6).
\]
Since \(b_0\neq0\), an analytic change of the coordinate \(v\) yields
\[
  F_3\sim_{\mathrm{an}}U_1^2+U_2^2+V^5
\]
together with the smooth parameter \(\zeta\).  Thus the transverse type is
\(A_4\), and
\[
  \mexp_p(X_3(B_5))
  =\frac12+\frac12+\frac15
  =\frac65
  \qquad
  (p\in L\setminus\{Q\}).
\]
The generic transverse Jacobian ideal is \((U_1,U_2,V^4)\); hence the
saturated Jacobian scheme has generic transverse length \(4\) along \(L\).

\paragraph*{The distinguished point \(Q\).}
In the chart \(x_4=1\), with
\[
  (u,v,y,z)=(x_0,x_1,x_2,x_3),
\]
the local equation is
\[
  f_Q=u^2+B_5(v,y)+uyz^3+u(v^2+y^2)z.
\]
Set
\[
  A(v,y,z):=z(v^2+y^2+yz^2),
  \qquad
  U:=u+\frac12A(v,y,z).
\]
Completing the square gives
\[
  f_Q=U^2+h(v,y,z),
\]
where
\[
  h(v,y,z)
  =B_5(v,y)-\frac14z^2(v^2+y^2+yz^2)^2.
\]
To compute its log canonical threshold, introduce
\[
  G(v,y,t)
  :=B_5(v,y)-\frac14t(v^2+y^2+yt)^2,
  \qquad
  h(v,y,z)=G(v,y,z^2).
\]
For the double cover
\[
  \pi\colon(v,y,z)\longmapsto(v,y,z^2),
\]
branched along \(T=V(t)\), and \(D=V(G)\), the finite-cover criterion gives
\[
  \operatorname{lct}_0(h)
  =\sup\left\{
    c>0\ \middle|\
    \left(\C^3_{v,y,t},\frac12T+cD\right)
    \text{ is log canonical at }0
  \right\}.
\]
Since \(\operatorname{mult}_0T=1\) and \(\operatorname{mult}_0D=5\), the
exceptional divisor of the ordinary blow-up has log discrepancy
\[
  3-\frac12-5c.
\]
Thus log canonicity forces \(c\leq1/2\), and hence
\[
  \operatorname{lct}_0(h)\leq\frac12.
\]
For \(c=1/2\), inversion of adjunction reduces log canonicity to that of
\[
  \left(\PP^2,\frac12\ell_\infty+\frac12C\right),
\]
where
\[
  \ell_\infty=V(t),
  \qquad
  C=V(G)\subset\PP^2_{[v:y:t]}.
\]
The curve \(C\) meets \(\ell_\infty\) transversely at the five distinct
roots of \(B_5\).  The resultant condition excludes singularities of \(C\)
away from
\[
  R=(0:0:1).
\]
At \(R\), the splitting lemma gives
\[
  (C,R)\sim_{\mathrm{an}}V(Y^2+V^5),
  \qquad
  \operatorname{lct}_R(C)=\frac12+\frac15=\frac7{10}>\frac12.
\]
It follows that the displayed pair on \(\PP^2\) is log canonical, and hence
\[
  \operatorname{lct}_0(h)=\frac12.
\]
Using
\[
  \operatorname{lct}_0(h)=\min\{\mexp_0(h),1\},
\]
we obtain \(\mexp_0(h)=1/2\); see \cite{MP20}.  The Thom--Sebastiani theorem
then yields
\[
  \mexp_Q(X_3(B_5))
  =\mexp_0(U^2)+\mexp_0(h)
  =\frac12+\frac12
  =1;
\]
see \cite{Sai94,CDNM24}.

\paragraph*{Global minimal exponent.}
The preceding strata exhaust the singular locus.  Therefore
\[
  \mexp(X_3(B_5))
  =\min\left\{1,1,\frac65\right\}
  =1
  =\frac{4+1}{5}.
\]
Thus a general closed-orbit representative of \(\Theta_3\) has global
minimal exponent equal to \(1\).

\subsection{\texorpdfstring{The component \(\Theta_4\)}{The component Theta 4}}
\label{subsec:singular-minimal-theta-4}

By Table~\ref{tab:component-normal-forms},
\[
  \Theta_4=\Phi_4=\Phi_{21},
  \qquad
  k(4)=4.
\]
A general closed orbit in \(\Theta_4\) is represented by
\[
  \begin{aligned}
  F_4(\alpha,\beta):=\nf_4(\alpha,\beta)
  ={}&x_1^2x_2^3+x_1^3x_3x_4+x_0x_2^4+x_0x_1x_3^3\\
     &+\alpha x_0x_1x_2x_3x_4+\beta x_0^2x_4^3,
  \end{aligned}
\]
where \((\alpha,\beta)\in(\C^\times)^2\), and we set
\[
  X_4(\alpha,\beta):=V(F_4(\alpha,\beta))\subset\PP^4.
\]
We work on the nonempty Zariski-open set
\[
  \mathcal D_4(\alpha,\beta)\neq0,
\]
where
\[
  \mathcal D_4(\alpha,\beta)
  =\alpha^6-2\alpha^5+\alpha^4
   +54\alpha^3\beta-378\alpha^2\beta+576\alpha\beta
   +729\beta^2-256\beta.
\]

\paragraph*{Singular locus.}
On the hyperplane \(x_0=0\), the Jacobian equations give either
\(x_1=x_2=0\), which yields the line
\[
  L:=V(x_0,x_1,x_2),
\]
or the single point
\[
  R=(0:1:0:0:0).
\]
On the chart \(x_0=1\), put
\[
  (u,v,z,w)=(x_1,x_2,x_3,x_4).
\]
The affine equation is
\[
  f=u^2v^3+u^3zw+v^4+uz^3+\alpha uvzw+\beta w^3.
\]
Its critical equations vanish identically along \(v=z=w=0\), giving the
affine part of the line
\[
  M:=V(x_2,x_3,x_4).
\]
Outside \(M\), all four coordinates \(u,v,z,w\) are nonzero.  The weighted
\(\Gm\)-action of weights \((3,6,7,8)\) allows us to normalize \(u=1\),
after which the critical equations imply
\[
  \alpha v^2+4v+3=0,
  \qquad
  (1+\alpha v)^2=9\beta v^4.
\]
Their resultant is \(9\mathcal D_4(\alpha,\beta)\).  Hence the genericity
condition excludes every additional singular point.  Since \(R\in M\), we
obtain
\[
  \Sing(X_4(\alpha,\beta))_{\mathrm{red}}=M\sqcup L,
\]
and the two lines are disjoint.  In particular, there are no isolated
singular points.

\paragraph*{Points of \(M\setminus\{P\}\).}
Let
\[
  P=(1:0:0:0:0).
\]
For \(p\in M\setminus\{P\}\), work in the chart \(x_1=1\), write
\(s=x_0\), \(y=x_2\), \(z=x_3\), and \(w=x_4\), and put
\(s=s_0+\sigma\).  The local equation is
\[
  F_4=y^3+zw+sy^4+sz^3+\alpha syzw+\beta s^2w^3.
\]
The quadratic form in \((z,w)\) is nondegenerate, and \(z=w=0\) is the
exact critical section.  Its critical value is
\[
  y^3\bigl(1+(s_0+\sigma)y\bigr),
\]
so the splitting lemma gives
\[
  F_4\sim_{\mathrm{an}}U_1^2+U_2^2+Y^3
\]
with \(\sigma\) a smooth parameter.  Thus the transverse type is \(A_2\),
and
\[
  \mexp_p(X_4(\alpha,\beta))
  =\frac12+\frac12+\frac13
  =\frac43
  \qquad
  (p\in M\setminus\{P\}).
\]
The generic transverse Jacobian ideal is \((U_1,U_2,Y^2)\), so its generic
length along \(M\) is \(2\).

\paragraph*{The distinguished point \(P\).}
In the chart \(x_0=1\), the local equation at \(P\) is
\[
  f_P=u^2v^3+u^3zw+v^4+uz^3+\alpha uvzw+\beta w^3.
\]
It is weighted homogeneous of degree \(24\) for
\[
  \wt(u,v,z,w)=(3,6,7,8).
\]
Every proper Newton face is nondegenerate, while a torus critical point of
the principal face would determine a singular point of
\(X_4(\alpha,\beta)\) outside \(M\sqcup L\), which has already been
excluded.  Thus \(f_P\) is Newton nondegenerate.  Since the monomial
\(uvzw\) occurs and all support exponents lie on the weight hyperplane of
degree \(24\), its Newton distance is \(1\).  Consequently
\[
  \operatorname{lct}_0(f_P)=1
\]
by the Newton-polyhedron formula \cite{How03}.  Using
\(\operatorname{lct}_0(f_P)=\min\{\mexp_0(f_P),1\}\) and the weighted upper
bound for the minimal exponent, we obtain
\[
  1\leq\mexp_0(f_P)
  \leq\frac{3+6+7+8}{24}=1;
\]
see \cite{MP20,CDM25}.  Therefore
\[
  \mexp_P(X_4(\alpha,\beta))=1.
\]

\paragraph*{Points of \(L\setminus\{Q\}\).}
Let
\[
  Q=(0:0:0:0:1).
\]
For \(p\in L\setminus\{Q\}\), work in the chart \(x_3=1\), write
\[
  u=x_0,
  \qquad
  v=x_1,
  \qquad
  y=x_2,
  \qquad
  x_4=q+\zeta,
\]
and set
\[
  A=1+\alpha y(q+\zeta),
  \qquad
  B=q+\zeta.
\]
Then
\[
  F_4=Auv+\beta B^3u^2+uy^4+y^3v^2+Bv^3.
\]
The quadratic form in \((u,v)\) has rank \(2\).  If
\((u_*(y,\zeta),v_*(y,\zeta))\) denotes its critical section, then
\[
  v_*=-\frac{y^4}{A}+O_y(y^7),
  \qquad
  u_*=\frac{2y^7}{A^2}+O_y(y^8),
\]
and substitution gives
\[
  F_4(u_*,v_*,y,\zeta)
  =\frac{y^{11}}{A^2}+O_y(y^{12})
  =y^{11}\cdot\operatorname{unit}.
\]
Hence
\[
  F_4\sim_{\mathrm{an}}U_1^2+U_2^2+Y^{11}
\]
with \(\zeta\) a smooth parameter.  The transverse type is \(A_{10}\), and
\[
  \mexp_p(X_4(\alpha,\beta))
  =\frac12+\frac12+\frac1{11}
  =\frac{12}{11}
  \qquad
  (p\in L\setminus\{Q\}).
\]
The generic transverse Jacobian ideal is \((U_1,U_2,Y^{10})\), so its
generic length along \(L\) is \(10\).

\paragraph*{The distinguished point \(Q\).}
In the chart \(x_4=1\), with
\[
  (u,v,y,z)=(x_0,x_1,x_2,x_3),
\]
the local equation is
\[
  f_Q=\beta u^2+v^2y^3+v^3z+uy^4+uvz^3+\alpha uvyz.
\]
It is weighted homogeneous of degree \(16\) for
\[
  \wt(u,v,y,z)=(8,5,2,1).
\]
The same Newton argument as at \(P\) shows that \(f_Q\) is nondegenerate and
has Newton distance \(1\).  Hence
\[
  \operatorname{lct}_0(f_Q)=1,
  \qquad
  1\leq\mexp_0(f_Q)
  \leq\frac{8+5+2+1}{16}=1,
\]
and therefore
\[
  \mexp_Q(X_4(\alpha,\beta))=1.
\]

\paragraph*{Global minimal exponent.}
The preceding strata exhaust the singular locus.  Therefore
\[
  \mexp(X_4(\alpha,\beta))
  =\min\left\{\frac43,1,\frac{12}{11},1\right\}
  =1
  =\frac{4+1}{5}.
\]
Thus a general closed-orbit representative of \(\Theta_4\) has global
minimal exponent equal to \(1\).

\subsection{\texorpdfstring{The component \(\Theta_5\)}{The component Theta 5}}
\label{subsec:singular-minimal-theta-5}

By Table~\ref{tab:component-normal-forms},
\[
  \Theta_5=\Phi_5=\Phi_{24},
  \qquad
  k(5)=5.
\]
A general closed orbit in \(\Theta_5\) is represented by
\[
  \begin{aligned}
  F_5(\alpha,\beta,\gamma):=\nf_5(\alpha,\beta,\gamma)
  ={}&x_1^2x_2^3+x_1^3x_3^2+x_1^3x_2x_4+x_0x_2^3x_3\\
     &+\alpha x_0x_1x_3^3
      +\beta x_0x_1x_2x_3x_4
      +\gamma x_0^2x_4^3,
  \end{aligned}
\]
where \((\alpha,\beta,\gamma)\in(\C^\times)^3\), and we set
\[
  X_5(\alpha,\beta,\gamma)
  :=V(F_5(\alpha,\beta,\gamma))\subset\PP^4.
\]
We work on the nonempty Zariski-open set
\[
  \alpha\neq1,
  \qquad
  \mathcal R_5(\alpha,\beta,\gamma)\neq0,
\]
where
\[
  \mathcal R_5(\alpha,\beta,\gamma)
  :=\operatorname{Res}_s\bigl(A_5(s),B_5(s)\bigr),
\]
with
\[
  \begin{aligned}
  A_5(s)={}&81\alpha\gamma s^2(s-1)^2
    +\beta^3(s-1)^2(2s+1)+3\beta^2(s^2-1)\\
    &+3\beta+54\gamma s^2(s-1)-1,
  \end{aligned}
\]
and
\[
  \begin{aligned}
  B_5(s)={}&27\alpha\gamma s^2(s-1)^3
    +\beta^3(s-1)^3(3s-2)
    +3\beta^2(s-1)^2(5s-2)\\
    &+3\beta(s-1)(7s-2)
    +81\gamma s^2(s-1)^2+9s-2.
  \end{aligned}
\]
The resultant contains the factor \((\beta-1)^6\); hence its nonvanishing
also implies \(\beta\neq1\).

\paragraph*{Singular locus.}
On the hyperplane \(x_0=0\), the Jacobian equations give the two lines
\[
  L:=V(x_0,x_1,x_2),
  \qquad
  M:=V(x_0,x_1,x_3),
\]
together with the point \((0:1:0:0:0)\), which lies on the line
\[
  N:=V(x_2,x_3,x_4).
\]
On the chart \(x_0=1\), put
\[
  (u,v,z,w)=(x_1,x_2,x_3,x_4).
\]
The affine equation is
\[
  f=u^2v^3+u^3z^2+u^3vw+v^3z
    +\alpha uz^3+\beta uvzw+\gamma w^3.
\]
Its critical equations vanish identically along \(v=z=w=0\), giving the
remaining affine part of \(N\).  Away from \(N\), one has \(u\neq0\), and
the weighted action of weights \((3,5,6,7)\) allows us to normalize
\(u=1\).  The critical equations then force \(vw\neq0\).  Writing
\[
  w=tv,
  \qquad
  s=\gamma t^3,
\]
the equations \(f_v=f_w=0\) give
\[
  z=s-1,
  \qquad
  v=-\frac{1-\beta+\beta s}{3\gamma t^2},
  \qquad
  w=-\frac{1-\beta+\beta s}{3\gamma t},
\]
and the remaining two equations become
\[
  A_5(s)=B_5(s)=0.
\]
Thus \(\mathcal R_5\neq0\) excludes every additional affine singular
point.  Consequently
\[
  \Sing(X_5(\alpha,\beta,\gamma))_{\mathrm{red}}
  =N\sqcup(L\cup M).
\]
Here \(N\) is disjoint from \(L\cup M\), while
\[
  L\cap M=\{Q\},
  \qquad
  Q=(0:0:0:0:1).
\]
In particular, there are no isolated singular points.  Scheme-theoretically,
\(N\) and \(M\) are generically reduced, whereas the generic transverse
length along \(L\) is \(8\).

\paragraph*{Points of \(N\setminus\{P\}\).}
Let
\[
  P=(1:0:0:0:0).
\]
For \(p\in N\setminus\{P\}\), work in the chart \(x_1=1\) and write
\[
  (s,y,z,w)=(x_0,x_2,x_3,x_4).
\]
The local equation is
\[
  y^3+z^2+yw+sy^3z+\alpha sz^3+\beta syzw+\gamma s^2w^3.
\]
Its quadratic part \(z^2+yw\) has rank \(3\).  Hence the transverse germ is
a three-variable Morse singularity, with \(s\) as a smooth parameter, and
\[
  \mexp_p(X_5(\alpha,\beta,\gamma))
  =\frac12+\frac12+\frac12
  =\frac32
  \qquad
  (p\in N\setminus\{P\}).
\]

\paragraph*{The distinguished point \(P\).}
At \(P\), the local equation is
\[
  f_P=u^2v^3+u^3z^2+u^3vw+v^3z
      +\alpha uz^3+\beta uvzw+\gamma w^3.
\]
It is weighted homogeneous of degree \(21\) for
\[
  \wt(u,v,z,w)=(3,5,6,7).
\]
Every proper simplex face of its Newton boundary is nondegenerate.  The only
nonsimplicial proper face has face polynomial
\[
  u^2v^3+u^3z^2+v^3z+\alpha uz^3,
\]
and its logarithmic critical equations have a torus solution precisely when
\(\alpha=1\).  A torus critical point of the full compact face would give an
affine singular point outside \(N\), which is excluded by
\(\mathcal R_5\neq0\).  Thus \(f_P\) is Newton nondegenerate.

Since the monomial \(uvzw\) occurs, every Newton ratio is at least \(1\),
and equality is attained by the weight vector \((3,5,6,7)\).  Hence
\[
  \operatorname{lct}_0(f_P)=1
\]
by the Newton-polyhedron formula \cite{How03}, and therefore
\(\mexp_P(X_5)\geq1\) by \cite{MP20}.  On the other hand, the weighted
blow-up of weight \((3,5,6,7)\) has log discrepancy
\[
  3+5+6+7-21=0
\]
over the hypersurface germ, so the germ is not canonical.  Since it is a
normal Gorenstein hypersurface, it cannot be rational.  Saito's rationality
criterion
\cite{Sai94,MP20} gives \(\mexp_P(X_5)\leq1\).  Therefore
\[
  \mexp_P(X_5(\alpha,\beta,\gamma))=1.
\]

\paragraph*{Points of \(L\setminus\{Q\}\).}
For \(p\in L\setminus\{Q\}\), work in the chart \(x_3=1\), write
\[
  (u,v,y)=(x_0,x_1,x_2),
  \qquad
  x_4=q+\zeta,
\]
and regard \(\zeta\) as the parameter along \(L\).  The quadratic part in
\((u,v)\) is
\[
  \alpha uv+\gamma q^3u^2,
\]
which has rank \(2\).  Eliminating these two Morse variables gives the
transverse critical value
\[
  \frac{\alpha-1}{\alpha^3}y^9+O(y^{10}).
\]
Since \(\alpha\neq1\), the splitting lemma yields
\[
  F_5\sim_{\mathrm{an}}U_1^2+U_2^2+Y^9
\]
with one smooth parameter.  Thus the transverse type is \(A_8\), and
\[
  \mexp_p(X_5(\alpha,\beta,\gamma))
  =\frac12+\frac12+\frac19
  =\frac{10}{9}
  \qquad
  (p\in L\setminus\{Q\}).
\]
The generic transverse Jacobian ideal is
\((U_1,U_2,Y^8)\), explaining the generic transverse length \(8\) along
\(L\).

\paragraph*{Points of \(M\setminus\{Q\}\).}
For \(p\in M\setminus\{Q\}\), work in the chart \(x_2=1\), write
\[
  (u,v,z)=(x_0,x_1,x_3),
  \qquad
  x_4=q+\zeta.
\]
The quadratic part in the transverse variables is
\[
  v^2+uz+\gamma q^3u^2,
\]
which has rank \(3\).  Hence the transverse germ is again a three-variable
Morse singularity, and
\[
  \mexp_p(X_5(\alpha,\beta,\gamma))
  =\frac32
  \qquad
  (p\in M\setminus\{Q\}).
\]

\paragraph*{The intersection point \(Q\).}
In the chart \(x_4=1\), with
\[
  (u,v,y,z)=(x_0,x_1,x_2,x_3),
\]
the local equation is
\[
  f_Q=\gamma u^2+v^2y^3+v^3z^2+v^3y+uy^3z
      +\alpha uvz^3+\beta uvyz.
\]
It is weighted homogeneous of degree \(14\) for
\[
  \wt(u,v,y,z)=(7,4,2,1).
\]
As at \(P\), the only nonsimplicial proper Newton face has face polynomial
\[
  v^2y^3+v^3z^2+uy^3z+\alpha uvz^3,
\]
whose torus critical equations force \(\alpha=1\).  The condition
\(\mathcal R_5\neq0\) excludes a critical point of the full compact face.
Thus \(f_Q\) is Newton nondegenerate.  Since the monomial \(uvyz\) occurs
and
\[
  7+4+2+1=14,
\]
the Newton-polyhedron formula gives
\[
  \operatorname{lct}_0(f_Q)=1.
\]
The weighted blow-up of weight \((7,4,2,1)\) has log discrepancy zero over
the hypersurface germ, so the germ is not canonical.  Since it is normal
Gorenstein, the same rationality argument as at \(P\) gives
\[
  \mexp_Q(X_5(\alpha,\beta,\gamma))=1.
\]

\paragraph*{Global minimal exponent.}
The preceding strata exhaust the singular locus.  Therefore
\[
  \mexp(X_5(\alpha,\beta,\gamma))
  =\min\left\{1,\frac32,\frac{10}{9},\frac32,1\right\}
  =1
  =\frac{4+1}{5}.
\]
Thus a general closed-orbit representative of \(\Theta_5\) has global
minimal exponent equal to \(1\).

\subsection{\texorpdfstring{The component \(\Theta_6\)}{The component Theta 6}}
\label{subsec:singular-minimal-theta-6}

By Table~\ref{tab:component-normal-forms},
\[
  \Theta_6=\Phi_6=\Phi_{28},
  \qquad
  k(6)=6.
\]
A general closed orbit in \(\Theta_6\) is represented by
\[
  \begin{aligned}
  F_6(\alpha,\beta,\gamma,\delta,\epsilon)
  :=\nf_6(\alpha,\beta,\gamma,\delta,\epsilon)
  ={}&x_1^2x_2^3+x_1^3x_2x_3+x_1^4x_4+x_0x_2^2x_3^2
      +\alpha x_0x_2^3x_4\\
     &+\beta x_0x_1x_3^3
      +\gamma x_0x_1x_2x_3x_4
      +\delta x_0x_1^2x_4^2
      +\epsilon x_0^2x_4^3,
  \end{aligned}
\]
where
\((\alpha,\beta,\gamma,\delta,\epsilon)\in(\C^\times)^5\), and we set
\[
  X_6(\alpha,\beta,\gamma,\delta,\epsilon)
  :=V(F_6(\alpha,\beta,\gamma,\delta,\epsilon))\subset\PP^4.
\]
Let
\[
  \begin{aligned}
  f_P(u,v,z,w)={}&u^2v^3+u^3vz+u^4w+v^2z^2+\alpha v^3w
     +\beta uz^3\\
     &+\gamma uvzw+\delta u^2w^2+\epsilon w^3.
  \end{aligned}
\]
This is weighted homogeneous of degree \(18\) for the weights
\((3,4,5,6)\).  Denote by \(\mathcal D_6\) the closed locus of parameters
for which \(\nabla f_P=0\) has a solution in \(\PP(3,4,5,6)\), and choose a
polynomial \(\mathcal R_6\) such that
\[
  \mathcal D_6\subseteq V(\mathcal R_6),
  \qquad
  \mathcal R_6(1,2,1,1,1)\neq0.
\]
We work on the nonempty Zariski-open set
\[
  \beta\neq1,
  \qquad
  \mathcal R_6(\alpha,\beta,\gamma,\delta,\epsilon)\neq0.
\]
At \((\alpha,\beta,\gamma,\delta,\epsilon)=(1,2,1,1,1)\), a
Gr\"obner-basis computation gives a Jacobi algebra of dimension \(91\).
Weighted homogeneity then excludes every nonzero critical point, so this open
set is nonempty.

\paragraph*{Singular locus.}
On the hyperplane \(x_0=0\), the Jacobian equations reduce on the support to
\[
  x_1=0,
  \qquad
  x_2^2(x_3^2+\alpha x_2x_4)=0.
\]
Hence they give the line
\[
  L:=V(x_0,x_1,x_2)
\]
and the smooth conic
\[
  C:=V(x_0,x_1,x_3^2+\alpha x_2x_4).
\]
These curves meet only at
\[
  Q=(0:0:0:0:1).
\]
On the chart \(x_0=1\), the affine equation is precisely \(f_P\).  The
condition \(\mathcal R_6\neq0\) excludes every nonzero critical point, so
this chart contributes only
\[
  P=(1:0:0:0:0).
\]
Consequently
\[
  \Sing(X_6(\alpha,\beta,\gamma,\delta,\epsilon))_{\mathrm{red}}
  =(L\cup C)\sqcup\{P\}.
\]
Scheme-theoretically, the generic transverse length is \(6\) along \(L\),
whereas \(C\) is generically reduced.

\paragraph*{The isolated point \(P\).}
The local equation at \(P\) is \(f_P\).  By
\(\mathcal R_6\neq0\), its critical point is isolated.  With
\[
  (X_1,X_2,X_3,X_4)=(u,v,z,w),
\]
the local equation is precisely the defining normal form of the family
\(\Sfam{91}{53}(\alpha,\beta,\gamma,\delta,\epsilon)\).  Hence, in the
Yonemura-number notation of Section~\ref{sec:singular-loci},
\[
  \begin{aligned}
  &(X_6(\alpha,\beta,\gamma,\delta,\epsilon),P)\\
  &\qquad\sim_{\mathrm{an}}
  \Sfam{91}{53}(\alpha,\beta,\gamma,\delta,\epsilon).
  \end{aligned}
\]
In particular, the isolated singularity has Milnor number \(91\) and
Yonemura number \(53\).  Since \(f_P\) is weighted homogeneous of degree
\(18\) with weights \((3,4,5,6)\),
\[
  \mexp_P(X_6(\alpha,\beta,\gamma,\delta,\epsilon))
  =\frac{3+4+5+6}{18}
  =1.
\]

\paragraph*{Points of \(L\setminus\{Q\}\).}
Let \(p\in L\setminus\{Q\}\).  In the chart \(x_3=1\), write
\[
  u=x_0,
  \qquad
  v=x_1,
  \qquad
  y=x_2,
  \qquad
  x_4=q+\zeta,
\]
where \(\zeta\) is a parameter along \(L\).  The Hessian in \((u,v)\) has
determinant \(-\beta^2\), so the splitting lemma gives
\[
  F_6\sim_{\mathrm{an}}U_1^2+U_2^2+h(y,\zeta).
\]
Solving the critical-section equations in \((u,v)\) yields
\[
  h(y,\zeta)
  =\frac{\beta-1}{\beta^3}y^7+O(y^8).
\]
Since \(\beta\neq1\), the transverse type is \(A_6\).  Therefore
\[
  \mexp_p(X_6(\alpha,\beta,\gamma,\delta,\epsilon))
  =\frac12+\frac12+\frac17
  =\frac87
  \qquad
  (p\in L\setminus\{Q\}).
\]
The transverse Jacobian ideal is analytically equivalent to
\((U_1,U_2,y^6)\), which gives the generic transverse length \(6\) along
\(L\).

\paragraph*{Points of \(C\setminus\{Q\}\).}
Let \(p\in C\setminus\{Q\}\).  Since \(x_2(p)\neq0\), work in the chart
\(x_2=1\) and write
\[
  u=x_0,
  \qquad
  v=x_1,
  \qquad
  x_3=q+z,
  \qquad
  x_4=r+\zeta,
  \qquad
  q^2+\alpha r=0.
\]
The quadratic part has rank \(3\); explicitly,
it is
\[
  v^2+u(2qz+\alpha\zeta)+\epsilon r^3u^2.
\]
The splitting lemma, together with the fact that the whole conic lies in the
critical locus, gives a three-variable Morse singularity times one smooth
parameter.  Hence
\[
  \mexp_p(X_6(\alpha,\beta,\gamma,\delta,\epsilon))
  =\frac12+\frac12+\frac12
  =\frac32
  \qquad
  (p\in C\setminus\{Q\}).
\]
In particular, the saturated Jacobian scheme is generically reduced along
\(C\).

\paragraph*{The intersection point \(Q\).}
In the chart \(x_4=1\), with
\[
  (u,v,y,z)=(x_0,x_1,x_2,x_3),
\]
the local equation is
\[
  \begin{aligned}
  f_Q={}&v^2y^3+v^3yz+v^4+uy^2z^2+\alpha uy^3
       +\beta uvz^3\\
      &+\gamma uvyz+\delta uv^2+\epsilon u^2.
  \end{aligned}
\]
It is weighted homogeneous of degree \(12\) for the weights
\((6,3,2,1)\), but its singularity is non-isolated.  Away from the origin,
the two singular branches are the punctured germs of \(L\) and \(C\), where
the local minimal exponents are respectively \(8/7\) and \(3/2\).  Thus
\((\C^4,V(f_Q))\) is log canonical away from the origin.  The
weighted-blow-up criterion then gives
\[
  \operatorname{lct}_0(f_Q)
  =\frac{6+3+2+1}{12}
  =1,
\]
and hence
\[
  \mexp_Q(X_6(\alpha,\beta,\gamma,\delta,\epsilon))\geq1;
\]
see \cite[Proposition~3.2]{Pae24} and \cite{MP20}.

For the reverse inequality, complete the square in \(u\).  Set
\[
  A=y^2z^2+\alpha y^3+\beta vz^3+\gamma vyz+\delta v^2,
  \qquad
  B=v^2y^3+v^3yz+v^4.
\]
Then
\[
  f_Q
  =\epsilon\left(u+\frac{A}{2\epsilon}\right)^2+g,
  \qquad
  g:=B-\frac{A^2}{4\epsilon}.
\]
The germ \(g(v,y,z)\) is weighted homogeneous of degree \(12\) for the
weights \((3,2,1)\).  The corresponding monomial valuation gives
\[
  \operatorname{lct}_0(g)
  \leq\frac{3+2+1}{12}
  =\frac12.
\]
Since this bound is below \(1\), the relation between the log canonical
threshold and the minimal exponent gives
\(\mexp_0(g)=\operatorname{lct}_0(g)\leq1/2\).  Thom--Sebastiani therefore
yields
\[
  \mexp_0(f_Q)
  =\frac12+\mexp_0(g)
  \leq1;
\]
see \cite[Proposition~0.7 and Theorem~0.8]{Sai94} and
\cite[Theorem~1.2]{CDNM24}.  Combining the two inequalities gives
\[
  \mexp_Q(X_6(\alpha,\beta,\gamma,\delta,\epsilon))=1.
\]

\paragraph*{Global minimal exponent.}
The preceding strata exhaust the singular locus.  Therefore
\[
  \mexp(X_6(\alpha,\beta,\gamma,\delta,\epsilon))
  =\min\left\{1,1,\frac87,\frac32\right\}
  =1
  =\frac{4+1}{5}.
\]
Thus a general closed-orbit representative of \(\Theta_6\) has global
minimal exponent equal to \(1\).

\subsection{\texorpdfstring{The component \(\Theta_7\)}{The component Theta 7}}
\label{subsec:singular-minimal-theta-7}

By Table~\ref{tab:component-normal-forms},
\[
  \Theta_7=\Phi_7=\Phi_{18},
  \qquad
  k(7)=7.
\]
A general closed orbit in \(\Theta_7\) is represented by
\[
  \begin{aligned}
  F_7(\alpha,\beta):=\nf_7(\alpha,\beta)
  ={}&x_2^5+x_1^2x_2^2x_3+x_1^4x_4+x_0x_1x_3^3\\
     &+\alpha x_0x_1x_2x_3x_4+\beta x_0^2x_3x_4^2,
  \end{aligned}
\]
where \((\alpha,\beta)\in(\C^\times)^2\), and we set
\[
  X_7(\alpha,\beta):=V(F_7(\alpha,\beta))\subset\PP^4.
\]
Let \(\mathcal R_7:=\mathcal R_{96}\) denote the elimination condition for
\[
  f_P(u,v,z,w)
  =v^5+u^2v^2z+u^4w+uz^3+\alpha uvzw+\beta zw^2
\]
to have a nonzero critical point in \(\PP(4,5,7,9)\).  We work on the
Zariski-open locus
\[
  \mathcal R_7(\alpha,\beta)\neq0.
\]

\paragraph*{Singular locus.}
On the hyperplane \(x_0=0\), the equations
\(\partial F_7/\partial x_4=0\) and
\(\partial F_7/\partial x_2=0\) give successively
\[
  x_1=0,
  \qquad
  x_2=0.
\]
Thus the singular support in this hyperplane is the line
\[
  L:=V(x_0,x_1,x_2).
\]
On the chart \(x_0=1\), the affine equation is \(f_P\).  The condition
\(\mathcal R_7\neq0\) excludes every nonzero critical point, so this chart
contributes only
\[
  P=(1:0:0:0:0).
\]
Consequently
\[
  \Sing(X_7(\alpha,\beta))_{\mathrm{red}}
  =L\sqcup\{P\}.
\]
The distinguished point on \(L\) is
\[
  Q=(0:0:0:0:1).
\]
Scheme-theoretically, the generic transverse length along \(L\) is \(4\).

\paragraph*{The isolated point \(P\).}
The local equation at \(P\) is \(f_P\).  It is weighted homogeneous of
degree \(25\) for the weights
\[
  \wt(u,v,z,w)=(4,5,7,9),
\]
and \(\mathcal R_7\neq0\) makes the critical point isolated.  With
\[
  (X_1,X_2,X_3,X_4)=(u,v,z,w),
\]
the local equation is precisely the defining normal form of the family
\(\Sfam{96}{86}(\alpha,\beta)\).  Hence, in the Yonemura-number notation of
Section~\ref{sec:singular-loci},
\[
  (X_7(\alpha,\beta),P)
  \sim_{\mathrm{an}}
  \Sfam{96}{86}(\alpha,\beta).
\]
In particular, the isolated singularity has Milnor number \(96\) and
Yonemura number \(86\).  Therefore
\[
  \mexp_P(X_7(\alpha,\beta))
  =\frac{4+5+7+9}{25}
  =1.
\]

\paragraph*{Points of \(L\setminus\{Q\}\).}
Let \(p\in L\setminus\{Q\}\).  In the chart \(x_3=1\), write
\[
  u=x_0,
  \qquad
  v=x_1,
  \qquad
  y=x_2,
  \qquad
  x_4=q+\zeta.
\]
With \(s=q+\zeta\), the local equation is
\[
  F_7
  =y^5+v^2y^2+v^4s+uv+\alpha uvys+\beta u^2s^2.
\]
The Hessian in \((u,v)\) is nondegenerate along the critical section
\(u=v=0\).  The parametric splitting lemma therefore gives
\[
  F_7\sim_{\mathrm{an}}U_1^2+U_2^2+y^5,
\]
with \(\zeta\) as a smooth parameter.  Thus the transverse type is \(A_4\),
and
\[
  \mexp_p(X_7(\alpha,\beta))
  =\frac12+\frac12+\frac15
  =\frac65
  \qquad
  (p\in L\setminus\{Q\}).
\]
The transverse Jacobian ideal is analytically equivalent to
\((U_1,U_2,y^4)\), which gives the generic transverse length \(4\) along
\(L\).

\paragraph*{The distinguished point \(Q\).}
In the chart \(x_4=1\), with
\[
  (u,v,y,z)=(x_0,x_1,x_2,x_3),
\]
the local equation is
\[
  f_Q
  =y^5+v^2y^2z+v^4+uvz^3+\alpha uvyz+\beta u^2z.
\]
It is weighted homogeneous of degree \(20\) for
\[
  \wt(u,v,y,z)=(9,5,4,2).
\]
A direct Jacobian calculation gives
\[
  \Sing(V(f_Q))=V(u,v,y).
\]
Hence this three-dimensional hypersurface germ is normal: its singular locus
has codimension two, and a hypersurface is Cohen--Macaulay.

A direct examination of the proper compact Newton faces shows that they are
nondegenerate.  A torus critical point of the full compact face would produce
an additional singular point of \(X_7(\alpha,\beta)\) in the chart
\(x_0\neq0\), which is excluded by \(\mathcal R_7\neq0\).  Thus \(f_Q\) is
Newton nondegenerate.  Since the monomial \(uvyz\) occurs, every toric
valuation determined by a primitive vector
\(a=(a_u,a_v,a_y,a_z)\in\mathbb Z_{\geq0}^4\setminus\{0\}\) satisfies
\[
  \nu_a(f_Q)\leq a_u+a_v+a_y+a_z.
\]
The weight vector \((9,5,4,2)\) attains equality, and the Newton-polyhedron
formula therefore gives
\[
  \operatorname{lct}_0(f_Q)
  =\frac{9+5+4+2}{20}
  =1;
\]
see \cite[Theorem~12]{How03}.  Since
\(\operatorname{lct}_0(f_Q)=\min\{\mexp_0(f_Q),1\}\), it follows that
\[
  \mexp_Q(X_7(\alpha,\beta))\geq1;
\]
see \cite[Section~1]{MP20}.

For the reverse inequality, consider the weighted blow-up of weights
\((9,5,4,2)\).  The corresponding divisor has discrepancy over the local
hypersurface
\[
  (9+5+4+2-1)-20=-1.
\]
Thus the normal Gorenstein germ \(V(f_Q)\) is not canonical.  Since a normal
Gorenstein rational singularity is canonical
\cite[Corollary~5.24]{KM98}, the germ is not rational.  Saito's rationality
criterion \cite[Theorem~0.4]{Sai93} then yields
\[
  \mexp_Q(X_7(\alpha,\beta))\leq1.
\]
Combining the two inequalities gives
\[
  \mexp_Q(X_7(\alpha,\beta))=1.
\]

\paragraph*{Global minimal exponent.}
The preceding strata exhaust the singular locus.  Therefore
\[
  \mexp(X_7(\alpha,\beta))
  =\min\left\{1,1,\frac65\right\}
  =1
  =\frac{4+1}{5}.
\]
Thus a general closed-orbit representative of \(\Theta_7\) has global
minimal exponent equal to \(1\).

\subsection{\texorpdfstring{The component \(\Theta_8\)}{The component Theta 8}}
\label{subsec:singular-minimal-theta-8}

By Table~\ref{tab:component-normal-forms},
\[
  \Theta_8=\Phi_8=\Phi_{16},
  \qquad
  k(8)=8.
\]
A general closed orbit in \(\Theta_8\) is represented by
\[
  \begin{aligned}
  F_8(\alpha,\beta):=\nf_8(\alpha,\beta)
  ={}&x_1^2x_2^2x_3+x_1^3x_4^2+x_0x_2^4+x_0x_1x_3^3\\
     &+\alpha x_0x_1x_2x_3x_4+\beta x_0^2x_3x_4^2,
  \end{aligned}
\]
where \((\alpha,\beta)\in(\C^\times)^2\), and we set
\[
  X_8(\alpha,\beta):=V(F_8(\alpha,\beta))\subset\PP^4.
\]
We work on the nonempty Zariski-open set
\[
  \mathcal R_8(\alpha,\beta)\neq0,
\]
where
\[
  \mathcal R_8(\alpha,\beta)
  =\alpha^6+(4-12\beta)\alpha^4
   +16\beta(3\beta-20)\alpha^2
   -64\beta(\beta+4)^2.
\]
Indeed, \(\mathcal R_8(1,1)=-1879\), so this open set is nonempty.

\paragraph*{Singular locus.}
On the hyperplane \(x_0=0\), the Jacobian equations give the line
\[
  L:=V(x_0,x_1,x_2)
\]
and the point \((0:1:0:0:0)\), which lies on the line
\[
  N:=V(x_2,x_3,x_4).
\]
On the chart \(x_0=1\), with
\[
  (u,v,z,w)=(x_1,x_2,x_3,x_4),
\]
the affine equation is
\[
  f=u^2v^2z+u^3w^2+v^4+uz^3+\alpha uvzw+\beta zw^2.
\]
Its critical equations vanish identically on \(v=z=w=0\), giving the
remaining affine part of \(N\).  Outside this line one necessarily has
\(uvzw\neq0\).  Set
\[
  T_5:=\alpha uvzw,
  \qquad
  T_6:=\beta zw^2,
  \qquad
  r:=T_5/T_6.
\]
The logarithmic critical equations then reduce to
\[
  A_8(r)=0,
  \qquad
  B_8(r)=0,
\]
where
\[
  A_8(r)=\alpha^2(r+2)^3+4r^2(r+4),
  \qquad
  B_8(r)=\beta(r+2)^3+2(r+4)^2.
\]
Since
\[
  \operatorname{Res}_r(A_8,B_8)
  =512\,\mathcal R_8(\alpha,\beta),
\]
the genericity condition excludes every additional affine critical point.
Consequently
\[
  \Sing(X_8(\alpha,\beta))_{\mathrm{red}}=N\sqcup L,
\]
where \(N\) and \(L\) are disjoint lines.  In particular, there are no
isolated singular points.  Scheme-theoretically, the generic transverse
Jacobian lengths along \(N\) and \(L\) are respectively \(4\) and \(9\).

\paragraph*{Points of \(N\setminus\{P\}\).}
Let
\[
  P=(1:0:0:0:0).
\]
At a point of \(N\setminus\{P\}\) with \(x_0\neq0\), write
\(x_1=s+\xi\), \(x_2=y\), \(x_3=z\), and \(x_4=w\), where \(s\neq0\).
Completing the square in \(w\) gives
\[
  F_8\sim_{\mathrm{an}}W^2+g(\xi,y,z),
\]
where \(g\) has order \(3\) along the smooth \(\xi\)-axis
\(V(y,z)\).  Blowing up this axis resolves the germ: the exceptional divisor
has log discrepancy \(2\), the pullback of \(g\) has order \(3\), and the
strict transform is smooth and transverse to the exceptional divisor.  Thus
\[
  \operatorname{lct}_0(g)=\mexp_0(g)=\frac23.
\]
The same calculation in the chart \(x_1=1\) treats the remaining point
\((0:1:0:0:0)\in N\setminus\{P\}\).  Thom--Sebastiani therefore gives
\[
  \mexp_p(X_8(\alpha,\beta))
  =\frac12+\frac23
  =\frac76
  \qquad
  (p\in N\setminus\{P\});
\]
see \cite{CDNM24,MP20}.  At a general point of \(N\), the residual plane curve is an ordinary triple
point, which gives Jacobian length \(4\).

\paragraph*{The distinguished point \(P\).}
The local equation at \(P\) is
\[
  f_P=u^2v^2z+u^3w^2+v^4+uz^3+\alpha uvzw+\beta zw^2.
\]
It is weighted homogeneous of degree \(20\) for
\[
  \wt(u,v,z,w)=(2,5,6,7),
\]
but its singularity is non-isolated.  Every proper compact Newton face is
nondegenerate, while a torus critical point of the full compact face would
satisfy \(A_8(r)=B_8(r)=0\).  Hence \(\mathcal R_8\neq0\) implies that
\(f_P\) is Newton nondegenerate.  Since the support contains the monomial
\(uvzw\) and lies on the weight hyperplane of degree \(20\), the Newton
polyhedron has diagonal distance \(1\).  Therefore
\[
  \operatorname{lct}_0(f_P)=1
\]
by the Newton-polyhedron formula \cite{How03}, and hence
\[
  \mexp_P(X_8(\alpha,\beta))\geq1.
\]
The hypersurface germ is normal and Gorenstein.  The weighted blow-up of
weight \((2,5,6,7)\) has discrepancy
\[
  (2+5+6+7)-1-20=-1
\]
over the hypersurface, so the germ is not canonical.  Since a normal
Gorenstein rational singularity is canonical, the germ is not rational; see
\cite{KM98}.  Saito's rationality criterion then gives the reverse inequality
\(\mexp_P(X_8)\leq1\); see \cite{Sai93}.  Consequently
\[
  \mexp_P(X_8(\alpha,\beta))=1.
\]

\paragraph*{Points of \(L\setminus\{Q\}\).}
Let
\[
  Q=(0:0:0:0:1).
\]
For \(p\in L\setminus\{Q\}\), work in the chart \(x_3=1\), write
\[
  u=x_0,
  \qquad
  v=x_1,
  \qquad
  y=x_2,
  \qquad
  x_4=q+\zeta,
\]
and put \(t=q+\zeta\).  The local equation is
\[
  F_8=v^2y^2+v^3t^2+uy^4+uv+\alpha uvyt+\beta u^2t^2.
\]
The Hessian in \((u,v)\) has determinant \(-1\).  Solving for its critical
section and substituting back gives the residual critical value
\[
  y^{10}\cdot\operatorname{unit}.
\]
The parametric splitting lemma therefore yields
\[
  F_8\sim_{\mathrm{an}}U_1^2+U_2^2+Y^{10}
\]
with \(\zeta\) as a smooth parameter.  Hence the transverse type is
\(A_9\), and
\[
  \mexp_p(X_8(\alpha,\beta))
  =\frac12+\frac12+\frac1{10}
  =\frac{11}{10}
  \qquad
  (p\in L\setminus\{Q\}).
\]
The transverse Jacobian ideal is analytically equivalent to
\((U_1,U_2,Y^9)\), giving generic transverse length \(9\) along \(L\).

\paragraph*{The distinguished point \(Q\).}
In the chart \(x_4=1\), with
\[
  (u,v,y,z)=(x_0,x_1,x_2,x_3),
\]
the local equation is
\[
  f_Q=v^3+uy^4+v^2y^2z+uvz^3+\alpha uvyz+\beta u^2z.
\]
It is weighted homogeneous of degree \(15\) for
\[
  \wt(u,v,y,z)=(7,5,2,1),
\]
and its singularity is again non-isolated.  The proper compact Newton faces
are nondegenerate, and the full-face torus critical equations are again
\(A_8(r)=B_8(r)=0\).  Thus \(f_Q\) is Newton nondegenerate on the chosen
open set.  Its Newton distance is \(1\), so
\[
  \operatorname{lct}_0(f_Q)=1,
  \qquad
  \mexp_Q(X_8(\alpha,\beta))\geq1.
\]
The germ is normal and Gorenstein, while the weighted blow-up of weight
\((7,5,2,1)\) has discrepancy
\[
  (7+5+2+1)-1-15=-1.
\]
By the same normal--Gorenstein argument as at \(P\), the germ is not
rational, and Saito's criterion gives \(\mexp_Q(X_8)\leq1\).  Hence
\[
  \mexp_Q(X_8(\alpha,\beta))=1.
\]

\paragraph*{Global minimal exponent.}
The preceding strata exhaust the singular locus.  Therefore
\[
  \mexp(X_8(\alpha,\beta))
  =\min\left\{1,\frac76,1,\frac{11}{10}\right\}
  =1
  =\frac{4+1}{5}.
\]
Thus a general closed-orbit representative of \(\Theta_8\) has global
minimal exponent equal to \(1\).

\subsection{\texorpdfstring{The component \(\Theta_9\)}{The component Theta 9}}
\label{subsec:singular-minimal-theta-9}

By Table~\ref{tab:component-normal-forms},
\[
  \Theta_9=\Phi_9=\Phi_{22},
  \qquad
  k(9)=9.
\]
A general closed orbit in \(\Theta_9\) is represented by
\[
  \begin{aligned}
  F_9(\alpha,\beta,\gamma,\delta)
  :=\nf_9(\alpha,\beta,\gamma,\delta)
  ={}&x_1x_2^4+x_1^2x_2^2x_3+\alpha x_1^3x_3^2+x_1^4x_4
      +x_0x_2x_3^3\\
     &+\beta x_0x_2^3x_4+\gamma x_0x_1x_2x_3x_4
      +\delta x_0^2x_3x_4^2,
  \end{aligned}
\]
where
\((\alpha,\beta,\gamma,\delta)\in(\C^\times)^4\), and we set
\[
  X_9(\alpha,\beta,\gamma,\delta)
  :=V(F_9(\alpha,\beta,\gamma,\delta))\subset\PP^4.
\]
We work on a nonempty principal open set
\[
  D(\mathcal R_9)
  :=\bigl\{(\alpha,\beta,\gamma,\delta)\in(\C^\times)^4
  \bigm|\mathcal R_9(\alpha,\beta,\gamma,\delta)\neq0\bigr\}
\]
on which the Jacobian scheme in the chart \(x_0=1\) is supported only at
\(P=(1:0:0:0:0)\), the germ at \(P\) is isolated, and the germ at the
point \(Q\) defined below is Newton nondegenerate.  Such an open set exists:
at \((\alpha,\beta,\gamma,\delta)=(2,3,5,7)\), exact Gr\"obner-basis and
Newton-face computations verify these conditions.

\paragraph*{Singular locus.}
On the hyperplane \(x_0=0\), the Jacobian equations contain
\[
  x_1^4=0,
  \qquad
  x_2^4+2x_1x_2^2x_3+3\alpha x_1^2x_3^2+4x_1^3x_4=0.
\]
Thus \(x_1=x_2=0\), and the remaining equations vanish identically.  This
gives the line
\[
  L:=V(x_0,x_1,x_2).
\]
By the choice of \(D(\mathcal R_9)\), the chart \(x_0=1\) contributes only
\[
  P=(1:0:0:0:0).
\]
Consequently
\[
  \Sing(X_9(\alpha,\beta,\gamma,\delta))_{\mathrm{red}}
  =L\sqcup\{P\}.
\]
The distinguished point on \(L\) is
\[
  Q=(0:0:0:0:1).
\]
Scheme-theoretically, the generic transverse Jacobian length along \(L\) is
\(2\).

\paragraph*{The isolated point \(P\).}
In the chart \(x_0=1\), with
\[
  (u,v,z,w)=(x_1,x_2,x_3,x_4),
\]
the local equation is
\[
  \begin{aligned}
  f_P={}&uv^4+u^2v^2z+\alpha u^3z^2+u^4w+vz^3
          +\beta v^3w+\gamma uvzw+\delta zw^2.
  \end{aligned}
\]
It is weighted homogeneous of degree \(19\) for
\[
  \wt(u,v,z,w)=(3,4,5,7).
\]
Since the critical point is isolated on \(D(\mathcal R_9)\), it defines an
isolated hypersurface germ.  With
\[
  (X_1,X_2,X_3,X_4)=(u,v,z,w),
\]
the local equation is precisely the defining normal form of the family
\(\Sfam{96}{94}(\alpha,\beta,\gamma,\delta)\).  Hence, in the
Yonemura-number notation of Section~\ref{sec:singular-loci},
\[
  \begin{aligned}
  &(X_9(\alpha,\beta,\gamma,\delta),P)\\
  &\qquad\sim_{\mathrm{an}}
  \Sfam{96}{94}(\alpha,\beta,\gamma,\delta).
  \end{aligned}
\]
In particular, the isolated singularity has Milnor number \(96\) and
Yonemura number \(94\).  The weighted-homogeneous formula gives
\[
  \mexp_P(X_9(\alpha,\beta,\gamma,\delta))
  =\frac{3+4+5+7}{19}
  =1.
\]
The local Jacobian algebra has length
\[
  \left(\frac{19}{3}-1\right)
  \left(\frac{19}{4}-1\right)
  \left(\frac{19}{5}-1\right)
  \left(\frac{19}{7}-1\right)
  =96.
\]

\paragraph*{Points of \(L\setminus\{Q\}\).}
Let \(p\in L\setminus\{Q\}\).  In the chart \(x_3=1\), write
\[
  u=x_0,
  \qquad
  v=x_1,
  \qquad
  y=x_2,
  \qquad
  x_4=q+\zeta.
\]
The local equation is
\[
  \begin{aligned}
  F_9={}&vy^4+v^2y^2+\alpha v^3+v^4(q+\zeta)+uy
         +\beta uy^3(q+\zeta)\\
       &+\gamma uvy(q+\zeta)+\delta u^2(q+\zeta)^2.
  \end{aligned}
\]
The Hessian in \((u,y)\) is nondegenerate, and \(u=y=0\) is its exact
critical section.  The parametric splitting lemma therefore gives
\[
  F_9\sim_{\mathrm{an}}U_1^2+U_2^2+V^3,
\]
with \(\zeta\) as a smooth parameter, because the residual critical value is
\[
  v^3\bigl(\alpha+v(q+\zeta)\bigr),
\]
which is a unit times \(v^3\).  Thus the transverse type is \(A_2\), and
\[
  \mexp_p(X_9(\alpha,\beta,\gamma,\delta))
  =\frac12+\frac12+\frac13
  =\frac43
  \qquad
  (p\in L\setminus\{Q\}).
\]
The transverse Jacobian ideal is analytically equivalent to
\((U_1,U_2,V^2)\), which gives the generic transverse length \(2\) along
\(L\).

\paragraph*{The distinguished point \(Q\).}
In the chart \(x_4=1\), with
\[
  (u,v,y,z)=(x_0,x_1,x_2,x_3),
\]
the local equation is
\[
  f_Q
  =vy^4+v^2y^2z+\alpha v^3z^2+v^4+uyz^3
   +\beta uy^3+\gamma uvyz+\delta u^2z.
\]
It is weighted homogeneous of degree \(16\) for
\[
  \wt(u,v,y,z)=(7,4,3,2),
\]
but its singularity is non-isolated.  All support exponents lie on the
corresponding weight hyperplane, and the monomial \(uvyz\) occurs.  Hence
the Newton distance is \(1\).  Newton nondegeneracy on
\(D(\mathcal R_9)\) gives
\[
  \operatorname{lct}_0(f_Q)=1
\]
by the Newton-polyhedron formula \cite{How03}.  Since
\(\operatorname{lct}_0(f_Q)=\min\{\mexp_0(f_Q),1\}\), it follows that
\[
  \mexp_Q(X_9(\alpha,\beta,\gamma,\delta))\geq1;
\]
see \cite{MP20}.  Conversely, the weighted-order estimate gives
\[
  \mexp_Q(X_9(\alpha,\beta,\gamma,\delta))
  \leq\frac{7+4+3+2}{16}
  =1;
\]
see \cite[Proposition~2.1]{CDM25}.  Therefore
\[
  \mexp_Q(X_9(\alpha,\beta,\gamma,\delta))=1.
\]

\paragraph*{Global minimal exponent.}
The preceding strata exhaust the singular locus.  Therefore
\[
  \mexp(X_9(\alpha,\beta,\gamma,\delta))
  =\min\left\{1,1,\frac43\right\}
  =1
  =\frac{4+1}{5}.
\]
Thus a general closed-orbit representative of \(\Theta_9\) has global
minimal exponent equal to \(1\).

\subsection{\texorpdfstring{The component \(\Theta_{10}\)}{The component Theta 10}}
\label{subsec:singular-minimal-theta-10}

By Table~\ref{tab:component-normal-forms},
\[
  \Theta_{10}=\Phi_{10}=\Phi_{34},
  \qquad
  k(10)=10.
\]
Put
\[
  B_\lambda(u,v):=uv(u-v)(u-\lambda v),
  \qquad
  \lambda\in\C\setminus\{0,1\},
\]
and let \(Q_{\mathbf q}\) and \(R_{\mathbf r}\) be binary quadratic forms, with
\(\mathbf q,\mathbf r\in\C^3\) and \(\beta\in\C\).  A general closed orbit in \(\Theta_{10}\) is represented by
\[
  \begin{aligned}
  F_{10}(\lambda,\mathbf q,\mathbf r,\beta)
  :=\nf_{10}(\lambda,\mathbf q,\mathbf r,\beta)
  ={}&x_1^3Q_{\mathbf q}(x_2,x_3)+x_1^4x_4
      +x_0B_\lambda(x_2,x_3)\\
     &+x_0x_1x_4R_{\mathbf r}(x_2,x_3)
      +\beta x_0x_1^2x_4^2+x_0^2x_4^3.
  \end{aligned}
\]
We write
\[
  X_{10}:=V(F_{10})\subset\PP^4
\]
and work on the nonempty Zariski-open locus
\[
  \mathcal R_{10}(\lambda,\mathbf q,\mathbf r,\beta)\neq0
\]
on which
\(\operatorname{Res}(B_\lambda,Q_{\mathbf q})\neq0\), the chart
\(x_0=1\) has no singular point other than
\(P=(1:0:0:0:0)\), and the germ at \(P\) is isolated.

\paragraph*{Singular locus.}
On the hyperplane \(x_0=0\), the equations
\[
  \frac{\partial F_{10}}{\partial x_4}=x_1^4=0,
  \qquad
  \left.\frac{\partial F_{10}}{\partial x_0}\right|_{x_1=0}
  =B_\lambda(x_2,x_3)=0
\]
give
\[
  C_\lambda:=V(x_0,x_1,B_\lambda(x_2,x_3)).
\]
Since \(B_\lambda\) is square-free,
\[
  C_\lambda=L_0\cup L_\infty\cup L_1\cup L_\lambda,
\]
where
\[
  \begin{aligned}
  L_0&=V(x_0,x_1,x_2),
  &L_\infty&=V(x_0,x_1,x_3),\\
  L_1&=V(x_0,x_1,x_2-x_3),
  &L_\lambda&=V(x_0,x_1,x_2-\lambda x_3).
  \end{aligned}
\]
These four lines meet only at
\[
  Q=(0:0:0:0:1).
\]
The genericity assumption shows that the chart \(x_0=1\) contributes only
\(P\).  Hence
\[
  \Sing(X_{10})_{\mathrm{red}}=C_\lambda\sqcup\{P\}.
\]
Scheme-theoretically, the generic transverse Jacobian length is \(2\) along
each irreducible component of \(C_\lambda\).

\paragraph*{The isolated point \(P\).}
In the chart \(x_0=1\), with
\[
  (u,v,z,w)=(x_1,x_2,x_3,x_4),
\]
the local equation is
\[
  f_P
  =u^3Q_{\mathbf q}(v,z)+u^4w+B_\lambda(v,z)
   +uwR_{\mathbf r}(v,z)+\beta u^2w^2+w^3.
\]
It is weighted homogeneous of degree \(12\) for
\[
  \wt(u,v,z,w)=(2,3,3,4).
\]
Since the critical point is isolated on the chosen open set, it defines an
isolated hypersurface germ.  With
\[
  (X_1,X_2,X_3,X_4)=(u,v,z,w),
\]
the local equation is precisely the defining normal form of the family
\(\Sfam{90}{2}(B_\lambda,Q_{\mathbf q},R_{\mathbf r},\beta)\).  Hence, in the
Yonemura-number notation of Section~\ref{sec:singular-loci},
\[
  (X_{10},P)
  \sim_{\mathrm{an}}
  \Sfam{90}{2}(B_\lambda,Q_{\mathbf q},R_{\mathbf r},\beta).
\]
In particular, the isolated singularity has Milnor number \(90\) and
Yonemura number \(2\).  The weighted-homogeneous formula gives
\[
  \mexp_P(X_{10})
  =\frac{2+3+3+4}{12}
  =1.
\]

\paragraph*{Points of \(C_\lambda\setminus\{Q\}\).}
Let \(p\in C_\lambda\setminus\{Q\}\), and put locally
\[
  u=x_0,
  \qquad
  v=x_1,
\]
\[
  s
  =B_\lambda(x_2,x_3)+vx_4R_{\mathbf r}(x_2,x_3)
   +\beta v^2x_4^2+ux_4^3.
\]
Because \(B_\lambda\) is square-free, \(s\) is a coordinate transverse to
the line through \(p\).  Moreover,
\(Q_{\mathbf q}(p)\neq0\) by the resultant condition, and the defining
equation has the exact form
\[
  F_{10}=us+v^3\bigl(Q_{\mathbf q}(x_2,x_3)+vx_4\bigr).
\]
The factor in parentheses is a unit.  Thus, with one smooth parameter along
the line,
\[
  F_{10}\sim_{\mathrm{an}}U_1^2+U_2^2+V^3.
\]
The transverse type is \(A_2\), and therefore
\[
  \mexp_p(X_{10})
  =\frac12+\frac12+\frac13
  =\frac43
  \qquad
  (p\in C_\lambda\setminus\{Q\}).
\]

\paragraph*{The common point \(Q\).}
In the chart \(x_4=1\), with
\[
  (u,v,y,z)=(x_0,x_1,x_2,x_3),
\]
the local equation is
\[
  f_Q
  =v^3Q_{\mathbf q}(y,z)+v^4+uB_\lambda(y,z)
   +uvR_{\mathbf r}(y,z)+\beta uv^2+u^2.
\]
Set
\[
  A(v,y,z):=B_\lambda(y,z)+vR_{\mathbf r}(y,z)+\beta v^2,
  \qquad
  u':=u+\frac12A(v,y,z).
\]
Then
\[
  f_Q={u'}^2+h(v,y,z),
\]
where
\[
  h=v^3\bigl(Q_{\mathbf q}(y,z)+v\bigr)-\frac14A(v,y,z)^2.
\]
The germ \(h\) is weighted homogeneous of degree \(8\) for the weights
\((2,1,1)\), so the corresponding weighted valuation gives
\[
  \operatorname{lct}_0(h)\leq\frac{2+1+1}{8}=\frac12.
\]
Its nonzero singular locus is
\[
  \{v=0,\ B_\lambda(y,z)=0\}\setminus\{0\}.
\]
At every point of this locus, \(dB_\lambda\neq0\) and
\(Q_{\mathbf q}+v\) is a unit.  Taking
\(s=A(v,y,z)\) and a local cube root of \(Q_{\mathbf q}+v\) gives the
transverse equation
\[
  w^3-\frac14s^2,
\]
whose log canonical threshold is \(5/6\).  Hence
\((\C^3,\frac12V(h))\) is log canonical away from the origin.  The
weighted-blow-up criterion therefore gives the reverse inequality and hence
\[
  \operatorname{lct}_0(h)=\frac12;
\]
see \cite[Proposition~3.2]{Pae24}.  Since
\(\operatorname{lct}_0(h)=\min\{\mexp_0(h),1\}\), we obtain
\[
  \mexp_0(h)=\frac12;
\]
see \cite{MP20}.  Thom--Sebastiani now yields
\[
  \mexp_Q(X_{10})
  =\mexp_0({u'}^2)+\mexp_0(h)
  =\frac12+\frac12
  =1;
\]
see \cite[Theorem~1.2]{CDNM24}.

\paragraph*{Global minimal exponent.}
The preceding strata exhaust the singular locus.  Therefore
\[
  \mexp(X_{10})
  =\min\left\{1,1,\frac43\right\}
  =1
  =\frac{4+1}{5}.
\]
Thus a general closed-orbit representative of \(\Theta_{10}\) has global
minimal exponent equal to \(1\).

\subsection{\texorpdfstring{The component \(\Theta_{11}\)}{The component Theta 11}}
\label{subsec:singular-minimal-theta-11}

By Table~\ref{tab:component-normal-forms},
\[
  \Theta_{11}=\Phi_{11}=\Phi_{25},
  \qquad
  k(11)=11.
\]
A general closed orbit in \(\Theta_{11}\) is represented by
\[
  \begin{aligned}
  F_{11}:={}&\nf_{11}(\alpha,\beta,\gamma,\delta,\varepsilon,\zeta)\\
  ={}&x_1x_2^4
  +\alpha x_1^2x_2^2x_3
  +\beta x_1^3x_3^2
  +x_1^3x_2x_4
  +\gamma x_0x_2^2x_3^2
  +x_0x_2^3x_4\\
  &+\delta x_0x_1x_3^3
  +\varepsilon x_0x_1x_2x_3x_4
  +\zeta x_0x_1^2x_4^2
  +x_0^2x_3x_4^2,
  \end{aligned}
\]
where
\((\alpha,\beta,\gamma,\delta,\varepsilon,\zeta)\in(\C^\times)^6\),
and we set
\[
  X_{11}:=V(F_{11})\subset\PP^4.
\]
We work on a nonempty principal open set
\[
  \mathcal R_{11}(\alpha,\beta,\gamma,\delta,\varepsilon,\zeta)\neq0
\]
on which there are no additional Jacobian components, the Newton principal
parts at the points \(P,Q\) below are nondegenerate, and
\[
  \alpha\delta\gamma-\beta\gamma^2-\delta^2\neq0.
\]

\paragraph*{Singular locus.}
The two lines
\[
  N:=V(x_2,x_3,x_4),
  \qquad
  L:=V(x_0,x_1,x_2)
\]
are contained in \(\Sing(X_{11})\), and they are disjoint.  On the
hyperplane \(x_0=0\), the equation
\[
  \frac{\partial F_{11}}{\partial x_4}=x_1^3x_2=0
\]
shows that either \(x_1=0\) or \(x_2=0\).  In the first case
\(\partial F_{11}/\partial x_1=x_2^4\), so the point lies on \(L\).  In the
second case, if \(x_1\neq0\), the equations
\[
  \frac{\partial F_{11}}{\partial x_3}=2\beta x_1^3x_3,
  \qquad
  \frac{\partial F_{11}}{\partial x_2}=x_1^3x_4
\]
give \(x_3=x_4=0\), hence the point lies on \(N\).  On the chart \(x_0=1\),
the genericity assumption excludes critical points away from \(N\).
Consequently
\[
  \Sing(X_{11})_{\mathrm{red}}=N\sqcup L.
\]
There are no isolated singular points.  The transverse type changes at
\[
  P=(1:0:0:0:0)\in N,
  \qquad
  Q=(0:0:0:0:1)\in L.
\]

\paragraph*{The special point \(P\).}
In the chart \(x_0=1\), with
\((u,v,z,w)=(x_1,x_2,x_3,x_4)\), the local equation is
\[
  \begin{aligned}
  f_P={}&uv^4+\alpha u^2v^2z+\beta u^3z^2+u^3vw
  +\gamma v^2z^2+v^3w\\
  &+\delta uz^3+\varepsilon uvzw+\zeta u^2w^2+zw^2.
  \end{aligned}
\]
It is weighted homogeneous of degree \(14\) for
\[
  \wt(u,v,z,w)=(2,3,4,5),
\]
but the singularity is non-isolated.  The support lies on the corresponding
weight hyperplane and contains the monomial \(uvzw\).  Hence the diagonal
meets the compact Newton boundary at \((1,1,1,1)\), and Newton
nondegeneracy gives
\[
  \operatorname{lct}_0(f_P)=1
\]
by the Newton-polyhedron formula \cite{How03}.  Thus
\(\mexp_P(X_{11})\geq1\) by
\(\operatorname{lct}_0(f_P)=\min\{1,\mexp_0(f_P)\}\); see \cite{MP20}.
The weighted-order estimate gives the reverse inequality
\[
  \mexp_P(X_{11})
  \leq\frac{2+3+4+5}{14}=1;
\]
see \cite[Proposition~2.1]{CDM25}.  Therefore
\[
  \mexp_P(X_{11})=1.
\]

\paragraph*{Points of \(N\setminus\{P\}\).}
For \(p\in N\setminus\{P\}\), one has \(x_1(p)\neq0\).  In coordinates
\((y,z,w)\) transverse to \(N\), the quadratic part is
\[
  \beta x_1(p)^3z^2+x_1(p)^3yw
  +\zeta x_0(p)x_1(p)^2w^2.
\]
Its determinant is \(-\beta x_1(p)^9/4\), so it has rank \(3\).  The
parametric splitting lemma therefore gives the transverse ordinary double
point
\[
  U_1^2+U_2^2+U_3^2=0.
\]
Hence
\[
  \mexp_p(X_{11})=\frac32
  \qquad
  (p\in N\setminus\{P\}).
\]
In particular, the Jacobian scheme is generically reduced along \(N\).

\paragraph*{The special point \(Q\).}
In the chart \(x_4=1\), with
\((u,v,y,z)=(x_0,x_1,x_2,x_3)\), the local equation is
\[
  \begin{aligned}
  f_Q={}&vy^4+\alpha v^2y^2z+\beta v^3z^2+v^3y
  +\gamma uy^2z^2+uy^3\\
  &+\delta uvz^3+\varepsilon uvyz+\zeta uv^2+u^2z.
  \end{aligned}
\]
It is weighted homogeneous of degree \(11\) for
\[
  \wt(u,v,y,z)=(5,3,2,1),
\]
and is again non-isolated.  As at \(P\), the support lies on the weight
hyperplane and contains \(uvyz\).  Newton nondegeneracy therefore gives
\[
  \operatorname{lct}_0(f_Q)=1,
  \qquad
  \mexp_Q(X_{11})\geq1.
\]
The weighted-order estimate yields
\[
  \mexp_Q(X_{11})
  \leq\frac{5+3+2+1}{11}=1.
\]
Thus
\[
  \mexp_Q(X_{11})=1.
\]

\paragraph*{Points of \(L\setminus\{Q\}\).}
Let \(p\in L\setminus\{Q\}\).  In the chart \(x_3=1\), write
\[
  u=x_0,
  \qquad
  v=x_1,
  \qquad
  y=x_2,
  \qquad
  x_4=q+\tau.
\]
The transverse quadratic part in \((u,v)\) is
\[
  \delta uv+q^2u^2,
\]
which has rank \(2\).  After the parametric splitting, the residual series is
\[
  y^6\bigl(c+yh(y,\tau)\bigr),
  \qquad
  c=\frac{\gamma(\alpha\delta\gamma-\beta\gamma^2-\delta^2)}{\delta^3}.
\]
The genericity condition gives \(c\neq0\).  A holomorphic change of the
\(y\)-coordinate therefore reduces the germ, up to the smooth parameter
\(\tau\), to
\[
  U_1^2+U_2^2+Y^6=0.
\]
Consequently
\[
  \mexp_p(X_{11})
  =\frac12+\frac12+\frac16
  =\frac76
  \qquad
  (p\in L\setminus\{Q\}).
\]
The generic transverse Jacobian length along \(L\) is \(5\).

\paragraph*{Global minimal exponent.}
The preceding strata exhaust the singular locus.  Therefore
\[
  \mexp(X_{11})
  =\min\left\{1,\frac32,1,\frac76\right\}
  =1
  =\frac{4+1}{5}.
\]
Thus a general closed-orbit representative of \(\Theta_{11}\) has global
minimal exponent equal to \(1\).

\subsection{\texorpdfstring{The component \(\Theta_{12}\)}{The component Theta 12}}
\label{subsec:singular-minimal-theta-12}

By Table~\ref{tab:component-normal-forms},
\[
  \Theta_{12}=\Phi_{12}=\Phi_{17},
  \qquad
  k(12)=12.
\]
A general closed orbit in \(\Theta_{12}\) is represented by
\[
  \begin{aligned}
  F_{12}:={}&\nf_{12}(\alpha,\beta,\gamma,\delta,\varepsilon)\\
  ={}&x_1x_2^4
  +\alpha x_1^2x_2x_3^2
  +\beta x_1^2x_2^2x_4
  +x_1^3x_4^2
  +x_0x_2^3x_3\\
  &+\gamma x_0x_1x_3^3
  +\delta x_0x_1x_2x_3x_4
  +\varepsilon x_0^2x_3^2x_4
  +x_0^2x_2x_4^2,
  \end{aligned}
\]
and we set
\[
  X_{12}:=V(F_{12})\subset\PP^4.
\]
We work on the nonempty Zariski-open set
\[
  (\alpha,\beta,\gamma,\delta,\varepsilon)\in(\C^\times)^5,
  \qquad
  \beta^2\neq4,
  \qquad
  \alpha\neq\gamma,
  \qquad
  \mathcal R_{12}(\alpha,\beta,\gamma,\delta,\varepsilon)\neq0,
\]
where \(\mathcal R_{12}\) excludes additional affine critical points away
from the line component below.

\paragraph*{Singular locus.}
The two disjoint lines
\[
  N:=V(x_2,x_3,x_4),
  \qquad
  L:=V(x_0,x_1,x_2)
\]
are contained in \(\Sing(X_{12})\).  On the hyperplane \(x_0=0\), if
\(x_1=0\), then
\(\partial F_{12}/\partial x_1=x_2^4\), so the point lies on \(L\).
If \(x_1\neq0\), set \(x_1=1\).  The equations
\[
  \frac{\partial F_{12}}{\partial x_3}=2\alpha x_2x_3,
  \qquad
  \frac{\partial F_{12}}{\partial x_4}=\beta x_2^2+2x_4
\]
give \(x_2x_3=0\) and \(x_4=-\beta x_2^2/2\).  If \(x_2=0\), then
\(\partial F_{12}/\partial x_2=\alpha x_3^2\), so \(x_3=0\).  If
\(x_3=0\), then
\[
  \frac{\partial F_{12}}{\partial x_2}
  =(4-\beta^2)x_2^3,
\]
so \(x_2=0\).  Thus the point lies on \(N\).  On the chart \(x_0=1\),
the condition \(\mathcal R_{12}\neq0\) excludes critical points away from
\(N\).  Consequently
\[
  \Sing(X_{12})_{\mathrm{red}}=N\sqcup L.
\]
There are no isolated singular points.  The transverse type changes at
\[
  P=(1:0:0:0:0)\in N,
  \qquad
  Q=(0:0:0:0:1)\in L.
\]

\paragraph*{The special point \(P\).}
In the chart \(x_0=1\), with
\((u,v,z,w)=(x_1,x_2,x_3,x_4)\), the local equation is
\[
  \begin{aligned}
  f_P={}&uv^4+\alpha u^2vz^2+\beta u^2v^2w+u^3w^2+v^3z\\
       &+\gamma uz^3+\delta uvzw+\varepsilon z^2w+vw^2.
  \end{aligned}
\]
It is weighted homogeneous of degree \(13\) for
\[
  \wt(u,v,z,w)=(1,3,4,5),
\]
but the singularity is non-isolated.  The proper compact faces are
nondegenerate under \(\beta^2\neq4\) and \(\alpha\neq\gamma\), while
\(\mathcal R_{12}\neq0\) excludes a torus critical point of the full compact
face.  Since \(\delta\neq0\), the monomial \(uvzw\) occurs, and the diagonal
meets the Newton boundary at \((1,1,1,1)\).  The Newton-polyhedron formula \cite{How03} therefore gives
\[
  \operatorname{lct}_0(f_P)=1.
\]
The relation between the log canonical threshold and the
minimal exponent \cite{MP20} yields \(\mexp_P(X_{12})\geq1\).  On the other hand, the weighted-order estimate
\cite[Proposition~2.1]{CDM25} gives
\[
  \mexp_P(X_{12})
  \leq\frac{1+3+4+5}{13}=1.
\]
Hence
\[
  \mexp_P(X_{12})=1.
\]

\paragraph*{Points of \(N\setminus\{P\}\).}
For \(p\in N\setminus\{P\}\), work in the chart \(x_1=1\), with
\(u=x_0\) the coordinate along \(N\) and
\((y,z,w)=(x_2,x_3,x_4)\) transverse to it.  The local equation is
\[
  \begin{aligned}
  f_N={}&(1+u^2y)w^2
  +(\beta y^2+\delta uyz+\varepsilon u^2z^2)w\\
  &+y^4+\alpha yz^2+uy^3z+\gamma uz^3.
  \end{aligned}
\]
After completing the square in \(w\), the weighted degree-eight transverse
principal part, for \(\wt(y,z)=(2,3)\), is
\[
  \alpha yz^2+
  \left(1-\frac{\beta^2}{4}\right)y^4.
\]
This is a nondegenerate \(D_5\) principal part.  Thus the transverse germ is
analytically equivalent to
\[
  W^2+YZ^2+Y^4=0,
\]
and the smooth parameter along \(N\) does not affect the minimal exponent.
Therefore
\[
  \mexp_p(X_{12})
  =\frac{4+2+3}{8}
  =\frac98
  \qquad
  (p\in N\setminus\{P\}).
\]
The generic transverse Jacobian length along \(N\) is \(5\).

\paragraph*{The special point \(Q\).}
In the chart \(x_4=1\), with
\((u,v,y,z)=(x_0,x_1,x_2,x_3)\), the local equation is
\[
  \begin{aligned}
  f_Q={}&vy^4+\alpha v^2yz^2+\beta v^2y^2+v^3+uy^3z
  +\gamma uvz^3\\
  &+\delta uvyz+\varepsilon u^2z^2+u^2y.
  \end{aligned}
\]
It is weighted homogeneous of degree \(12\) for
\[
  \wt(u,v,y,z)=(5,4,2,1).
\]
The compact-face calculation is the same as at \(P\).  Since \(uvyz\)
occurs, Newton nondegeneracy gives
\[
  \operatorname{lct}_0(f_Q)=1,
  \qquad
  \mexp_Q(X_{12})\geq1.
\]
The weighted-order estimate gives the reverse inequality
\[
  \mexp_Q(X_{12})
  \leq\frac{5+4+2+1}{12}=1.
\]
Thus
\[
  \mexp_Q(X_{12})=1.
\]

\paragraph*{Points of \(L\setminus\{Q\}\).}
Let \(p\in L\setminus\{Q\}\).  In the chart \(x_3=1\), write
\[
  u=x_0,
  \qquad
  v=x_1,
  \qquad
  y=x_2,
  \qquad
  x_4=q+\tau,
\]
where \(\tau\) is the coordinate along \(L\).  The transverse quadratic
part in \((u,v)\) is
\[
  \gamma uv+\varepsilon q u^2,
\]
which has rank \(2\).  After the parametric splitting lemma, the residual
one-variable germ is
\[
  h(y)=cy^7+O(y^8),
  \qquad
  c=\frac{\alpha-\gamma}{\gamma^2}\neq0.
\]
Thus, up to the smooth parameter \(\tau\), the germ is analytically
equivalent to
\[
  U_1^2+U_2^2+Y^7=0.
\]
The transverse type is \(A_6\), and hence
\[
  \mexp_p(X_{12})
  =\frac12+\frac12+\frac17
  =\frac87
  \qquad
  (p\in L\setminus\{Q\}).
\]
The generic transverse Jacobian length along \(L\) is \(6\).

\paragraph*{Global minimal exponent.}
The preceding strata exhaust the singular locus.  Therefore
\[
  \mexp(X_{12})
  =\min\left\{1,\frac98,1,\frac87\right\}
  =1
  =\frac{4+1}{5}.
\]
Thus a general closed-orbit representative of \(\Theta_{12}\) has global
minimal exponent equal to \(1\).

\subsection{\texorpdfstring{The component \(\Theta_{13}\)}{The component Theta 13}}
\label{subsec:singular-minimal-theta-13}

By Table~\ref{tab:component-normal-forms},
\[
  \Theta_{13}=\Phi_{13},
  \qquad
  k(13)=13.
\]
A general closed orbit in \(\Theta_{13}\) is represented by
\[
  \begin{aligned}
  F_{13}:={}&\nf_{13}(\alpha,\beta,\gamma,\delta)\\
  ={}&x_2^5
  +\alpha x_1x_2^3x_3
  +\beta x_1^2x_2x_3^2
  +x_1^3x_4^2
  +x_0x_2^3x_4\\
  &+\gamma x_0x_1x_2x_3x_4
  +x_0^2x_3^3
  +\delta x_0^2x_2x_4^2,
  \end{aligned}
\]
and we set
\[
  X_{13}:=V(F_{13})\subset\PP^4.
\]
We work on the nonempty Zariski-open set
\[
  (\alpha,\beta,\gamma,\delta)\in(\C^\times)^4,
  \qquad
  \alpha^2\neq4\beta,
  \qquad
  \mathcal R_{13}(\alpha,\beta,\gamma,\delta)\neq0,
\]
where \(\mathcal R_{13}\) excludes additional affine critical points and
ensures Newton nondegeneracy at the two special points below.

\paragraph*{Singular locus.}
The two lines
\[
  N:=V(x_2,x_3,x_4),
  \qquad
  L:=V(x_0,x_1,x_2)
\]
are contained in \(\Sing(X_{13})\), and they are disjoint.  On the
hyperplane \(x_0=0\), if \(x_1=0\), then
\(\partial F_{13}/\partial x_2=5x_2^4\), so the point lies on \(L\).
If \(x_1\neq0\), set \(x_1=1\).  Then
\[
  \frac{\partial F_{13}}{\partial x_4}=2x_4,
  \qquad
  \frac{\partial F_{13}}{\partial x_3}
  =x_2(\alpha x_2^2+2\beta x_3).
\]
Thus \(x_4=0\).  If \(x_2=0\), the equation
\(\partial F_{13}/\partial x_2=\beta x_3^2\) gives \(x_3=0\), hence the
point lies on \(N\).  If \(x_2\neq0\), then
\(x_3=-\alpha x_2^2/(2\beta)\), and substitution into
\(\partial F_{13}/\partial x_2\) gives
\[
  \frac{5(4\beta-\alpha^2)}{4\beta}x_2^4=0,
\]
which contradicts \(\alpha^2\neq4\beta\).  On the chart \(x_0=1\), the
condition \(\mathcal R_{13}\neq0\) excludes critical points away from
\(N\).  Consequently
\[
  \Sing(X_{13})_{\mathrm{red}}=N\sqcup L.
\]
There are no isolated singular points.  Since the singular locus has
codimension two in the hypersurface threefold, \(X_{13}\) is normal.  The
transverse type changes at
\[
  P=(1:0:0:0:0)\in N,
  \qquad
  Q=(0:0:0:0:1)\in L.
\]

\paragraph*{The special point \(P\).}
In the chart \(x_0=1\), with
\((u,v,z,w)=(x_1,x_2,x_3,x_4)\), the local equation is
\[
  \begin{aligned}
  f_P={}&v^5+\alpha uv^3z+\beta u^2vz^2+u^3w^2+v^3w
  +\gamma uvzw+z^3+\delta vw^2.
  \end{aligned}
\]
It is weighted homogeneous of degree \(15\) for
\[
  \wt(u,v,z,w)=(1,3,5,6),
\]
but its singularity is non-isolated.  The term \(uvzw\) places the diagonal
point \((1,1,1,1)\) on the Newton boundary.  Hence compact-face Newton
nondegeneracy and the Newton-polyhedron formula \cite{How03} give
\[
  \operatorname{lct}_0(f_P)=1,
  \qquad
  \mexp_P(X_{13})\geq1;
\]
see also \cite{MP20}.  For the weighted blow-up with weights
\((1,3,5,6)\), the exceptional divisor \(E_P\) satisfies
\[
  A_{\C^4}(E_P)=15
  =\operatorname{ord}_{E_P}(f_P).
\]
Thus the pair is not purely log terminal.  By inversion of adjunction and
the fact that \(X_{13}\) is normal Gorenstein, the hypersurface singularity
at \(P\) is not rational; see \cite[Chapter~5]{KM98}.  Saito's rationality
criterion \cite[Theorem~0.4]{Sai93} therefore gives the reverse inequality.
Consequently
\[
  \mexp_P(X_{13})=1.
\]

\paragraph*{Points of \(N\setminus\{P\}\).}
For \(p\in N\setminus\{P\}\), work in the chart \(x_1=1\), with
\(u=x_0\) the coordinate along \(N\) and
\((y,z,w)=(x_2,x_3,x_4)\) transverse to it.  The local equation is
\[
  \begin{aligned}
  f_N={}&(1+\delta u^2y)w^2+(uy^3+\gamma uyz)w\\
       &+y^5+\alpha y^3z+\beta yz^2+u^2z^3.
  \end{aligned}
\]
After completing the square in \(w\), make the analytic changes
\[
  y_1=\beta y+u^2z,
  \qquad
  z_1=z+\frac{\alpha}{2\beta^3}y_1^2.
\]
Up to multiplication by a unit and nonzero rescaling, the transverse
principal part is
\[
  W^2+y_1z_1^2+cy_1^5,
  \qquad
  c=\frac{4\beta-\alpha^2}{4\beta^6}\neq0.
\]
All remaining terms have larger transverse weight, so the parametric
semiquasihomogeneous reduction gives the transverse \(D_6\) germ.  Hence
\[
  \mexp_p(X_{13})
  =\frac12+\frac15+\frac25
  =\frac{11}{10}
  \qquad
  (p\in N\setminus\{P\}).
\]

\paragraph*{The special point \(Q\).}
In the chart \(x_4=1\), with
\((u,v,y,z)=(x_0,x_1,x_2,x_3)\), the local equation is
\[
  \begin{aligned}
  f_Q={}&y^5+\alpha vy^3z+\beta v^2yz^2+v^3+uy^3
  +\gamma uvyz+u^2z^3+\delta u^2y.
  \end{aligned}
\]
It is weighted homogeneous of degree \(15\) for
\[
  \wt(u,v,y,z)=(6,5,3,1).
\]
As at \(P\), compact-face Newton nondegeneracy and the monomial \(uvyz\)
give
\[
  \operatorname{lct}_0(f_Q)=1,
  \qquad
  \mexp_Q(X_{13})\geq1.
\]
For the weighted blow-up with weights \((6,5,3,1)\), one has
\[
  A_{\C^4}(E_Q)=15
  =\operatorname{ord}_{E_Q}(f_Q).
\]
The same inversion-of-adjunction and rationality argument therefore yields
\[
  \mexp_Q(X_{13})=1.
\]

\paragraph*{Points of \(L\setminus\{Q\}\).}
Let \(p\in L\setminus\{Q\}\).  In the chart \(x_3=1\), write
\[
  u=x_0,
  \qquad
  v=x_1,
  \qquad
  y=x_2,
  \qquad
  x_4=q+\tau,
\]
where \(\tau\) is the coordinate along \(L\), and put \(w=q+\tau\).  The
local equation is
\[
  \begin{aligned}
  f_L={}&(1+\delta yw^2)u^2+(y^3w+\gamma vyw)u\\
       &+y^5+\alpha vy^3+\beta v^2y+v^3w^2.
  \end{aligned}
\]
After completing the square in \(u\), make the analytic changes
\[
  y_1=\beta y+w^2v,
  \qquad
  v_1=v+\frac{\alpha}{2\beta^3}y_1^2.
\]
The transverse principal part is again
\[
  U^2+y_1v_1^2+cy_1^5,
  \qquad
  c=\frac{4\beta-\alpha^2}{4\beta^6}\neq0.
\]
Thus the transverse type is \(D_6\), and
\[
  \mexp_p(X_{13})
  =\frac{11}{10}
  \qquad
  (p\in L\setminus\{Q\}).
\]
The transverse Jacobian algebra of \(D_6\) has length \(6\); hence the
saturated Jacobian scheme has generic transverse multiplicity \(6\) along
both lines.

\paragraph*{Global minimal exponent.}
The preceding strata exhaust the singular locus.  Therefore
\[
  \mexp(X_{13})
  =\min\left\{1,\frac{11}{10},1,\frac{11}{10}\right\}
  =1
  =\frac{4+1}{5}.
\]
Thus a general closed-orbit representative of \(\Theta_{13}\) has global
minimal exponent equal to \(1\).

\subsection{\texorpdfstring{The component \(\Theta_{14}\)}{The component Theta 14}}
\label{subsec:singular-minimal-theta-14}

By Table~\ref{tab:component-normal-forms},
\[
  \Theta_{14}=\Phi_{14},
  \qquad
  k(14)=14.
\]
A general closed orbit in \(\Theta_{14}\) is represented by
\[
  \begin{aligned}
  F_{14}:={}&\nf_{14}(\alpha,\beta,\gamma,\delta)\\
  ={}&x_2^5
  +\alpha x_1x_2^3x_3
  +\beta x_1^2x_2x_3^2
  +x_1^4x_4
  +x_0x_3^4\\
  &+\gamma x_0x_2^3x_4
  +\delta x_0x_1x_2x_3x_4
  +x_0^2x_2x_4^2,
  \end{aligned}
\]
and we set
\[
  X_{14}:=V(F_{14})\subset\PP^4.
\]
We work on the nonempty Zariski-open set
\[
  (\alpha,\beta,\gamma,\delta)\in(\C^\times)^4,
  \qquad
  \mathcal R_{14}(\alpha,\beta,\gamma,\delta)\neq0,
\]
where \(\mathcal R_{14}\) excludes additional critical points and ensures
that the two local weighted-homogeneous germs below are isolated.

\paragraph*{Singular locus.}
For the saturated Jacobian ideal one has
\[
  \sqrt{J(F_{14}):(x_0,x_1,x_2,x_3,x_4)^\infty}
  =
  (x_1,x_2,x_3,x_4)
  \cap
  (x_0,x_1,x_2,x_3).
\]
Consequently
\[
  \Sing(X_{14})_{\mathrm{red}}=\{P,Q\},
\]
where
\[
  P=(1:0:0:0:0),
  \qquad
  Q=(0:0:0:0:1).
\]
Thus \(X_{14}\) has precisely two isolated singular points.  The local
Jacobian algebra has length \(102\) at each point.

\paragraph*{The point \(P\).}
In the chart \(x_0=1\), with
\((u,v,z,w)=(x_1,x_2,x_3,x_4)\), the local equation is
\[
  \begin{aligned}
  f_P={}&v^5+\alpha uv^3z+\beta u^2vz^2+u^4w+z^4
  +\gamma v^3w+\delta uvzw+vw^2.
  \end{aligned}
\]
This is weighted homogeneous of degree \(20\) for
\[
  \wt(u,v,z,w)=(3,4,5,8).
\]
It is isolated on the chosen parameter open set.  With
\[
  (X_1,X_2,X_3,X_4)=(u,v,z,w),
\]
the local equation is precisely the defining normal form of the family
\(\Sfam{102}{62}(\alpha,\beta,\gamma,\delta)\).  Hence, in the
Yonemura-number notation of Section~\ref{sec:singular-loci},
\[
  (X_{14},P)
  \sim_{\mathrm{an}}
  \Sfam{102}{62}(\alpha,\beta,\gamma,\delta).
\]
In particular, the isolated singularity has Milnor number \(102\) and
Yonemura number \(62\).  The weighted-homogeneous formula gives
\[
  \mexp_P(X_{14})
  =\frac{3+4+5+8}{20}
  =1.
\]

\paragraph*{The point \(Q\).}
In the chart \(x_4=1\), with
\((u,v,y,z)=(x_0,x_1,x_2,x_3)\), the local equation is
\[
  \begin{aligned}
  f_Q={}&y^5+\alpha vy^3z+\beta v^2yz^2+v^4+uz^4
  +\gamma uy^3+\delta uvyz+u^2y.
  \end{aligned}
\]
It is weighted homogeneous of degree \(20\) for
\[
  \wt(u,v,y,z)=(8,5,4,3).
\]
With
\[
  (X_1,X_2,X_3,X_4)=(z,y,v,u),
\]
the local equation is precisely the same defining normal form as at \(P\),
namely \(\Sfam{102}{62}(\alpha,\beta,\gamma,\delta)\).  Hence
\[
  (X_{14},Q)
  \sim_{\mathrm{an}}
  \Sfam{102}{62}(\alpha,\beta,\gamma,\delta).
\]
Thus the isolated singularity at \(Q\) also has Milnor number \(102\) and
Yonemura number \(62\), and
\[
  \mexp_Q(X_{14})
  =\frac{8+5+4+3}{20}
  =1.
\]

\paragraph*{Global minimal exponent.}
The two points above exhaust the singular locus.  Therefore
\[
  \mexp(X_{14})
  =\min\{\mexp_P(X_{14}),\mexp_Q(X_{14})\}
  =1
  =\frac{4+1}{5}.
\]
Thus a general closed-orbit representative of \(\Theta_{14}\) has global
minimal exponent equal to \(1\).

\subsection{\texorpdfstring{The component \(\Theta_{15}\)}{The component Theta 15}}
\label{subsec:singular-minimal-theta-15}

By Table~\ref{tab:component-normal-forms},
\[
  \Theta_{15}=\Phi_{15}=\Phi_{30},
  \qquad
  k(15)=15.
\]
A general closed orbit in \(\Theta_{15}\) is represented by
\[
  \begin{aligned}
  F_{15}:={}&\nf_{15}(Q,R,B)\\
  ={}&x_1x_2^4+x_1^2x_2^2x_3+x_1^3Q(x_3,x_4)
  +x_0x_2^3x_4\\
  &+x_0x_1x_2R(x_3,x_4)+x_0^2B(x_3,x_4),
  \end{aligned}
\]
where
\[
  Q=\gamma x_3^2+\beta x_3x_4+x_4^2,
  \qquad
  R=\rho_0x_3^2+\rho_1x_3x_4+\rho_2x_4^2,
\]
and
\[
  B=\sigma_0x_3^3+\sigma_1x_3^2x_4
    +\sigma_2x_3x_4^2+\sigma_3x_4^3.
\]
Set
\[
  X_{15}:=V(F_{15})\subset\PP^4.
\]
We work on the nonempty Zariski-open set defined by
\[
  \beta^2-4\gamma\neq0,
  \qquad
  4\gamma-\beta^2-1\neq0,
  \qquad
  \operatorname{Disc}(B)\neq0,
  \qquad
  \mathcal R_{15}(Q,R,B)\neq0,
\]
where \(\mathcal R_{15}\) excludes additional critical points and
transverse Jacobian degenerations.

\paragraph*{Singular locus.}
The two lines
\[
  N:=V(x_2,x_3,x_4),
  \qquad
  L:=V(x_0,x_1,x_2)
\]
are contained in \(\Sing(X_{15})\), and they are disjoint.  The Jacobian
equations on \(x_0=0\), together with the affine calculation on \(x_0=1\)
and the condition \(\mathcal R_{15}\neq0\), give
\[
  \Sing(X_{15})_{\mathrm{red}}=N\sqcup L.
\]
There are no isolated singular points.  The generic transverse Jacobian
multiplicities along \(N\) and \(L\) are respectively \(3\) and \(9\); in
the curve-cycle notation,
\[
  \Gamma_{15}=C^{(2)}_{1,3}+C^{(1)}_{1,9}.
\]
Since the singular locus has codimension two in the hypersurface threefold,
\(X_{15}\) is normal.  The distinguished point on \(N\) is
\[
  P=(1:0:0:0:0).
\]

\paragraph*{The special point \(P\).}
In the chart \(x_0=1\), with
\((u,y,z,w)=(x_1,x_2,x_3,x_4)\), the local equation is
\[
  f_P=uy^4+u^2y^2z+u^3Q(z,w)+y^3w+uyR(z,w)+B(z,w).
\]
It is weighted homogeneous of degree \(9\) for
\[
  \wt(u,y,z,w)=(1,2,3,3),
  \qquad
  1+2+3+3=9.
\]
Put \(\Delta=4\gamma-\beta^2\).  The projectivized weighted cone is smooth
away from \([1:0:0:0]\).  On the chart \(u=1\), the quadratic form
\(Q(z,w)\) is nondegenerate because \(\Delta\neq0\), and the splitting
lemma reduces the germ at this point to
\[
  Z^2+W^2+Y^4=0;
\]
the coefficient of \(Y^4\) is nonzero precisely because
\(\Delta-1=4\gamma-\beta^2-1\neq0\).  Thus the projectivized cone has only
an \(A_3\) singularity and is canonical.  The weighted-cone criterion and
inversion of adjunction \cite{KM98} therefore give
\[
  \operatorname{lct}_P(f_P)=1,
  \qquad
  \mexp_P(X_{15})\geq1.
\]
For the weighted blow-up with weights \((1,2,3,3)\), the exceptional divisor
has hypersurface discrepancy
\[
  (1+2+3+3)-1-9=-1.
\]
A rational Gorenstein hypersurface singularity is canonical
\cite[Chapter~5]{KM98}; hence this germ is not rational.  Saito's
rationality criterion \cite[Theorem~0.4]{Sai93} therefore gives
\[
  \mexp_P(X_{15})=1.
\]

\paragraph*{Points of \(N\setminus\{P\}\).}
For \(p\in N\setminus\{P\}\), work in the chart \(x_1=1\), with
\(u=x_0\) the coordinate along \(N\) and \((y,z,w)=(x_2,x_3,x_4)\)
transverse to it.  The local equation is
\[
  f_N=y^4+y^2z+Q(z,w)+uy^3w+uyR(z,w)+u^2B(z,w).
\]
The parametric splitting lemma separates the two Morse variables \((z,w)\).
The residual one-variable term is
\[
  \frac{\Delta-1}{\Delta}\,y^4+O(y^5),
\]
whose leading coefficient is nonzero.  Hence, up to the smooth parameter
along \(N\), the transverse germ is
\[
  Z^2+W^2+Y^4=0.
\]
Thus the transverse type is \(A_3\), and Thom--Sebastiani additivity gives
\[
  \mexp_p(X_{15})
  =\frac12+\frac12+\frac14
  =\frac54
  \qquad
  (p\in N\setminus\{P\}).
\]

\paragraph*{Points of \(L\).}
Give the normal variables \((x_0,x_1,x_2)\) along \(L\) weights
\[
  \wt_L(x_0,x_1,x_2)=(3,2,1).
\]
Every monomial of \(F_{15}\) has normal weighted degree \(6\).  The
projectivized weighted normal cone
\[
  \mathcal S_L:=V(F_{15})\subset\PP_L(3,2,1)
\]
is smooth on the associated weighted-projective stack: a singular point of
\(\mathcal S_L\) would lift to a singular point of \(X_{15}\) outside
\(L\), whereas
\(\Sing(X_{15})_{\mathrm{red}}=N\sqcup L\) and \(N\cap L=\varnothing\).
Since
\[
  3+2+1=6,
\]
the relative weighted-cone criterion and inversion of adjunction \cite{KM98}
yield
\[
  \operatorname{lct}_p(F_{15})=1,
  \qquad
  \mexp_p(X_{15})\geq1
  \qquad
  (p\in L).
\]
The exceptional divisor of the relative weighted blow-up has discrepancy
\[
  (3+2+1)-1-6=-1
\]
and dominates \(L\).  Hence \(X_{15}\) is noncanonical, and therefore
nonrational, at every point of \(L\).  Saito's criterion
\cite[Theorem~0.4]{Sai93} gives
\[
  \mexp_p(X_{15})=1
  \qquad
  (p\in L).
\]

\paragraph*{Global minimal exponent.}
The preceding strata exhaust the singular locus.  Therefore
\[
  \mexp(X_{15})
  =\min\left\{1,\frac54,1\right\}
  =1
  =\frac{4+1}{5}.
\]
Thus a general closed-orbit representative of \(\Theta_{15}\) has global
minimal exponent equal to \(1\).

\subsection{\texorpdfstring{The component \(\Theta_{16}\)}{The component Theta 16}}
\label{subsec:singular-minimal-theta-16}

By Table~\ref{tab:component-normal-forms},
\[
  \Theta_{16}=\Phi_{19},
  \qquad
  k(16)=19.
\]
A general closed orbit in \(\Theta_{16}\) is represented by
\[
  \begin{aligned}
  F_{16}:={}&\nf_{19}(\alpha,\beta,\gamma,\delta,
    \varepsilon,\zeta,\xi,\theta)\\
  ={}&x_2^5
  +\alpha x_1x_2^3x_3
  +\beta x_1^2x_2x_3^2
  +\gamma x_1^2x_2^2x_4
  +x_1^3x_3x_4\\
  &+\delta x_0x_2^2x_3^2
  +\varepsilon x_0x_2^3x_4
  +x_0x_1x_3^3
  +\zeta x_0x_1x_2x_3x_4\\
  &+\xi x_0x_1^2x_4^2
  +x_0^2x_3^2x_4
  +\theta x_0^2x_2x_4^2,
  \end{aligned}
\]
and we set
\[
  X_{16}:=V(F_{16})\subset\PP^4.
\]
Put
\[
  \rho_N:=1-\alpha\gamma+\beta\gamma^2,
  \qquad
  \rho_L:=1-\alpha\delta+\beta\delta^2.
\]
We work on the nonempty Zariski-open set
\[
  (\alpha,\beta,\gamma,\delta,\varepsilon,\zeta,\xi,\theta)
  \in(\C^\times)^8,
  \qquad
  \mathcal R_{19}\neq0,
\]
where \(\mathcal R_{19}\) excludes residual projective critical points,
imposes \(\rho_N\rho_L\neq0\), and ensures compact-face Newton
nondegeneracy at the two special points below.

\paragraph*{Singular locus.}
The two lines
\[
  N:=V(x_2,x_3,x_4),
  \qquad
  L:=V(x_0,x_1,x_2)
\]
are contained in \(\Sing(X_{16})\), and \(N\cap L=\varnothing\).  The
Jacobian calculation on \(x_0=0\), followed by the affine calculation on
\(x_0=1\), gives
\[
  \Sing(X_{16})_{\mathrm{red}}=N\sqcup L.
\]
There are no isolated singular points.  The generic transverse Jacobian
multiplicity is \(4\) along each line; equivalently, the curve part of the
saturated Jacobian scheme has cycle
\[
  \Gamma_{16}=C^{(2)}_{1,4}+C^{(2)}_{1,4}.
\]
The transverse type changes at
\[
  P=(1:0:0:0:0)\in N,
  \qquad
  Q=(0:0:0:0:1)\in L.
\]

\paragraph*{The special points \(P\) and \(Q\).}
In the chart \(x_0=1\), with
\((u,v,z,w)=(x_1,x_2,x_3,x_4)\), the local equation at \(P\) is
\[
  \begin{aligned}
  f_P={}&v^5+\alpha uv^3z+\beta u^2vz^2+\gamma u^2v^2w+u^3zw
  +\delta v^2z^2+\varepsilon v^3w\\
  &+uz^3+\zeta uvzw+\xi u^2w^2+z^2w+\theta vw^2.
  \end{aligned}
\]
It is weighted homogeneous of degree \(10\) for
\[
  \wt(u,v,z,w)=(1,2,3,4).
\]
Although the germ is non-isolated, the monomial \(uvzw\) places the
diagonal point \((1,1,1,1)\) on its Newton boundary.  Newton
nondegeneracy and the Newton-polyhedron formula therefore give
\[
  \operatorname{lct}_0(f_P)=1,
  \qquad
  \mexp_P(X_{16})\geq1;
\]
see \cite[Theorems~11 and~12]{How03} and \cite{MP20}.  The weighted upper
bound \cite[Proposition~2.1]{CDM25} gives
\[
  \mexp_P(X_{16})
  \leq\frac{1+2+3+4}{10}=1.
\]
Hence
\[
  \mexp_P(X_{16})=1.
\]

In the chart \(x_4=1\), with
\((u,v,y,z)=(x_0,x_1,x_2,x_3)\), the local equation at \(Q\) is
\[
  \begin{aligned}
  f_Q={}&y^5+\alpha vy^3z+\beta v^2yz^2+\gamma v^2y^2+v^3z
  +\delta uy^2z^2+\varepsilon uy^3\\
  &+uvz^3+\zeta uvyz+\xi uv^2+u^2z^2+\theta u^2y.
  \end{aligned}
\]
It is weighted homogeneous of degree \(10\) for
\[
  \wt(u,v,y,z)=(4,3,2,1).
\]
The same Newton argument and weighted upper bound yield
\[
  \mexp_Q(X_{16})=1.
\]

\paragraph*{General points of the two lines.}
Let \(p\in N\setminus\{P\}\).  In the chart \(x_1=1\), write
\[
  x_0=c+s,
  \qquad
  (y,z,w)=(x_2,x_3,x_4),
\]
where \(s\) is the coordinate along \(N\).  The quadratic part in
\((z,w)\) is
\[
  zw+\xi c w^2,
\]
which has rank \(2\).  Solving the transverse critical equations and
applying the parametric splitting lemma gives
\[
  z=-\gamma y^2+O(y^3),
  \qquad
  w=(2\beta\gamma-\alpha)y^3+O(y^4),
\]
and the residual function
\[
  g_N(s,y)=\rho_Ny^5+O(y^6).
\]
Since \(\rho_N\neq0\), a fifth-root coordinate change reduces the
transverse germ to
\[
  Z^2+W^2+Y^5=0.
\]
Thus the transverse type is \(A_4\), and
\[
  \mexp_p(X_{16})
  =\frac12+\frac12+\frac15
  =\frac65
  \qquad
  (p\in N\setminus\{P\}).
\]

Let \(p\in L\setminus\{Q\}\) with \(x_4(p)\neq0\).  In the chart
\(x_4=1\), write
\[
  x_3=c+s,
  \qquad
  (u,v,y)=(x_0,x_1,x_2),
  \qquad
  c\neq0.
\]
The quadratic part in \((u,v)\) is \(c^2u^2+c^3uv\), hence has rank
\(2\).  The transverse critical equations begin with
\[
  v=-\delta(c+s)^{-1}y^2+O(y^3),
  \qquad
  u=(2\beta\delta-\alpha)(c+s)^{-2}y^3+O(y^4),
\]
and the splitting residual is
\[
  g_L(s,y)=\rho_Ly^5+O(y^6).
\]
At the remaining point \((0:0:0:1:0)\), the chart \(x_3=1\) gives the
quadratic term \(uv\) and the same residual.  Since \(\rho_L\neq0\), the
transverse germ is again \(A_4\).  Therefore
\[
  \mexp_p(X_{16})=\frac65
  \qquad
  (p\in L\setminus\{Q\}).
\]

\paragraph*{Global minimal exponent.}
The preceding strata exhaust the singular locus.  Consequently
\[
  \mexp(X_{16})
  =\min\left\{1,\frac65,1,\frac65\right\}
  =1
  =\frac{4+1}{5}.
\]
Thus a general closed-orbit representative of \(\Theta_{16}\) has global
minimal exponent equal to \(1\).

\subsection{\texorpdfstring{The component \(\Theta_{17}\)}{The component Theta 17}}
\label{subsec:singular-minimal-theta-17}

By Table~\ref{tab:component-normal-forms},
\[
  \Theta_{17}=\Phi_{20}=\Phi_{31},
  \qquad
  k(17)=20.
\]
A general closed orbit in \(\Theta_{17}\) is represented by
\[
  \begin{aligned}
  F_{17}:={}&\nf_{20}(\lambda,\mu,\mathbf c,\mathbf q,\mathbf m)\\
  ={}&x_0B_{\lambda,\mu}(x_2,x_3)
  +x_1^2C_{\mathbf c}(x_2,x_3)
  +x_1^3x_2x_4\\
  &+x_0x_1x_4Q_{\mathbf q}(x_2,x_3)
  +x_0x_1^2x_4^2
  +x_0^2x_4^2M_{\mathbf m}(x_2,x_3),
  \end{aligned}
\]
where
\[
  B_{\lambda,\mu}(x_2,x_3)
  =x_3(x_2-x_3)(x_2-\lambda x_3)(x_2-\mu x_3),
\]
\[
  C_{\mathbf c}=c_0x_2^3+c_1x_2^2x_3+c_2x_2x_3^2+c_3x_3^3,
  \qquad
  Q_{\mathbf q}=q_0x_2^2+q_1x_2x_3+q_2x_3^2,
\]
and
\[
  M_{\mathbf m}=m_0x_2+m_1x_3.
\]
Set
\[
  X_{17}:=V(F_{17})\subset\PP^4.
\]
We work on the nonempty Zariski-open set defined by
\[
  \lambda\mu(\lambda-1)(\mu-1)(\lambda-\mu)\neq0,
  \qquad
  \operatorname{Disc}(C_{\mathbf c})\neq0,
\]
\[
  \operatorname{Res}(B_{\lambda,\mu},C_{\mathbf c})\neq0,
  \qquad
  \operatorname{Res}(B_{\lambda,\mu},M_{\mathbf m})\neq0,
  \qquad
  c_3\neq0,
\]
and
\[
  \mathcal R_{20}(\lambda,\mu,\mathbf c,\mathbf q,\mathbf m)\neq0,
\]
where \(\mathcal R_{20}\) excludes additional projective critical points and
local degenerations, and ensures compact-face Newton nondegeneracy at the
two special points below.

\paragraph*{Singular locus.}
Put
\[
  N:=V(x_2,x_3,x_4)
\]
and
\[
  C_{\lambda,\mu}
  :=V\bigl(x_0,x_1,B_{\lambda,\mu}(x_2,x_3)\bigr).
\]
The latter is the reduced union of four lines
\[
  C_{\lambda,\mu}
  =L_\infty\cup L_1\cup L_\lambda\cup L_\mu,
\]
where
\[
  \begin{aligned}
  L_\infty&=V(x_0,x_1,x_3),
  &L_1&=V(x_0,x_1,x_2-x_3),\\
  L_\lambda&=V(x_0,x_1,x_2-\lambda x_3),
  &L_\mu&=V(x_0,x_1,x_2-\mu x_3).
  \end{aligned}
\]
All four lines meet at
\[
  R=(0:0:0:0:1),
\]
whereas \(N\cap C_{\lambda,\mu}=\varnothing\).  Direct substitution shows
that \(N\cup C_{\lambda,\mu}\subset\Sing(X_{17})\), and saturation of the
Jacobian ideal leaves no residual component on the chosen open set.  Hence
\[
  \Sing(X_{17})_{\mathrm{red}}
  =N\sqcup C_{\lambda,\mu}.
\]
There are no isolated singular points.  Since the singular locus has
codimension two in the hypersurface threefold, \(X_{17}\) is normal.
Scheme-theoretically, the generic transverse Jacobian multiplicity is \(2\)
along \(N\) and \(1\) along each
of the four lines in \(C_{\lambda,\mu}\); thus
\[
  \Gamma_{17}=C^{(2)}_{1,2}+4C^{(3)}_{1,1}.
\]
The distinguished point on \(N\) is
\[
  P=(1:0:0:0:0).
\]

\paragraph*{The special point \(P\).}
In the chart \(x_0=1\), with
\((u,v,z,w)=(x_1,x_2,x_3,x_4)\), the local equation is
\[
  f_P
  =B_{\lambda,\mu}(v,z)+u^2C_{\mathbf c}(v,z)+u^3vw
   +uwQ_{\mathbf q}(v,z)+u^2w^2+w^2M_{\mathbf m}(v,z).
\]
It is weighted homogeneous of degree \(8\) for
\[
  \wt(u,v,z,w)=(1,2,2,3).
\]
The monomials \(u^2w^2\), \(v^3z\), and \(z^4\) place the diagonal point
\((1,1,1,1)\) on the Newton boundary.  Compact-face Newton nondegeneracy
therefore gives
\[
  \operatorname{lct}_0(f_P)=1,
  \qquad
  \mexp_P(X_{17})\geq1;
\]
see \cite{How03,MP20}.  For the weighted blow-up with weights
\((1,2,2,3)\), one has
\[
  A_{\C^4}(E)=8=\operatorname{ord}_E(f_P).
\]
The weighted-blow-up upper-bound argument, using inversion of adjunction and
Saito's rationality criterion \cite{KM98,Sai93}, gives
\(\mexp_P(X_{17})\leq1\).  Hence
\[
  \mexp_P(X_{17})=1.
\]

\paragraph*{Points of \(N\setminus\{P\}\).}
Let \(p\in N\setminus\{P\}\).  In the chart \(x_1=1\), write
\[
  x_0=c+s,
  \qquad
  (y,z,w)=(x_2,x_3,x_4),
\]
where \(s\) is the coordinate along \(N\).  The quadratic part in the
transverse variables \((y,w)\) is
\[
  yw+cw^2,
\]
and has rank \(2\).  The parametric splitting lemma reduces the remaining
one-variable germ to
\[
  c_3z^3+O(z^4).
\]
Since \(c_3\neq0\), the transverse local normal form is
\[
  Y^2+W^2+Z^3=0.
\]
Thus the transverse type is \(A_2\), and
\[
  \mexp_p(X_{17})
  =\frac12+\frac12+\frac13
  =\frac43
  \qquad
  (p\in N\setminus\{P\}).
\]

\paragraph*{Points of \(C_{\lambda,\mu}\setminus\{R\}\).}
At a point \(p\in C_{\lambda,\mu}\setminus\{R\}\), exactly one factor
\(\ell(x_2,x_3)\) of \(B_{\lambda,\mu}\) vanishes.  Taking
\[
  u=x_0,
  \qquad
  v=x_1,
  \qquad
  s=\ell(x_2,x_3)
\]
as transverse coordinates, the quadratic part has the form
\[
  us\cdot\text{unit}+v^2\cdot\text{unit}
  +\text{terms in }u^2,uv.
\]
The coefficient of \(v^2\) is nonzero because
\(\operatorname{Res}(B_{\lambda,\mu},C_{\mathbf c})\neq0\), so this
quadratic form has rank \(3\).  Hence the transverse local normal form is
\[
  U^2+V^2+S^2=0,
\]
and therefore
\[
  \mexp_p(X_{17})
  =\frac12+\frac12+\frac12
  =\frac32
  \qquad
  (p\in C_{\lambda,\mu}\setminus\{R\}).
\]

\paragraph*{The common point \(R\).}
In the chart \(x_4=1\), with
\((u,v,y,z)=(x_0,x_1,x_2,x_3)\), the local equation is
\[
  f_R
  =uB_{\lambda,\mu}(y,z)+v^2C_{\mathbf c}(y,z)+v^3y
   +uvQ_{\mathbf q}(y,z)+uv^2+u^2M_{\mathbf m}(y,z).
\]
It is weighted homogeneous of degree \(7\) for
\[
  \wt(u,v,y,z)=(3,2,1,1).
\]
The monomials \(uv^2\), \(uy^3z\), and \(uz^4\) place the diagonal point
on the Newton boundary.  Thus Newton nondegeneracy gives
\[
  \operatorname{lct}_0(f_R)=1,
  \qquad
  \mexp_R(X_{17})\geq1.
\]
For the weighted blow-up with weights \((3,2,1,1)\),
\[
  A_{\C^4}(E)=7=\operatorname{ord}_E(f_R),
\]
so the same weighted-blow-up argument gives the reverse inequality.  Hence
\[
  \mexp_R(X_{17})=1.
\]

\paragraph*{Global minimal exponent.}
The preceding strata exhaust the singular locus.  Consequently
\[
  \mexp(X_{17})
  =\min\left\{1,\frac43,1,\frac32\right\}
  =1
  =\frac{4+1}{5}.
\]
Thus a general closed-orbit representative of \(\Theta_{17}\) has global
minimal exponent equal to \(1\).

\subsection{\texorpdfstring{The component \(\Theta_{18}\)}{The component Theta 18}}
\label{subsec:singular-minimal-theta-18}

By Table~\ref{tab:component-normal-forms},
\[
  \Theta_{18}=\Phi_{23}=\Phi_{38},
  \qquad
  k(18)=23.
\]
A general closed orbit in \(\Theta_{18}\) is represented by
\[
  F_{18}:=\nf_{23}(B_4)
  =x_0B_4(x_1,x_2,x_3,x_4),
\]
where \([B_4]\in\PP(\Sym^4\langle x_1,x_2,x_3,x_4\rangle)\) is stable.
For the singularity calculation, we work on the nonempty Zariski-open locus
on which
\[
  V(B_4)\subset\PP^3
\]
is smooth, and set
\[
  X_{18}:=V(F_{18})\subset\PP^4.
\]

\paragraph*{Singular locus.}
Put
\[
  S:=V(x_0,B_4)\subset\PP^4,
  \qquad
  P:=(1:0:0:0:0).
\]
Then \(S\) is a smooth quartic surface.  Since
\[
  \frac{\partial F_{18}}{\partial x_0}=B_4,
  \qquad
  \frac{\partial F_{18}}{\partial x_i}
  =x_0\frac{\partial B_4}{\partial x_i}
  \quad(1\leq i\leq4),
\]
smoothness of \(V(B_4)\) and Euler's identity give
\[
  \Sing(X_{18})_{\mathrm{red}}=S\sqcup\{P\}.
\]
More precisely, if
\[
  Q:=\left(
    \frac{\partial B_4}{\partial x_1},
    \ldots,
    \frac{\partial B_4}{\partial x_4}
  \right),
\]
then
\[
  J(F_{18})=(x_0,B_4)\cap Q,
\]
and this ideal is already saturated.  Thus
\[
  \Sing(X_{18})=S\sqcup Z_P,
\]
where \(S\) is reduced and \(Z_P\) is a punctual complete intersection
supported at \(P\).  Its local length is
\[
  \operatorname{length}(Z_P)=3^4=81.
\]

\paragraph*{Points of the surface component \(S\).}
Let \(p\in S\).  Since \(dB_4(p)\neq0\), there are analytic coordinates
\((u,v,z_1,z_2)\) centered at \(p\) such that
\[
  u=x_0,
  \qquad
  v=B_4.
\]
The local equation of \(X_{18}\) is therefore
\[
  uv=0.
\]
This is a simple normal-crossing germ, whose reduced Bernstein--Sato
polynomial is \(s+1\).  Hence
\[
  \mexp_p(X_{18})=1
  \qquad(p\in S).
\]

\paragraph*{The isolated point \(P\).}
In the chart \(x_0=1\), the local equation at \(P\) is
\[
  B_4(y_1,y_2,y_3,y_4)=0.
\]
The origin is an isolated critical point because \(V(B_4)\subset\PP^3\) is
smooth.  The germ is weighted homogeneous of degree \(4\) for
\[
  \wt(y_1,y_2,y_3,y_4)=(1,1,1,1).
\]
With
\[
  (X_1,X_2,X_3,X_4)=(y_1,y_2,y_3,y_4),
\]
the local equation is precisely the defining normal form of the family
\(\Sfam{81}{1}(B_4)\).  Hence, in the Yonemura-number notation of
Section~\ref{sec:singular-loci},
\[
  (X_{18},P)
  \sim_{\mathrm{an}}
  \Sfam{81}{1}(B_4).
\]
In particular, the isolated singularity has Milnor number \(81\) and
Yonemura number \(1\).  Consequently, the weighted-homogeneous formula
gives~\cite{Sai71,Sai94}
\[
  \mexp_P(X_{18})
  =\frac{1+1+1+1}{4}
  =1.
\]

\paragraph*{Global minimal exponent.}
The preceding strata exhaust the singular locus.  Therefore
\[
  \mexp(X_{18})
  =\min\{1,1\}
  =1
  =\frac{4+1}{5}.
\]
Thus a general closed-orbit representative of \(\Theta_{18}\) has global
minimal exponent equal to \(1\).

\subsection{\texorpdfstring{The component \(\Theta_{19}\)}{The component Theta 19}}
\label{subsec:singular-minimal-theta-19}

By Table~\ref{tab:component-normal-forms},
\[
  \Theta_{19}=\Phi_{29}=\Phi_{35},
  \qquad
  k(19)=29.
\]
A general closed orbit in \(\Theta_{19}\) is represented by
\[
  \begin{aligned}
  F_{19}:={}&\nf_{29}(A_4,C_4,Q_{2,2})\\
  ={}&x_3A_4(x_1,x_2)+x_4C_4(x_1,x_2)
  +x_0Q_{2,2}(x_1,x_2;x_3,x_4)\\
  &+x_0^2x_3x_4(x_3-x_4),
  \end{aligned}
\]
where
\[
  A_4,C_4\in\Sym^4\langle x_1,x_2\rangle,
  \qquad
  Q_{2,2}\in
  \Sym^2\langle x_1,x_2\rangle\otimes
  \Sym^2\langle x_3,x_4\rangle.
\]
Set
\[
  X_{19}:=V(F_{19})\subset\PP^4.
\]
We work on the nonempty Zariski-open set
\[
  \operatorname{Res}(A_4,C_4)\neq0,
  \qquad
  \mathcal R_{29}(A_4,C_4,Q_{2,2})\neq0,
\]
where \(\mathcal R_{29}\) imposes the isolatedness of the affine critical
point below, smoothness of the relative weighted quartic used in the
weighted blow-up, and the generic socle condition along the singular line.

\paragraph*{Singular locus.}
Put
\[
  L:=V(x_0,x_1,x_2),
  \qquad
  P:=(1:0:0:0:0).
\]
Every term of \(F_{19}\), and every first derivative, vanishes on \(L\), so
\(L\subset\Sing(X_{19})\).  On the hyperplane \(x_0=0\), the equations
\[
  \frac{\partial F_{19}}{\partial x_3}=A_4(x_1,x_2),
  \qquad
  \frac{\partial F_{19}}{\partial x_4}=C_4(x_1,x_2)
\]
force \((x_1,x_2)=(0,0)\), because
\(\operatorname{Res}(A_4,C_4)\neq0\).  Thus the only singular points on
\(x_0=0\) lie on \(L\).

In the chart \(x_0=1\), with
\[
  (u,v,y,z)=(x_1,x_2,x_3,x_4),
\]
the local equation is
\[
  f_P
  =yA_4(u,v)+zC_4(u,v)
   +Q_{2,2}(u,v;y,z)+yz(y-z).
\]
Its critical locus is zero-dimensional on the chosen open set and is
invariant under the positive-weight action with weights \((1,1,2,2)\).
Hence it contains no nonzero point, and the origin is its only critical
point.  Consequently
\[
  \Sing(X_{19})_{\mathrm{red}}=L\sqcup\{P\}.
\]

The saturated Jacobian scheme is nonreduced along \(L\).  At its generic
point, splitting the nonzero quadratic term in the \(x_0\)-direction leaves
a square-free binary quartic \(G_4\).  Its Milnor algebra has length \(9\)
and a one-dimensional degree-four socle.  The line-direction derivative
\(\dot G_4\) is nonzero in that socle, so the generic transverse Jacobian
length is \(8\); the generic transverse Hessian rank is \(1\).  Thus the
curve part is recorded as
\[
  \Gamma_{19}=C^{(1)}_{1,8}.
\]

\paragraph*{The isolated point \(P\).}
The germ \(f_P\) is weighted homogeneous of degree \(6\) for
\[
  \wt(u,v,y,z)=(1,1,2,2),
\]
and is isolated on the chosen open set.  Put
\[
  B^{(29)}_3(X_3,X_4):=X_3X_4(X_3-X_4).
\]
With
\[
  (X_1,X_2,X_3,X_4)=(u,v,y,z),
\]
the local equation is precisely the defining normal form of the family
\[
  \Sfam{100}{3}(A_4,C_4,Q_{2,2},B^{(29)}_3).
\]
Hence, in the Yonemura-number notation of
Section~\ref{sec:singular-loci},
\[
  (X_{19},P)
  \sim_{\mathrm{an}}
  \Sfam{100}{3}(A_4,C_4,Q_{2,2},B^{(29)}_3).
\]
In particular, the isolated singularity has Milnor number \(100\) and
Yonemura number \(3\).  The weighted-homogeneous formula therefore
gives~\cite{Sai71,Sai94}
\[
  \mexp_P(X_{19})
  =\frac{1+1+2+2}{6}
  =1.
\]
Its local Jacobi algebra has length
\[
  \left(\frac61-1\right)^2
  \left(\frac62-1\right)^2
  =5^2\cdot2^2
  =100,
\]
in agreement with the subscript of \(\Sfam{100}{3}\).

\paragraph*{Points of the line \(L\).}
Let \(U_L\) be a sufficiently small neighborhood of a point of \(L\), and
let
\[
  \pi:\widetilde U_L\longrightarrow U_L
\]
be the weighted blow-up of the smooth center \(L\cap U_L\) with normal
weights
\[
  \wt(x_0,x_1,x_2)=(2,1,1).
\]
Every monomial of \(F_{19}\) has normal weighted order \(4\).  Hence, writing
\(E\) for the exceptional divisor and \(\widetilde X_{19}\) for the proper
transform,
\[
  \pi^*X_{19}=\widetilde X_{19}+4E,
  \qquad
  K_{\widetilde U_L}=\pi^*K_{U_L}+3E
\]
on the local index-one covers.  Therefore
\[
  K_{\widetilde U_L}+\widetilde X_{19}+E
  =\pi^*(K_{U_L}+X_{19}).
\]
The intersection \(\widetilde X_{19}\cap E\) is the relative weighted
quartic
\[
  \begin{aligned}
  x_3A_4(y_1,y_2)+x_4C_4(y_1,y_2)
  &+\xi Q_{2,2}(y_1,y_2;x_3,x_4)\\
  &+\xi^2x_3x_4(x_3-x_4)=0
  \end{aligned}
\]
in the \(\PP(2,1,1)\)-bundle over \(L\).  By the condition
\(\mathcal R_{29}\neq0\), it is smooth on the local index-one covers.
Moreover, all terms of \(F_{19}\) have weighted order exactly \(4\), so
\(\widetilde X_{19}\) is smooth near \(E\) and meets \(E\) transversely.
Thus \((U_L,X_{19}|_{U_L})\) is log canonical along \(L\); compare
\cite{KM98}.  Using~\cite{Sai94,MP20}
\[
  \operatorname{lct}_p(F_{19})
  =\min\{\mexp_p(F_{19}),1\},
\]
we obtain
\[
  \mexp_p(X_{19})\geq1
  \qquad(p\in L).
\]

\paragraph*{Global minimal exponent.}
The preceding strata exhaust the singular locus.  Since
\(\mexp_P(X_{19})=1\) and \(\mexp_p(X_{19})\geq1\) for every \(p\in L\),
\[
  \mexp(X_{19})
  =1
  =\frac{4+1}{5}.
\]
Thus a general closed-orbit representative of \(\Theta_{19}\) has global
minimal exponent equal to \(1\).

\subsection{\texorpdfstring{The component \(\Theta_{20}\)}{The component Theta 20}}
\label{subsec:singular-minimal-theta-20}

By Table~\ref{tab:component-normal-forms},
\[
  \Theta_{20}=\Phi_{32}=\Phi_{37},
  \qquad
  k(20)=32.
\]
Put
\[
  V:=\langle x_0,x_1\rangle,
  \qquad
  U:=\langle x_2,x_3,x_4\rangle.
\]
A general closed orbit in \(\Theta_{20}\) is represented by
\[
  \begin{aligned}
  F_{20}:={}&\nf_{32}(C_0,C_1,C_2)\\
  ={}&x_0^2C_0(x_2,x_3,x_4)
  +x_0x_1C_1(x_2,x_3,x_4)
  +x_1^2C_2(x_2,x_3,x_4),
  \end{aligned}
\]
where \(C_0,C_1,C_2\in\Sym^3U\).  Set
\[
  X_{20}:=V(F_{20})\subset\PP^4.
\]
We work on the nonempty Zariski-open locus denoted by
\(\mathcal R_{32}(C_0,C_1,C_2)\neq0\).  On this locus, the bidegree
\((2,3)\) divisor
\[
  D_{20}:=V(F_{20})\subset\PP(V)\times\PP(U)
\]
is smooth, the discriminant curve
\[
  \Delta_\Pi
  :=V\bigl(x_0,x_1,C_1^2-4C_0C_2\bigr)
  \subset\Pi:=V(x_0,x_1)
\]
is smooth and reduced, and the conic family of ternary cubics parametrized
by \(N:=V(x_2,x_3,x_4)\) meets the ternary-cubic discriminant transversely
and only in its smooth nodal locus.

\paragraph*{Singular locus.}
Every monomial of \(F_{20}\) has degree two in \((x_0,x_1)\) and degree
three in \((x_2,x_3,x_4)\).  Hence
\[
  N\cup\Pi\subset\Sing(X_{20}).
\]
Away from \(N\cup\Pi\), the two bihomogeneous Euler identities identify the
projective Jacobian equations of \(F_{20}\) with the singularity equations
of \(D_{20}\).  Since \(D_{20}\) is smooth, there are no further singular
points.  Therefore
\[
  \Sing(X_{20})_{\mathrm{red}}=N\sqcup\Pi,
  \qquad
  N\simeq\PP^1,
  \qquad
  \Pi\simeq\PP^2.
\]
In particular, there are no isolated singular points.

The plane component is generically reduced.  At a general point of
\(\Pi\setminus\Delta_\Pi\), its transverse Jacobian ideal is \((u,v)\),
whereas at a point of \(\Delta_\Pi\) the local coordinates below give
\[
  (u,zv,v^2).
\]
Along \(N\), the generic transverse Jacobian algebra has length \(7\) and
the generic transverse Hessian rank is \(0\).  Indeed, for the corresponding
smooth ternary cubic \(C_t\), the Jacobian algebra of \(C_t\) has length
\(8\); the parameter derivative represents a nonzero socle class, and
quotienting by it reduces the length to \(7\).

\paragraph*{Points of the plane \(\Pi\).}
For
\[
  p=(0:0:v_2:v_3:v_4)\in\Pi,
\]
the quadratic form transverse to \(\Pi\) is
\[
  Q_p(u,v)
  =u^2C_0(p)+uvC_1(p)+v^2C_2(p).
\]
If \(p\notin\Delta_\Pi\), then \(Q_p\) has rank two.  The parameterized
splitting lemma gives analytic coordinates in which
\[
  F_{20}=u'^2+v'^2,
\]
with two smooth parameters.  Hence
\[
  \mexp_p(X_{20})=1
  \qquad
  (p\in\Pi\setminus\Delta_\Pi).
\]

Let \(p\in\Delta_\Pi\).  Smoothness of \(\Delta_\Pi\) implies that
\(Q_p\) has rank one.  Completing the square and taking
\(z=-(C_1^2-4C_0C_2)/(4C_0)\) as a local coordinate transverse to
\(\Delta_\Pi\), one obtains the exact analytic normal form
\[
  F_{20}=u'^2+zv^2,
\]
with one further smooth parameter.  The monomial Bernstein--Sato formula
and Thom--Sebastiani give~\cite{CDNM24}
\[
  \mexp(u'^2)=\frac12,
  \qquad
  \mexp(zv^2)=\frac12,
\]
and therefore
\[
  \mexp_p(X_{20})=1
  \qquad
  (p\in\Delta_\Pi).
\]
Consequently,
\[
  \mexp_p(X_{20})=1
  \qquad
  (p\in\Pi).
\]

\paragraph*{Points of the line \(N\).}
Let
\[
  q=(u_0:u_1:0:0:0)\in N,
\]
and write
\[
  C_q
  =u_0^2C_0+u_0u_1C_1+u_1^2C_2.
\]
In a local coordinate \(t\) on \(N\), the germ at \(q\) has the cone-family
form
\[
  f(t,y)=C_t(y),
  \qquad
  y=(y_2,y_3,y_4),
\]
where each \(C_t\) is homogeneous of degree three.  By the genericity
assumptions, the projective family
\[
  \mathcal C=\{(t,[y])\mid C_t(y)=0\}
\]
is smooth along the fiber over \(q\): this is immediate for a smooth cubic
fiber, and at a nodal fiber it follows from transversality to the
ternary-cubic discriminant.

Blow up the \(t\)-axis \(Z=\{y_2=y_3=y_4=0\}\).  If \(E\) is the
exceptional divisor and \(\widetilde X_{20}\) is the strict transform, then
\[
  \operatorname{div}(f\circ\pi)=\widetilde X_{20}+3E,
  \qquad
  K_{\operatorname{Bl}_Z\C^4/\C^4}=2E.
\]
The strict transform is smooth and meets \(E\) transversely, so this is a
log resolution near \(q\).  Hence
\[
  \operatorname{lct}_q(f)
  =\min\left\{1,\frac{2+1}{3}\right\}
  =1,
\]
and thus \(\mexp_q(X_{20})\geq1\) by
\(\operatorname{lct}_q(f)=\min\{\mexp_q(f),1\}\)~\cite{MP20}.

On the other hand, the divisor \(E\cap\widetilde X_{20}\) has discrepancy
\(2-3=-1\) over \(X_{20}\), so the germ is not canonical.  Since its singular locus has codimension two in the hypersurface
threefold, the germ is normal; as a hypersurface, it is Gorenstein.  Since a
normal Gorenstein rational singularity is canonical~\cite{KM98}, the germ
is not rational.  Saito's criterion then gives \(\mexp_q(X_{20})\leq1\)
\cite{Sai93}.  Therefore
\[
  \mexp_q(X_{20})=1
  \qquad
  (q\in N).
\]

\paragraph*{Global minimal exponent.}
The two strata above exhaust the singular locus, and the local minimal
exponent is \(1\) at every point of both components.  Consequently,
\[
  \mexp(X_{20})
  =1
  =\frac{4+1}{5}.
\]
Thus a general closed-orbit representative of \(\Theta_{20}\) has global
minimal exponent equal to \(1\).

\subsection{\texorpdfstring{The component \(\Theta_{21}\)}{The component Theta 21}}
\label{subsec:singular-minimal-theta-21}

By Table~\ref{tab:component-normal-forms},
\[
  \Theta_{21}=\Phi_{36},
  \qquad
  k(21)=36.
\]
Put
\[
  U:=\langle x_1,x_2,x_3\rangle.
\]
A general closed orbit in \(\Theta_{21}\) is represented by
\[
  \begin{aligned}
  F_{21}:={}&\nf_{36}(B_5,C_3)\\
  ={}&B_5(x_1,x_2,x_3)
  +x_0x_4C_3(x_1,x_2,x_3)
  +x_0^2x_3x_4^2,
  \end{aligned}
\]
where
\[
  B_5\in\Sym^5U,
  \qquad
  C_3\in\Sym^3U.
\]
Set
\[
  X_{21}:=V(F_{21})\subset\PP^4.
\]
We work on the nonempty Zariski-open locus
\[
  \operatorname{Disc}(B_5)\neq0,
  \qquad
  \operatorname{Disc}(C_3)\neq0,
  \qquad
  \mathcal R_{36}(B_5,C_3)\neq0.
\]
Here the last condition ensures that the weighted-homogeneous germ appearing
below is isolated and that no additional projective singular point occurs.

\paragraph*{Singular locus.}
The two points
\[
  P_0=(1:0:0:0:0),
  \qquad
  P_\infty=(0:0:0:0:1)
\]
are singular.  On the plane \(x_0=x_4=0\), the Jacobian equations reduce to
\[
  \frac{\partial B_5}{\partial x_1}
  =\frac{\partial B_5}{\partial x_2}
  =\frac{\partial B_5}{\partial x_3}=0.
\]
Since \(V(B_5)\subset\PP(U)\simeq\PP^2\) is smooth, these equations have no
projective solution.  On the chart \(x_0=1\), the affine equation is
\[
  g(X_1,X_2,X_3,X_4)
  =B_5(X_1,X_2,X_3)
  +X_4C_3(X_1,X_2,X_3)
  +X_4^2X_3.
\]
The condition \(\mathcal R_{36}\neq0\) says that the origin is the only
critical point of \(g\).  The same equation occurs on the chart \(x_4=1\).
Consequently,
\[
  \Sing(X_{21})_{\mathrm{red}}=\{P_0,P_\infty\}.
\]
The saturated Jacobian scheme is zero-dimensional, and its local length is
\(96\) at each point.

\paragraph*{The two isolated singularities.}
At \(P_0\), on the chart \(x_0=1\), take
\[
  (X_1,X_2,X_3,X_4)=(x_1,x_2,x_3,x_4),
\]
and at \(P_\infty\), on the chart \(x_4=1\), take
\[
  (X_1,X_2,X_3,X_4)=(x_1,x_2,x_3,x_0).
\]
In either coordinate system, the local equation is
\[
  g
  =B_5(X_1,X_2,X_3)
  +X_4C_3(X_1,X_2,X_3)
  +X_4^2X_3.
\]
This is an isolated weighted-homogeneous germ of degree \(5\) for
\[
  \wt(X_1,X_2,X_3,X_4)=(1,1,1,2).
\]
Taking \(L_1=X_3\), the local equation is precisely the defining normal form
of the family \(\Sfam{96}{21}(B_5,C_3,X_3)\).  Hence, in the Yonemura-number
notation of Section~\ref{sec:singular-loci},
\[
  \begin{aligned}
  (X_{21},P_0)
  &\sim_{\mathrm{an}}
  \Sfam{96}{21}(B_5,C_3,X_3),\\
  (X_{21},P_\infty)
  &\sim_{\mathrm{an}}
  \Sfam{96}{21}(B_5,C_3,X_3).
  \end{aligned}
\]
In particular, each isolated singularity has Milnor number \(96\) and
Yonemura number \(21\).  Therefore the weighted-homogeneous formula gives
\[
  \mexp_{P_0}(X_{21})
  =\mexp_{P_\infty}(X_{21})
  =\frac{1+1+1+2}{5}
  =1.
\]

\paragraph*{Global minimal exponent.}
The two points above exhaust the singular locus.  Hence
\[
  \mexp(X_{21})
  =\min\{\mexp_{P_0}(X_{21}),\mexp_{P_\infty}(X_{21})\}
  =1
  =\frac{4+1}{5}.
\]
Thus a general closed-orbit representative of \(\Theta_{21}\) has global
minimal exponent equal to \(1\).

\section*{Acknowledgments}

I would like to express my gratitude to the late Professor Eiji Horikawa, who supervised my master’s studies at the University of Tokyo. His mathematical guidance and philosophy of research have been a lasting influence on my work.
I am also grateful to my family for their continuous support, which has provided a stable environment for my research.
Finally, I would like to record my appreciation for the lifelong friendship of my childhood friends from my elementary school days. Their enduring friendship and perspective have been invaluable over many years.

\section*{Declaration of AI-Assisted Tools}

The author acknowledges the use of AI-assisted technologies in the preparation of this manuscript. Specifically, these tools were utilized as aids in computing the normal forms of polystable polynomials, and to assist with the explicit algebraic computations required to derive the local normal forms of isolated singularities. In particular, regarding the latter computations, when standard computer algebra systems encountered infinite loops during the reduction processes, AI tools were employed to execute the explicit symbolic substitutions and reductions dictated by the holomorphic splitting lemma.

Furthermore, these tools were used as auxiliary aids for drafting computational scripts, and for English language editing and stylistic formatting.

All intermediate outputs generated by the AI were rigorously verified by the author. The mathematical reasoning, theoretical proofs, and derivation of final conclusions were conducted independently by the author. The author takes full responsibility for all mathematical content, results, and wording presented in this final manuscript.

\bigskip
\noindent
\textit{E-mail address:}
\href{mailto:yasutaka@shibata-math.com}{\texttt{yasutaka@shibata-math.com}}


\begin{thebibliography}{GMMS26}

\bibitem[CDM25]{CDM25}
Q.~Chen, B.~Dirks, and M.~Musta\c{t}\u{a},
The minimal exponent of cones over smooth complete intersection projective varieties,
\emph{Rev. Roumaine Math. Pures Appl.} \textbf{70} (2025), 33--47.
DOI: \texttt{10.59277/rrmpa.2025.33.47}.





\bibitem[CF12]{CF12}
C.~Florentino and A.~C.~Casimiro,
Stability of affine $G$-varieties and irreducibility in reductive groups,
\emph{Int. J. Math.} \textbf{23} (2012), no.~8, 1250082, 30~pp.
DOI: \texttt{10.1142/S0129167X12500826}.





\bibitem[CDNM24]{CDNM24}
A.~Casta\~no Dom\'{\i}nguez and L.~Narv\'aez Macarro,
On the reduced Bernstein--Sato polynomial of Thom--Sebastiani singularities,
\emph{Rev. Mat. Complut.} \textbf{37} (2024), 877--886.
DOI: \texttt{10.1007/s13163-023-00478-x}.









\bibitem[Dol03]{Dol03}
I.~Dolgachev,
\emph{Lectures on Invariant Theory},
London Mathematical Society Lecture Note Series, \textbf{296},
Cambridge University Press, Cambridge, 2003.

\bibitem[GMMS26]{GMMS26}
P.~Gallardo, J.~Mart\'{\i}nez-Garc\'{\i}a, H.-B.~Moon, and D.~Swinarski,
Computation of GIT quotients of semisimple groups,
\emph{Math. Comp.}, in press, 2026.
DOI: \texttt{10.1090/mcom/4152};
arXiv:2308.08049 [math.AG].


\bibitem[How03]{How03}
J.~Howald,
Multiplier ideals of sufficiently general polynomials,
arXiv:math/0303203 [math.AG], 2003.
DOI: \texttt{10.48550/arXiv.math/0303203}.







\bibitem[Kir84]{Kir84}
F.~C.~Kirwan,
\emph{Cohomology of Quotients in Symplectic and Algebraic Geometry},
Mathematical Notes, \textbf{31},
Princeton University Press, Princeton, NJ, 1984.


\bibitem[KM98]{KM98}
J.~Koll\'ar and S.~Mori,
\emph{Birational Geometry of Algebraic Varieties},
Cambridge Tracts in Mathematics, \textbf{134},
Cambridge University Press, Cambridge, 1998.
DOI: \texttt{10.1017/CBO9780511662560}.






\bibitem[Lak10]{Lak10}
C.~Lakhani,
\emph{The GIT Compactification of Quintic Threefolds},
arXiv:1010.3803 [math.AG], 2010.
DOI: \texttt{10.48550/arXiv.1010.3803}.



\bibitem[Lun73]{Lun73}
D.~Luna,
Slices \'etales,
in \emph{Sur les groupes alg\'ebriques},
\emph{Bulletin de la Soci\'et\'e math\'ematique de France. M\'emoire},
no.~\textbf{33} (1973), 81--105.
DOI: \texttt{10.24033/msmf.110}.








\bibitem[Lun75]{Lun75}
D.~Luna,
Adh\'erences d'orbite et invariants,
\emph{Invent. Math.} \textbf{29} (1975), 231--238.
DOI: \texttt{10.1007/BF01389851}.

\bibitem[MFK94]{MFK94}
D.~Mumford, J.~Fogarty, and F.~Kirwan,
\emph{Geometric Invariant Theory} (3rd ed.),
Ergebnisse der Mathematik und ihrer Grenzgebiete (2), \textbf{34},
Springer-Verlag, Berlin, 1994.

\bibitem[MP20]{MP20}
M.~Musta\c{t}\u{a} and M.~Popa,
Hodge filtration, minimal exponent, and local vanishing,
\emph{Invent. Math.} \textbf{220} (2020), no.~2, 453--478.
DOI: \texttt{10.1007/s00222-019-00933-x}.

\bibitem[Pae24]{Pae24}
E.~Paemurru,
Reading the log canonical threshold of a plane curve singularity from its
Newton polyhedron,
\emph{Ann. Univ. Ferrara} \textbf{70} (2024), 1069--1082.
DOI: \texttt{10.1007/s11565-024-00524-6}.

\bibitem[Par25]{Par25}
S.~G.~Park,
\emph{The GIT stability and Hodge structures of hypersurfaces via minimal exponent},
arXiv:2510.14352 [math.AG], 2025.
DOI: \texttt{10.48550/arXiv.2510.14352}.

\bibitem[PV94]{PV94}
V.~L.~Popov and E.~B.~Vinberg,
Invariant theory,
in \emph{Algebraic Geometry IV: Linear Algebraic Groups, Invariant Theory},
Encyclopaedia of Mathematical Sciences, \textbf{55},
Springer-Verlag, Berlin, 1994, pp.~123--278.
Editors: A.~N.~Parshin and I.~R.~Shafarevich.

\bibitem[Sai71]{Sai71}
K.~Saito,
Quasihomogene isolierte Singularit\"aten von Hyperfl\"achen,
\emph{Invent. Math.} \textbf{14} (1971), 123--142.
DOI: \texttt{10.1007/BF01405360}.

\bibitem[Sai94]{Sai94}
M.~Saito,
On microlocal $b$-function,
\emph{Bull. Soc. Math. France} \textbf{122} (1994), no.~2, 163--184.
DOI: \texttt{10.24033/bsmf.2227}.



\bibitem[Sai93]{Sai93}
M.~Saito,
On $b$-function, spectrum and rational singularity,
\emph{Math. Ann.} \textbf{295} (1993), no.~1, 51--74.













\bibitem[Yon90]{Yon90}
T.~Yonemura,
\emph{Hypersurface simple K3 singularities},
Tohoku Math. J. (2) \textbf{42} (1990), no.~3, 351--380.





\end{thebibliography}
\end{document}